\documentclass[12pt]{article}

\usepackage[english]{babel}
\usepackage{titlesec}
\usepackage{geometry}
\usepackage{enumitem}
\usepackage[utf8x]{inputenc}
\usepackage{amsmath}
\usepackage{amssymb}
\usepackage{amsthm}
\usepackage[mathcal]{euscript}
\usepackage{color}
\usepackage{tikz-cd}
\usepackage{mathabx}
\usepackage{physics}
\usepackage{MnSymbol}
\usepackage{hyperref}

\numberwithin{equation}{section}

\title{Rough Collisions \thanks{Mathematics Subject Classification: 70E18; 37C83; 70L99. Keywords: rigid body, contact dynamics, frictional collisions, stochastic billiards, invariant measure}}
\author{Peter Rudzis \thanks{Email: \url{prudzis@uw.edu}}}
\date{March 9, 2022}

\newtheorem{theorem}{Theorem}[section]
\newtheorem{lemma}[theorem]{Lemma}
\newtheorem{corollary}[theorem]{Corollary}
\newtheorem{proposition}[theorem]{Proposition}
\newtheorem{remark}[theorem]{Remark}
\newtheorem{definition}[theorem]{Definition}

\newtheorem{manualtheoreminner}{Claim}
\newenvironment{claim}[1]{%
  \renewcommand\themanualtheoreminner{#1}%
  \manualtheoreminner
}{\endmanualtheoreminner}

\newenvironment{subproof}[1][\proofname]{%
  \begin{proof}[#1]%
}{%
  \end{proof}%
}

\DeclareMathOperator{\cyl}{cyl}
\DeclareMathOperator{\Int}{Int}
\DeclareMathOperator{\supp}{supp}
\DeclareMathOperator{\Lip}{Lip}
\DeclareMathOperator{\reg}{reg}
\DeclareMathOperator{\roll}{roll}
\DeclareMathOperator{\Img}{Im}

\DeclareMathOperator{\dist}{dist}
\DeclareMathOperator{\fin}{fin}
\DeclareMathOperator{\argmin}{argmin}
\DeclareMathOperator{\CES}{CES}

\begin{document}

\maketitle

\begin{abstract}
A rough collision law describes the limiting contact dynamics of a pair of rough rigid bodies, as the scale of the rough features (asperities) on the surface of each body goes to zero. The class of rough collision laws is quite large and includes random elements. Our main results characterize the rough collision laws for a freely moving rough disk and a fixed rough wall in dimension 2. Any collision law which (i) is symmetric with respect to a certain well-known invariant measure from billiards theory, and (ii) conserves the projection of the phase space velocity onto the ``rolling velocity'' is a rough collision law. We also provide a method for explicitly constructing rough collision laws for a broad range of choices of microstructure on the disk and wall. In our introduction, we review past work in billiards, including characterizations of other rough billiard systems, which our results build upon.
\end{abstract}

\section{Introduction} \label{sec_intro}

\subsection{Motivation and main results}

\subsubsection{Frictional collisions}

Frictional forces between colliding physical bodies arise from a combination of electrical forces and asperities (microscopic rough features) on the surface of each body. Most mathematical models for friction are phenomenological, in the sense that they do not reduce to more fundamental physical principles and typically contain basic parameters (e.g. the coefficient of friction) depending on the physical materials in play, which must be determined through empirical measurement. Models for frictional collisions can lead to paradoxical results, and there is no single model which describes friction well in all scenarios (see the review [\ref{PainleveParadox_review}]). The statistical mechanical point of view has been taken much more rarely, and the relationship between the microscopic surface features on each body and the macroscopic contact dynamics is only vaguely understood.

This monograph is a mathematical work concerned with an idealized statistical model for frictional collisions. We derive dynamics under the following assumptions: (1) the frictional forces arise only from rigid asperities on the surfaces of each body (and not from electrical forces); and (2) the kinetic energy of the colliding bodies is conserved. 

These postulates allow us to frame our objective in the language of mathematical billiards. Consider two rigid bodies whose surfaces are endowed with small geometric features -- bumps, crevices, etc. Associated with the two bodies is a ``collision law'' which governs the dynamics when the two bodies collide. The physical assumptions of our model imply that a collision may be represented by a point particle undergoing specular (mirror) reflection from the boundary of the configuration space. A \textit{rough collision law} will be defined as a limit of a sequence of collision laws as the scale of the asperities on each body goes to zero. The limiting collision law may in general have a random ``noise'' component, and thus an appropriate sense of convergence must be defined to capture the full breadth of possible limiting behavior. Our goal is to describe the kinds of collision laws which may arise from such a limiting procedure.

\subsubsection{Rigid body collisions}

The mathematical literature on rigid body interactions falls into two categories. On the one hand, we find extensive literature on hard sphere models, where the particle-to-particle interactions are simple to describe. On the other hand, the literature on colliding rigid bodies of more general shape is much more restricted in scope, being mainly concerned with foundational issues (well-posedness) and describing the local (in space and time) contact dynamics. 

Problems about interacting rigid bodies become an order of magnitude harder when one passes from spherical to non-spherical bodies. In the latter case, the configuration space can contain complicated singularities, and the dynamical evolution may not be well-defined for a small set of initial conditions, even for smooth bodies. This leads to paradoxes. The authors of [\ref{impactparadox}], for example, construct convex non-spherical rigid bodies which, for certain initial conditions, must either interpenetrate upon collision or violate the classical balance laws of rigid body mechanics. Cox, Feres, and Ward have developed a theory of rigid body collisions from a differential geometric point of view [\ref{diffgeo_rbcollisions}]. To avoid issues with singularities, these authors assume that the difference in the shape operators on the boundaries of the two bodies, expressed in a certain common frame, are non-singular. One can also consider weak solutions to the dynamical equations governing rigid body interactions. Ballard has developed an existence theory along these lines [\ref{Ballard}]. Wilkinson shows that typically such systems are underdetermined in the weak sense [\ref{Wilk_nonunique}]. A rare case in which a well-known hard sphere model has been extended to the non-spherical setting is Saint-Raymond and Wilkinson's study of the Boltzmann equation [\ref{SrW_colinv}]. The challenges described by these authors in their introduction exemplify the general difficulty of working outside of the hard sphere paradigm.

Rough collisions have the potential to provide a kind of mean between well-understood questions in the hard sphere setting, and their corresponding generalizations to rigid bodies of more arbitrary shape. In the rough collisions setting, one can choose the microscopic features to be quite complicated, even fractal-like, while keeping the macroscopic shape of each body relatively simple (e.g. a sphere). In the limit as the scale of the rough features goes to zero, the complicated singularities in the configuration space become invisible, but some information about the rough features is still preserved in the limiting rough collision law.

\subsubsection{Model and main results: informal description} \label{sssec_minformal}

Our main results characterize collisions between a freely moving rough disk and a fixed rough wall. This characterization provides a way to explicitly construct the collision dynamics for various choices of microstructure on the disk and the wall. We give a mathematically rigorous description of our model in \S\ref{ssec_main_results}, stating our main results in \S\ref{sssec_main}. Here we limit ourselves to an informal description of our model and results.

Consider a disk $D$ with unit radius, moving freely in two-dimensional space and colliding with a fixed wall $W$ lying in the lower half-plane $\{(x_1,x_2) \in \mathbb{R}^2 : x_2 \leq 0\}$. The surfaces of the disk and the wall are covered in small asperities. We allow the asperities on the wall to be fairly arbitrary in shape, while requiring the asperities of the disk to be of a quite specific type, namely ``geostationary satellites,'' as illustrated in Figure \ref{disk_plus_wall_fig}. The satellites should be spaced far enough apart that the event that multiple satellites interact with the wall during a single collision is rare. During a collision event, a single satellite may hit the wall multiple times however. The limiting (possibly random) collision dynamics, obtained as the scale of the roughness on $D$ and $W$ goes to zero, are governed by a \textit{rough collision law}, which specifies the post-collision linear and angular velocities of the disk after it leaves the wall. 

The somewhat unrealistic surface structure on $D$ is necessary to avoid some of the difficulties encountered in rigid body mechanics, described above. If the satellites are too close together, then the boundary of the configuration space will be too singular to derive the kinds of estimates needed to prove our main results. The model is nonetheless ``universal,'' in a sense to be described shortly.

\begin{figure}
    \centering
    \includegraphics[width = \linewidth]{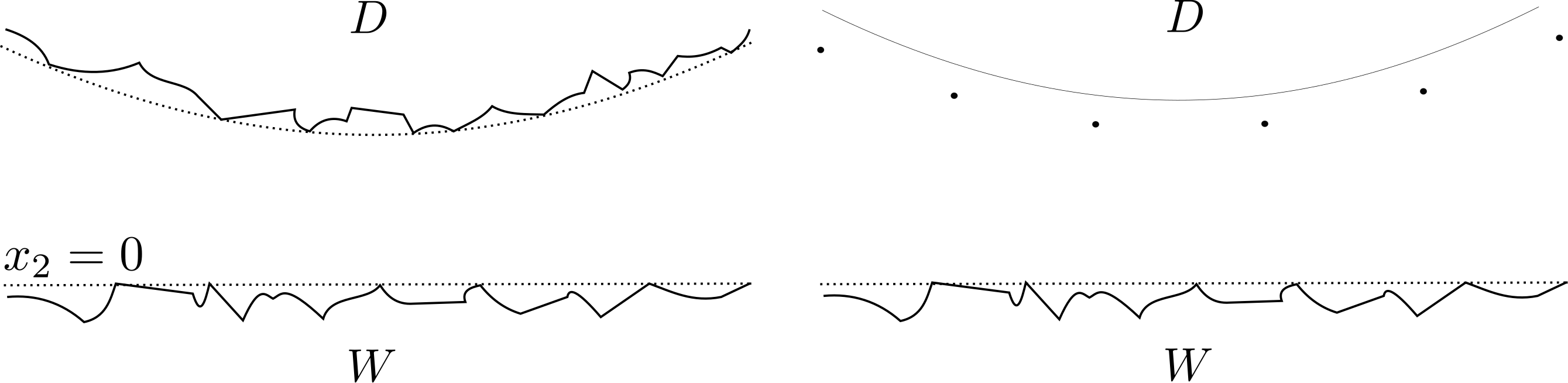}
    \caption{The collision dynamics are too complicated to describe when we put arbitrarily shaped asperities on both $D$ and $W$ (left). The problem becomes more tractable if we assume the asperities on $D$ are well-spaced ``geostationary satellites'' (right).}
    \label{disk_plus_wall_fig}
\end{figure}

We represent the state of the system  by a sextuple $(x_1,x_2,\alpha, v_1, v_2, \omega)$, where $(x_1,x_2)$ is the center of mass of $D$ in $\mathbb{R}^2$ and $\alpha$ its angular orientation, and $(v_1,v_2)$ is the linear velocity of the disk and $\omega$ its angular velocity. We assume in our analysis that the mass density of the disk is rotationally symmetric, and that the kinetic energy of the system is conserved.

There are three properties which we expect the disk and wall system with rough collision dynamics to satisfy.

\begin{enumerate}[label = (\Roman*)]
    \item Liouville measure on the phase space $\dd x_1 \dd x_2 \dd\alpha \dd v_1 \dd v_2 \dd\omega$ should be preserved.
    
    \item The collision dynamics should be ``time-reversible,'' in the sense that the evolution of the system will look the same from a statistical point of view, whether time is run forward or backward.
    
    \item The quantity 
    \begin{equation} \label{eq1.1'}
        - m v_1 + J \omega,
    \end{equation}
    where $m$ is the mass of the disk and $J$ is the moment of inertia of the disk about its center of mass, should be conserved.
\end{enumerate}

The basis for properties (I) and (II) comes from billiards theory. It is well-known that the dynamics of classical billiard systems preserve Liouville measure and are time-reversible. Consequently, the rough collision dynamics, obtained in the weak limit, should also preserve Liouville measure and be time-reversible.

The intuition behind (III) is that the quantity (\ref{eq1.1'}) is the projection (with respect to an inner product coming from kinetic energy) of the phase space velocity onto the ``rolling velocity'' $\chi = (-1,0,1)$. If the disk comes into contact with the wall with velocity $(v_1,v_2,\omega) \approx (-1,0,1)$, then the disk will ``roll'' along the wall. The relative velocity of the wall and the point of contact on the disk will be negligible. Consequently the impact will be negligible, and the disk will continue rolling indefinitely with approximately the same velocity as before. In other words, translation in the direction $\chi$ should be a ``symmetry'' of the system.

Modulo some technical assumptions, our main results (Theorems \ref{thm_pure_scaling}, \ref{thm_dichotomy}, and \ref{thm_classification}) say that a rough collision law not only satisfies properties (I)-(III), but these properties characterize the class of rough collision laws for the disk and wall system described above. That is, any collision law which produces dynamics satisfying (I)-(III) may be approximated by a deterministic collision law obtained by equipping $W$ with small, appropriately shaped asperities.

Note that the heuristic justification for properties (I)-(III) does not depend on the special surface structure imposed on $D$. Thus the range of dynamics manifested in our model is much broader than the setup suggests.

Contained in our results is a way to construct rough collision laws. A rough collision law is described by a Markov kernel $\mathbb{K}(y_1, y_3, \theta, \psi; \dd y_1' \dd y_3' \dd\theta' \dd\psi')$, where $(y_1,y_3)$ are a certain choice of coordinates on the $x_1\alpha$-plane in the configuration space, and $(\theta, \psi)$ are spherical coordinates on the velocity space. We will see that rough collision laws always take the form 
\begin{equation}
    \mathbb{K}(y_1, y_3, \theta, \psi; \dd y_1' \dd y_3' \dd\theta' \dd\psi') = \delta_{(y_1,y_3)}(\dd y_1' \dd y_3') \widetilde{\mathbb{P}}(\theta, \dd\theta') \delta_{\pi - \psi}(\dd\psi').
\end{equation}
The single non-trivial factor $\widetilde{\mathbb{P}}$ describes the way a point particle reflects from a rough wall $\widetilde{W}$, obtained by foreshortening the original wall $W$ in one direction by a factor of $(1 + m/J)^{1/2}$. In many cases, the Markov kernel $\widetilde{\mathbb{P}}$ can be computed explicitly.

The proof of our main results depends on a characterization of rough reflection laws, discovered independently by Plakhov and by Angel, Burdzy, and Sheffield (see \S\ref{sssec_rfrefchar}. and \S\ref{sssec_Plakhov}). The main novelty in this work  -- as well as the main technical challenge -- is to prove rigorously that the quantity (\ref{eq1.1'}) is conserved.

The results of this book relate to the work of R. Feres and collaborators on two separate fronts -- first, in relation to \textit{rough reflections} (see \S\ref{sssec_Feres}), and second, in relation to \textit{no-slip collisions}, a type of idealized, deterministic frictional collision (see \S \ref{sssec_sm_ns}). Our results imply that no-slip collisions belong to the class of rough collisions; thus the dynamics of a freely moving disk and fixed wall undergoing no-slip collisions can be approximated by a pair of bodies undergoing classical non-frictional collisions.

\subsubsection{Organization of book}

Rough collisions belong to a subbranch of stochastic billiards which we refer to as \textit{rough billiards}. An introduction to past work in this subject area may be found in \S\ref{ssec_rough_bill}. A more rigorous description of our model and main results is given in \S\ref{ssec_main_results}.

The purpose of \S\ref{sec_examples} is to apply our main results to construct a number of examples of rough reflection laws and rough collision laws, for various choices of microstructure on the wall $W$.

A collision between two rigid bodies may be represented by a point mass reflecting specularly from the boundary of the configuration space. This fact allows us to apply techniques from billiards theory to analyze our model. In \S\ref{sec_specularreflection} we derive from physical principles in rigid body mechanics the specular reflection law for our model.

\S\ref{sec_config} is devoted to preliminaries for the proof of our main results. First comes a careful description of the elementary properties of the configuration space of the disk and wall system. Subsequent sections provide a rigorous definition of the collision law associated with the system, and introduce two auxiliary collision laws which play a role in our proofs.

Our main results are proved in \S\ref{sec_main_results}. For a high-level summary of our arguments, see also \S\ref{sssec_prfoutline}.

\S\ref{sec_genref} is concerned with the ``abstract theory'' of rough billiards. Both the rough reflections described in \S\ref{ssec_rough_bill} and the rough collision laws introduced in \S\ref{ssec_main_results} are special cases of the rough reflections defined in \S\ref{sec_genref}. We will refer to results proved in this section a number of times throughout the book.

\subsubsection{Acknowledgments}

I would like to sincerely thank my advisor Krzysztof Burdzy, who has been an invaluable source of help and insight from start to finish. This book owes much to his patience and unabating encouragement. I am also grateful to David Clancy and Robin Graham for their helpful comments on the draft of this work.

\subsection{Rough billiards} \label{ssec_rough_bill}

The following section serves two purposes: first, to introduce results concerning rough billiards in the upper half-plane, upon which the main results of this work depend; and second, to provide a general survey of past work in rough billiards. Consequently, we are careful about giving technically accurate statements in \S\S\ref{sssec_uphalf}-\ref{sssec_rfrefchar}, whereas the style of \S\S\ref{sssec_Feres}-\ref{sssec_Plakhov} is a bit more informal. 

\subsubsection{Notation and other conventions} \label{sssec_noteconv}

We will use the following notation throughout the book. If $X$ is a topological space and $Y \subset X$, then $\Int Y$ and $\overline Y$ denote, respectively, the topological interior and closure of $Y$ relative to $X$. The notation $\partial Y := \overline{Y} \cap \overline{Y^c}$ always refers to the topological boundary of $Y$ relative to $X$, and should not be confused with the boundary of a manifold. Context will be sufficient to distinguish the ambient space $X$ in most cases (usually $X = \mathbb{R}^d$ for some $d \geq 1$).

We let $C_c(X)$ denote the space of compactly supported functions on $X$. If $X$ has the structure of a differentiable manifold and $k \geq 0$, then $C^k(X)$ denotes the space of $k$-times continuously differentiable functions on $X$, and $C^\infty(X)$ denotes the space of infinitely differentiable functions on $X$. We let $C_c^k(X) = C_c(X) \cap C^k(X)$ and $C_c^\infty(X) = C_c(X) \cap C^\infty(X)$.

If $(X,\mu)$ is a measure space and $1 \leq p \leq \infty$, then $L^p_\mu(X)$ denotes the space of $p$-integrable functions on $(X,\mu)$. We denote the $L^p$-norm on this space by $||\cdot||_{L^p_\mu(X)}$. We suppress $X$ and $\mu$ from our notation if they are clear from the context.

If $f, g : \mathbb{R} \to \mathbb{R}_+$ are real functions, we write $f = O(g)$ if $\limsup_{t \to 0} f(t)/g(t) < \infty$, and we write $f = o(g)$ if $\lim_{t \to 0} f(t)/g(t) = 0$.

If $B$ is a subset of $\mathbb{R}^d$, and $v$ is a vector in $\mathbb{R}^d$, then we denote the translate of $B$ by $v$ as follows: $B + v := \{p \in \mathbb{R}^d : p - v \in B\}$.

Here and throughout the book, the term \textit{billiard} refers to any dynamical system in which a point particle moves linearly in the complement of a closed subset $W \subset \mathbb{R}^d$ and reflects from the boundary of $W$ (specularly or according to some other rule). The subset $W$ is called the \textit{wall}, and is usually assumed to have a piecewise smooth boundary (where the meaning of ``piecewise smooth'' is made precise in more specific contexts). The complement of the interior of $W$ is referred to as either the \textit{billiard table} or \textit{billiard domain}. The piecewise linear curve traced out by the point particle for some choice of initial conditions is called the \textit{billiard trajectory}. For more background on mathematical billiards, we refer the reader to [\ref{Chernov&Markarian}] and [\ref{Tabachnikov}].

\subsubsection{Rough billiards in the upper half-plane} \label{sssec_uphalf}

We begin by describing the construction of rough billiards in the upper half-plane. Our approach is essentially the same as that of [\ref{ABS_refl}], and similar to that of [\ref{feres_rw}].

The billiard table we initially consider is the complement of a closed set $W \subset \mathbb{R}^2$ satisfying the following assumptions:
\begin{enumerate}[label = A\arabic*.]
    \item $W$ is the closure of its interior in $\mathbb{R}^2$.
    \item $\mathbb{R}^2 \smallsetminus W$ is path-connected.
    \item The following inclusions hold: $\mathbb{R} \times (-\infty,-1] \subset W \subset \mathbb{R} \times (-\infty,0]$; thus $\partial W \subset \mathbb{R} \times [-1,0]$.
    \item $\partial W = \bigcup_{i = 1}^\infty \Gamma_i$, where $\{\Gamma_i\}_{i \geq 1}$ is some collection of compact curve segments satisfying the following conditions: 
    \begin{enumerate}[label = (\roman*)]
        \item The collection $\{\Gamma_i\}_{i \geq 1}$ is locally finite, in the sense that any bounded set intersects only finitely many of the curve segments $\Gamma_i$; 
        \item each $\Gamma_i$ is the image of an injective $C^2$ map $\gamma_i : [0,1] \to \mathbb{R}^2$ with nonvanishing left and right-hand derivatives (where $C^2$ means that $\gamma_i$ has a $C^2$ extension to an open interval containing $[0,1]$);
        \item the curves $\Gamma_i$ are allowed to intersect each other only at their endpoints; and
        \item for each $i$, if $\Gamma_i$ intersects the line $\{(x_1,x_2) : x_2 = 0\}$ at a point other than one of its two endpoints, then $\Gamma_i \subset \{(x_1,x_2) : x_2 = 0\}$.
    \end{enumerate}
\end{enumerate}
In condition A4, the decomposition of $\partial W$ into curve segments $\Gamma_i$ is not unique. We shall refer more generally to a curve $\Gamma$ with a decomposition $\Gamma = \bigcup_{i \geq 1} \Gamma_i$ such that conditions (i)-(iii) are satisfied as a \textit{piecewise $C^2$ curve}. Condition (iv) lets us avoid certain pathological situations when defining the macro-reflection law below (we would like to avoid the situation where some $\Gamma_i$ intersects the line $x_2 = 0$ in a ``fat Cantor set'' for example). In most typical situations, it will be easy to choose a decomposition of $\partial W$ such that (iv) holds.

Let $\mathbb{R}^2_{\pm} = \{(x_1,x_2) \in \mathbb{R}^2 : \pm x_2 > 0\}$. Consider a point particle moving freely in $W^c$ and reflecting \textit{specularly} (angle of incidence equals angle of reflection) from $\partial W$. When the point particle leaves the upper half-plane $\mathbb{R}^2_+$, the particle may hit $\partial W$ multiple times before returning to the upper half-plane, as illustrated in Figure \ref{reflection1_fig}. The limiting behavior of this interaction as $\epsilon \to 0$ will be described by a rough reflection law.

The \textit{kinetic energy} of a point particle with velocity $v = (v_1,v_2)$ is the quantity $\frac{1}{2}v_1^2 + \frac{1}{2}v_2^2$. We assume that kinetic energy is conserved for all time. Without loss generality we take the velocity of the point particle to be restricted to the Euclidean unit circle $\mathbb{S}^1 \subset \mathbb{R}^2$ for all time. We identify points in $\mathbb{S}^1$ with angles $\theta$ in the interval $[0,2\pi)$, and we let $\mathbb{S}^1_+ = \mathbb{S}^1 \cap \mathbb{R}^2_+ = (0,\pi)$ and $\mathbb{S}^1_- = \mathbb{S}^1 \cap \mathbb{R}^2_- = (\pi,2\pi)$.

The \textit{macro-reflection law associated with $W$} is the map $P^W : \mathbb{R} \times \mathbb{S}^1_+ \to \mathbb{R} \times \mathbb{S}^1_+$ defined as follows. As shown in Figure \ref{reflection1_fig}, if initially the point particle starts on the $x_1$-axis with velocity pointing into the lower half-plane $\mathbb{R}^2_-$, its state may be represented by a pair $(x,\theta) \in \mathbb{R} \times \mathbb{S}^1_+$, where $x$ is first coordinate of the particle on the $x_1$-axis, and $\pi + \theta \in \mathbb{S}^1_-$ is its velocity. After reflecting from the boundary $\partial W$ a certain number of times, the particle returns to the $x_1$-axis at a position $x' \in \mathbb{R}$ and with velocity $\theta' \in \mathbb{S}^1_+$. We define
\begin{equation}
    P^W(x,\theta) = (x',\theta').
\end{equation}
The term ``macro-reflection law'' should be understood in contradistinction to the specular reflection law which describes the ``micro'' reflection of the trajectory from $\partial W$ at a single instant in time.

\begin{figure}
    \centering
    \includegraphics[width = 0.7\linewidth]{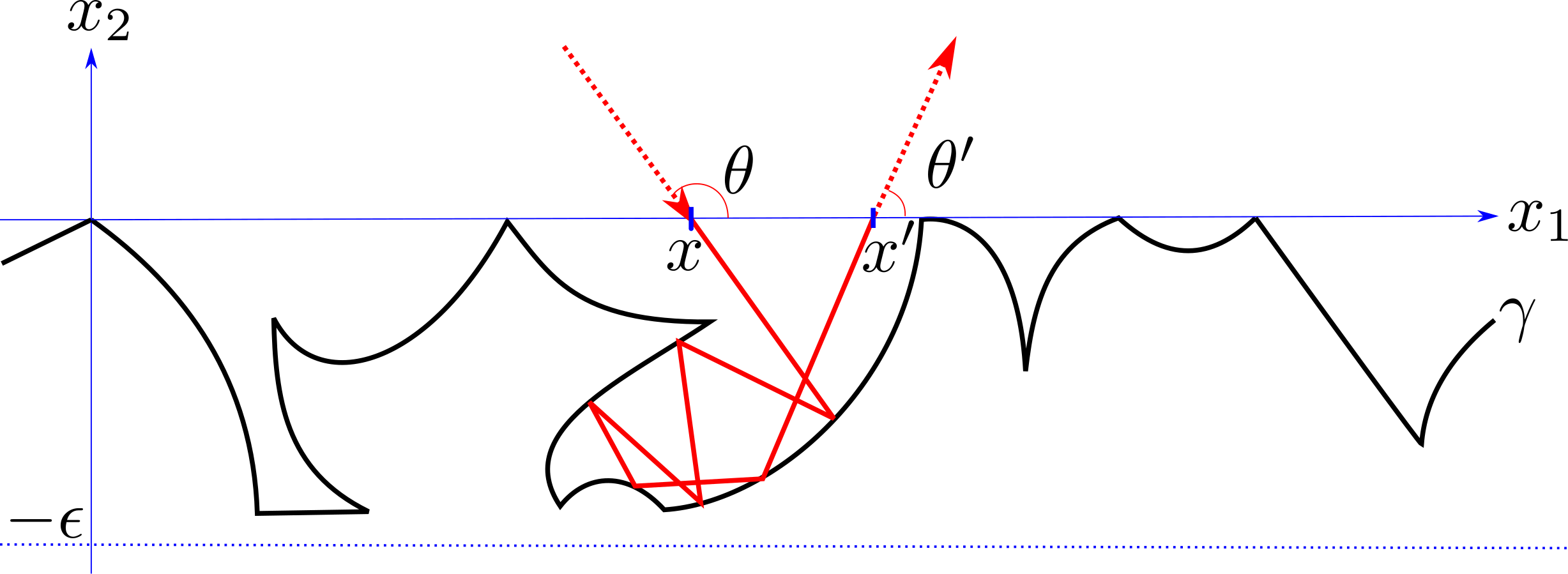}
    \caption{The macro-reflection law maps the initial state $(x,\theta)$ to the state $(x',\theta')$.}
    \label{reflection1_fig}
\end{figure}

The map $P^W$ may fail to be defined at pairs $(x,\theta)$ such that the billiard trajectory hits the boundary $\partial W$ tangentially or at a ``corner'' or never returns to the $x_1$-axis. Thus we impose the following additional assumption.
\begin{enumerate}
    \item[A5.] For almost every $(x,\theta) \in \mathbb{R} \times \mathbb{S}^1_+$, the billiard trajectory starting from $(x,\theta)$ is well-defined for all time, and returns to the $x_1$-axis after only finitely many collisions with $W$.
\end{enumerate}
This condition is implied by conditions A1-A4 together with either one of the following conditions:
\begin{enumerate}
    \item[A5a.] There exists a countable collection $\{R_j\}_{j \geq 1}$ of disjoint bounded open subsets of $\mathbb{R}^2$ such that $\mathbb{R}^2_- \smallsetminus W = \bigcup_{j = 1}^\infty R_j$.
    \item[A5b.] The wall $W$ is $\epsilon$-periodic in the $x_1$-coordinate, in the sense that
    \begin{equation} \label{eq1.16}
       W = \{(x_1,x_2) : (x_1 - \epsilon, x_2) \in W\} = W + \epsilon e_1, \quad \text{ where }e_1 = (1,0).
    \end{equation}
\end{enumerate}
The first condition means that the point particle will always become trapped in some bounded region when it interacts with the wall. Both conditions allow us to apply the Poincar\'{e} Recurrence Theorem to obtain A5. For more details, we refer the reader to \S\ref{sec_genref}, where we define macro-reflection laws in a more general setting. See specifically the discussion of upper half-space billiards in \S\ref{sssec_special}.

Define a measure $\Lambda^1$ on $\mathbb{R} \times \mathbb{S}^1_+$ by 
\begin{equation} \label{eq1.5}
    \Lambda^1(\dd x \dd\theta) = \sin\theta \dd x \dd\theta.
\end{equation}
The most important elementary properties of $P^W$ are summed up in the following proposition.

\begin{proposition} \label{prop_det_ref_properties}
(i) The map $P^W$ is involutive in the sense that $P^W \circ P^W(x,\theta) = (x,\theta)$ whenever the left-hand side is defined.

(ii) The map $P^W$ preserves the measure $\Lambda^1$ in the sense that, for any measurable set $A \subset \mathbb{R} \times \mathbb{S}^1_+$, 
\begin{equation}
    \Lambda^1((P^W)^{-1}(A)) = \Lambda^1(A).
\end{equation}
\end{proposition}
To understand (i), note that specular reflection is involutive; so ``running the evolution backward'' from $(x',\theta') := P^W(x,\theta)$, the trajectory is guaranteed to return to the $x_1$-axis in state $(x,\theta)$. Part (ii) is a corollary of a well-known theorem in billiards theory (see Lemma \ref{lem_bilinvar}). If we accept that the continuous billiard evolution should preserve Liouville measure $\dd x_1 \dd x_2 \dd\theta$ on the phase space, then Figure \ref{invariant_fig} should make property (ii) quite believable. A more general version of Proposition \ref{prop_det_ref_properties} is proved in \S\ref{sec_genref} (see Proposition \ref{prop_reflproperties}).

The macro-reflection law $P^W$ is naturally associated with a deterministic Markov kernel on $\mathbb{R} \times \mathbb{S}^1_+$, defined by 
\begin{equation}
    \mathbb{P}^W(x,\theta; \dd x' \dd\theta') = \delta_{P^W(x,\theta)}(\dd x' \dd\theta').
\end{equation}

\begin{figure}
    \centering
    \includegraphics[width = 0.6\linewidth]{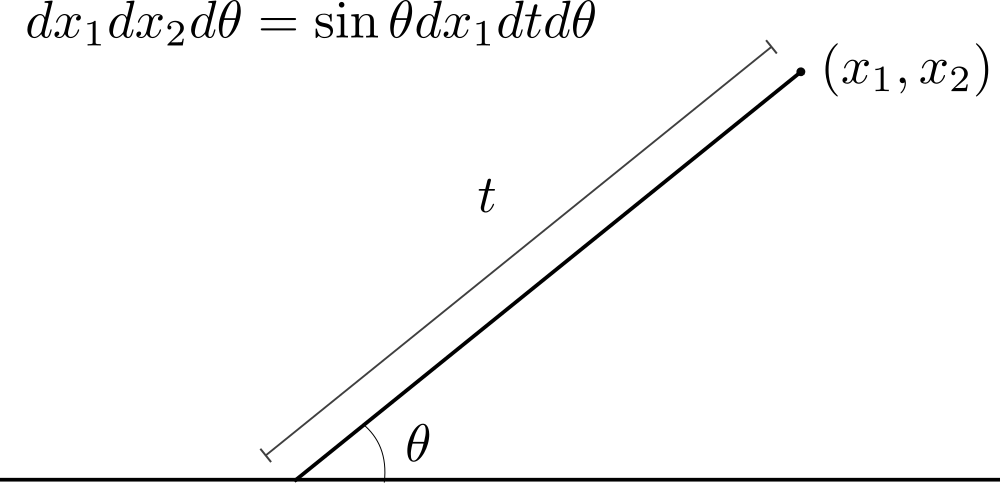}
    \caption{If $(x_1,x_2,\theta)$ is the state of the point particle, then $x_2 = t \sin\theta$, where $t$ is the time to hit the $x_1$-axis. Liouville measure in these new coordinates is $\sin\theta \dd x_1 \dd t \dd\theta$.}
    \label{invariant_fig}
\end{figure}

\noindent In what follows, by ``wall'' we mean a subset $W \subset \mathbb{R}^2$ satisfying conditions A1-A5.

\begin{definition} \normalfont \label{roughref_def}
We call a Markov kernel $\mathbb{P}$ on $\mathbb{R} \times \mathbb{S}^1_+$ a \textit{rough reflection law} in the upper half-plane $\mathbb{R}^2_+$ if there exists a sequence of positive numbers $\epsilon_n \to 0$ and a sequence of walls $W_n$ such that $\partial W_n \subset \{(x_1,x_2) : -\epsilon_n \leq x_2 \leq 0\}$ and
\begin{equation} \label{eq1.6}
    \mathbb{P}^{W_n}(x,\theta; \dd x' \dd\theta') \Lambda^1(\dd x \dd\theta) \to  \mathbb{P}(x,\theta; \dd x' \dd\theta') \Lambda^1(\dd x \dd\theta)
\end{equation}
weakly in the space of measures on $\mathbb{R} \times \mathbb{S}^2_+$.
\end{definition}

The two properties of macro-reflection laws $P^W$ described in Proposition \ref{prop_det_ref_properties} carry over to rough reflection laws in the following sense.

\begin{proposition} \label{prop_dualref}
Let $\mathbb{P}(x,\theta; \dd x' \dd\theta')$ be a rough reflection law. The Markov kernel $\mathbb{P}$ is symmetric with respect to the measure $\Lambda^1$, in the sense that, for any $f \in C_c((\mathbb{R} \times \mathbb{S}^1_+)^2)$, 
\begin{equation} \label{eq1.14}
\begin{split}
    & \int_{(\mathbb{R} \times \mathbb{S}^1_+)^2} f(x,\theta, x', \theta')\mathbb{P}(x,\theta; \dd x' \dd\theta') \Lambda^1(\dd x \dd\theta) \\
    & \quad\quad = \int_{(\mathbb{R} \times \mathbb{S}^1_+)^2} f(x',\theta', x, \theta)\mathbb{P}(x,\theta; \dd x' \dd\theta')\Lambda^1(\dd x \dd\theta).
\end{split}
\end{equation}
\end{proposition}

Symmetry generalizes time-reversibility in the sense that the left-hand side of (\ref{eq1.14}) is transformed into the right-hand side by interchanging the pre- and post-reflection variables $(x,\theta)$ and $(x',\theta')$. Symmetry also implies that $\mathbb{P}$ preserves $\Lambda^1$. Indeed, by letting $f(x,\theta,x',\theta') \uparrow g(x',\theta') \in C_c(\mathbb{R} \times \mathbb{S}^2_+)$ in (\ref{eq1.14}), we obtain 
\begin{equation}
    \int_{(\mathbb{R} \times \mathbb{S}^1_+)^2} g(x',\theta')\mathbb{P}(x,\theta; \dd x' \dd\theta')\Lambda^1(\dd x \dd\theta) = \int_{\mathbb{R} \times \mathbb{S}^1_+} g(x,\theta)\Lambda^1(\dd x \dd\theta).
\end{equation}
Proposition \ref{prop_dualref} is a special case of Proposition \ref{prop_gensymm}, proved in \S\ref{sec_genref}.

\begin{remark} \label{rem_replacelambda1} \normalfont
In Definition \ref{roughref_def}, it is equivalent to replace $\Lambda^1$ with any measure which is mutually absolutely continuous with respect to Lebesgue measure on $\mathbb{R} \times \mathbb{S}^1_+$. The measure $\Lambda^1$ happens to be a convenient choice, due to Proposition \ref{prop_dualref}.
\end{remark}

\begin{remark} \label{rem_verify1} \normalfont
The convergence (\ref{eq1.6}) means that, for any $h(x,\theta,x',\theta') \in C_c((\mathbb{R} \times \mathbb{S}^1_+)^2)$, 
\begin{equation} \label{eq1.11}
\begin{split}
    &\lim_{n \to \infty}\int_{(\mathbb{R} \times \mathbb{S}^1_+)^2} h(x,\theta,x',\theta') \mathbb{P}^{W_n}(y,w; \dd y' \dd w') \Lambda^2(\dd x \dd\theta) \\
    & \quad\quad = \int_{(\mathbb{R} \times \mathbb{S}^1_+)^2} h(x,\theta,x',\theta') \mathbb{P}(y,w; \dd y' \dd w') \Lambda^2(\dd x \dd\theta).
\end{split}
\end{equation}
Since the tensor product $C_c^{\infty}(\mathbb{R} \times \mathbb{S}^1_+) \otimes C_c^{\infty}(\mathbb{R} \times \mathbb{S}^1_+)$ is dense in $C_c((\mathbb{R} \times \mathbb{S}^1_+)^2)$, it is sufficient to verify (\ref{eq1.11}) for functions for form $h(y,w,y',w') = f(y,w)g(y',w')$, where $f,g \in C_c^{\infty}(\mathbb{R} \times \mathbb{S}^1_+)$.
\end{remark}

\begin{remark} \label{rem_convsense0} \normalfont
There is a sense in which $\mathbb{P}^{W_n}$ converges to $\mathbb{P}$ as a limit with respect to a pseudometric topology on the space of Markov kernels on $\mathbb{R} \times \mathbb{S}^1_+$. This topology is described in \S \ref{sssec_pseudo}. 

Some care must be taken when working with this sense of convergence, because limits may not be unique. With respect to the pseudometric, the distance between two Markov kernels $\mathbb{P}$ and $\mathbb{P}'$ is zero if and only if $\mathbb{P}(\cdot; \dd x \dd\theta)$ and $\mathbb{P}'(\cdot; \dd x \dd\theta)$ agree on a $\Lambda^1$-full measure subset of $\mathbb{R} \times \mathbb{S}^1_+$. If we identify Markov kernels which agree on a $\Lambda^1$-full measure subset of $\mathbb{R} \times \mathbb{S}^1_+$, then limits will be unique and the pseudometric will be a metric.
\end{remark}

From this point on, we will write $\lim_{i \to \infty} \mathbb{P}^{W_n} = \mathbb{P}$ to indicate that (\ref{eq1.6}) holds.

\subsubsection{Simple example: the rectangular teeth microstructure}

A simple example of a rough reflection law may be obtained by considering a sequence of walls $W_n$ with periodic boundary structure consisting of ``rectangular teeth.'' That is, we first define real functions
\begin{equation}
    t_n(x) = \begin{cases}
    0 & \text{ if } 2k \epsilon_n \leq x \leq (2k + 1) \epsilon_n, \\
    -r\epsilon_n & \text{ if } (2k + 1) \epsilon_n < x < (2k+2) \epsilon_n  
    \end{cases} \quad \text{ for } k \in \mathbb{Z}.
\end{equation}
See Figure \ref{square_teeth_fig} in \S\ref{sec_examples}. The quantity $r > 0$ is a fixed parameter representing the ratio of the height of the teeth to the width. We define 
\begin{equation}
    W_n = \{(x_1,x_2) : x_2 \leq t_n(x_1)\}.
\end{equation}

If $v^- = (v_1^-, v_2^-)$ is the incoming velocity of a point particle, then after hitting $W_n$ a certain number of times, the particle will return to the upper half-plane with velocity either $v^+ = (v_1^-, - v_2^-)$ or $v^+ = (-v_1^-, -v_2^-) = -v^-$. The first of these velocities corresponds to a specular reflection, while the second corresponds to a \textit{retroreflection} -- i.e. a reflection in which the outgoing trajectory of the point particle goes in the opposite direction as the incoming trajectory. Thus, as $\epsilon_n \to 0$, we expect the limiting rough reflection law to randomly select between specular reflection and  retroreflection. 

In \S\ref{sec_examples}, we derive explicit formulas for rough reflections from several different types of microstructures, including the rectangular teeth microstructure described above. 

\subsubsection{Periodic case} \label{sssec_periodic}

We now comment on the special case where the wall $W$ satisfies the periodicity condition A5b. In this setting, it is useful to abstract the \textit{shape} of the wall from the \textit{scale}. In the limit, as the scale of the wall goes to zero, we expect at least some information about the shape of the wall to be preserved, and we would like to be able to talk about the shape of the wall independently of the scale. 

To accomplish this, we observe that a periodic wall is determined uniquely by a pair $(\Sigma,\epsilon)$, where $\Sigma$ is a subset of $\mathbb{S}^1 \times (-\infty,0]$ with $\partial \Sigma \subset \mathbb{S}^1 \times [-1,0]$, and $\epsilon > 0$. In particular, $W$ is the unique wall satisfying periodicity condition (\ref{eq1.16}) such that the image of $W$ under the covering map 
\begin{equation}
    (x_1,x_2) \mapsto (e^{2\pi i x_1/\epsilon}, \epsilon^{-1}x_2) : \mathbb{R} \times [-\epsilon,0] \to \mathbb{S}^1 \times [-1,0]
\end{equation}
is $\Sigma$. We denote the wall so determined by $W(\Sigma,\epsilon)$. Drawing on terminology from Feres [\ref{feres_rw}], we refer to $\Sigma$ as the \textit{cell} and we refer to $\epsilon > 0$ as the \textit{roughness scale}.

If the wall $W = W(\Sigma, \epsilon)$ arises from a cell and roughness scale, as described above, then we will denote the corresponding macro-reflection law by $\mathbb{P}^{\Sigma,\epsilon}$.

To illustrate the use of this concept, consider rough reflections from a ``fractal microstructure.'' In general, fractals do not have well-defined normal vectors at most boundary points, so we cannot define specular reflection on such a surface directly. But sense can be made of this in the case of rough reflections. For example, we might take $\Sigma_n$ to be a sequence of sets generating a fractal whose boundary is a Koch curve -- see Figure \ref{Koch_reflections_fig}. A rough reflection ``from a Koch curve microstructure'' can then be defined as a rough reflection obtained as the limit (in the sense of (\ref{eq1.6})) of a sequence of deterministic Markov kernels $\mathbb{P}^{\Sigma_n,\epsilon_n}$, where $\epsilon_n$ is a sequence of positive numbers converging to zero. The paper [\ref{LapNiem}] carries out numerical experiments for a related model.

\begin{figure}
    \centering
    \includegraphics[width = 0.8\linewidth]{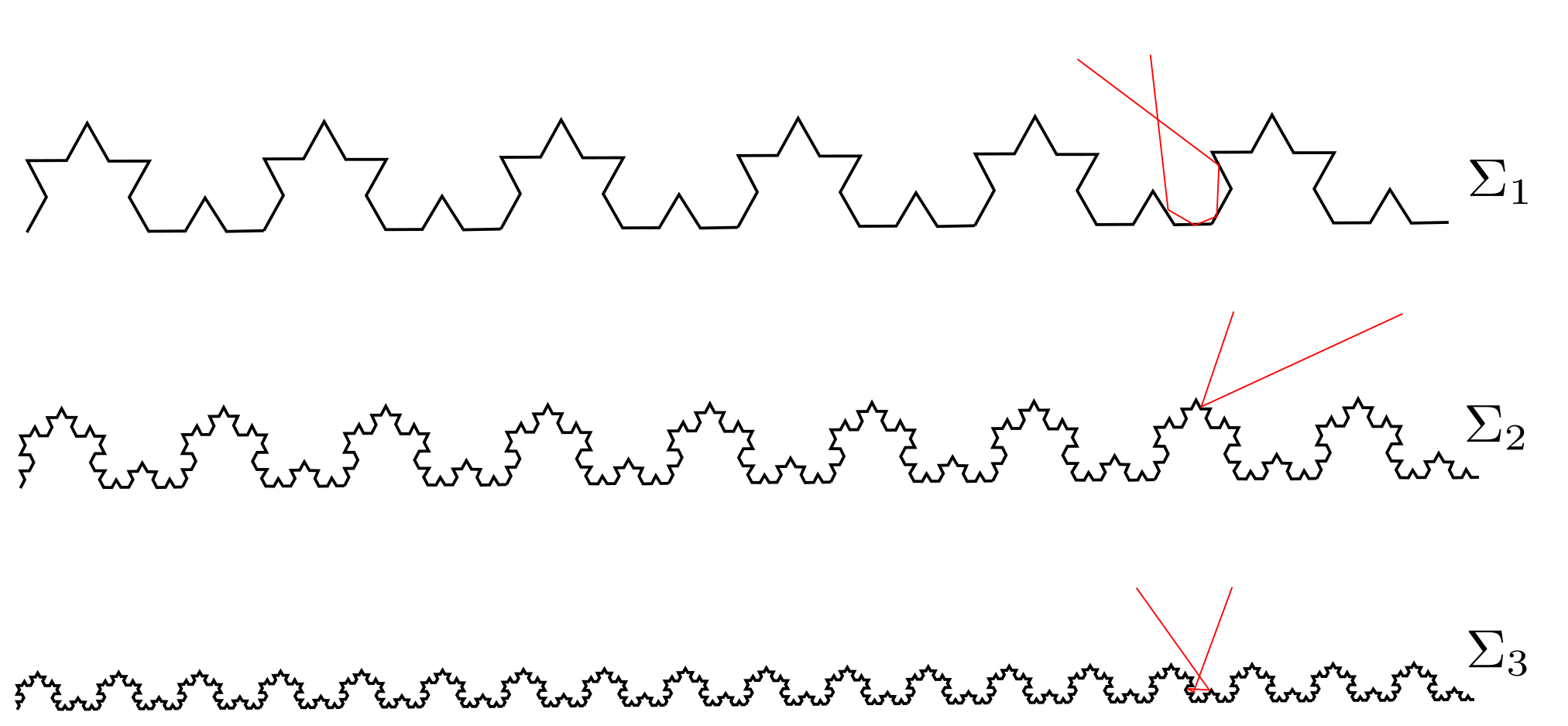}
    \caption{Reflections from a Koch curve.}
    \label{Koch_reflections_fig}
\end{figure}

Conditions on $\Sigma$ which are sufficient to guarantee that the wall $W = W(\Sigma,\epsilon)$ satisfies conditions A1-A5 are the following:
\begin{enumerate}[label = B\arabic*.]
    \item $\Sigma$ is the closure of its interior in $\mathbb{S}^1 \times \mathbb{R}$.
    \item $\mathbb{S}^1 \times \mathbb{R} \smallsetminus \Sigma$ is connected.
    \item The following inclusions hold: $\mathbb{S}^1 \times (-\infty,-1] \subset \Sigma \subset \mathbb{S}^1 \times (-\infty,0]$; thus $\partial \Sigma \subset \mathbb{S}^1 \times [-1,0]$.
    \item There exists a finite collection of compact $C^2$ curve segments $\{\widetilde{\Gamma}_i\}_{i = 1}^m$ such that $\partial \Sigma = \bigcup_{i = 1}^m \widetilde{\Gamma}_i$. The curve segments $\widetilde{\Gamma}_i$ satisfy conditions A4(i)-(iv), with the obvious modifications.
\end{enumerate}
We of course get the periodicity condition A5b for free. 

One condition which is sufficient to guarantee that A5a holds is:
\begin{enumerate}
    \item [B5.] There exists a non-trivial loop $\gamma \subset \mathbb{S}^1 \times \mathbb{R}$, starting and ending at the point $(1,0)$, which lies entirely in $\Sigma$.
\end{enumerate}
Here ``non-trivial'' means that $\gamma$ cannot be contracted in $\mathbb{S}^1 \times \mathbb{R}$ to a point. This condition implies the points $(\epsilon k,0) \in \mathbb{R}^2$, where $k \in \mathbb{Z}$, all lie in a single connected component of $W = W(\Sigma,\epsilon)$. Condition B5 will always be satisfied if $\Sigma$ is connected, satisfies conditions B1-B4, and contains the point $(1,0) \in \mathbb{S}^1 \times \mathbb{R}$.

The main reason we would want to impose the condition B5 is the following. Under generic circumstances, we expect a rough collision law obtained from a sequence of periodic walls to take the form
\begin{equation} \label{eq1.17'}
    \mathbb{P}(x,\theta; \dd x' \dd\theta') = \delta_x(\dd x') \widetilde{\mathbb{P}}(\theta, \dd\theta').
\end{equation}
The intuition behind this is that the point particle should leave a rough wall at approximately the same spatial position that it hits. Moreover, if the microstructure on the wall is periodic, then the distribution of the angle of reflection $\theta'$ should only depend on the angle of incidence $\theta$, and not on the position where the particle hits the wall. Lemma \ref{lem_construct} from \S\ref{sec_examples} implies that if $\mathbb{P} = \lim_{n \to \infty} \mathbb{P}^{\Sigma_n,\epsilon_n}$, where the cells $\Sigma_n$ satisfy B1-B5, then $\mathbb{P}$ takes the form (\ref{eq1.17'}). The assumption B5 guarantees that the billiard trajectory will get trapped in small ``hollow'' within a single period of the wall $W$, and consequently the distance between the points where the trajectory first hits and returns to the $x_1$-axis will be of order $\epsilon$ apart. 

Although we do not know of a specific counter-example, it is most likely not possible to obtain (\ref{eq1.17'}) if we just assume conditions B1-B4. One can imagine constructing a periodic wall with a large ``asteroid field'' of connected components, such that the point particle will be forced to travel a great distance, reflecting from the various components many times, before eventually leaving the wall. If each wall $W_n$ is constructed in this way, the limiting reflection law might not satisfy (\ref{eq1.17'}).

\begin{remark} \normalfont
We can, however, weaken condition B5 as follows, and (\ref{eq1.17'}) will still hold:
\begin{enumerate}
\item[B5'.] There exists $a \in [-1,0]$ and $u \in \mathbb{S}^1$ such that 
\begin{equation}
    \mathbb{S}^1 \times (-\infty,-1] \subset \Sigma \subset \mathbb{S}^1 \times (-\infty,a],
\end{equation}
and there exists a nontrivial loop $\gamma \subset \mathbb{S}^1 \times \mathbb{R}$ starting and ending at the point $(u,a)$, which lies entirely in $\Sigma$.
\end{enumerate}
Under this assumption, the proof of Lemma \ref{lem_construct} in \S\ref{sec_examples} goes through with only minor modifications.
\end{remark}

\begin{remark} \normalfont
The factor $\widetilde{\mathbb{P}}(\theta, \dd\theta')$ is a Markov kernel on $\mathbb{S}^1_+$. When $\mathbb{P}$ takes the form (\ref{eq1.17'}), the symmetry property (\ref{eq1.14}) reduces to the following: for any $f \in C_c(\mathbb{S}^1_+ \times \mathbb{S}^1_+)$, 
\begin{equation} \label{eq1.18'}
    \int_{\mathbb{S}^1_+ \times \mathbb{S}^1_+} f(\theta,\theta')\widetilde{\mathbb{P}}(\theta, \dd\theta') \sin\theta \dd\theta = \int_{\mathbb{S}^1_+ \times \mathbb{S}^1_+} f(\theta', \theta)\widetilde{\mathbb{P}}(\theta, \dd\theta')\sin\theta \dd\theta.
\end{equation}
\end{remark}

For more information on the situation when $W$ is periodic, see \S\ref{ssec_lemconstruct}.

\subsubsection{Characterization of rough reflections laws in upper half-plane} \label{sssec_rfrefchar}

Rough reflections were originally characterized by Plakhov in the context of optimization problems in aerodynamics. This author's setting is quite general and at least superficially different from the one above, considering the \textit{scattering law} on a bounded convex body in $\mathbb{R}^d$, instead of the rough reflection law in the upper half-plane. Angel, Burdzy, and Sheffield independently obtained a characterization of the rough reflection laws as we have defined them above. We state this characterization first, since it is the one most closely related to the main results obtained in this work. In \S\ref{sssec_Plakhov} we discuss Plakhov's ideas and how they are related to the result which we state presently.

\begin{theorem}[\ref{ABS_refl}, Theorem 2.3] \label{thm_rough_ref_char} Suppose $\mathbb{P}(x,\theta; \dd x' \dd\theta') = \delta_x(\dd x')\widetilde{\mathbb{P}}(x,\theta; \dd\theta')$ and is a symmetric with respect to the measure $\Lambda^1$ in the sense of (\ref{eq1.14}). Then there exists a sequence of walls $W_n$ with piecewise analytic boundaries $\partial W_n \subset \{(x_1,x_2) : -1/n < x_2 \leq 0\}$ such that 
\begin{equation}
    \mathbb{P}^{W_n}(x,\theta; \dd x' \dd\theta')\Lambda^1(\dd x \dd\theta) \to \mathbb{P}(x,\theta; \dd x' \dd\theta')\Lambda^1(\dd x \dd\theta),
\end{equation}
weakly on the space of measures on $\mathbb{R} \times \mathbb{S}^1_+$.
\end{theorem}

\begin{remark} \normalfont
The actual walls $W_n$ constructed in [\ref{ABS_refl}, Theorem 2.3] have boundaries $\partial W_n$ which satisfy the following conditions: (1) the boundary is composed of a locally finite collection of compact, analytic curve segments (where analytic on a compact interval means having an analytic extension to an open interval); (2) the curve segments intersect only at their endpoints and do not form cusps at their intersection points (i.e. the angle between two intersecting curve segments at their endpoints is not zero); (3) each curve segment either has non-vanishing curvature of one sign or is a line segment; and (4) the condition A5a (see above) is satisfied. Thus the walls satisfy the same hypotheses as those of [\ref{Chernov&Markarian}, Chapter 2] for example.
\end{remark}

To appreciate the significance of Theorem \ref{thm_rough_ref_char}, consider the following two Markov kernels which are easily shown to be symmetric with respect to the measure $\Lambda^1$.
\begin{itemize}
    \item retroreflection: $\mathbb{P}(x,\theta; \dd x' \dd\theta') = \delta_{(x,\theta)}(\dd x' \dd\theta')$.
    
    \item Lambertian reflection: $\mathbb{P}(x,\theta; \dd x' \dd\theta') = \frac{1}{2}\delta_x(\dd x') \sin\theta' \dd\theta'$. 
\end{itemize}
The first of these reflection laws is deterministic. Examples of approximate retro-reflectors in real life include cat's eyes and street signs with reflective paint [\ref{retroref_wiki}]. The second reflection law is random. It was introduced by Lambert in 1760 to model light reflecting from a matte surface [\ref{Lambert_history}]. There are many other examples of Markov kernels which preserve the measure $\Lambda^1$ and have a trivial spatial factor $\delta_x(\dd x')$ (the collection of deterministic reflection laws alone is isomorphic to the space of measure preserving transformations of $(0,1)$ with Lebesgue measure -- see Remark \ref{rem_deterministic}). Theorem \ref{thm_rough_ref_char} says that each of these is a rough reflection law; that is, each may be approximated by a deterministic reflection from a surface with a geometric microstructure.

In the case where $\mathbb{P}(x,\theta; \dd x' \dd\theta') = \delta_x(\dd x')\widetilde{\mathbb{P}}(\theta, \dd\theta')$, the sequence of approximating reflectors may be taken to be periodic. 

\begin{corollary} \label{cor_rough_ref_char}
Suppose $\mathbb{P}(x,\theta; \dd x' \dd\theta') = \delta_x(\dd x') \widetilde{\mathbb{P}}(\theta, \dd\theta')$, where $\widetilde{\mathbb{P}}$ is a Markov kernel which is symmetric with respect to the measure $\sin\theta \dd\theta$ in the sense of (\ref{eq1.18'}). Then there exists a sequence of cells $\{\Sigma_n\}_{n \geq 1}$ (satisfying conditions B1-B5) and positive numbers $\epsilon_n \to 0$ such that 
\begin{equation}
    \mathbb{P}^{\Sigma_n,\epsilon_n}(x,\theta; \dd x' \dd\theta')\Lambda^1(\dd x \dd\theta) \to \mathbb{P}(x,\theta; \dd x' \dd\theta')\Lambda^1(\dd x \dd\theta),
\end{equation}
weakly on the space of measures on $\mathbb{R} \times \mathbb{S}^1_+$.
\end{corollary}

\begin{remark} \normalfont
We call this a ``corollary'' because it follows from the proof of Theorem 2.3 in [\ref{ABS_refl}]. In this proof, the sequence of approximating reflectors $W_n$ is obtained by locally constructing reflectors beneath the intervals $[k/n, (k+1)/n]$ in the $x_1$-axis for $k \in \mathbb{Z}$, and then piecing the reflectors together. When $\widetilde{\mathbb{P}}$ does not depend on $x$, these local reflectors can be taken to be of the same type on each interval, and consequently $W_n$ is periodic.
\end{remark}

\begin{remark} \label{rem_indep} \normalfont
It will follow from Lemma \ref{lem_construct} that the limit of $\mathbb{P}^{\Sigma_n,\epsilon_n}$ does not depend on the choice of positive numbers $\epsilon_n \to 0$, but only on the sequence of cells $\{\Sigma_n\}$.
\end{remark}

\subsubsection{Operators on a Hilbert space} \label{sssec_Feres}

What information about the microgeometry of the rough surface can be recovered from the rough reflection law $\mathbb{P}$? In a series of papers, Feres and collaborators have sought to address this and related questions.

These authors are motivated in part by a problem in gas kinematics. Suppose that an inert gas at low pressure is released from a long but finite cylindrical chamber with rough interior walls. How is the microgeometry of the interior walls related to the time of escape of the gas particles? Questions of this nature go back to Knudsen's studies of gas kinematics in 1907 (see [\ref{Knudsen}], [\ref{FeresYab}]). Knudsen assumed, based on physical heuristics, that the angle of reflection of a gas particle from the interior wall is independent of the angle of incidence, and hence the distribution of the angle of reflection is $\frac{1}{2}\sin\theta \dd\theta$ (the same distribution introduced by Lambert in optical studies, as noted above). This is of course directly related to the fact that $\widetilde{\mathbb{P}}$ preserves the measure $\sin\theta \dd\theta$.

Consider the case of a periodic wall where the sequence of cells $\Sigma_n = \Sigma$ is constant. Recall that in this case, rough reflection laws take the form (\ref{eq1.17'}). It is proved in [\ref{feres_rw}] that $\widetilde{\mathbb{P}}$ is a bounded self-adjoint operator on the Hilbert space $L^2(\mathbb{S}^1_+, \sin\theta \dd\theta)$. Self-adjointness is a direct consequence of the time-reversibility property mentioned above. Moreover, under additional assumptions (more or less, the sides of $\Sigma$ should be dispersing or Sinai), $\widetilde{\mathbb{P}}$ becomes a Hilbert-Schmidt operator.

The time for a gas particle to exit an open-ended cylindrical chamber of length $L$ and radius $r$ is related in [\ref{feres_rw}] to the Markov kernel $\widetilde{\mathbb{P}}$ as follows. If the gas particle repeatedly hits the interior boundary of the cylinder with angle of incidence $\theta_n$ and angle of reflection $\theta_n'$, then we must have 
\begin{equation}
    \theta_n' \sim \widetilde{\mathbb{P}}(\theta_n, \dd\theta') \quad\quad \text{ and } \quad\quad \theta_{n+1} = \theta_n', \quad\quad n = 0, 1, 2, ....
\end{equation}
The sequence of random angles $\theta_n'$ together with the constant speed $v$ of the gas particle determine a continuous, piecewise linear process $X_t$ on the real line which is the projection of the position of the gas particle at time $t$ onto the cylindrical axis. The exit time for the gas particle from the cylinder is then 
\begin{equation}
    \tau_{L} = \inf\{t \geq 0 : X_t > L\}.
\end{equation}
By a central limit theorem argument, under restrictive assumptions on the stationary distribution of the angle process $\theta_n'$, the scaling limit of the process $\{\xi X_{t/\xi}, t \geq 0\}$ as $\xi \to 0$ is Brownian motion with variance depending on the spectrum of $\widetilde{\mathbb{P}}$. The assumptions in [\ref{feres_rw}] under which this is proved exclude the case where the process $\theta_n'$ is ergodic (i.e. the unique stationary distribution is $\sin\theta \dd\theta$). In the ergodic case, the linear increments of the process $X_t$ have infinite variance, and to obtain a Brownian motion scaling limit, one must instead use the scaling $\{\xi X_{t|\log \xi|/\xi}, t \geq 0\}$, $\xi \to 0$. The analysis of the latter case is carried out in [\ref{feres_rw}] for one example.

With this motivation, the spectrum of $\widetilde{\mathbb{P}}$ is further analyzed in [\ref{FeresZhang_spectrum1}] and [\ref{FeresZhang_spectrum2}]. The first of these papers examines $\widetilde{\mathbb{P}}$ for a special class of cells $\Sigma$ whose sides consist of dispersive circular arcs. By analyzing the moments of $\widetilde{\mathbb{P}}$, a relationship to spherical harmonics emerges. Namely, for smooth functions $u$ on $\mathbb{S}^2$ which are rotationally invariant about the vertical axis, 
\begin{equation}
    \widetilde{\mathbb{P}} u - u = \frac{k_2}{6}\Delta_{\mathbb{S}^2} u + O(K^3),
\end{equation}
where $K$ is a scale invariant curvature parameter depending on $\Sigma$, and $\Delta_{\mathbb{S}^2}$ is the spherical Laplacian on $\mathbb{S}^2$. (The meaning of $\widetilde{\mathbb{P}} u$ is as follows: If we give $\mathbb{S}^2$ spherical coordinates $(\phi,\psi)$ where $\psi \in \mathbb{S}^1_+$ is the angle from vertical axis, then $u(\phi,\psi) = \widetilde{u}(\psi)$ for some $\widetilde{u}$ by rotational invariance, and $\widetilde{\mathbb{P}} u := \widetilde{\mathbb{P}}\widetilde{u}$.)

The paper [\ref{FeresZhang_spectrum2}] examines the spectrum of $\widetilde{\mathbb{P}}$ for more general types of cells, and the authors obtain explicit bounds on the spectral gap of $\widetilde{\mathbb{P}}$. These bounds depend on the curvature of the ``most exposed'' parts of the boundary of the cell $\Sigma$. These results are applied to estimate the rate of convergence of the Markov process determined by $\widetilde{\mathbb{P}}$ to stationary distribution.

The paper [\ref{CookFeres_temp}] considers similar questions but in an even more general setting, where the wall $W$ is allowed to have moving parts and energy can be exchanged between the wall and the point particle.

\subsubsection{Rough scattering laws on bounded convex bodies} \label{sssec_Plakhov}

In the context of optimization problems in aerodynamics, Plakhov has considered rough reflections on general bounded convex bodies. Here the main object of interest is the scattering law, which describes the equilibrium distribution of the incoming and outgoing particle velocities and the normal direction at the point of contact. The scattering law does not contain precisely the same information as the rough reflection law, but the important ideas are similar. The concepts and results summarized here originally appeared in [\ref{Plakhov1}, \ref{Plakhov2}, \ref{Plakhov3}, \ref{Plakhov4}] and were later assembled in a book [\ref{Plakhov5}, Chapter 4].

Given a bounded convex body $C \subset \mathbb{R}^d$ and a body $B \subset C$ with piecewise smooth boundary, define subspaces
\begin{equation}
    (\partial C \times \mathbb{S}^{d-1})_\pm = \{(x,u) \in \partial C \times \mathbb{S}^{d-1} : \pm u \cdot n(x) > 0\},
\end{equation}
where $n(x)$ is the outward-pointing unit normal vector at $x \in \partial C$, and $\cdot$ is the Euclidean dot product. Define a measure $\Lambda_C$ on $(\partial C \times \mathbb{S}^{d-1})_{+}$ by 
\begin{equation}
    \Lambda_C(\dd x \dd u) = (u \cdot n(x))_+ \dd x \dd u, 
\end{equation}
where $\dd x$ and $\dd u$ are Lebesgue measure on $\partial C$ and $\mathbb{S}^{d-1}$ respectively. The macro-reflection law $P^{B,C}$ is defined in the same way as before. That is, if $(x, -u) \in (\partial C \times \mathbb{S}^{d-1})_+$ is the initial state of a point particle, then the particle will enter the $C$ and hit the body $B$ some number (possibly zero) times before returning to the boundary $\partial C$ in state $(x',u') = (x'(x,u),u'(x,u)) \in (\partial C \times \mathbb{S}^{d-1})_+$. The \textit{macro-reflection law determined by $B$} is the map $P^{B,C} : (\partial C \times \mathbb{S}^{d-1})_+ \to (\partial C \times \mathbb{S}^{d-1})_+$ such that
\begin{equation}
    P^{B,C}(x,u) = (x',u').
\end{equation}
The map $P^{B,C}$ may not be defined on a measure zero subset of $(\partial C \times \mathbb{S}^{d-1})_+$. Like in the case of the upper half-plane, $P^{B,C}$ is an involution and preserves the measure $\Lambda_C$.

Let $\nu_{B,C}$ be the measure on $\mathbb{S}^{d-1} \times \mathbb{S}^{d-1} \times \mathbb{S}^{d-1} = (\mathbb{S}^{d-1})^3$ giving the equilibrium joint distribution of the incoming velocity $-u$, the outgoing velocity $u'$, and the unit normal vector $n(x)$ at the point of contact $x \in \partial C$. That is, $\nu_{B,C}$ is defined by  
\begin{equation}
    \int_{(\mathbb{S}^{d-1})^3} f(u,u',n)\nu_{B,C}(\dd u \dd u' \dd n) = \int_{(\partial C \times \mathbb{S}^{d-1})_+} f(u, u'(x,u), n(x)) \Lambda_C(\dd x \dd u).
\end{equation}

A measure $\nu$ on $(\mathbb{S}^{d-1})^{3}$ is called a \textit{rough scattering law} on $C$ if there exists a sequence of bodies $B_n \subset C$ such that 
\begin{enumerate}  [label = (\roman*)]
    \item $\text{Vol}(C \smallsetminus B_n) \to 0$ as $n \to \infty$, and
    
    \item the sequence of measures $\nu_{_{B_n, C}}$ converges weakly to the measure $\nu$.
\end{enumerate}
\begin{remark} \normalfont
This terminology departs slightly from the terminology in [\ref{Plakhov5}, Chapter 4]. Here one considers equivalence classes $\mathcal{B}$ of sequences of bodies $B_n$ satisfying (i) such that the sequences of measures $\{\nu_{B_n,C}\}$ have the same weak limit $\nu_{\mathcal{B}}$. One says that $\mathcal{B}$ is a \textit{rough body} obtained by \textit{grooving} $C$, and the measure $\nu_{\mathcal{B}}$ is given no special name.
\end{remark}

\begin{remark} \normalfont
Condition (i) looks different from the corresponding condition in the definition of a rough reflection law, where the wall boundaries $\Gamma_n$ approach the line $x_2 = 0$ uniformly. Nonetheless, an equivalent definition of a rough scattering law is obtained by replacing (i) with the more restrictive condition:
\begin{enumerate}
    \item[(i)'] $\sup\{\dist(x, \partial C) : x \in B_n\} \to 0$ as $n \to \infty$.
\end{enumerate}
This is not immediate, but it follows from the characterization of rough scattering laws (Theorem \ref{thm_scattering_char} -- below) and its proof. The sequence of bodies $B_n$ constructed in the proof of the characterization theorem [\ref{Plakhov5}, Thm 4.5] can in fact be taken to satisfy (i)'. 
\end{remark}

Define a measure $\tau_C$ on $(\mathbb{S}^{d-1})^2$ by 
\begin{equation}
    \int_{(\mathbb{S}^{d-1})^2} g(u,n)\tau_C(\dd u \dd n) = \int_{\partial C \times \mathbb{S}^{d-1}} g(u,n(x)) \Lambda_C(\dd x \dd u).
\end{equation}
Let $\pi_{u,n} : (u,u',n) \mapsto (u,n)$ and $\pi_{u',n} : (u,u',n) \mapsto (u',n)$ be the natural projections. Let $\text{Adj}(u,u',n) = (u',u,n)$. The most important elementary properties of a rough scattering law are summarized in the following proposition.

\begin{proposition} A rough scattering law $\nu$ on $C$ has the following properties:

(i) $\pi_{u,n}^{\#} \nu = \tau_{C} = \pi_{u',n}^{\#} \nu$.

(ii) $\text{Adj}^{\#}\nu = \nu$.
\end{proposition}

\noindent (Here if $g : X \to Y$ is a map between measurable spaces, and $\mu$ is a measure on $X$, then $g^\#\mu$ is the \textit{pushforward measure} on $Y$ defined by $g^\#\mu(A) = \mu(g^{-1}(A))$.)

The proposition above follows from the fact that any measure $\nu_{B,C}$ must satisfy (i) and (ii), and these properties are preserved in weak limits. For $\nu_{B,C}$, property (i) is a consequence of the fact that $P^{B,C}$ preserves the measure $\Lambda_C$, while property (ii) follows from involutivity of $P^{B,C}$.

In fact, the properties (i) and (ii) completely characterize the rough scattering laws.

\begin{theorem}[\ref{Plakhov5}, Theorem 4.5] \label{thm_scattering_char}
A measure $\nu$ on $\mathbb{S}^3$ is a rough scattering law on $C$ if and only if $\nu$ satisfies properties (i) and (ii).
\end{theorem}

The original motivation for considering the scattering law and its characterization is its relationship to certain \textit{resistance functionals} of the form
\begin{equation}
    R_{\chi}[P^{B,C}] := \int_{(\partial C \times \mathbb{S}^{d-1})_+} c(u,u'(x,u))(u \cdot n(x)) \dd x \widetilde{\chi}(\dd u),
\end{equation}
where $c$ is a ``cost function'' on $(\mathbb{S}^{d-1})^2$, and $\chi$ is a Borel measure on $\mathbb{S}
^{d-1}$. When $\chi$ is Lebesgue measure, the above functional may be expressed as 
\begin{equation}
\begin{split}
    R[P^{B,C}] & = \int_{(\partial C \times \mathbb{S}^{d-1})_+} c(u,u'(x,u))\Lambda_C(\dd x \dd u) \\
    & = \int_{(\mathbb{S}^{d-1})^3} c(u,u') u \cdot n \ \nu_{B,C}(\dd u \dd u' \dd n).
\end{split}
\end{equation}
By taking weak limits, we can extend the definition of resistance functionals to rough scattering laws:
\begin{equation}
    R[\nu] := \int_{(\mathbb{S}^{d-1})^3} c(u,u') u \cdot n \ \nu(\dd u \dd u' \dd n).
\end{equation}
With the characterization given by Theorem \ref{thm_scattering_char}, one can thus reduce problems of minimizing air resistance of rough convex bodies to problems in mass optimal transport. This leads to some counterintuitive results. It is possible, for example, to actually \textit{decrease} the air resistance of a convex body by appropriately roughening its surface. In fact, one can construct nonconvex bodies which have arbitrarily small air resistance in one direction, as well as bodies which are invisible in one direction. By contrast, the air resistance of a smooth convex body has long been known to have a strictly positive lower bound [\ref{Plakhov5}, Chapters 5 and 8].

Let us now describe the connection between rough reflections and rough scattering laws. Just as in the upper half-plane case, we can define deterministic Markov kernels on $(\partial C \times \mathbb{S}^{d-1})_+$ by
\begin{equation}
    \mathbb{P}^{B,C}(x,u; \dd x' \dd u') = \delta_{P^{B,C}(x,u)} \dd x' \dd u'.
\end{equation}
We say that Markov kernel $\mathbb{P}$ is a \textit{rough reflection law on $C$} if there exists a sequence of bodies $B_n \subset C$ such that $\text{Vol}(C \smallsetminus B_n) \to 0$, and 
\begin{equation}
    \mathbb{P}^{B_n,C}(x,u; \dd x' \dd u') \Lambda_C(\dd x \dd u) \to \mathbb{P}(x,u; \dd x' \dd u') \Lambda_C(\dd x \dd u) 
\end{equation}
weakly in the space of measures on $(\partial C \times \mathbb{S}^{d-1})_+ \times (\partial C \times \mathbb{S}^{d-1})_+$. A measure $\nu$ is a rough scattering law if and only if there exists a rough reflection law $\mathbb{P}$ such that
\begin{equation} \label{eq1.41}
\begin{split}
    &\int_{(\mathbb{S}^{d-1})^3} f(u,u',n)\nu(\dd u \dd u' \dd n) \\
    & \quad\quad = \int_{((\partial C \times \mathbb{S}^{d-1})_+)^2} f(u,u',n(x)) \mathbb{P}(x,u; \dd x' \dd u') \Lambda_C(\dd x \dd u).
\end{split}
\end{equation}

Rough scattering laws and rough reflection laws are not in one-to-one correspondence. One way to see this is to consider a convex body $C$ with a flat side $V \subset \partial C$. The rough scattering law will not distinguish pointwise variation in reflection from $V$, because the unit normal vector at each point in $V$ is the same.

As noted above, the characterization of rough scattering laws predates the characterization of rough reflection laws. Although Theorem \ref{thm_scattering_char} does not imply Theorem \ref{thm_rough_ref_char}, the proof of the former can almost certainly be modified to obtain the latter. In fact, higher dimensional versions of Theorem \ref{thm_rough_ref_char} can probably be proved by an appropriate modification of the arguments in [\ref{Plakhov5}, Chapter 4]. The proofs of both theorems begin with a local construction of reflectors which redirect a point particle hitting the boundary in a prescribed way. It is then a matter of appropriately assembling the local reflectors to produce the desired rough scattering law or rough reflection law.

\subsection{Disk and wall model} \label{ssec_main_results}

\subsubsection{Rigid body system} \label{sssec_rigid}

The model we study consists of a fixed wall in the lower half-plane of $\mathbb{R}^2$, together with a disk which is given some rotationally symmetric mass density and allowed to move freely in the complement of the wall. Each body is furnished with asperities (microscopic structures) on its surface. Aside from a periodicity requirement, the asperities on the surface of the wall are allowed to be fairly arbitrary in shape. On the other hand, the asperities on the disk are of a specific type, namely ``geostationary satellites'' spaced in such a way as to guarantee that non-local interactions between the two bodies are rare.

The wall $W$ will be built from a cell $\Sigma$ and a roughness scale $\epsilon > 0$, in the same way as we have done in \S\ref{sssec_uphalf}. The cell $\Sigma$ is assumed to satisfy conditions B1-B5, stated in \S\ref{sssec_periodic}. The definition of the wall $W = W(\Sigma,\epsilon)$ is slightly modified as follows: $W(\Sigma,\epsilon)$ is the unique subset of $\mathbb{R}^2$ which is $\epsilon$-periodic in the $x_1$-direction, and whose image under the covering map 
\begin{equation} \label{eq1.35}
    (x_1,x_2) \mapsto (e^{2\pi i x_1/\epsilon}, \epsilon^{-1}(x_2 + 1))
\end{equation}
is $\Sigma$. Such a wall will satisfy conditions A1-A4 and A5a and A5b, stated in \S\ref{sssec_uphalf}, except that condition A3 must be modified as follows:
\begin{enumerate}
    \item[A3'.] The following inclusions hold: $\mathbb{R} \times (-\infty, -1 - \epsilon] \subset W \subset \mathbb{R} \times (-\infty, - 1]$; thus $\partial W \subset [-1 - \epsilon, -1] \times \mathbb{R}$.
\end{enumerate}
The advantage of having $W$ located below the line $x_2 = -1$, instead of $x_2 = 0$, will become apparent later.

By assumption, the cell boundary $\partial \Sigma$ decomposes into a finite collection of closed $C^2$ curve segments $\{\widetilde{\Gamma}_i\}_{i = 0}^{m-1}$. Such a decomposition determines a decomposition of $\partial W$ into a countable, locally finite collection $\{\Gamma_i\}_{i \in \mathbb{Z}}$ of $C^2$ curve segments such that distinct pairs $\Gamma_i$ and $\Gamma_j$ can intersect only at their endpoints, and $\Gamma_i$ maps to $\widetilde{\Gamma}_{i \mod m}$ via the covering map (\ref{eq1.35}) defined above. 

\begin{remark} \label{rem_finkappa} \normalfont
Recall what it means for a closed curve segment $\Gamma_i$ to be $C^2$: There an open interval $I \supset [0,1]$ and a $C^2$ map $\gamma_i : I \to \mathbb{R}^2$ such that $\gamma_i([0,1]) = \Gamma_i$. Consequently, by compactness, each of the curve segments $\Gamma_i$ has bounded curvature.
\end{remark}

The decomposition of $\partial \Sigma$ into $C^2$ curve segments $\widetilde{\Gamma}_i$ is not unique, and correspondingly the decomposition of $\partial W$ into $C^2$ curve segments $\Gamma_i$ is not unique. To ensure that terms introduced below are well-defined, we assume from this point on that some decomposition $\{\widetilde{\Gamma}_i\}_{0 \leq i \leq m-1}$ of $\partial \Sigma$ and correspondingly $\{\Gamma_i\}_{i \in \mathbb{Z}}$ of $\partial W$ has been fixed, and we shall refer to these as the \textit{given decompositions} of $\partial \Sigma$ and $\partial W$ respectively.

We denote the relative interior of a curve segment $\Gamma_i$ by $\Int \Gamma_i$. We call a point $p \in \partial \Sigma$ \textit{regular} if there is some $C^2$ curve segment $\Gamma \subset \partial \Sigma$ (not necessarily coming from the given decomposition) such that $p \in \Int \Gamma$. We denote the set of regular points in $\partial \Sigma$ by $\partial_{\reg} \Sigma$, and we denote the set of regular points in $\partial W$ by $\partial_{\reg} W$. We let $\partial_s \Sigma = \partial \Sigma \smallsetminus \partial_{\reg} \Sigma$, and we let $\partial_s W = \partial W \smallsetminus \partial_{\reg} W$. We refer to points in $\partial_s \Sigma$ and $\partial_s W$ as \textit{singular points} of $\partial \Sigma$ and $\partial W$ respectively.  The sets $\partial_s W$ and $\partial_s \Sigma$ are measure zero subsets of $\partial \Sigma$ and $\partial W$ respectively.

For $p \in \partial_{\reg} W$, we let $\kappa(p)$ denote the unsigned curvature of $\partial W$ at $p$. We define
\begin{equation}
    \kappa_{\max} = \sup\{\kappa(p) : p \in \partial_{\reg} W(\Sigma, 1)\}.
\end{equation}
The quantity above depends only on the cell $\Sigma$. Per Remark \ref{rem_finkappa} and periodicity of $\partial W$, $\kappa_{\max}$ is finite.

We take $\rho(\epsilon)$ to be some positive, non-decreasing function of $\epsilon$ such that 
\begin{equation}
    \rho(\epsilon) \to 0 \quad \text{ and } \frac{\epsilon^{1/2}}{\rho(\epsilon)} \to 0 \quad \text{ as } \epsilon \to 0.
\end{equation}
In addition, we assume that $2\pi/\rho(\epsilon)$ is an integer.

The freely moving body in our system is a ``disk with satellites.'' Namely, we first define a \textit{reference body} 
\begin{equation} \label{eq1.2}
D = D(\epsilon) = \left(\bigcup_{k = 0}^{N-1} S_k \right) \cup D_0 \subset \mathbb{R}^2,
\end{equation}
where
\begin{equation} \label{eq1.38}
    N = \frac{2\pi}{\rho}, \quad\quad S_k = (\sin k\rho(\epsilon), -\cos k\rho(\epsilon)),
\end{equation}
and 
\begin{equation} \label{eq1.34}
    D_0 \subset \{(x_1,x_2) : x_1^2 + x_2^2 \leq (1 - 2\rho(\epsilon)^2)^2\}.
\end{equation}
See Figure \ref{disk_plus_wall00_fig}. The quantity $N$ is the number of satellites, and $S_k$ is the position of the $k$'th satellite in the reference body.

\begin{figure}
    \centering
    \includegraphics[width = 0.6\linewidth]{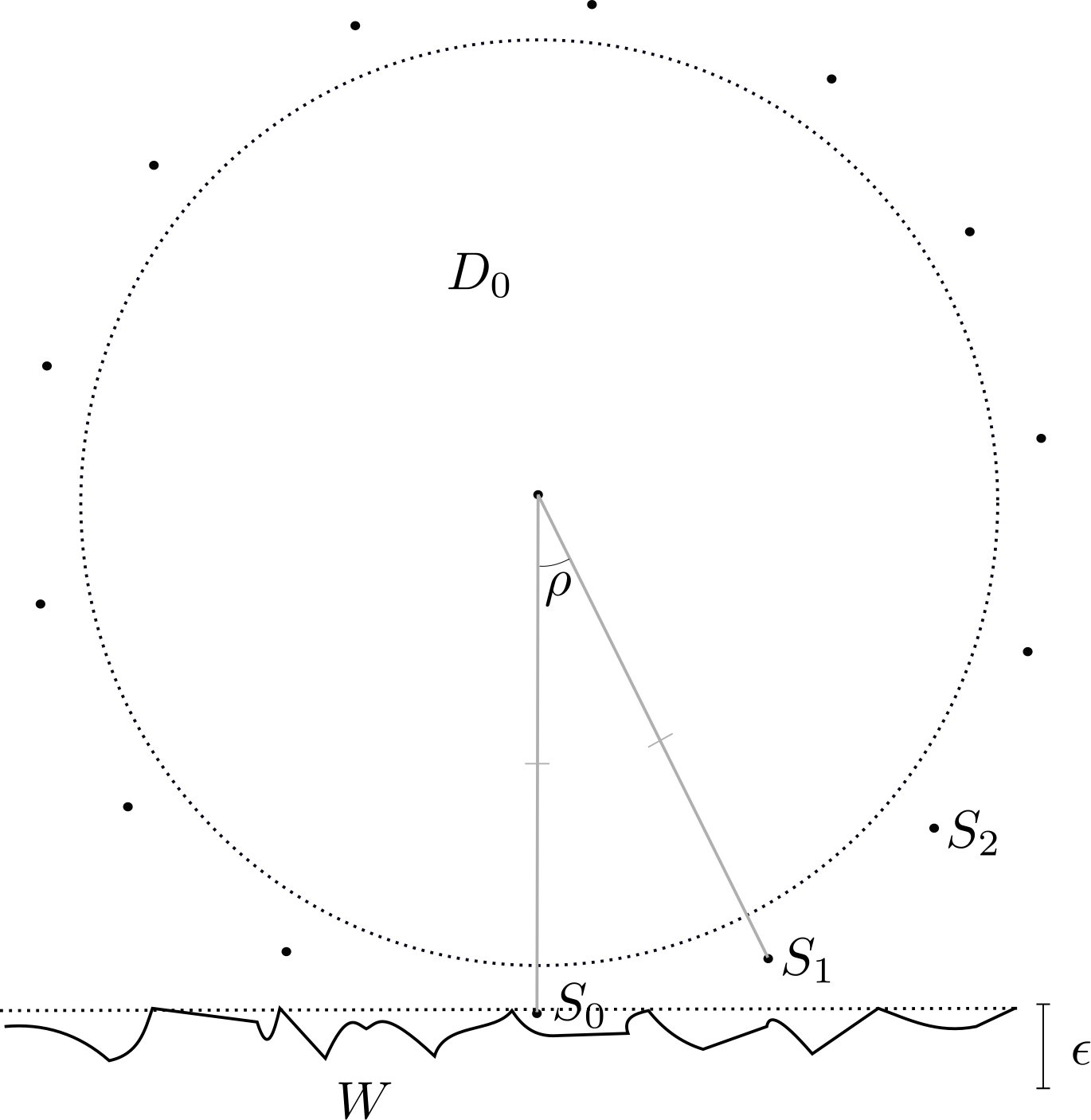}
    \caption{A freely moving disk with satellites, and a fixed wall.}
    \label{disk_plus_wall00_fig}
\end{figure}

For $x = (x_1,x_2) \in \mathbb{R}^2$ and $\alpha \in \mathbb{R}$, we let $D(x,\alpha)$ denote the subset of the plane occupied by the reference body after rotating it counterclockwise about the origin by an angle $\alpha$ and translating it by $x$ from the reference configuration (\ref{eq1.2}). We will sometimes abuse terminology by using $D$ to denote both the reference body (\ref{eq1.2}) (a fixed subset of the plane) and the freely moving physical body which it represents.

We assume that $D$ has some mass density $\lambda$ which is rotationally symmetric about the origin when $D$ is in reference configuration (\ref{eq1.2}), and we introduce parameters: 
\begin{equation}
    m = \int_D \lambda(\dd x), \quad\quad J = \int_D |x|^2 \lambda(\dd x)
\end{equation}
-- the \textit{mass} and \textit{moment of inertia}, respectively, of the disk. We assume that $\lambda$ does not depend on the parameter $\epsilon > 0$. 

In the analysis of our model, we will see that there is no loss of generality in assuming that $m = J = 1$; however, we will not introduce this simplification until later. 

\begin{remark} \label{rem_no_contact} \normalfont
Additional motivation for our choice of the body $D$ can be gained from the following remarks.
\begin{itemize}
    \item The first member of the union (\ref{eq1.2}) is a discrete set of points $S_k$ (the ``satellites'') which are disconnected from the rest of the body and evenly spaced at angles of $\rho(\epsilon)$ along the unit circle. We have stipulated that $\rho(\epsilon)$ should converge to zero, but at a rate more slowly than $\epsilon^{1/2}$. This will guarantee that interactions between multiple satellites during a single collision event are rare as the roughness scale goes to zero.
    
    \item The subset $D_0$ is of no mathematical significance in the analysis of the disk and wall model, but serves only to persuade the reader of the physical generality of the model. One may imagine that $D_0$ carries the mass of the body $D$. The inclusion (\ref{eq1.34}) guarantees that $D_0$ can never come into contact with the wall $W$. Only the satellites $S_k$ can come into contact with $W$.
    
    \item Choosing the ``roughness'' on $D$, as we have above, so that only isolated points can interact with the wall greatly simplifies the description of the configuration space of the system. Unfortunately, the choice of a disconnected body is physically unrealistic. One way to reconcile this with our intuition is to think of $D$ as a ``hockey puck'' with additional features in the third (vertical) dimension. In the same spirit as the analysis of polygonal chains in [\ref{DemaineORourke}, \S 5.3], we obtain a physical system which is equivalent to the one above by joining each satellite to the inner body $D_0$ by a curved rod which extends into the third dimension. 
    
    \item Whether theorems analogous to the ones we state in \S\ref{sssec_main} can be proved for a connected body $D$ is an open question. On the other hand, we expect the range of possible dynamics for a system consisting of a freely moving rough disk and fixed rough wall to be exhausted by our model. See the discussion of our main results in \S\ref{sssec_minformal}.  
\end{itemize}
\end{remark}

\subsubsection{Configuration space} \label{sssec_config}

The \textit{configuration space} of the disk and wall system is the topological closure in $\mathbb{R}^3$ of the set of configurations $y = (x_1,x_2,\alpha)$ of the disk such that the disk and the wall are disjoint, i.e. 
\begin{equation} \label{eq1.37}
\mathcal{M} = \overline{\{(x_1,x_2,\alpha) : D(x_1,x_2,\alpha) \cap W = \emptyset\}}.
\end{equation}
Note that points in $\mathcal{M}$ which differ only in their angular coordinates by multiples of $2\pi$ represent the same physical configuration. It is in fact easy to see that $\mathcal{M}$ is doubly periodic: Let 
\begin{equation}
    e_1 = (1,0,0), \quad e_2 = (0,1,0), \quad e_3 = (0,0,1)
\end{equation}
be the standard basis for $\mathbb{R}^3$. Then
\begin{equation}
\begin{split}
    &\mathcal{M} + \epsilon e_1 = \{(x_1,x_2,\alpha) : (x_1 - \epsilon, x_2, \alpha) \in \mathcal{M}\} = \mathcal{M}, \quad \text{ and } \\
    &\mathcal{M} + \rho e_3 = \{(x_1,x_2,\alpha) : (x_1, x_2, \alpha - \rho) \in \mathcal{M}\} = \mathcal{M}.
\end{split}
\end{equation}
Indeed, $\mathcal{M}$ is $\epsilon$-periodic in the $x_1$-coordinate because $W$ is $\epsilon$-periodic in the $x_1$-coordinate, and $\mathcal{M}$ is $\rho$-periodic in the $\alpha$-coordinate because, as noted above, only the satellites of $D$ can come into contact with the wall, and rotating the disk about its center of mass by an angle of $\rho(\epsilon)$ maps the set of satellites onto itself. 

We can alternatively study the configuration space $\mathcal{M}/\sim$, where $\sim$ is the equivalence relation which identifies points $y$ and $y'$ such that $y' = y + i\epsilon e_1 + j\rho e_3$, for some $i,j \in \mathbb{Z}$. The space $\mathcal{M}/\sim$ may be regarded as a subspace of $\mathbb{S}^1 \times \mathbb{R} \times \mathbb{S}^1$. This configuration space will be useful for applying dynamical results which require compactness, but otherwise we will mainly just work with $\mathcal{M}$.

The topological boundary $\partial \mathcal{M}$ of the configuration space corresponds to the set of \textit{collision configurations} of the system, i.e. the set of configurations in which the rough disk and the wall are in contact. The boundary of $\mathcal{M}$ lies just below the plane 
\begin{equation}
    \mathbf{P} := \{(x_1,x_2,\alpha) \in \mathbb{R}^3 : x_2 = 0\}.
\end{equation}

The space $\mathcal{M}$ does not fit neatly into a well-studied class of manifolds. It is probably not even a manifold with corners for many choices of the wall $W$. The space $\mathcal{M}$ may loosely be described as a $C^2$ manifold with boundary and ``singularities.'' We will see that there exists a closed subset $\mathcal{S} \subset \partial \mathcal{M}$ such that the 2-dimensional Hausdorff measure of $\mathcal{S}$ is zero, and $\mathcal{M}_{\reg} := \mathcal{M} \smallsetminus \mathcal{S}$ is an embedded $C^2$ submanifold of $\mathbb{R}^3$ with boundary. In particular, $\partial_{\reg} \mathcal{M} := \partial \mathcal{M} \smallsetminus \mathcal{S}$ is a full-measure subset of the boundary on which there exists a $C^1$ field of inward-pointing unit normal vectors $n(q), q \in \partial_{\reg} \mathcal{M}$.

A more detailed description of the configuration space $\mathcal{M}$ is given in \S\ref{ssec_configprops}.  In \S\ref{sec_genref}, we reprove a number of standard results from billiards for a very general class of manifolds to which $\mathcal{M}$ may be shown to belong.

\subsubsection{Cylindrical approximation} \label{sssec_cylapprox}

Some additional understanding of the structure of $\mathcal{M}$ can be gained by imagining the following physical situation. Suppose the configuration of the disk $D$ is initially $y = (x_1,x_2,0)$. Then at most one satellite of $D$ can be in contact with $\partial W$, and this satellite is $S_0$ (this follows from Proposition \ref{prop_Mreg} for example). In this configuration, the coordinates of $S_0$ are $(x_1,x_2 - 1)$. All other satellites must lie in $W^c$. Consequently, we may rotate the disk counterclockwise about the satellite $S_0$ by a small angle $\Delta\alpha$, and the resulting configuration $y'$ will still lie in $\mathcal{M}$. Moreover, if $y \in \partial \mathcal{M}$ then $y' \in \partial \mathcal{M}$. We compute explicitly:
\begin{equation} \label{eq1.40}
    y' = (x_1 - \sin\Delta \alpha, x_2 - 1 + \cos \Delta\alpha, \Delta \alpha) \approx (x_1 - \Delta \alpha, x_2, \Delta \alpha).
\end{equation}
Let 
\begin{equation}
    \mathbf{Q}_0 = \{(x_1,x_2,\alpha) : \alpha = 0\}.
\end{equation}
We will see that $\mathcal{M} \cap \mathbf{Q}_0 = \overline{W^c} + e_2$. Thus (\ref{eq1.40}) suggests that, in a neighborhood of the plane $\mathbf{Q}_0$, the space $\mathcal{M}$ is well modeled by the cylinder
\begin{equation}
    \mathcal{M}_{\cyl} := \{(x_1,x_2,\alpha) : (x_1 + \alpha, \alpha) \in \overline{W^c} + e_2\}.
\end{equation}
This is the cylinder with base
\begin{equation}
    \widehat{B} := \overline{W^c} + e_2
\end{equation}
and axis
\begin{equation} \label{eq1.42}
    \widetilde{\chi} := (-1,0,1).
\end{equation}
The vector $\widetilde{\chi}$ is the ``rolling velocity'' described in the introduction.

More generally, we will see in \S\ref{ssec_configprops} that $\mathcal{M}$ may be decomposed into large ``cylindrical regions'' and small ``gap regions.'' The ``gap regions'' are difficult to describe explicitly and correspond to configurations where multiple satellites can come into contact with $\partial W$. Near any point in one of the ``cylindrical regions'' $\mathcal{M}$ is well modeled by a cylinder whose base is an order $\epsilon$ translate of $\widehat{B}$ and whose axis is $\widetilde{\chi}$. 

The main results of this work were originally derived informally by assuming that we can replace $\mathcal{M}$ with $\mathcal{M}_{\cyl}$ in our arguments.

\subsubsection{Phase space and dynamics} \label{sssec_phasespace}

The dynamical state of the system is completely specified by a pair $(y,w)$, where $y = (x,\alpha) = (x_1,x_2,\alpha) \in \mathcal{M}$ and $w = (v,\omega) = (v_1,v_2,\omega) \in \mathbb{R}^3$. Here $x = (x_1,x_2)$ is the center of mass of the disk, and $\alpha$ is the angular orientation of the disk, as above; $v = (v_1,v_2)$ is the linear velocity of the center of mass of the disk, and $\omega$ is the angular velocity. 

The \textit{phase space} is the set of all possible states $(y,w)$ of the system:
\begin{equation}
    T\mathcal{M} \cong \mathcal{M} \times \mathbb{R}^3.
\end{equation}
The notation $T\mathcal{M}$ refers to the fact that geometrically we regard the phase space as the tangent space of the configuration space. For technical purposes, the phase space can be taken simply as the Cartesian product of the configuration space $\mathcal{M}$ and the ``velocity space'' $\mathbb{R}^3$; however it will occasionally be convenient to use the above geometric language, for example when referring to the tangent space $T_y\mathcal{M} \subset T\mathcal{M}$ at a point $y$ in $\mathcal{M}$, which in this context is just the subset $\{y\} \times \mathbb{R}^3$. 

We equip $\mathbb{R}^3$ with the inner product 
\begin{equation} \label{eq1.49'}
\langle w,w' \rangle = m v_1v_1' + m v_2v_2' + J\omega\omega', \quad \text{ for } \quad w = (v_1,v_2,\omega), \quad w' = (v_1',v_2',\omega'),
\end{equation}
and we denote the corresponding norm by $||\cdot||$. If the disk moves with velocity $w = (v_1,v_2,\omega)$, then the kinetic energy of the disk is
\begin{equation}
    \frac{1}{2}m v_1^2 + \frac{1}{2}m v_2^2 + \frac{1}{2}J\omega^2 = \frac{1}{2}||w||^2.
\end{equation}
The sum of the first two terms is the kinetic energy arising from the linear motion of the disk, while the third term is the kinetic energy arising from rotational motion. For this reason we will often refer to the above inner product as the \textit{kinetic energy inner product}.

Unless specified otherwise, \textit{distances, angles, etc. in $\mathcal{M}$ are assumed to be with respect to the inner product $\langle \cdot,\cdot \rangle$ defined above.}

The main physical assumptions which govern the dynamics of the model are: 
\begin{enumerate}[label = P\arabic*.]
    \item The bodies $D$ and $W$ do not interpenetrate.
    \item The system is subject to Euler's laws of rigid body motion (see \S\ref{sec_specularreflection}).
    \item When not in contact, the net force applied to each body is zero, and upon contact, a single impulsive force is applied to the disk at the point of contact and directed parallel to the unit normal vector on the wall.
    \item The kinetic energy of the disk is conserved for all time.
\end{enumerate}
These assumptions are discussed in greater detail and given a more precise mathematical form in \S\ref{sec_specularreflection}.

Assumptions P1-P3 are standard for rigid body interactions. Assumption P4 is the limiting case for a system of bodies of finite mass in which the total kinetic energy, linear momentum, and angular momentum are conserved. After letting the mass and moment of inertia of one body diverge to infinity, the kinetic energy of the other body is conserved in the limit. (Note that the wall $W$ may be regarded as having infinite mass and infinite moment of inertia.) For more details, see Proposition \ref{prop_energylimit}.

The dynamics of the disk and wall system may be represented by a moving point mass in $\mathcal{M}$, where the position and velocity of the point mass is given by the state $(y,w)$ of the disk and wall system. It follows from the assumptions P3 and P4 that the point mass moves linearly with constant velocity in the interior of $\mathcal{M}$. When the point mass comes into contact with the boundary at a point $q \in \partial \mathcal{M}$, the velocity should be redirected back into $\mathcal{M}$ by a mapping $R_{q} : T_{q}\mathcal{M} \to T_{q}\mathcal{M}$. The dynamical description of our system is completed by the following proposition.

\begin{proposition} \label{prop_spec}
Let $q \in \partial_{\reg} \mathcal{M}$ and let $n(q)$ denote the inward-pointing unit normal vector at $q$. The unique map $R_{q} : T_{q}\mathcal{M} \to T_{q}\mathcal{M}$ for which assumptions P1-P4 are satisfied is specular reflection with respect to the kinetic energy inner product:
\begin{equation}
R_{q}(w) = w - 2\langle w,n(q) \rangle n(q), \quad\quad w \in T_{q}\mathcal{M}.
\end{equation}
\end{proposition}
Proposition \ref{prop_spec} is proved in \S\ref{sec_specularreflection}. The main takeaway is that the system under consideration has the same dynamical evolution as a classical billiard: free motion in the interior of the billiard domain and specular reflection on the boundary. 

\begin{remark} \normalfont
The inner product $\langle \cdot,\cdot \rangle$ makes $\mathbb{R}^3$ into a geodesically complete Riemannian manifold, and thus results proved in \S\ref{sec_genref} for billiards in general Riemannian manifolds will apply in this case. To avoid confusion, we will refer to $\langle \cdot,\cdot \rangle$ as an ``inner product'' in \S\S\ref{sec_intro}-\ref{sec_main_results}, reserving the word ``metric'' to refer to either (a) the distance functional on a metric space, or (b), in \S\ref{sec_genref}, a Riemannian metric on a general geodesically complete Riemannian manifold. Context will be sufficient to distinguish senses (a) and (b).
\end{remark}

\subsubsection{Collision laws} \label{sssec_collaws}

The configuration space $\mathcal{M}$ contains the plane $\mathbf{P} := \{(x_1,x_2,\alpha) : x_2 = 0\} \subset \mathbb{R}^3$. The boundary $\partial \mathcal{M}$ lies just below $\mathbf{P}$ within distance $O(\rho(\epsilon)^2)$ from the plane (this is proved in Proposition \ref{prop_Melem}). 

Let $\mathbb{R}^3_{\pm} := \{(x_1,x_2,\alpha) : x_2 \geq 0\}$. Let $\mathbb{S}^2$ denote the unit sphere in $\mathbb{R}^3$ with respect to the kinetic energy norm $||\cdot||$ defined above. By conservation of energy, the velocity of a point particle representing the system may be taken to lie in the unit sphere $\mathbb{S}^2$ for all time. Let $\mathbb{S}^2_\pm = \mathbb{S}^2 \cap \mathbb{R}^3_{\pm} = \{(v_1,v_2,\omega) \in \mathbb{S}^2 : \pm v_2 > 0\}$. 

Let $(y,w) \in \mathbf{P} \times \mathbb{S}^2_+$, and suppose the initial state of the point particle is $(y,-w)$. The velocity $-w$ is directed toward the boundary $\partial \mathcal{M}$. Therefore, the point particle will reflect from $\partial \mathcal{M}$ specularly a certain number of times before eventually returning to a point $y'$ in the plane $\mathbf{P}$ with some velocity $w' \in \mathbb{S}^2_+$. The \textit{collision law} is the map $\mathbf{P} \times \mathbb{S}^2_+ \to \mathbf{P} \times \mathbb{S}^2_+$, defined by
\begin{equation}
K^{\Sigma,\epsilon}(y,w) = (y',w').
\end{equation}
The set of inputs $(y,w)$ such that the billiard trajectory starting from $(y,-w)$ cannot be continued for all time (say because it hits a singular point of $\partial \mathcal{M}$) or never returns to $\mathbf{P}$ constitutes a measure zero subset of $\mathbf{P} \times \mathbb{S}^2_+$, and thus $K^{\Sigma,\epsilon}$ is well-defined up to null sets. This fact is rigorously proved in \S\ref{ssec_collaws}.

\begin{remark} \normalfont
The reader should be careful not to confuse the terms ``collision law'' and ``reflection law.'' In our usage, the latter expression is a generic term for any rule whereby the particle trajectory is redirected into the billiard domain at the instant in time when it hits the wall (e.g. the specular reflection law). The collision law, by contrast, is an analogue of the macro-reflection law defined in \S\ref{sssec_uphalf}. In fact, we will see that the collision law is a special case of the general macro-reflection laws defined in \S\ref{sssec_detref}.
\end{remark}

Just as in the random reflections case, the map $K^{\Sigma,\epsilon}$ is naturally associated with the deterministic Markov kernel on $\mathbf{P} \times \mathbb{S}^2_+$,
\begin{equation} \label{eq1.17}
    \mathbb{K}^{\Sigma,\epsilon}(y, w; \dd y'\dd w') := \delta_{K^{\Sigma,\epsilon}(y,w)}(y',w') \dd y'\dd w',
\end{equation}
We equip $\mathbf{P} \times \mathbb{S}^2_+$ with the well-known measure from billiards theory,
\begin{equation} \label{eq1.18}
    \Lambda^2(\dd y\dd w) := \langle w, e_2 \rangle \dd y\sigma(\dd w),
\end{equation}
where $\dd y$ denotes Lebesgue measure on $\mathbf{P}$ and $\sigma(\dd w)$ denotes surface measure on $\mathbb{S}^2$. This is the 3-dimensional analogue of the Lambertian measure $\Lambda^1$, defined by (\ref{eq1.5}).

\begin{definition} \normalfont
We call a Markov kernel $\mathbb{K}$ on $\mathbf{P} \times \mathbb{S}^2_+$ a \textit{rough collision law} if there exists a sequence of cells $\Sigma_i$ (satisfying conditions B1-B5 of \S\ref{sssec_periodic}) and a sequence of positive numbers $\epsilon_i \to 0$ such that, 
\begin{equation} \label{eq1.49}
    \mathbb{K}^{\Sigma_i,\epsilon_i}(y,w; \dd y' \dd w') \Lambda^2(\dd y \dd w) \to \mathbb{K}(y,w; \dd y' \dd w') \Lambda^2(\dd y \dd w)
\end{equation}
weakly in the space of measures on $\mathbf{P} \times \mathbb{S}^2_+$.
\end{definition}

\begin{remark} \normalfont
Similar comments to Remarks \ref{rem_replacelambda1} and \ref{rem_verify1} apply. In the above definition, it is equivalent to replace $\Lambda^2$ with any measure which is mutually absolutely continuous with respect to surface measure on $\mathbf{P} \times \mathbb{S}^2_+$. The convergence (\ref{eq1.49}) is in duality to functions $h(y,w,y',w') \in C_c((\mathbf{P} \times \mathbb{S}^2_+)^2)$, but by a density argument it is sufficient to verify the convergence in duality to functions of form $h(y,w,y',w') = f(y,w)g(y',w')$ where $f,g \in C_c^{\infty}(\mathbf{P} \times \mathbb{S}^2_+)$.
\end{remark}

\begin{remark} \label{rem_conv_sense} \normalfont
In (\ref{eq1.49}), there is a sense in which $\mathbb{K}^{\Sigma_i,\epsilon_i}$ converges to $\mathbb{K}$ as a limit with respect to a pseudometric topology on the space of Markov kernels on $\mathbf{P} \times \mathbb{S}^2_+$. This topology is described in \S\ref{sssec_pseudo}. 
\end{remark}

From this point on, we will write $\lim_{i \to \infty} \mathbb{K}^{\Sigma_i,\epsilon_i} = \mathbb{K}$ to indicate that (\ref{eq1.49}) holds.

The most important elementary properties of collision laws are summarized in the following proposition. (Compare with Propositions \ref{prop_det_ref_properties} and \ref{prop_dualref}.)

\begin{proposition} \label{prop_dualityetc}
(i) There exists a full measure open set $\mathcal{F} \subset \mathbf{P} \times \mathbb{S}^2_+$ such that $K^{\Sigma,\epsilon} : \mathcal{F} \to \mathcal{F}$ is a $C^2$ diffeomorphism and $K^{\Sigma,\epsilon} \circ K^{\Sigma,\epsilon} = \text{Id}_{\mathcal{F}}$ on $\mathcal{F}$. 

Moreover, $K^{\Sigma,\epsilon}$ preserves the measure $\Lambda^2$, in the sense that for any set $A \subset \mathbf{P} \times \mathbb{S}^2_+$, $\Lambda^2((K^{\Sigma,\epsilon})^{-1}(A)) = \Lambda^2(A)$.

(ii) A rough collision law $\mathbb{K}$ is symmetric with respect to $\Lambda^2$, in the sense that for any $f,g \in C_c(\mathbf{P} \times \mathbb{S}^2_+)$, 
\begin{equation} \label{eq1.20}
\begin{split}
    \int_{(\mathbf{P} \times \mathbb{S}^2_+)^2} & g(y',w')\mathbb{K}(y,w; \dd y' \dd w')f(y,w)\Lambda^2(\dd y \dd w) \\
    & = \int_{(\mathbf{P} \times \mathbb{S}^2_+)^2} f(y',w')\mathbb{K}(y,w; \dd y' \dd w')g(y,w)\Lambda^2(\dd y \dd w).
\end{split}
\end{equation}
Consequently, $\mathbb{K}$ preserves $\Lambda^2$, in the sense that 
\begin{equation} \label{eq1.21}
\int_{(\mathbf{P} \times \mathbb{S}^2_+)^2} g(y',w')\mathbb{K}(y,w; \dd y' \dd w')\Lambda^2(\dd y \dd w) \\
 = \int_{\mathbf{P} \times \mathbb{S}^2_+} g(y,w)\Lambda^2(\dd y \dd w).  
\end{equation}
\end{proposition}

This result is proved in \S\ref{ssec_collaws}.

\subsubsection{Coordinates on $\mathbf{P} \times \mathbb{S}^2_+$} \label{sssec_tilted}

Recall the cylindrical axis $\widetilde{\chi}$, defined by (\ref{eq1.42}). We equip $\mathbb{R}^3 \times \mathbb{S}^2$ with coordinates $(y_1,y_2, y_3,\theta,\psi)$ defined as follows. Let 
\begin{equation}
\chi = ||\widetilde{\chi}||^{-1}\widetilde{\chi} = (m + J)^{-1/2}(1,0,-1), \quad\quad \chi^\perp = (m^{-1} + J^{-1})^{-1/2}(m^{-1},0,J^{-1}).
\end{equation}
Also let
\begin{equation}
    \widehat{e}_2 = m^{-1/2}e_2 = (0,m^{-1/2},0).
\end{equation}
Then $(\chi^\perp, \widehat{e}_2, \chi)$ is an orthonormal basis for $\mathbb{R}^3$ (with respect to the inner product $\langle \cdot,\cdot \rangle$). The two vectors $\chi$ and $\chi^\perp$ span $\mathbf{P}$. We define a coordinate map $G : \mathbb{R}^3 \times (0,\pi)^2 \to \mathbb{R}^3 \times \mathbb{S}^2$ by 
\begin{equation}
    G : (y_1,y_2,y_3,\theta,\psi) \mapsto (y,w),
\end{equation}
where 
\begin{equation} \label{eq1.27'}
    \begin{split}
        y & = y_1 \chi^\perp + y_2 \widehat{e}_2 + y_3\chi, \\
        w & = (\cos\theta \sin\psi) \chi^\perp + (\sin\theta\sin\psi) \widehat{e}_2 + (\cos\psi) \chi.
    \end{split}
\end{equation}
In other words, the spatial coordinates $(y_1,y_2,y_3)$ are obtained by rotating (with respect to the kinetic energy inner product) the original coordinates $(x_1,x_2,\alpha)$ so that the vertical axis coincides with $\chi$. The velocity coordinates $(\theta,\psi)$ are just spherical coordinates with the ``north pole'' at $\chi$.

In these coordinates, $\mathbf{P} = \{(y_1,y_2,y_3) : y_2 = 0\}$. The invariant measure $\Lambda^2$ has the following coordinate representation:
\begin{equation} \label{eq1.28'}
    \Lambda^2(\dd y_1 \dd y_3 \dd\theta \dd\psi) = \sin\theta\sin^2\psi \dd y_1 \dd y_3 \dd\theta \dd\psi. 
\end{equation}
To see this, note that $\langle w, e_2 \rangle = \sin\theta\sin\psi$ and the spherical volume form is $\sin\psi \dd\theta \dd\psi$.

\subsubsection{Main results} \label{sssec_main}

The first of our main results concerns the case in which the rough collision law is obtained through \textit{pure scaling}. This means that the sequence of cells $\Sigma_i = \Sigma$ is constant.

\begin{theorem} \label{thm_pure_scaling}
Consider a constant sequence of cells $\Sigma_i = \Sigma$ for $i \geq 1$. There exists a Markov kernel $\mathbb{K}$ such that for any sequence of positive numbers $\epsilon_i \to 0$, the limit $\lim_{i \to \infty} \mathbb{K}^{\Sigma_i,\epsilon_i}$ exists and is equal to $\mathbb{K}$. Moreover, $\mathbb{K}$ takes the form
\begin{equation} \label{rough_col0}
\mathbb{K}(y_1,y_3,\theta,\psi ; \dd y_1' \dd y_3' \dd\theta' \dd\psi') = \delta_{(y_1,y_3)}(\dd y_1'\dd y_3') \widetilde{\mathbb{P}}(\theta,\dd\theta') \delta_{\pi - \psi}(\dd\psi'),
\end{equation}
where $\widetilde{\mathbb{P}}$ is a Markov kernel on $\mathbb{S}^1_+$ satisfying the following properties:
\begin{enumerate}[label = \roman*.]
    \item $\widetilde{\mathbb{P}}$ is symmetric with respect to the measure $\sin\theta \dd\theta$, in the sense of (\ref{eq1.18'})
    \item Let 
    \begin{equation}
    \begin{split}
        &\widetilde{\Sigma} = \{(x_1,x_2) \in \mathbb{S}^1 \times \mathbb{R} : (x_1,(1 + mJ^{-1})^{-1/2}x_2) \in \Sigma\}, \\
        &\widetilde{\epsilon}_i = (1 + mJ^{-1})^{-1/2}\epsilon_i, \\
        &\mathbb{P} = \delta_y(\dd y') \widetilde{\mathbb{P}}(\theta, \dd\theta').
    \end{split}
    \end{equation}
    Then 
    \begin{equation}
        \mathbb{P} = \lim_{i \to \infty} \mathbb{P}^{\widetilde{\Sigma},\widetilde{\epsilon}_i}.
    \end{equation}
\end{enumerate}
Consequently, $\mathbb{K}$ is symmetric with respect to the measure $\Lambda^2$.
\end{theorem}

Our next two results concern more general rough collision laws. For an arbitrary sequence of cells $\{\Sigma_i\}_{i \geq 1}$, there is no guarantee that the limit of $\mathbb{K}^{\Sigma_i,\epsilon_i}$ exists or is uniquely determined by the sequence of cells $\Sigma_i$. Nonetheless, if $\epsilon_i \to 0$ sufficiently fast (where the rate depends on the sequence of cells $\Sigma_i$), a strict dichotomy holds.

\begin{theorem} \label{thm_dichotomy}
For any sequence of cells $\{\Sigma_i\}_{i \geq 1}$, there exist numbers $b_1 \geq b_2 \geq b_3 \geq \cdots > 0$ such that exactly one of the following is true:
\begin{enumerate}[label = (\Alph*)]
    \item There exists a unique Markov kernel $\mathbb{K}$ such that for any sequence of positive numbers $\epsilon_i \leq b_i$ with $\epsilon_i \to 0$, 
    \begin{equation} \label{eq1.27}
        \lim_{i \to \infty} \mathbb{K}^{\Sigma_i,\epsilon_i} = \mathbb{K}.
    \end{equation}
    \item For any sequence of positive numbers $\epsilon_i \leq b_i$ with $\epsilon_i \to 0$, $\lim_{i \to \infty} \mathbb{K}^{\Sigma_i,\epsilon_i}$ does not exist.
\end{enumerate}
\end{theorem}

\begin{remark} \normalfont
Both possibilities in the dichotomy are realized. In  \S\ref{sec_examples} we will construct a few different examples of rough collision laws. Denote two of these by $\mathbb{K} = \lim \mathbb{K}^{\Sigma,\epsilon_i}$ and $\mathbb{K}' = \lim \mathbb{K}^{\Sigma',\epsilon_i}$ where $\Sigma$ and $\Sigma'$ are fixed cells and $\mathbb{K} \neq \mathbb{K}'$ (where inequality means that $\mathbb{K}(\cdot ; \dd y' \dd w')$ and $\mathbb{K}'(\cdot ; \dd y' \dd w')$ disagree on a non-null set). Note that by Theorem \ref{thm_pure_scaling}, these limits always exist and do not depend on the sequence $\epsilon_i \to 0$, and thus we may take $b_i = \infty$ for example. Hence we see that (A) is realized. Define a sequence of cells by $\Sigma_{2i - 1} = \Sigma$ and $\Sigma_{2i} = \Sigma'$ for $i \geq 1$. Then the sequence $\mathbb{K}^{\Sigma_i,\epsilon_i}$ has two distinct limit points $\mathbb{K}$ and $\mathbb{K}'$ for any sequence $\epsilon_i \to 0$, and thus (B) is realized.
\end{remark}

\begin{remark} \normalfont
The proof of this theorem depends on the Poincar\'{e} Recurrence Theorem. As such, it does not yield quantitative estimates for the numbers $b_i$. 
\end{remark}

We let $\mathcal{A}_0$ denote the set of all Markov kernels $\mathbb{K}$ obtained as a limit of form $\lim_{i \to \infty} \mathbb{K}^{\Sigma_i, \epsilon_i}$, where $\epsilon_i \leq b_i$ for all $i$, and the $b_i$ are chosen as in Theorem \ref{thm_dichotomy}. The following theorem says that the collision laws in $\mathcal{A}_0$ have essentially the same properties as those stated in Theorem \ref{thm_pure_scaling}. Moreover, these properties completely characterize the members of $\mathcal{A}_0$.

\begin{theorem} \label{thm_classification}
Let $\mathbb{K} \in \mathcal{A}_0$. In the coordinates $(y_1,y_3,\theta,\psi)$, $\mathbb{K}$ takes the form 
\begin{equation} \label{rough_col}
\mathbb{K}(y_1,y_3,\theta,\psi ; \dd y_1' \dd y_3' \dd\theta' \dd\psi') = \delta_{(y_1,y_3)}(\dd y_1'\dd y_3')\widetilde{\mathbb{P}}(\theta,\dd\theta')\delta_{\pi - \psi}(\dd\psi'),
\end{equation}
where $\widetilde{\mathbb{P}}$ is a Markov kernel on $\mathbb{S}^1_+$ satisfying the following properties:
\begin{enumerate}[label = \roman*.]
    \item $\widetilde{\mathbb{P}}$ is symmetric with respect to the measure $\sin\theta \dd\theta$ on $\mathbb{S}^1_+$, in the sense of (\ref{eq1.18'}).
    \item Suppose $\{\Sigma_i\}$ is a sequence of cells such that for any $\epsilon_i \leq b_i$ with $\epsilon_i \to 0$, $\mathbb{K}^{\Sigma_i,\epsilon_i} \to \mathbb{K}$ in the sense of Theorem \ref{thm_dichotomy}(A). Let 
    \begin{equation} \label{eq1.66}
    \begin{split}
        &\widetilde{\Sigma}_i = \{(x_1,x_2) \in \mathbb{S}^1 \times \mathbb{R} : (x_1,(1 + mJ^{-1})^{-1/2}x_2) \in \Sigma_i\}, \\
        &\widetilde{\epsilon}_i = (1 + mJ^{-1})^{-1/2}\epsilon_i, \\
        &\mathbb{P} = \delta_y(\dd y') \widetilde{\mathbb{P}}(\theta, \dd\theta').
    \end{split}
    \end{equation}
    Then 
    \begin{equation}
        \mathbb{P} = \lim_{i \to \infty} \mathbb{P}^{\widetilde{\Sigma}_i,\widetilde{\epsilon}_i}.
    \end{equation}
\end{enumerate}
Consequently, $\mathbb{K}$ is symmetric with respect to the measure $\Lambda^2$.

Conversely, if $\mathbb{K}$ is a Markov kernel on $\mathbf{P} \times \mathbb{S}^2_+$ of form (\ref{rough_col}) such that $\widetilde{\mathbb{P}}$ is symmetric with respect to the measure $\sin\theta \dd\theta$ on $\mathbb{S}^1_+$, then $\mathbb{K} \in \mathcal{A}_0$.
\end{theorem}

Theorems \ref{thm_pure_scaling} and \ref{thm_classification} establish a one-to-one correspondence between rough collision laws and rough reflection laws in the upper half-plane $\{(x_1,x_2) : x_2 \geq 0\}$, indicated schematically as follows:
\begin{equation} \label{eq1.68}
\begin{tikzcd}
\{(\Sigma_i,\epsilon_i)\}_{i \geq 1} \arrow[r] \arrow[d] & \{(\widetilde{\Sigma}_i, \widetilde{\epsilon_i})\}_{i \geq 1} \arrow[l] \arrow[d] \\
\mathbb{K} \arrow[r] & \mathbb{P} \arrow[l]
\end{tikzcd}
\end{equation}
The correspondence at the top is between the sequences of cell-roughness scale pairs giving rise to the walls $W_i = W(\Sigma_i,\epsilon_i)$ and foreshortened walls $\widetilde{W}_i = W(\widetilde{\Sigma}_i,\widetilde{\epsilon}_i)$ respectively. The correspondence on the bottom is given by (\ref{rough_col}). The downward arrows map the sequence $\{(\Sigma_i,\epsilon_i)\}$ (resp. $\{(\widetilde{\Sigma}_i,\widetilde{\epsilon}_i)\}$) to $\lim_{i \to \infty} \mathbb{K}^{\Sigma_i,\epsilon_i}$ (resp. $\lim_{i \to \infty} \mathbb{P}^{\widetilde{\Sigma}_i,\widetilde{\epsilon}_i}$). Provided $\epsilon_i \to 0$ sufficiently fast, the limit on the left exists if and only if the limit on the right exists. We will take advantage of this correspondence in \S\ref{sec_examples} for building examples of rough collision laws.

The reader may wish to compare the statement of Theorem \ref{thm_classification}, above, to the informal description of our main results in \S\ref{sssec_minformal}. The fact that $\mathbb{K}$ preserves the measure $\Lambda^2$ is equivalent to the fact that the rough collision dynamics preserve the Liouville measure on the phase space (one approach for proving this is suggested by Figure \ref{invariant_fig}), while symmetry with respect to $\Lambda^2$ makes precise the notion of ``time-reversibility.'' The product decomposition (\ref{rough_col}) implies that projection of the phase space velocity onto $\chi$ is a conserved quantity.

\subsubsection{Outline of proof of main results} \label{sssec_prfoutline}

The proofs of our main results are given in \S\ref{sec_main_results}. Here we provide a high level synopsis of our arguments.

\textit{Step 1.} The first step will be to show that versions of Theorems \ref{thm_pure_scaling}, \ref{thm_dichotomy}, \ref{thm_classification} hold if we replace the configuration space $\mathcal{M}$ with its cylindrical approximation $\mathcal{M}_{\cyl}$. (More precisely, see Theorem \ref{thm_cylindrical}.)

Consider a point particle moving freely in $\mathcal{M}_{\cyl}$ and reflecting specularly from $\partial \mathcal{M}_{\cyl}$. The boundary of $\mathcal{M}_{\cyl}$ lies just below the plane $\mathbf{P}$, within $\epsilon$ distance from the plane. In analogy to the collision law defined for the space $\mathcal{M}$, we define the \textit{cylindrical collision law} $K_{\cyl}^{\Sigma,\epsilon} : \mathbf{P} \times\mathbb{S}^2_+ \to \mathbf{P} \times \mathbb{S}^2_+$ by
\begin{equation}
    K_{\cyl}^{\Sigma,\epsilon} : (y,w) \mapsto (y',w,'),
\end{equation}
where $(y',w')$ is the state of the freely moving point particle upon its first return to $\mathbf{P}$, after reflecting from $\partial \mathcal{M}_{\cyl}$ some number of times. We then consider Markov kernels $\mathbb{K}$ on $\mathbf{P} \times \mathbb{S}^2_+$ such that for some sequence of cells $\{\Sigma_i\}$ and positive numbers $\epsilon_i \to 0$,
\begin{equation} \label{eq1.72}
    \delta_{K_{\cyl}^{\Sigma_i,\epsilon_i}(y,w)}(\dd y' \dd w') \Lambda^2(\dd y' \dd w') \to \mathbb{K}(y,w; \dd y' \dd w')\Lambda^2(\dd y' \dd w'),
\end{equation}
weakly in the space of measures on $(\mathbf{P} \times \mathbb{S}^2_+)^2$. 

To prove that such a Markov kernel $\mathbb{K}$ takes the form (\ref{rough_col}), the key observation is that the billiard trajectory $y(t)$ in $\mathcal{M}_{\cyl}$ decouples into two independent evolutions: $y(t) = (y_{1,2}(t), y_3(t))$, where the evolution $y_3(t)$ is the projection of $y(t)$ onto the cylindrical axis $\chi$, and the evolution $y_{1,2}(t)$ is the projection of $y(t)$ onto the orthogonal complement of $\chi$, $\mathbf{Q}_1 := \{y \in \mathbb{R}^3 : \langle y, \chi \rangle = 0\}$. By virtue of the cylindrical structure of $\mathcal{M}_{\cyl}$, $y_3(t)$ evolves linearly, with constant velocity for all time, while $y_{1,2}(t)$ follows the trajectory of a point particle moving freely in $\mathbf{Q}_1 \smallsetminus \mathcal{M}_{\cyl}$ and reflecting specularly from the boundary $\mathbf{Q}_1 \cap \partial \mathcal{M}_{\cyl}$. After appropriately identifying $\mathbf{Q}_1$ with $\mathbb{R}^2$, this billiard domain in $\mathbf{Q}_1$ may be shown to coincide with the complement of the foreshortened wall $\widetilde{W}_i = W(\widetilde{\Sigma}_i, \widetilde{\epsilon}_i)$. This accounts for the non-trivial factor $\widetilde{\mathbb{P}}$ in (\ref{rough_col}).

On the other hand, if $\mathbb{K}$ is a Markov kernel on $\mathbf{P} \times \mathbb{S}^2_+$ of form (\ref{rough_col}), then one may show that for some sequence of cells $\{\Sigma_i\}$ and positive numbers $\epsilon_i \to 0$, a limit of form (\ref{eq1.72}) holds. The idea is to apply the characterization of rough reflection laws in the upper half-plane given by Theorem \ref{thm_rough_ref_char} to show that there exist $\Sigma_i$ and $\epsilon_i$ such that the limiting rough reflection law on $\widetilde{W}_i = W(\widetilde{\Sigma}_i, \widetilde{\epsilon}_i)$ is $\widetilde{\mathbb{P}}$ in (\ref{rough_col}). 

For more details on the cylindrical configuration space and the cylindrical collision law, see \S\ref{ssec_cylconfig} and \S\ref{sssec_cylcol}. The argument in this step is presented in \S\ref{ssec_cyl}.

\textit{Step 2.} Most of the work involved with proving our main results is concerned with the case of pure scaling -- that is, the case where the sequence of cells $\Sigma_i = \Sigma$ is constant. We will prove that, for any $f, g \in C_c^\infty(\mathbf{P} \times \mathbb{S}^2_+)$,  
\begin{equation} \label{eq1.73}
    \int_{\mathbf{P} \times \mathbb{S}^2_+} g(y,w) [f \circ K^{\Sigma,\epsilon}(y,w) - f \circ K^{\Sigma,\epsilon}_{\cyl}(y,w)] \dd\Lambda^2(\dd y \dd w) \to 0 \quad \text{ as } \epsilon \to 0.
\end{equation}
This is the statement of Lemma \ref{lem_comparison}. The above limit together with the previous step can be used to prove Theorem \ref{thm_pure_scaling}.

To obtain (\ref{eq1.73}), we introduce a \textit{modified collision law} $\widetilde{K}^{\Sigma,\epsilon}$ to which we may compare both $K^{\Sigma,\epsilon}$ and $K^{\Sigma,\epsilon}_{\cyl}$. The modified collision law is defined on a large (but not full measure) open subset $\widetilde{\mathcal{F}} \subset \mathbf{P} \times \mathbb{S}^2_+$, and satisfies
\begin{equation}
    \widetilde{K}^{\Sigma,\epsilon} = \eta^{-1} \circ K^{\Sigma,\epsilon} \circ \eta,
\end{equation}
where $\eta : \widetilde{\mathcal{F}} \to \widetilde{\mathcal{F}}$ is a smooth perturbation which ``corrects'' for the large-scale spherical shape of the body $D$. The map $\eta$ is defined in \S\ref{sssec_modcol}. We will see that $\eta$ converges to $\text{Id}$ in $C^1(\widetilde{\mathcal{F}})$ as $\epsilon \to 0$. The comparison between $\widetilde{K}^{\Sigma,\epsilon}$ and $K^{\Sigma,\epsilon}$ will be carried out in \S\ref{ssec_pure_scaling} by making estimates on the differential of $\eta$.

The comparison between $\widetilde{K}^{\Sigma,\epsilon}$ and $K_{\cyl}^{\Sigma,\epsilon}$ will be carried out in \S\ref{ssec_zoom} via a ``zooming argument.'' By double-periodicity of the configuration spaces $\mathcal{M}$ and $\mathcal{M}_{\cyl}$, it is enough to compare the behavior of  $\widetilde{K}^{\Sigma,\epsilon}$ and $K_{\cyl}^{\Sigma,\epsilon}$ on a single parallelogram
\begin{equation}
    R_\epsilon = \{(x_1,\alpha) \in \mathbf{P} : 0 \leq x_1 + \alpha \leq \epsilon, -\rho/2 \leq x_2 \leq \rho/2\}.
\end{equation}
We will see that there is a large subset $\Omega \subset \mathbf{P} \times \mathbb{S}^2_+$ such that 
\begin{equation} \label{eq1.76}
    \sup\left\{||\widetilde{K}^{\Sigma,\epsilon}(y,w) - K_{\cyl}^{\Sigma,\epsilon}(y,w)|| : (y,w) \in \Omega \cap (R_\epsilon \times \mathbb{S}^2_+)\right\} \to 0, \text{ as } \epsilon \to 0.
\end{equation}
This is the content of Lemma \ref{lem_mod}. The comparison is best carried out in ``zoomed'' coordinates. That is, we let
\begin{equation}
    \sigma_{\epsilon^{-1}}(y,w) = (\epsilon^{-1}y,w), \quad\quad (y,w) \in \mathbf{P} \times \mathbb{S}^2_+,
\end{equation}
and we compare the two maps 
\begin{equation}
    \widetilde{K}^*: = \sigma_{\epsilon^{-1}} \circ \widetilde{K}^{\Sigma,\epsilon} \circ \sigma_{\epsilon}, \quad\quad K_{\cyl}^*: = \sigma_{\epsilon^{-1}} \circ K_{\cyl}^{\Sigma,\epsilon} \circ \sigma_{\epsilon}.
\end{equation}
The advantage of this point of view can be seen by observing that the scaled cylindrical configuration space $\mathcal{M}^*_{\cyl} := \epsilon^{-1}\mathcal{M}_{\cyl}$ is simply the cylinder with base $\overline{W(\Sigma,1)^c} + e_2$ and axis $\chi$. Thus $\mathcal{M}_{\cyl}^*$ does not depend on $\epsilon$. This will allow us to control the billiard trajectories in the zoomed spaces $\mathcal{M}^* := \epsilon^{-1}\mathcal{M}$ and $\mathcal{M}^*_{\cyl}$ in terms of properties of their ``projections'' onto the fixed cylindrical base $\overline{W(\Sigma,1)^c} + e_2$. This idea is fleshed out in Lemmas \ref{lem_control} and \ref{lem_control2}.

\textit{Step 3.} Finally, to prove Theorems \ref{thm_dichotomy} and \ref{thm_classification}, we will apply the previous two steps and take advantage of the fact that the convergence (\ref{eq1.49}) comes from a pseudometric $d_{\mathcal{G}}^{\Lambda^2}$ on the space of Markov kernels on $\mathbf{P} \times \mathbb{S}^2_+$. (For the definition of this pseudometric, see \S\ref{sssec_pseudo}.) For any $\Sigma_i$ and $\epsilon_i$, the triangle inequality gives us
\begin{equation} \label{eq1.74}
    d_{\mathcal{G}}^{\Lambda^2}(\mathbb{K}, \mathbb{K}^{\Sigma_i,\epsilon_i}) \leq d_{\mathcal{G}}^{\Lambda^2}(\mathbb{K}, \mathbb{K}^{\Sigma_i,\epsilon_i}_{\cyl}) + d_{\mathcal{G}}^{\Lambda^2}(\mathbb{K}^{\Sigma_i,\epsilon_i}_{\cyl}, \mathbb{K}^{\Sigma_i,\epsilon_i}),
\end{equation}
where $\mathbb{K}^{\Sigma_i,\epsilon_i}_{\cyl}(y,w; \dd y' \dd w') := \delta_{K^{\Sigma_i,\epsilon_i}_{\cyl}(y,w)}(\dd y' \dd w')$. The convergence (\ref{eq1.73}) implies that there exist positive numbers $b_i^{(0)}$ such that if $\epsilon_i \leq b_i^{(0)}$ for all $i$, then the second term in the right-hand side of (\ref{eq1.74}) is negligible. Consequently, Theorems \ref{thm_dichotomy} and \ref{thm_classification} will follow by applying the corresponding results obtained for the cylindrical configuration space in Step 1.

\section{Examples} \label{sec_examples}

Here we construct examples of rough reflection laws and rough collision laws. Most of the work goes into building rough reflection laws. For each rough reflection law we construct, the correspondence (\ref{eq1.68}) gives us a rough collision law ``for free.''

\subsection{Lemma for constructing rough reflections} \label{ssec_lemconstruct}

Throughout this section, we assume that the cells $\Sigma$ satisfy conditions B1-B5 of \S\ref{sssec_periodic}. 

Fix a cell $\Sigma$ and a positive number $\epsilon > 0$. Suppose that $(x,\theta) \in \mathbb{R} \times \mathbb{S}^1_+$, and let $(x',\theta')$ be the random variable in $\mathbb{R} \times \mathbb{S}^1_+$ whose law is given by $\mathbb{P}^{\Sigma,\epsilon}(x,\theta, \dd x' \dd\theta')$. We define $\widetilde{\mathbb{P}}^{\Sigma,\epsilon}(\theta,\dd\theta')$ to be the law of $\theta'$ as $x$ varies uniformly in the period $[0,\epsilon]$ and $\theta$ stays fixed. In other words, for any $f \in C_c(\mathbb{S}^1_+)$,
\begin{equation}
    \int_{\mathbb{S}^1_+} f(\theta') \widetilde{\mathbb{P}}^{\Sigma,\epsilon}(\theta,\dd\theta') := \frac{1}{\epsilon}\int_0^\epsilon \int_{\mathbb{R} \times \mathbb{S}^1_+} f(\theta')\mathbb{P}^{\Sigma,\epsilon}(x,\theta, \dd x' \dd\theta') \dd x.
\end{equation}

The following lemma gives us a way to construct rough reflection laws from a periodic microstructure.

\begin{lemma} \label{lem_construct}
(i) If the limit 
\begin{equation}
    \lim_{i \to \infty} \mathbb{P}^{\Sigma_i,\epsilon_i}(x,\theta; \dd x' \dd\theta')
\end{equation} 
exists, then the limit 
\begin{equation}
    \lim_{i \to \infty}\widetilde{\mathbb{P}}^{\Sigma_i,\epsilon_i}(\theta, \dd\theta')
\end{equation} 
exists. 

(ii) If the limit 
\begin{equation}
    \widetilde{\mathbb{P}}(\theta, \dd\theta') := \lim_{i \to \infty}\widetilde{\mathbb{P}}^{\Sigma_i,\epsilon_i}(\theta, \dd\theta')
\end{equation}
exists, then the limit 
\begin{equation}
    \lim_{i \to \infty} \mathbb{P}^{\Sigma_i,\epsilon_i}(x,\theta; \dd x' \dd\theta')
\end{equation}
exists and is equal to $\delta_{x}(\dd x')\widetilde{\mathbb{P}}(\theta, \dd\theta')$, where ``equal'' means the two Markov kernels belong to the same equivalence class (see Remark \ref{rem_convsense0}).

(iii) Suppose that $\Sigma_i = \Sigma$ is constant. Then for any sequence $\epsilon_i \to 0$, the limit $\lim_{i \to \infty} \mathbb{P}^{\Sigma, \epsilon_i}(x,\theta,\dd x' \dd\theta')$ exists and is equal to $\delta_x(\dd x') \widetilde{\mathbb{P}}(\theta, \dd\theta')$, where 
\begin{equation}
    \widetilde{\mathbb{P}}(\theta, \dd\theta') = \widetilde{\mathbb{P}}^{\Sigma, 1}(\theta, \dd\theta').
\end{equation}
\end{lemma}

The lemma is proved in \S\ref{ssec_proof}. The benefit of the lemma is that the law $\widetilde{\mathbb{P}}^{\Sigma_i, \epsilon_i}(\theta, \dd\theta')$ is usually much easier to compute than $\mathbb{P}^{\Sigma_i,\epsilon_i}(x,\theta; \dd x' \dd\theta')$, owing to the fact that in the former case we do not have to worry about the spatial variable $x'$.

\subsection{Examples of rough reflection laws}

\subsubsection{Rectangular teeth}

First we consider a microstructure of ``rectangular teeth.'' That is, we define real functions
\begin{equation}
    t_n(x) = \begin{cases}
    0 & \text{ if } 2k \epsilon_n \leq x \leq (2k + 1) \epsilon_n, \\
    -r\epsilon_n & \text{ if } (2k + 1) \epsilon_n < x < (2k+2)\epsilon_n  
    \end{cases} \quad \text{ for } k \in \mathbb{Z}.
\end{equation}
The quantity $r > 0$ is a fixed parameter representing the ratio of the height of the teeth to the width. Define $\epsilon_n$-periodic walls
\begin{equation}
    W_n = W(\Sigma,\epsilon_n) = \{(x_1,x_2) : x_2 \leq t_n(x_1)\}, 
\end{equation}

To find the macro-reflection law, let $X$ be uniform in $[0,1]$, and let $\theta \in (0,\pi)$ be fixed. Let $(X'_{\theta},\Theta'_{\theta}) = P^{\Sigma,1}(X,\theta)$. By Lemma \ref{lem_construct}(iii), the limiting reflection law $\widetilde{\mathbb{P}}(\theta, \dd\theta')$ is the law of the random variable $\Theta'_{\theta}$. 

As illustrated in Figure \ref{square_teeth_fig}, $\Theta'_{\theta}$ is equal to either $\pi - \theta$ (specular reflection) or $\theta$ (retroreflection). With probability 1/2 the starting point $X$ of the point particle is on top of a tooth, in which case $\Theta_\theta' = \pi - \theta$. Otherwise, $X$ will be in the interval above the crevice between two teeth. 

Conditioned on the latter event, the probability that $\Theta_\theta' = \pi - \theta$ may be determined by ``unfolding'' the rectangular billiard between the two teeth, as shown in Figure \ref{square_teeth3_fig}. By inspecting this figure, we see that if $\lfloor 2 r |\cot\theta| \rfloor$ is even, then the probability of specular reflection is $\{2 r |\cot\theta|\}$. If on the other hand $\lfloor 2 r |\cot\theta| \rfloor$ is odd, then the probability of specular reflection is $1 - \{2 r |\cot\theta|\}$.

Putting these observations together, we conclude that 
\begin{equation} \label{eq2.5}
    \Theta'_\theta = \begin{cases}
    \pi - \theta & \text{ w.p. } p_r(\theta), \\
    \theta & \text{ w.p. } 1 - p_r(\theta),
    \end{cases}
\end{equation}
where 
\begin{equation}
    p_r(\theta) = \begin{cases}
    \frac{1}{2} + \frac{1}{2}\{2r|\cot\theta|\} & \text{ if } \lfloor 2 r |\cot\theta| \rfloor \text{ is even}, \\
    1 - \frac{1}{2}\{2r|\cot\theta|\} & \text{ if } \lfloor 2 r |\cot\theta| \rfloor \text{ is odd}.
    \end{cases}
\end{equation}
Thus the limiting rough reflection law is $\mathbb{P}(x,\theta,\dd x' \dd\theta') = \delta_x(\dd x') \widetilde{\mathbb{P}}(\theta, \dd\theta')$, where $\widetilde{\mathbb{P}}(\theta, \dd\theta')$ is the law of $\Theta'_{\theta}$.

\begin{figure}
    \centering
    \includegraphics[width = 0.4\linewidth]{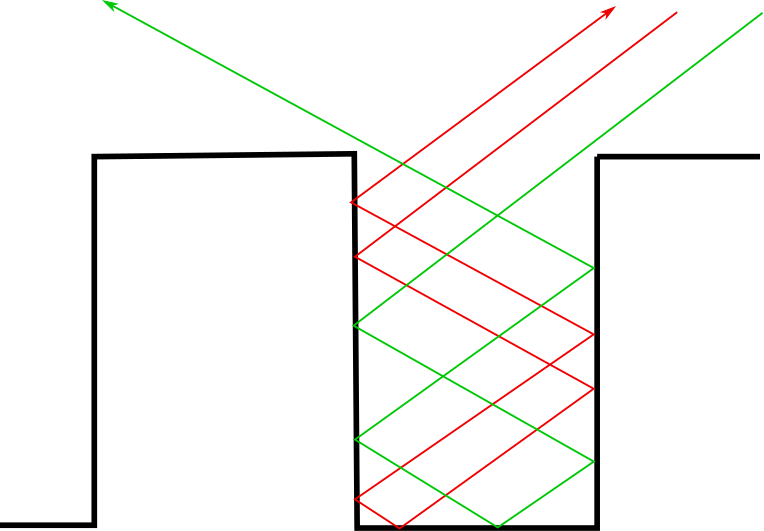}
    \caption{Reflection from rectangular teeth. The post-reflection angle will be either $\theta' = \pi - \theta$ (specular reflection -- green) or $\theta' = \theta$ (retroreflection -- red).}
    \label{square_teeth_fig}
\end{figure}

\begin{figure}
    \centering
    \includegraphics[width = 0.7\linewidth]{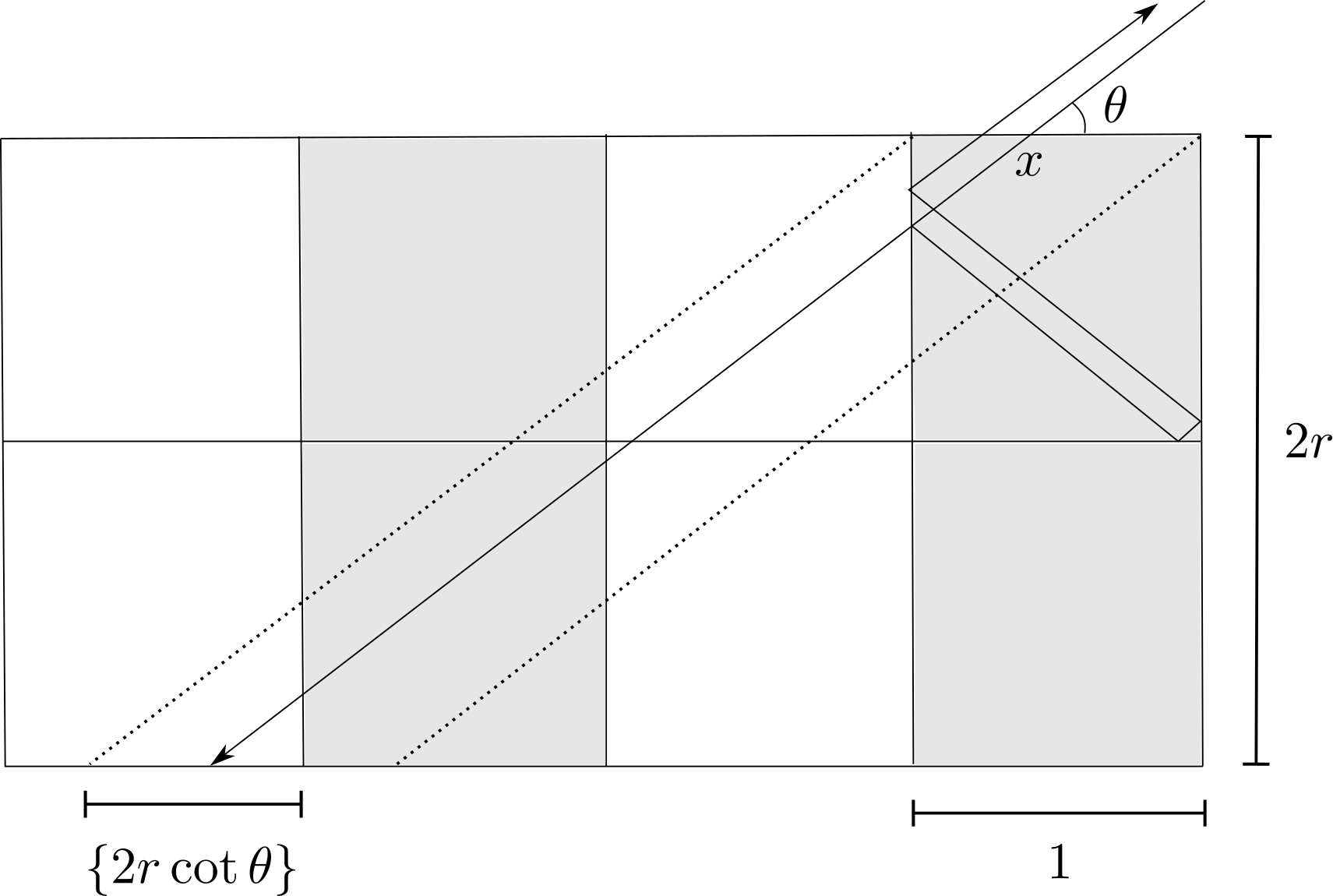}
    \caption{If the end of the trajectory of the point particle in the unfolded rectangular billiard falls in a gray region, then the reflection will be specular. Otherwise it will be a retroreflection.}
    \label{square_teeth3_fig}
\end{figure}

\subsubsection{Triangular teeth}

We can similarly construct reflections from triangular teeth. Let $\psi \in (0,\pi)$. Define a ``tooth function''
\begin{equation}
    t_n(x) = \begin{cases}
    -\cot(\frac{\psi}{2})x & \epsilon_n k \leq x < \epsilon_n(\frac{2k+1}{2}), \\
    \cot(\frac{\psi}{2})(x - 1) & \epsilon_n(\frac{2k+1}{2}) \leq x < \epsilon_n(k+1)
    \end{cases} \quad\quad \text{ for } k \in \mathbb{Z}.
\end{equation}
Then define $\epsilon_n$-periodic walls
\begin{equation}
    W_n = W(\Sigma_n,\epsilon_n) = \{(x_1,x_2) : x_2 \leq t_n(x_1)\}.
\end{equation}
Each period of the wall is an isosceles triangle such that each peak and valley spans an angle of $\psi$. See Figure \ref{triangle_teeth1_fig}. 

\begin{figure}
    \centering
    \includegraphics[width = 0.5\linewidth]{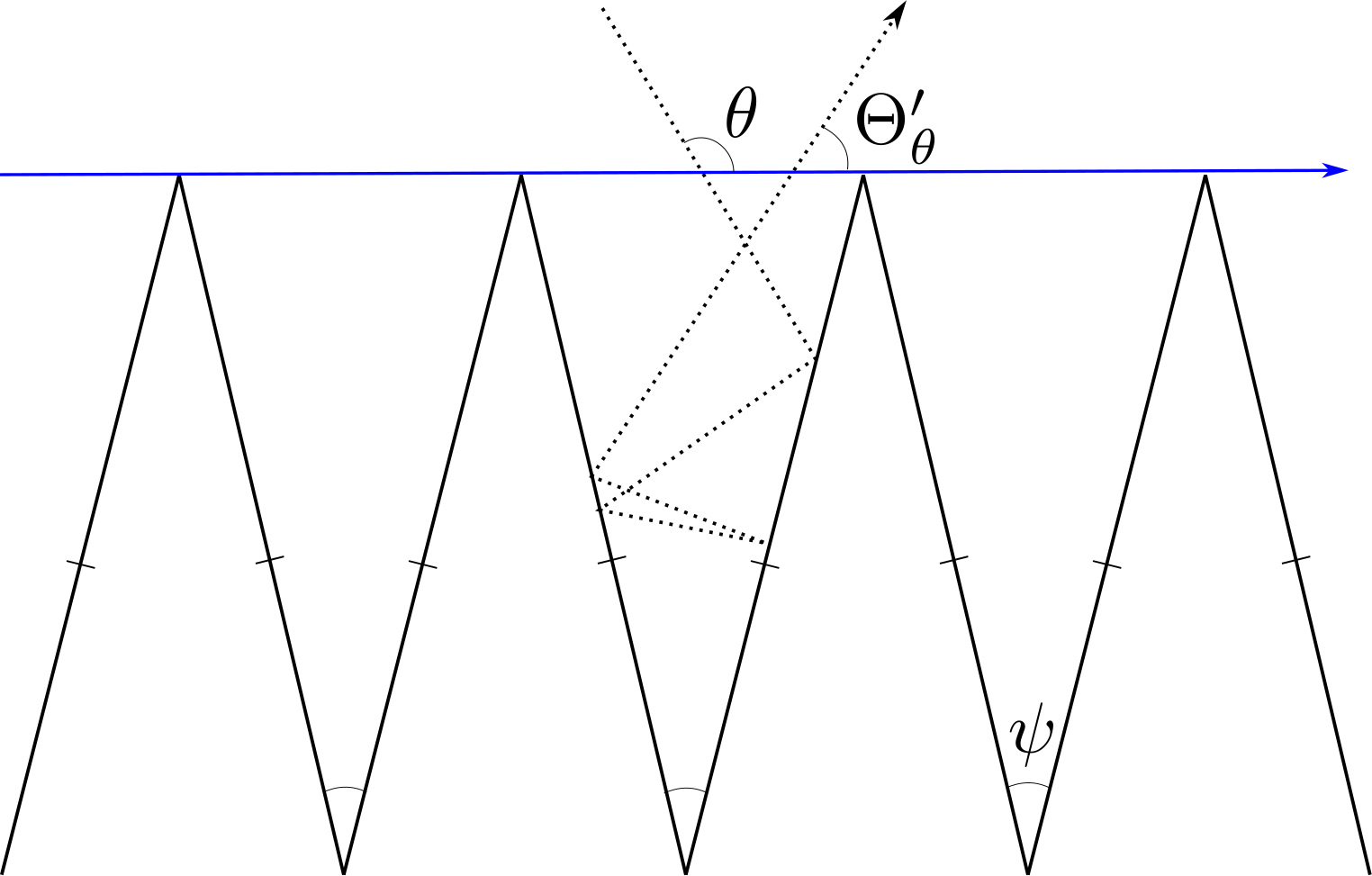}
    \caption{Triangular teeth}
    \label{triangle_teeth1_fig}
\end{figure}

As before, let $X$ be uniform in $[0,1]$, let $\theta \in (0,\pi)$ be fixed, and let $(X_\theta',\Theta_\theta') = P^{\Sigma,1}(X,\theta)$. The limiting reflection law $\widetilde{\mathbb{P}}(\theta, \dd\theta')$ is the law of $\Theta_\theta'$.

\begin{figure}
    \centering
    \includegraphics[width = \linewidth]{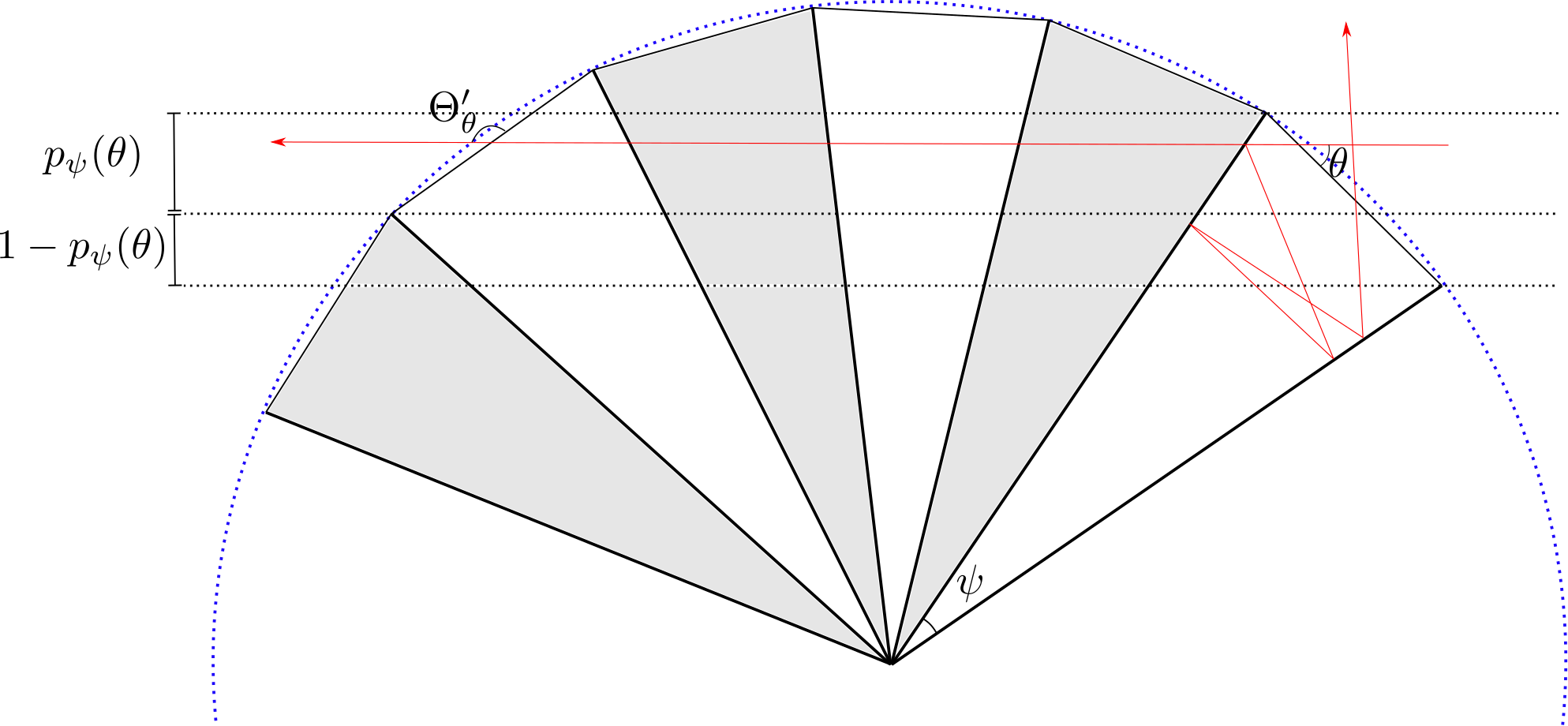}
    \caption{Unfolding of the billiard between two teeth.}
    \label{triangle_teeth2_fig}
\end{figure}

Similarly to the case of rectangular teeth, we can deduce the distribution of $\Theta_\theta'$ by considering the unfolding of the region between two triangular teeth. See Figure \ref{triangle_teeth2_fig}. In the figure, the number of triangles crossed by the top of the horizontal strip is
\begin{equation}
    N = \left\lceil \frac{2\theta}{\psi} \right\rceil.
\end{equation}
The angle of the outgoing trajectory depends on whether $N$ is even or odd and whether the unfolded trajectory exits the unfolded billiard in a white region or a gray region. The trajectory leaves in a white region with probability $p_\psi(\theta)$ and leaves in gray region with probability $1 - p_\psi(\theta)$.

It is tedious but elementary to compute the angle of exit in each case, and the probability $p_\psi(\theta)$. One obtains that 
\begin{equation} \label{eq2.10}
    \begin{split}
\text{if $\left\lceil \frac{2\theta}{\psi} \right\rceil$ is even, } \quad & \Theta_\theta' = \begin{cases}
\pi + \theta - \psi \lceil \frac{2\theta}{\psi} \rceil & \text{w.p. } p_\psi(\theta), \\
\psi \left(\lceil \frac{2\theta}{\psi} \rceil + 1\right) - \theta & \text{w.p. } 1 - p_\psi(\theta),
\end{cases} \\
\text{and if $\left\lceil \frac{2\theta}{\psi} \right\rceil$ is odd, } \quad & \Theta_\theta' = \begin{cases}
\psi \lceil \frac{2\theta}{\psi} \rceil - \theta & \text{w.p. } p_\psi(\theta), \\
\pi + \theta - \psi \left(\lceil \frac{2\theta}{\psi} \rceil + 1\right) - \theta & \text{w.p. } 1 - p_\psi(\theta),
\end{cases}
\end{split}
\end{equation}
where
\begin{equation}
    p_{\psi}(\theta) = \frac{\left[\cos(\theta - \frac{\psi}{2}) - \cos(\frac{\psi}{2}\lceil\frac{2\theta}{\psi}\rceil )\right]_+}{\cos(\theta - \frac{\psi}{2}) - \cos(\theta + \frac{\psi}{2})}.
\end{equation}
Here $[x]_+ = x$ if $x \geq 0$ and zero otherwise. The limiting reflection law is $\mathbb{P}(x,\theta;\dd x' \dd\theta') = \delta_{x}(\dd x')\widetilde{\mathbb{P}}(\theta, \dd\theta')$, where $\widetilde{\mathbb{P}}(\theta, \dd\theta')$ is the law of $\Theta_{\theta}'$.

\subsubsection{Focusing circular arcs} \label{sssec_circarc}

The discrete distribution of both of the previous reflection laws is an artifact of the polygonal boundary of the walls. When the boundary of the wall contains curve segments with non-zero curvature, we expect the distribution to be non-singular in general.

As an example of this, consider the wall whose periods consist of focusing (i.e. concave-up) circular arcs. Let $0 < \xi \leq \pi/2$. We define
\begin{equation}
    t_n(x) = \frac{\epsilon_n}{2}\cot\xi - \sqrt{\frac{\epsilon_n^2}{4}\csc^2\xi - \left(x - \frac{\epsilon_n}{2}\right)^2} \quad\quad \text{ for } x \in [0, \epsilon_n],
\end{equation}
and extend $t_n$ to be $\epsilon_n$-periodic on the line $\mathbb{R}$. The wall is 
\begin{equation}
    W_n = W(\Sigma, \epsilon_n) = \{(x_1,x_2) : x_2 \leq t_n(x_1)\}.
\end{equation}
The arc forming one period of the wall spans an angle of $2\xi$.

Consider the wall $W(\Sigma,1)$, depicted in Figure \ref{circ_arc_fig}. Let $X$ be uniform in $[0,1]$, and let $\theta \in (0,\pi)$ be fixed. Let $(X_\theta',\Theta_\theta') = P^{\Sigma,1}(X,\theta)$. Let $C$ denote the center of the circle whose arc forms the period of the wall below $[0,1]$, and let $R$ denote the radius. In coordinates,
\begin{equation}
    C = \left(\frac{1}{2}, \frac{1}{2}\cot\xi\right), \quad\quad R = \frac{1}{2}\csc\xi.
\end{equation}
Let $P$ be the first point in the circular arc hit by the billiard trajectory starting from $(X,\theta)$, and let $\gamma$ be the signed angle measured counterclockwise from the vector $-e_2$ to the ray $\overrightarrow{CP}$. (Thus if $X = 0$ then $\gamma = -\xi$ and if $X = 1$ then $\gamma = \xi$.)

\begin{figure}
    \centering
    \includegraphics[width = \linewidth]{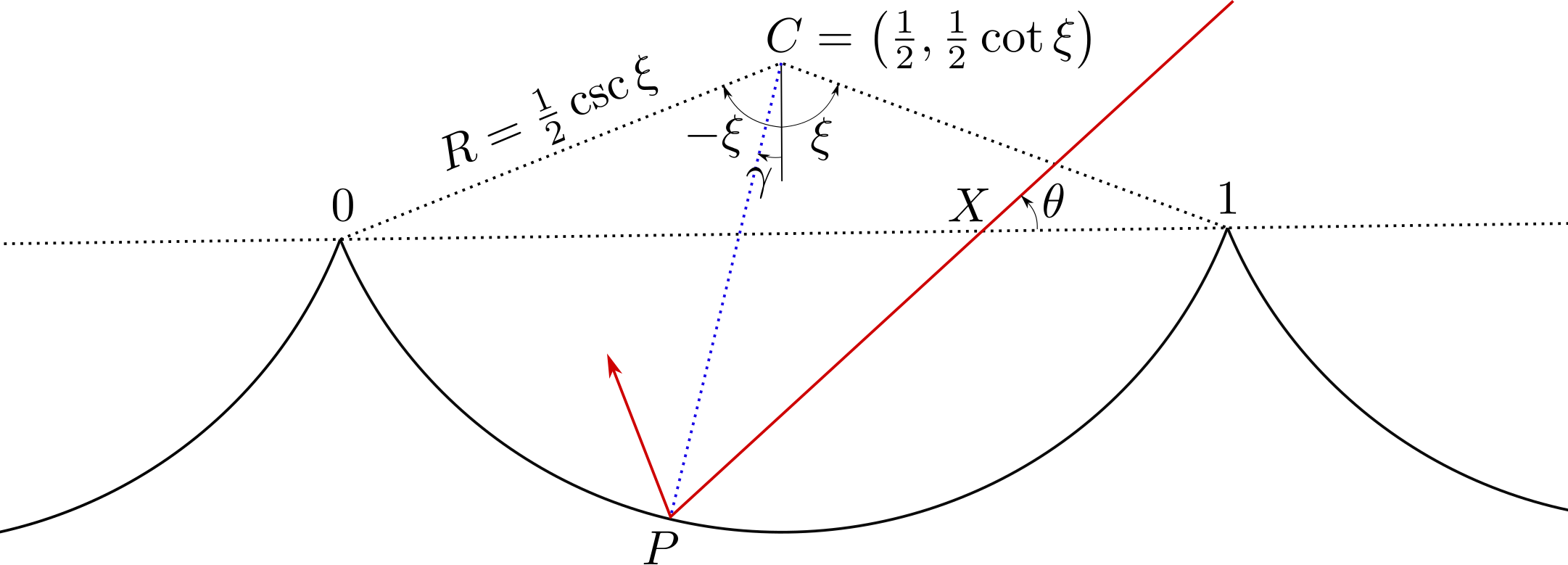}
    \caption{A wall formed from circular arcs.}
    \label{circ_arc_fig}
\end{figure}

It turns out to be more convenient to express $\Theta_\theta'$ in terms of $\gamma$, rather than $X$. 

Let $N$ denote the number of times that the point particle hits the circular arc before leaving the wall. Referring to Figure \ref{circ_arc2_fig}, let $P_0 = P$, and suppose that the point particle subsequently hits the wall at points $P_1, P_2, \dots, P_{N-1}$. The sequence of points $P_i$ will progress along the arc in the counterclockwise direction if $\theta \geq \frac{\pi}{2} + \gamma$, and in the clockwise direction if $\theta < \frac{\pi}{2} + \gamma$ (strictly speaking, the case where $\theta = \frac{\pi}{2} + \gamma$ is vacuous since the trajectory will hit the boundary only once). The signed angle from $P_{i-1}$ to $P_i$ is 
\begin{equation}
    \Delta \gamma = \begin{cases}
    2(\pi + \gamma - \theta) & \text{ if }  \theta \geq \frac{\pi}{2} + \gamma, \\
    2(\gamma - \theta) & \text{ if } \theta < \frac{\pi}{2} + \gamma.
    \end{cases}
\end{equation}
The number of times which the billiard trajectory hits the arc is 
\begin{equation}
    N = \begin{cases}
    \lceil \frac{\xi - \gamma}{|\Delta \gamma|} \rceil & \text{ if } \theta \geq \frac{\pi}{2} + \gamma, \\
    \lceil \frac{\xi + \gamma}{|\Delta \gamma|} \rceil & \text{ if } \theta < \frac{\pi}{2} + \gamma. 
    \end{cases}
\end{equation}
Each time the point particle hits the circular arc, the angle which the billiard trajectory makes with the $+x_1$-axis increments by $\Delta \gamma$. Thus, 
\begin{equation} \label{eq2.17}
    \Theta_\theta' = \theta + N\Delta\gamma = \begin{cases}
    \theta + \lceil \frac{\xi - \gamma}{2|\pi + \gamma - \theta|} \rceil 2(\pi + \gamma - \theta) & \text{ if } \theta \geq \frac{\pi}{2} + \gamma, \\
    \theta + \lceil \frac{\xi + \gamma}{2|\gamma - \theta|} \rceil 2(\gamma - \theta) & \text{ if } \theta < \frac{\pi}{2} + \gamma.
    \end{cases}
\end{equation}

\begin{figure}
    \centering
    \includegraphics[width = 0.8\linewidth]{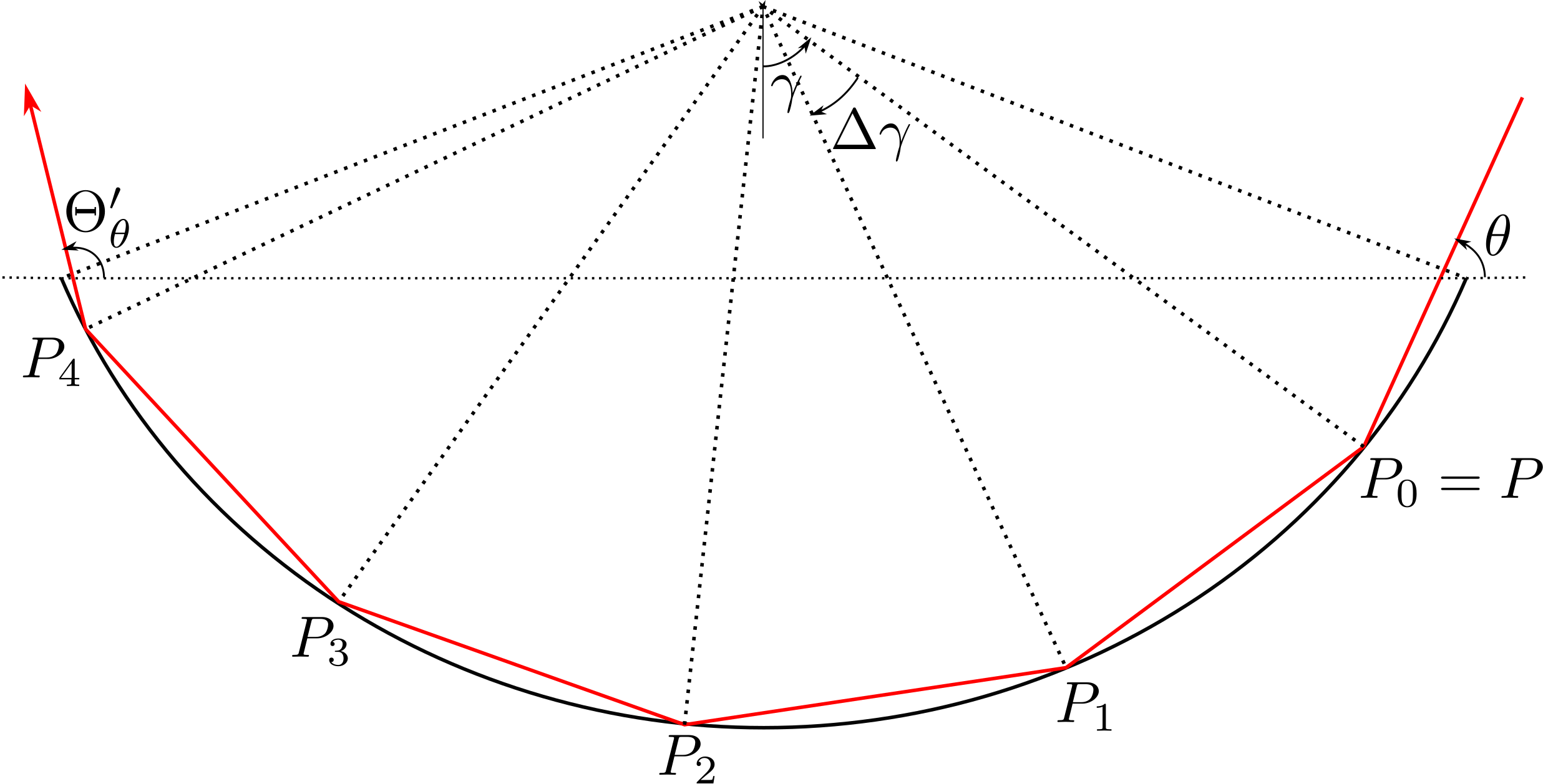}
    \caption{Multiple reflections from a circular arc.}
    \label{circ_arc2_fig}
\end{figure}

We will now find an explicit formula for $\gamma$ in terms of $X$ and $\theta$. This, together with (\ref{eq2.17}), will give us the distribution of $\Theta_\theta'$.

An elementary calculation shows that
\begin{equation} \label{eq2.18}
\begin{split}
    X & = \frac{\cot\theta\cos\gamma}{2\sin\xi} + \frac{\sin\gamma}{2\sin\xi} + \frac{1}{2} - \frac{1}{2}\cot\theta\cot\xi \\
    & = \frac{\cos(\theta + \gamma) - \cos(\theta + \xi)}{\cos(\theta - \xi) - \cos(\theta + \xi)}.
\end{split}
\end{equation}
Geometrically, it is clear that $X$ is uniquely determined by $\gamma$ and $\theta$ and that $X$ should increase with $\gamma$. Note that $-\pi/2 \leq -\xi \leq \theta + \gamma < \pi + \xi \leq 3\pi/2$. On the interval $[-\pi/2, 3\pi/2)$, the cosine function increases on the intervals $[-\pi/2,0]$ and $[\pi,3\pi/2)$. Thus $\theta + \gamma \in [-\pi/2,0] \cup [\pi,3\pi/2)$. Thus if we define
\begin{equation}
    \text{Arccos}(x) = \begin{cases}
        2\pi - \arccos(x) & \text{ if } -1 \leq x < 0, \\
        -\arccos(x) & \text{ if } 0 \leq x \leq 1,
    \end{cases}
\end{equation}
then $\text{Arccos}(\cos(\theta + \gamma)) = \theta + \gamma$. Hence, solving (\ref{eq2.18}) gives us
\begin{equation}
    \gamma = -\theta + \text{Arccos}\left(X\cos(\theta - \xi) + (1 - X)\cos(\theta + \xi) \right).
\end{equation}

\subsubsection{Retroreflection} \label{sssec_retro}

Consider the retroreflection law
\begin{equation}
    \mathbb{P}(x,\theta; \dd x' \dd\theta') = \delta_{(x,\theta)}(\dd x' \dd\theta').
\end{equation}
It follows from Theorem \ref{thm_rough_ref_char} that there is a sequence of walls $W_i$ such that $\mathbb{P}^{W_i} \to \mathbb{P}$ as $i \to \infty$.

We obtain a sequence of walls which generates retroreflection by considering the fractal of mushroom billiards depicted in Figure \ref{mushrooms_fig}. The ``cap'' of each mushroom is a semicircular arc. The walls consist of partial iterations of the mushroom fractal. To each $W_i$ is added an additional row of ``mushroom hallows,'' while the scale of the wall converges to zero as $u \to \infty$. The ratio of the width of the stem to the width of the cap of each mushroom is taken to converge to zero as $i \to \infty$, but at not too fast a rate. The horizontal segments at the top the boundary of $W_i$ form a partial iteration of a Cantor set. Provided the width of the stems does not converge to zero too quickly, this set will become negligible as $i \to \infty$.

\begin{figure}
    \centering
    \includegraphics[width = 0.8\linewidth]{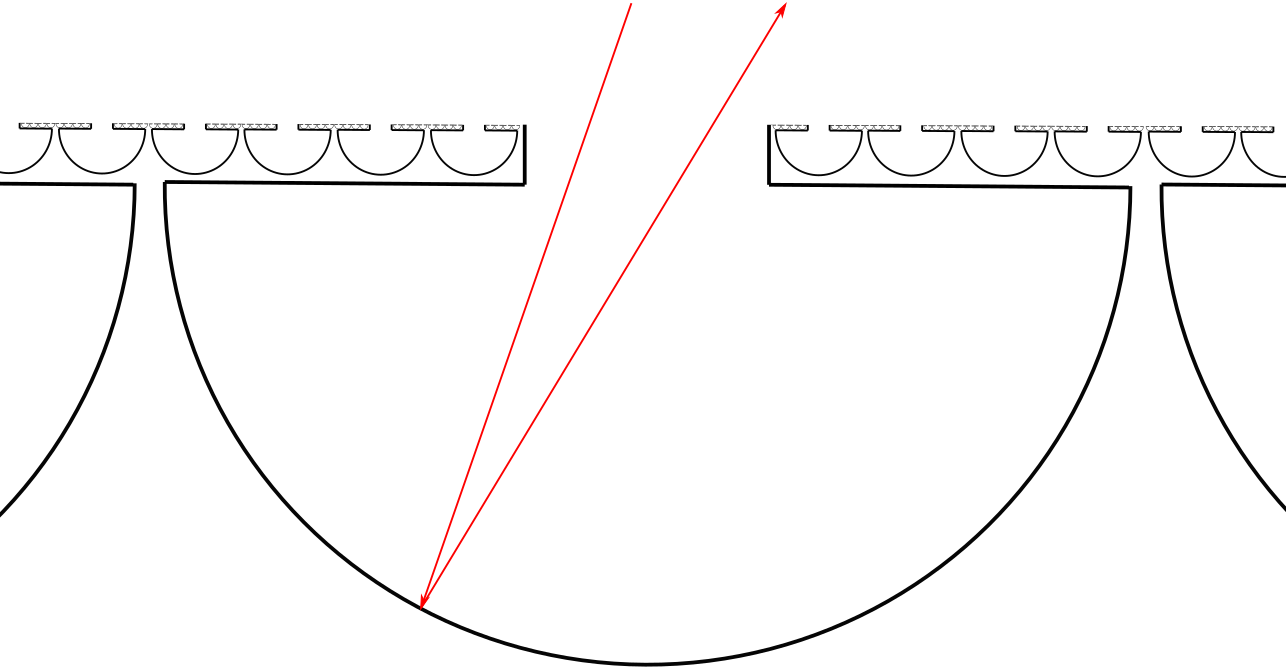}
    \caption{Retroreflection from a fractal of mushroom billiards.}
    \label{mushrooms_fig}
\end{figure}

For additional examples and discussion of walls which generate retroreflection, see Chapter 9 of [\ref{Plakhov5}].

\subsection{Examples of rough collision laws}

Given a sequence of periodic walls, $W_i = W(\Sigma, \epsilon_i)$, let $\widetilde{W}_i = W(\widetilde{\Sigma}, \widetilde{\epsilon}_i)$ denote the foreshortened wall (where $\widetilde{\Sigma}$ and $\widetilde{\epsilon}_i$ are defined by (\ref{eq1.66})). Let $X$ be uniform in $[0,1]$, and let 
\begin{equation}
    (X_\theta',\Theta_\theta') = P^{\widetilde{\Sigma},1}(X,\theta),
\end{equation}
The angle $\Theta_\theta'$ is the angle of exit of the trajectory hitting the foreshortened wall, as $X$ varies uniformly over one period. Theorem \ref{thm_classification} tells us that the limiting collision law for the disk and wall system $\mathbb{K} = \lim_{i \to \infty} \mathbb{K}^{W_i}$ is given by 
\begin{equation} \label{eq2.23}
    \mathbb{K}(y_1,y_3,\theta,\psi; \dd y_1' \dd y_3' \dd\theta' \dd\xi') = \delta_{(y_1,y_3)}(\dd y_1' \dd y_3') \widetilde{\mathbb{P}}(\theta, \dd\theta') \delta_{\pi - \psi}(\dd\psi'),
\end{equation}
where $\widetilde{\mathbb{P}}(\theta, \dd\theta')$ is the law of the random angle of exit $\Theta_\theta'$.

We obtain examples of rough collision laws by choosing walls $W_i$, such that the limiting reflection law for the sequence of foreshortened walls $\widetilde{W}_i$ is known.

\subsubsection{Rectangular teeth}

Consider walls $W_i = W(\Sigma,\epsilon_i)$ consisting of rectangular teeth with parameter $r$ (the ratio of the width of the teeth to the height). For this choice of walls, we denote the angle of exit by $\Theta_\theta'(r)$ to indicate the dependence of $\Theta_\theta'$ on $r$ in the formula (\ref{eq2.5}).  

The class of walls with rectangular teeth is invariant under foreshortening. After foreshortening by a factor $(1 + m/J)^{1/2}$, the new wall $\widetilde{W}_i$ is composed of rectangular teeth with parameter
\begin{equation}
    \widetilde{r} = (1 + m/J)^{-1/2}r.
\end{equation}
Thus the rough collision law $\mathbb{K} = \lim_{i \to \infty} \mathbb{K}^{W_i}$ is given by (\ref{eq2.23}), where $\widetilde{\mathbb{P}}(\theta, \dd\theta')$ is the law of $\Theta_\theta'(\widetilde{r})$.

Using (\ref{eq1.27'}) to convert back to the original coordinates, we obtain that the collision law $\mathbb{K}$ maps the incoming velocity $w = (v_1,v_2,\omega)$ to $A_{\text{smooth}}w$ with probability $p_{\widetilde{r}}(\theta)$ and to $A_{\text{no-slip}}w$ with probability $1 - p_{\widetilde{r}}(\theta)$, where $A_{\text{smooth}}$ and $A_{\text{no-slip}}$ are the matrices defined as follows: 
\begin{equation} \label{eq2.25}
    A_{\text{smooth}} = \begin{bmatrix}
    1 & 0 & 0 \\
    0 & - 1 & 0 \\
    0 & 0 & 1
    \end{bmatrix}, \quad\quad A_{\text{no-slip}} = \begin{bmatrix}
    \frac{m - J}{m + J} & 0 & \frac{-2J}{m + J} \\
    0 & - 1 & 0 \\
    \frac{-2m}{m + J} & 0 & \frac{-m + J}{m + J}
    \end{bmatrix}.
\end{equation}
The significance of the notation $A_{\text{smooth}}$ and $A_{\text{no-slip}}$ is explained in \S\ref{sssec_sm_ns}.

\subsubsection{Triangular teeth}

Consider walls $W_i = W(\Sigma,\epsilon_i)$, consisting of triangular teeth with parameter $\psi$. We denote the random angle or reflection by $\Theta_\theta'(\psi)$, to indicate the dependence of $\Theta_\theta'$ on $\psi$ in the formula (\ref{eq2.10})

For each $i$, the foreshortened wall $\widetilde{W}_i$ consists of triangular teeth with parameter 
\begin{equation}
\widetilde{\psi} = 2\arctan\left( \left(1 + \frac{m}{J}\right)^{-1/2}\tan(\frac{\psi}{2}) \right).
\end{equation}
The rough collision law $\mathbb{K} = \lim_{i \to \infty} \mathbb{K}^{W_i}$ is given by (\ref{eq2.23}), where $\widetilde{\mathbb{P}}(\theta, \dd\theta')$ is the law of $\Theta_\theta'(\widetilde{\psi})$.

\begin{sloppypar}
In principle, one can use (\ref{eq1.27'}) to express $\mathbb{K}$ in the original coordinates $(x_1,x_2,\alpha, v_1, v_2, \theta)$, although the formula is not as simple as in the previous case.
\end{sloppypar}

\subsubsection{Focusing elliptical arcs}

\begin{figure}
    \centering
    \includegraphics[width = 0.8\linewidth]{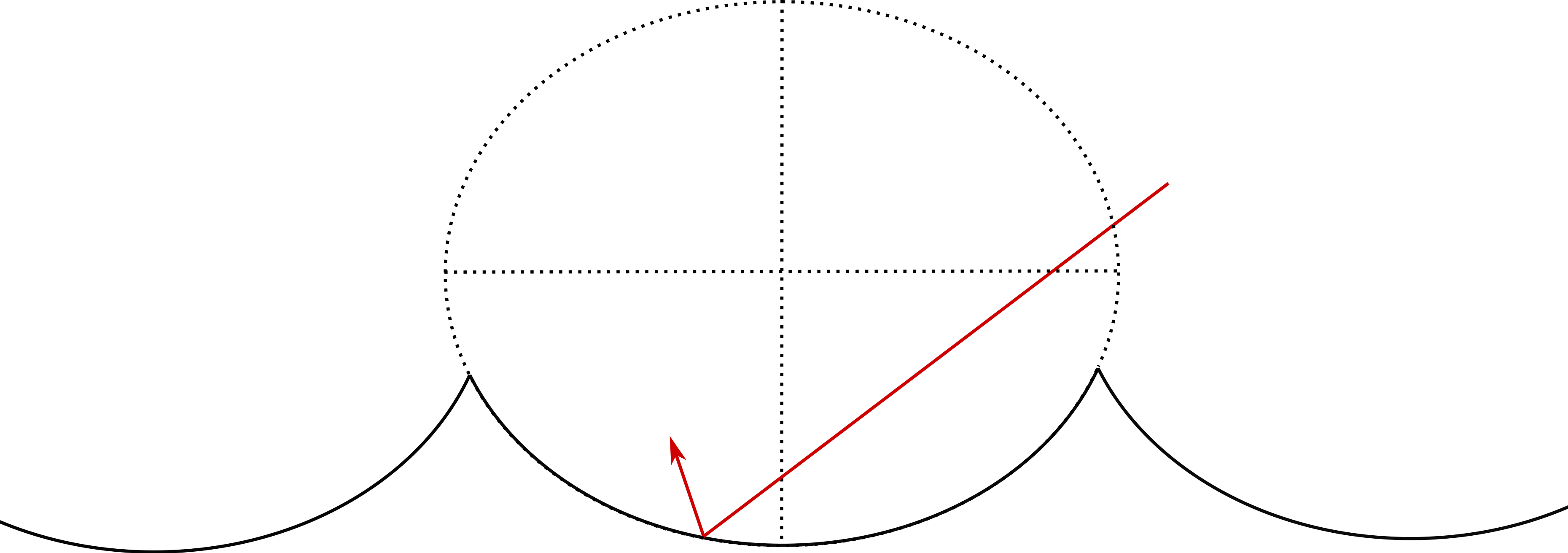}
    \caption{Focusing elliptical arcs}
    \label{elliptical_arc_fig}
\end{figure}

Consider walls $W_i = W(\Sigma,\epsilon_i)$ constructed of focusing (concave-up) elliptical arcs, where the ratio of the horizontal to vertical axes is $(1 + m/J)^{1/2}$. In other words, we define 
\begin{equation}
    t_n(x) = \frac{\epsilon_n}{2}\cot\xi - (1+m/J)^{-1/2}\sqrt{\frac{\epsilon_n^2}{4}\csc^2\xi - \left(x - \frac{\epsilon_n}{2}\right)^2} \quad\quad \text{ for } x \in [0, \epsilon_n],
\end{equation}
and extend $t_n$ to be periodic on $\mathbb{R}$. We define
\begin{equation}
    W_n = W(\Sigma, \epsilon_n) = \{(x_1,x_2) : x_2 \leq t_n(x_1)\}.
\end{equation}
See Figure \ref{elliptical_arc_fig}

The foreshortened wall $\widetilde{W}_i$ is precisely the wall of circular arcs with parameter $\xi$, discussed in \S\ref{sssec_circarc}. Therefore, the rough collision law $\mathbb{K} = \lim_{i \to \infty} \mathbb{K}^{W_i}$ is given by (\ref{eq2.23}), where $\widetilde{\mathbb{P}}(\theta,\dd\theta')$ is the law of $\Theta_\theta' = \Theta_\theta'(\xi)$, given by the formula (\ref{eq2.17}).

\subsubsection{Smooth and no-slip collisions} \label{sssec_sm_ns}

As we have seen, there are two basic examples of deterministic rough reflection laws: specular reflection 
\begin{equation}
    \mathbb{P}_{\text{spec}}(x,\theta; \dd x' \dd\theta') := \delta_{(x,\pi - \theta)}(\dd x' \dd\theta')
\end{equation}
and retroreflection
\begin{equation}
    \mathbb{P}_{\text{retro}}(x,\theta; \dd x' \dd\theta') := \delta_{(x,\theta)}(\dd x' \dd\theta').
\end{equation}
Trivially, $\mathbb{P}_{\text{spec}}$ is the limit of the reflection laws on the constant sequence of flat walls $W_i = \{(x_1,x_2) : x_2 \leq 0\}$. Under the correspondence (\ref{eq1.68}), $\mathbb{P}_{\text{spec}}$ corresponds to 
\begin{equation}
    \mathbb{K}_{\text{smooth}}(y_1,y_3,\theta,\psi; \dd y_1' \dd y_3' \dd\theta' \dd\psi') = \delta_{(y_1,y_3,\pi - \theta,\pi - \psi)}(\dd y_1', \dd y_3' \dd\theta' \dd\psi').
\end{equation}
This is just the classical collision law describing a frictionless collision between a hard disk and fixed wall in the plane.

Under the correspondence (\ref{eq1.68}), $\mathbb{P}_{\text{retro}}$ corresponds to 
\begin{equation}
    \mathbb{K}_{\text{no-slip}}(y_1,y_3,\theta,\psi; \dd y_1' \dd y_3' \dd\theta' \dd\psi') = \delta_{(y_1,y_3,\theta,\pi - \psi)}(\dd y_1', \dd y_3' \dd\theta' \dd\psi').
\end{equation}
This type of deterministic collision is known as a \textit{no-slip collision}. 

Using (\ref{eq1.27'}) to convert back to our original coordinates $(x_1,x_2,\alpha,v_1,v_2,\beta)$, one may show that the collision law $\mathbb{K}_{\text{smooth}}$ maps the velocity $w = (v_1,v_2,\omega)$ to $A_{\text{smooth}}w$, and the collision law $\mathbb{K}_{\text{no-slip}}$ maps $w$ to $A_{\text{no-slip}}w$, where $A_{\text{smooth}}$ and $A_{\text{no-slip}}$ are the matrices defined by (\ref{eq2.25}).

Let $W_i$ be a sequence of walls such that $\mathbb{P}_{\text{spec}} = \lim_{i \to \infty} \mathbb{P}^{W_i}$ (for example, the walls of \S\ref{sssec_retro}). Let $\widehat{W}_i = \{(x_1,x_2) : (x_1,(1 + m/J)^{-1/2}x_2) \in W_i\}$. That is, we choose walls $\widehat{W}_i$ whose foreshortening is $W_i$. Theorem \ref{thm_classification} tells us that $\mathbb{K}_{\text{no-slip}} = \lim_{i \to \infty} \mathbb{K}^{\widehat{W}_i}$.

No-slip collisions were first introduced by Broomhead and Gutkin in [\ref{GutkinBroomhead}] and have been further investigated Cox, Feres, and Ward [\ref{diffgeo_rbcollisions}, \ref{NoSlip}].

\begin{remark} \label{rem_deterministic} \normalfont
It turns out that $\mathbb{P}_{\text{spec}}$ and $\mathbb{P}_{\text{retro}}$ are the only two deterministic rough reflection laws which act continuously on $\mathbb{R} \times \mathbb{S}^1_+$. Thus, by the correspondence (\ref{eq1.68}), $\mathbb{K}_{\text{smooth}}$ and $\mathbb{K}_{\text{no-slip}}$ are the only two rough collision laws which act continuously on $\mathbf{P} \times \mathbb{S}^2_+$. The latter statement is also closely related to Corollary 2.2 in [\ref{diffgeo_rbcollisions}], which classifies deterministic collision laws for more general rigid bodies.

To see why the first statement is true, let $\mu(\dd\theta) = \sin\theta \dd\theta$, and let $m$ denote Lebesgue measure on the line. Recall that by Proposition \ref{prop_reflproperties}, if $\mathbb{P}(x,\theta; \dd x' \dd\theta') = \delta_x(\dd x') \widetilde{\mathbb{P}}(\theta, \dd\theta')$ is a rough reflection law, then $\widetilde{\mathbb{P}}(\theta, \dd\theta')$ preserves the measure $\mu$. The measure space $((0,\pi),\mu)$ is isomorphic to $((-1,1),m)$, via the mapping $f(\theta) = -\cos\theta$. The only continuous transformations $(-1,1) \to (-1,1)$ which preserve Lebesgue measure are $T(x) = x$ and $T(x) = -x$, and these transformations pull back via $f$ to retroreflection and specular reflection respectively.

By considering transformations which are not continuous, one can obtain a much larger class of deterministic reflection laws and collision laws. For example, any involutive interval exchange transformation of $(-1,1)$ corresponds to some deterministic reflection law and to some deterministic collision law.
\end{remark}

\subsection{Proof of the lemma for constructing random reflections} \label{ssec_proof}

\begin{proof}[Proof of Lemma \ref{lem_construct}]

For $(x,\theta) \in \mathbb{R} \times \mathbb{S}^1_+$, let $(x',\theta') = P^{\Sigma_i,\epsilon_i}(x,\theta)$. By condition B5 on the cells $\Sigma_i$ (see \S\ref{sssec_periodic}), the boundary of the wall $W(\Sigma_i,\epsilon_i)$ is periodic and has maximums at integer multiples of $\epsilon_i$ along the line $\{(x_1,x_2) : x_2 = 0\}$. Moreover, if $x \in [k\epsilon_i, (k+1)\epsilon_i]$ for $k \in \mathbb{Z}$, then the billiard trajectory will remain in the hollow bounded between the two maximums and will return to line $x_2 = 0$ at $x' \in [k\epsilon_i, (k+1)\epsilon_i]$. Therefore, 
\begin{equation} \label{eq2.7'}
    |x - x'| \leq \epsilon_i.
\end{equation}
We will use this estimate several times in the argument below.

\textit{Proof of (i)} Suppose that $\mathbb{P}(x,\theta; \dd x' \dd\theta') = \lim_{i \to \infty} \mathbb{P}^{\Sigma_i,\epsilon_i}(x,\theta; \dd x' \dd\theta')$ exists. Let $f, g \in C_c(\mathbb{S}^1_+)$. We have
\begin{equation} \label{eq2.8}
\begin{split}
    &\int_{\mathbb{S}^1_+} g(\theta) \int_{\mathbb{S}^1_+} f(\theta') \widetilde{\mathbb{P}}^{\Sigma_i,\epsilon_i}(\theta, \dd\theta')\sin\theta \dd\theta \\
    & = \int_{\mathbb{S}^1_+} g(\theta) \frac{1}{\epsilon_i}\int_0^{\epsilon_i} \int_{\mathbb{R} \times \mathbb{S}^1_+} f(\theta')\mathbb{P}^{\Sigma_i,\epsilon_i}(x,\theta; \dd x' \dd\theta')\dd x \sin\theta \dd\theta \\
    & = \int_{\mathbb{S}^1_+} g(\theta) \frac{1}{\epsilon_i\lceil 1/\epsilon_i \rceil}\int_0^{\epsilon_i\lceil 1/\epsilon_i \rceil} \int_{\mathbb{R} \times \mathbb{S}^1_+} f(\theta')\mathbb{P}^{\Sigma_i,\epsilon_i}(x,\theta; \dd x' \dd\theta')\dd x \sin\theta \dd\theta, \\
\end{split}
\end{equation}
where the last line uses the fact that $x \mapsto \int_{\mathbb{R} \times \mathbb{S}^1_+} f(\theta')\mathbb{P}^{\Sigma_i,\epsilon_i}(x,\theta; \dd x' \dd\theta')$ is $\epsilon_i$-periodic. Noting that $\mathbb{P}^{\Sigma_i,\epsilon_i}$ is a probability measure, we also have the following bound:
\begin{equation} \label{eq2.9'}
    \int_{\mathbb{R} \times \mathbb{S}^1_+} f(\theta')\mathbb{P}^{\Sigma_i,\epsilon_i}(x,\theta; \dd x' \dd\theta') \leq ||f||_{L^\infty}.
\end{equation}
Using this bound, we see that the last line in (\ref{eq2.8}) is equal to 
\begin{equation} \label{eq2.10'}
    \int_{\mathbb{S}^1_+} g(\theta)  \int_0^1 \int_{\mathbb{R} \times \mathbb{S}^1_+} f(\theta')\mathbb{P}^{\Sigma_i,\epsilon_i}(x,\theta; \dd x' \dd\theta')\dd x \sin\theta \dd\theta + O(\epsilon_i),
\end{equation}
where the error term may depend on $f$ and $g$. Let $\phi_\delta \in C_c(\mathbb{R})$ be chosen so that $0 \leq \phi_\delta \leq 1$ and $\supp\phi_\delta \subset [0,1]$ and $\int|\phi_\delta - \mathbf{1}_{[0,1]}|\dd x < \delta$. Let $\eta_\delta \in C_c(\mathbb{R})$ be chosen so that $0 \leq \eta_\delta \leq 1$ and $\eta_\delta = 1$ on the interval $[-\delta^{-1}, 1 + \delta^{-1}]$. We see that (\ref{eq2.10'}) is equal to 
\begin{equation} \label{eq2.11'}
\begin{split}
    & \int_{\mathbb{R} \times \mathbb{S}^1_+} \phi_\delta(x)g(\theta) \int_{\mathbb{R} \times \mathbb{S}^1_+} f(\theta')\mathbb{P}^{\Sigma_i,\epsilon_i}(x,\theta; \dd x' \dd\theta')\Lambda^1(\dd x \dd\theta) + O(\delta + \epsilon_i) \\
    & = \int_{\mathbb{R} \times \mathbb{S}^1_+} \phi_\delta(x)g(\theta) \int_{\mathbb{R} \times \mathbb{S}^1_+} \eta_\delta(x')f(\theta')\mathbb{P}^{\Sigma_i,\epsilon_i}(x,\theta; \dd x' \dd\theta')\Lambda^1(\dd x \dd\theta) + O(\delta + \epsilon_i),
\end{split}
\end{equation}
noting that by the bound (\ref{eq2.7'}), $x' \in [-\delta^{-1},1 + \delta^{-1}]$ (for $i$ sufficiently large) whenever $x \in [0,1]$; thus $\eta_\delta(x') = 1$ whenever $\phi_\delta(x) \neq 0$. Taking the limit as $i \to \infty$, (\ref{eq2.11'}) converges to 
\begin{equation} \label{eq2.13}
     \int_{\mathbb{R} \times \mathbb{S}^1_+} \phi_\delta(x)g(\theta) \int_{\mathbb{R} \times \mathbb{S}^1_+} \eta_\delta(x')f(\theta')\mathbb{P}(x,\theta; \dd x' \dd\theta')\Lambda^1(\dd x \dd\theta) + O(\delta).
\end{equation}
Letting $\delta \to 0$, by boundedness and pointwise convergence of $\phi_\delta$ and $\eta_\delta$, dominated convergence implies (\ref{eq2.13}) converges to 
\begin{equation} \label{eq2.14'}
\begin{split}
    &\int_{[0,1] \times \mathbb{S}^1_+} g(\theta) \int_{\mathbb{R} \times \mathbb{S}^1_+} f(\theta') \mathbb{P}(x,\theta; \dd x' \dd\theta') \dd x \sin\theta \dd\theta \\
    & \quad\quad = \int_{\mathbb{S}^1_+} g(\theta) \int_{\mathbb{S}^1_+} f(\theta') \widetilde{\mathbb{P}}_0(\theta, \dd\theta') \sin\theta \dd\theta,
\end{split}
\end{equation}
where $\widetilde{\mathbb{P}}_0$ is the Markov kernel on $\mathbb{S}^1_+$ defined by 
\begin{equation}
    \int_{\mathbb{S}^1_+} h(\theta')\widetilde{\mathbb{P}}_0(\theta, \dd\theta') = \int_{[0,1]} \int_{\mathbb{R} \times \mathbb{S}^1_+} h(\theta') \mathbb{P}(x,\theta; \dd x' \dd\theta').
\end{equation}
Putting all this together, we see that (\ref{eq2.8}) converges to (\ref{eq2.14'}) as $i \to \infty$, as desired.

\begin{sloppypar}
\textit{Proof of (ii).} Conversely, suppose that $\widetilde{\mathbb{P}} = \lim_{i \to \infty} \widetilde{\mathbb{P}}^{\Sigma_i,\epsilon_i}$ exists. Given $h(x,\theta, x',\theta') = f(x,\theta)g(x',\theta')$, where $f,g \in C_c^\infty(\mathbb{R} \times \mathbb{S}^1_+)$, we want to show that a limit of form (\ref{eq1.11}) holds. Since the tensor product $C_c^\infty(\mathbb{R}) \otimes C_c^\infty(\mathbb{S}^1_+)$ is dense in $C_c^\infty(\mathbb{R} \times \mathbb{S}^1_+)$, we may assume without loss of generality that $f(x,\theta) = f_1(x)f_2(\theta)$ and $g(x,\theta) = g_1(x)g_2(\theta)$ for some $f_1, g_1 \in C_c^\infty(\mathbb{R})$ and $f_2, g_2 \in C_c^\infty(\mathbb{S}^1_+)$. We have 
\end{sloppypar}
\begin{equation} \label{eq2.15}
\begin{split}
    &\int_{\mathbb{R} \times \mathbb{S}^1_+} g_1(x)g_2(\theta) \int_{\mathbb{R} \times \mathbb{S}^1_+} f_1(x')f_2(\theta') \mathbb{P}^{\Sigma_i,\epsilon_i}(x,\theta; \dd x' \dd\theta')\Lambda^1(\dd x \dd\theta) \\
    & = \int_{\mathbb{R} \times \mathbb{S}^1_+} g_1(x)g_2(\theta) \int_{\mathbb{R} \times \mathbb{S}^1_+} f_1(x)f_2(\theta') \mathbb{P}^{\Sigma_i,\epsilon_i}(x,\theta; \dd x' \dd\theta')\Lambda^1(\dd x \dd\theta) \\
    & \quad + \int_{\mathbb{R} \times \mathbb{S}^2_+} g_1(x)g_2(\theta) \int_{\mathbb{R} \times \mathbb{S}^1_+} [f_1(x') - f_1(x)]f_2(\theta') \mathbb{P}^{\Sigma_i,\epsilon_i}(x,\theta; \dd x' \dd\theta')\Lambda^1(\dd x \dd\theta) \\
    & \hspace{2.5in} + O(\delta^2) \\ 
    & = \int_{\mathbb{R} \times \mathbb{S}^1_+} g_1(x)f_1(x)g_2(\theta) \int_{\mathbb{R} \times \mathbb{S}^1_+} f_2(\theta') \mathbb{P}^{\Sigma_i,\epsilon_i}(x,\theta; \dd x' \dd\theta')\Lambda^1(\dd x \dd\theta) + O(E_i),
\end{split}
\end{equation}
where 
\begin{equation}
    E_i := \sup\{|f_1(x') - f_1(x)| : x \in \mathbb{R}\}.
\end{equation}
By the bound (\ref{eq2.7'}) and the fact that $f_1$ is continuous and compactly supported, we see that $E_i$ converges to zero as $i \to \infty$. By approximation by simple functions, we may write 
\begin{equation}
    g_1(x)f_1(x) = \sum_{k \in \mathbb{Z}} g_1(\epsilon_i k)f_1(\epsilon_i k) \mathbf{1}_{[\epsilon_i k,\epsilon_i(k+1)]}(x) + h_i(x),
\end{equation}
where $h_i \to 0$ in $L^1(\mathbb{R})$. Also, observing that the function
\begin{equation}
    x \mapsto \int_{\mathbb{R} \times \mathbb{S}^1_+} f_2(\theta') \mathbb{P}^{\Sigma_i,\epsilon_i}(x,\theta; \dd x' \dd\theta')
\end{equation}
is $\epsilon_i$-periodic, we may write the last quantity in (\ref{eq2.15}) as 
\begin{equation}
\begin{split}
    &\sum_{k \in \mathbb{Z}} g_1(\epsilon_i k)f_1(\epsilon_i k) \int_k^{k+\epsilon_i} \int_{\mathbb{S}^1_+} g_2(\theta) \int_{\mathbb{R} \times \mathbb{S}^1_+} f_2(\theta') \mathbb{P}^{\Sigma_i,\epsilon_i}(x,\theta; \dd x' \dd\theta') \sin\theta \dd\theta \dd x \\
    & \hspace{2.5in} + O(||h_i||_{L^1} + E_i) \\
    & = \left( \sum_{k \in \mathbb{Z}} g_1(\epsilon_i k)f_1(\epsilon_i k) \epsilon_i \right)\left(\int_{\mathbb{S}^1_+} g_2(\theta) \frac{1}{\epsilon_i} \int_0^{\epsilon_i} \int_{\mathbb{R} \times \mathbb{S}^1_+} f_2(\theta') \mathbb{P}^{\Sigma_i,\epsilon_i}(x,\theta; \dd x' \dd\theta') \dd x \sin\theta \dd\theta \right) \\
    & \hspace{2.5in} + O(||h_i||_{L^1} + E_i). \\
    & = \left( \sum_{k \in \mathbb{Z}} g_1(\epsilon_i k)f_1(\epsilon_i k) \epsilon_i \right)\left(\int_{\mathbb{S}^1_+} g_2(\theta) \int_{\mathbb{S}^1_+} f_2(\theta') \widetilde{\mathbb{P}}^{\Sigma_i,\epsilon_i}(\theta; \dd\theta') \sin\theta \dd\theta \right) \\
    & \hspace{2.5in} + O(||h_i||_{L^1} + E_i).
\end{split}
\end{equation}
Taking the limit first as $i \to \infty$, and noting that the first factor is a Riemann sum, the quantity above converges to 
\begin{equation}
\begin{split}
    &\left(\int_{\mathbb{R}} f_1(x)g_1(x)\dd x \right)\left( \int_{\mathbb{S}^1_+} g_2(\theta) \int_{\mathbb{S}^1_+} f_2(\theta') \widetilde{\mathbb{P}}(\theta; \dd\theta') \sin\theta \dd\theta \right) \\
    & = \int_{\mathbb{R} \times \mathbb{S}^1_+} g_1(x) g_2(\theta) \int_{\mathbb{S}^1_+} f_1(x)f_2(\theta') \widetilde{\mathbb{P}}(\theta; \dd\theta') \Lambda^1(\dd x \dd\theta) \\
    & = \int_{\mathbb{R} \times \mathbb{S}^1_+} g_1(x) g_2(\theta) \int_{\mathbb{R} \times \mathbb{S}^1_+} f_1(x')f_2(\theta') \delta_x(\dd x')\widetilde{\mathbb{P}}(\theta; \dd\theta') \Lambda^1(\dd x \dd\theta). 
\end{split}
\end{equation}
This proves that $\lim_{i \to \infty} \mathbb{P}^{\Sigma_i,\epsilon_i}$ exists and is equal to $\delta_x \times \widetilde{\mathbb{P}}$.

\textit{Proof of (iii).} Take $\Sigma_i = \Sigma$ to be constant. For all $\epsilon_i$, the macro-reflection map $P^{\Sigma, \epsilon_i}$ is just the macro-reflection map $P^{\Sigma,1}$ re-expressed after scaling the spatial coordinates by $\epsilon_i$, that is 
\begin{equation}
    P^{\Sigma,\epsilon_i} = \sigma_{\epsilon_i} \circ P^{\Sigma,1} \circ \sigma_{\epsilon_i}^{-1}, \quad \text{ where } \sigma_{\epsilon_i}(x,\theta) = (\epsilon_i x, \theta).
\end{equation}
Consequently, 
\begin{equation} \label{eq2.20}
    \mathbb{P}^{\Sigma,\epsilon_i}(x,\theta; \dd x' \dd\theta') = (\sigma_{\epsilon_i})_{\#}\mathbb{P}^{\Sigma,1}(\sigma_{\epsilon_i}^{-1}(x,\theta); \dd x' \dd\theta').
\end{equation}
Thus, for $f, g \in C_c(\mathbb{S}^1_+)$, 
\begin{equation}
\begin{split}
    &\int_{\mathbb{S}^1_+} g(\theta) \int_{\mathbb{S}^1_+} f(\theta')\widetilde{\mathbb{P}}^{\Sigma,\epsilon_i}(\theta, \dd\theta') \sin\theta \dd\theta \\
    & = \int_{\mathbb{S}^1_+} g(\theta) \frac{1}{\epsilon_i}\int_0^{\epsilon_i}\int_{\mathbb{R} \times \mathbb{S}^1_+} f(\theta')\mathbb{P}^{\Sigma,\epsilon_i}(x,\theta; \dd x' \dd\theta') \dd x \sin\theta \dd\theta  \\
    & = \int_{\mathbb{S}^1_+} g(\theta) \frac{1}{\epsilon_i}\int_0^{\epsilon_i}\int_{\mathbb{R} \times \mathbb{S}^1_+} f(\theta')(\sigma_{\epsilon_i})_{\#}\mathbb{P}^{\Sigma,1}(\sigma_{\epsilon_i}^{-1}(x,\theta); \dd x' \dd\theta') \dd x \sin\theta \dd\theta \\
    & = \int_{\mathbb{S}^1_+} g(\theta) \int_0^1\int_{\mathbb{R} \times \mathbb{S}^1_+} f(\theta')\mathbb{P}^{\Sigma,1}(x,\theta; \dd x' \dd\theta') \dd x \sin\theta \dd\theta \\
    & = \int_{\mathbb{S}^1_+} g(\theta) \int_{\mathbb{S}^1_+} f(\theta')\widetilde{\mathbb{P}}^{\Sigma,1}(\theta, \dd\theta') \sin\theta \dd\theta.
\end{split}
\end{equation}
Here the second equality follows from (\ref{eq2.20}) and the third equality follows by making the changes of variables $(x',\theta') \mapsto \sigma_{\epsilon_i}(x',\theta')$ and $(x,\theta) \mapsto \sigma_{\epsilon_i}(x,\theta)$. This shows that the sequence $\widetilde{\mathbb{P}}^{\Sigma,\epsilon_i}$ is constant and equal to $\widetilde{\mathbb{P}}^{\Sigma,1}$. Thus (iii) follows from (i) and (ii).
\end{proof}

\section{Specular Reflection Law} \label{sec_specularreflection}

\begin{sloppypar}
The dynamics of the disk and wall system are described by an evolution $t \mapsto (y(t),w(t))$ in the phase space $T\mathcal{M} \cong \mathcal{M} \times \mathbb{R}^3$. This evolution should satisfy, and ideally be uniquely determined by, accepted physical laws governing rigid body interactions. We assume that:
\end{sloppypar}
\begin{enumerate}[label = P\arabic*.]
    \item The bodies $D$ and $W$ do not interpenetrate. 
    \item The system is subject to Euler's laws of rigid body motion (see equations (\ref{eq3.1}) and (\ref{eq3.2})). 
    \item When not in contact, the net force applied to each body is zero, and upon contact, a single impulsive force is applied to the disk at the point of contact and directed parallel to the unit normal vector on the wall.  
    \item The kinetic energy of the disk is conserved for all time. 
\end{enumerate}
The first three assumptions are standard for rigid body interactions in which no friction is present. The fourth is the limiting case for a system of bodies of finite mass in which the total kinetic energy, linear momentum, and angular momentum are conserved. After letting the mass and moment of inertia of one body diverge to infinity, the kinetic energy of the other body is conserved in the limit (see in Proposition  \ref{prop_energylimit}). 

The material in this section does not depend on results from the other sections, with the exception of two results proved in \S\ref{sec_config}, namely (i) that there is a full measure open subset $\partial_{\reg} \mathcal{M} \subset \partial \mathcal{M}$ on which a $C^1$ field of unit normal vectors is defined, and (ii) a formula for the unit normal vector to the surface. See Proposition \ref{prop_Mreg} and Remark \ref{rem_normal}. The material on the geometric properties of the configuration space does not depend on this section, so there is no circularity.

Let us give a more precise mathematical formulation to the physical assumptions P1-P4. 

The assumption that the bodies do not interpenetrate means that the configuration of the system $y$ is confined to the set $\mathcal{M} \subset \mathbb{R}^3$ for all time. The boundary $\partial \mathcal{M}$ corresponds to collision configurations for the system. The collision dynamics cannot be defined when multiple satellites simultaneously make contact with the wall $W$, or when a satellite makes contact with a singular point of $W$. We therefore only derive the collision law for when $y$ lies in $\partial_{\reg} \mathcal{M}$, which is a full measure subset of $\partial \mathcal{M}$ (see Proposition \ref{prop_Mreg}). 

We use the following notation: for $y \in \partial_{\reg} \mathcal{M}$, we let $\sigma(y) \in \partial W$ denote the position in $\mathbb{R}^2$ of the unique satellite of $D(y)$ which lies in $\partial W$. As usual, we write $y = (x,\alpha)$ and $w = (v,\omega)$ where $x$ is the center of mass of the disk, $\alpha$ is the angular configuration, $v$ is the linear velocity of the center of mass, and $\omega$ is the angular velocity of the disk.

When $y \in \Int \mathcal{M}$, the net force on $D$ is zero. Thus Euler's equations of motion say that
\begin{equation} \label{eq3.1}
    m\dv{v}{t} = 0, \quad\text{ and } \quad J\dv{\omega}{t} = 0 \quad \text{ whenever } y \in \Int \mathcal{M},
\end{equation}
here recalling that the total mass of the disk is $m$ and the moment of inertia of the disk about its center of mass is $J$. In other words, $D$ moves freely in the complement of $W$.

Euler's laws must be given an impulsive interpretation when the disk makes contact with the wall. For $p \in \partial_{\reg} W$, let $k(p)$ denote the outward pointing unit normal vector to the wall at the point $p$. Our assumptions say that the impulse when the disk makes contact with the wall in configuration $y \in \partial_{\reg} \mathcal{M}$ is $\lambda k(\sigma(y))$ for some $\lambda \in \mathbb{R}$. Since each satellite lies at unit distance from the center of mass, the induced impulsive torque at the point of contact is then $\lambda \sin\beta$ where $\beta = \beta(y) \in (-\pi/2,\pi/2)$ is the signed angle measured counterclockwise from $k(\sigma(y))$ to $x - \sigma(y)$. In impulsive form, Euler's laws say that 
\begin{equation} \label{eq3.2}
    m(v^+ - v^-)= \lambda k(\sigma(y)), \quad\quad J(\omega^+ - \omega^-) = \lambda \sin\beta(y) \quad \text{ whenever } y \in \partial_{\reg} \mathcal{M}, 
\end{equation}
where 
\begin{equation}
    v^\pm(t) := \lim_{s \to t\pm} v(t), \quad \text{ and } \quad \omega^\pm := \lim_{s \to t\pm} \omega(t).
\end{equation}

The kinetic energy of the disk in state $(y,w)$ is 
\begin{equation}
    \frac{1}{2}mv_1^2 + \frac{1}{2}mv_2^2 + \frac{1}{2}J \omega^2 = \frac{1}{2} ||w||^2.
\end{equation}
By conservation of the kinetic energy of the disk, up to a change of units, we may assume that $||w||^2 = 1$ for all time.

Summarizing, we wish to solve the following initial value problem: for $(y^0, w^0) \in \Int \mathcal{M}$, to find an evolution $t \mapsto (y(t),w(t)) \in \mathbb{R}^3 \times \mathbb{R}^3$ which satisfies
\begin{equation} \label{eq6.15}
\begin{cases}
y \in \mathcal{M}, \\
y \in \Int \mathcal{M} \quad \Rightarrow \quad \dv{w}{t} = 0, \\
y \in \partial_{\reg} \mathcal{M} \quad \Rightarrow \quad (\exists \lambda  \in \mathbb{R}) : w^+ - w^- = \begin{pmatrix} m^{-1}\lambda k(\sigma(y)) \\ J^{-1}\lambda \sin\beta(y) \end{pmatrix}, \\
||w|| = 1, \\
(y(0),w(0)) = (y^0, w^0).
\end{cases}
\end{equation}
Contained in these conditions is the assumption that the derivative $\dv{w}{t}$ exists when $y \in \Int \mathcal{M}$ and that the limits $w^-$ and $w^+$ exist when $y \in \partial_{\reg} \mathcal{M}$.

\begin{remark} \normalfont
Our approach to formulating the above problem aims to obtain a precise mathematical statement quickly while keeping the amount of technical machinery to a minimum. One drawback to our approach is that it requires taking the distinct forms of Euler's laws for impulsive and non-impulsive interactions as basic. The reader may wonder whether there is a unified framework in which to handle these different types of interactions. In addition, while Proposition \ref{prop_specvolution} shows that the problem (\ref{eq6.15}) is well-posed, it is not so clear how well the approach would generalize to other rigid body settings.

Much work has been done in the last 40 years to put physical systems with impulsive interactions within a rigorous and unified mathematical framework. One such approach is based on \textit{differential inclusions}. The idea is to solve a differential relation of form 
\begin{equation}
    F(q,t) - \ddot{q} \in \partial I_{V(q)},
\end{equation}
where $q$ is the position in the configuration space $\mathcal{M}$, $F$ is the external force acting on the system, $V(q)$ is the tangent cone in $\mathcal{M}$ at the point $q$, $I_{V(q)}(r) = 0$ if $r \in V(q)$ and $I_{V(q)}(r) = \infty$ otherwise, and $\partial I_{V(q)}$ is the subdifferential of $I_{V(q)}$ in the convex analysis sense. Some early work to formulate and solve rigid body problems within this setting may be found in [\ref{Marq1}, \ref{Mor1}, \ref{Mor3}, \ref{PaoSch1}, \ref{PaoSch2}, \ref{Marq2}, \ref{Mab1}].

Another approach is to formulate impulse problems as variational problems in Lagrangian mechanics. This leads to geometric integration algorithms for systems with impacts, as detailed in [\ref{FMOW}].
\end{remark}

We now solve the problem (\ref{eq6.15}).

Recall that $n$ denotes the field of inward-pointing unit normal vectors on $\partial_{\reg} \mathcal{M}$. Define the subbundles of the tangent space
\begin{equation}
\begin{split}
    V_{\text{in}} & = \{(y,w) \in T\mathcal{M} : y \in \partial_{\reg} \mathcal{M} \text{ and } \langle w, n(y) \rangle > 0\}, \\
    V_{\text{out}} & = -V_{\text{in}} = \{(y,w) \in T\mathcal{M} : y \in \partial_{\reg} \mathcal{M} \text{ and } \langle w, n(y) \rangle < 0\}.
\end{split}
\end{equation}
Given $(y,w) \in \mathcal{M} \times \mathbb{R}^3$, let
\begin{equation}
    \overline{t}(y,w) = \inf\{t > 0 : y + tw \in \partial \mathcal{M}\}.
\end{equation}
If $y \in \Int \mathcal{M}$, or $(y,w) \in V_{\text{in}}$, then by compactness and continuity, if $\overline{t}(y,w) < \infty$ then it is a minimum.

\begin{proposition} \label{prop_specvolution}
Let $(y^0, w^0) \in \Int \mathcal{M} \times \mathbb{R}^3$, and assume that $\overline{t}_1 := \overline{t}(y^0,w^0) < \infty$. Let $y^1 = y^0 + \overline{t}_1w^0$, and assume that $(y^1,w^0) \in V_{\text{out}}$. There exists $\delta > 0$ and a unique evolution $t \mapsto (y(t),w(t))$, $t \in [0,\overline{t}_1 + \delta)$ such that $y$ is continuous, $w$ is right-continuous, and the conditions (\ref{eq6.15}) hold. Explicitly, the evolution is given by linear motion in the interior of $\mathcal{M}$ and specular reflection on the boundary:
\begin{equation} \label{eq6.16'}
    (y(t),w(t)) = \begin{cases}
    (y^0 + tw^0, w^0) \quad & \text{ if } 0 \leq t < \overline{t}_1, \\
    (y^1, w^0 - 2\langle w^0, n(y^1) \rangle n(y^1) ) & \text{ if } t = \overline{t}_1, \\
   (y^1 + (t - \overline{t}_1)w(\overline{t}_1), w(\overline{t}_1)) & \text{ if } \overline{t}_1 < t < \overline{t}_1+ \delta.
    \end{cases}
\end{equation}
\end{proposition}

\begin{proof}[Proof of Proposition \ref{prop_specvolution}]
Let $w^1 = w^0 - 2\langle w^0, n(y^1) \rangle n(y^1)$. It follows from $(y^1,w^0) \in V_{\text{out}}$ that $(y^1,w^1) \in V_{\text{in}}$, and consequently $\overline{t}_2 := \overline{t}(y^1,w^1) > 0$. We choose $\delta = \overline{t}_2$.

Assume that $t \mapsto (y(t), w(t))$ is right-continuous and satisfies the conditions (\ref{eq6.15}) for $0 \leq t \leq \overline{t}_1 + \delta$. To prove uniqueness, it is enough to show that (\ref{eq6.16'}) holds. Since the motion is free in the interior of $\mathcal{M}$, it is clear that $(y(t),w(t)) = (y^0 + tw^0, w^0)$ for $0 \leq t < \overline{t}_1$, and in particular $y^1 = y(\overline{t}_1)$. Write $y^1 = (x^1,\alpha^1) \in \partial_{\reg} \mathcal{M}$. We have $w^- = (v^-,\omega^-) = \lim_{t \to \overline{t}_1-} w(t) = w^0$, and $w^+ = (v^-,\omega^-) = \lim_{t \to \overline{t}_1+} w(t) = w(\overline{t}_1)$, where the last equality follows by right-continuity. By assumption, we have 
\begin{equation}
    w^+ - w^- = w(\overline{t}_1) - w^0 = \begin{pmatrix} m^{-1}\lambda k(\sigma(y^1)) \\ J^{-1}\lambda \sin\beta(y^1) \end{pmatrix}.
\end{equation}
Write $k = k(\sigma(y^1)) = (k_1,k_2) \in \mathbb{R}^2$, and write $y^1 = (x^1,\alpha^1)$. The vector $x^1 - \sigma(y^1)$ is equal to $(-\sin\alpha^1, \cos\alpha^1)$, and consequently $\frac{\pi}{2} - \beta(y^1)$ is the angle measured counterclockwise from the vector $(\cos\alpha^1,\sin\alpha^1)$ to the normal $k(y^1)$. We therefore have 
\begin{equation}
    \sin\beta = \cos(\frac{\pi}{2} - \beta) = k \cdot (\cos\alpha^1, \sin\alpha^1) = k_1\cos\alpha^1 + k_2\sin\alpha^1.
\end{equation}
Thus
\begin{equation} \label{eq6.19}
    w(\overline{t}_1) - w^0 =  \lambda \begin{pmatrix} m^{-1}k_1 \\ m^{-1}k_2 \\ J^{-1}(k_1\cos\alpha + k_2\sin\alpha)
    \end{pmatrix} = R\lambda n,
\end{equation}
where $R = \left(m^{-1} + J^{-1}(k_1\cos\alpha^1 + k_2\sin\alpha^1)^2 \right)^{1/2}$, here using the formula (\ref{eq3.26}) for the unit normal $n$ in the configuration space given in Remark \ref{rem_normal}. By conservation of energy and the above, we have
\begin{equation} \label{eq6.9}
\begin{split}
    ||w^0||^2 = ||w(\overline{t}_1)||^2 & = ||w^0 + R \lambda n||^2 \\
    & = ||w^0||^2 + 2R\lambda\langle w^0,n \rangle + (R\lambda)^2 \\
    \Rightarrow  \quad 0 & = 2R\lambda\langle w^0,n \rangle + (R\lambda)^2 \\
    \Rightarrow \quad R\lambda & = -2\langle w^0, n \rangle \quad \text{ or } \quad R\lambda = 0.
\end{split}
\end{equation}
As $(y^1,w^0) \in V_{\text{out}}$, we see that $\langle w^0, n \rangle \neq 0$. By (\ref{eq6.19}) and (\ref{eq6.9}) either $w(\overline{t}_1) = w^0$ or $w(\overline{t}_1) = w^0 - 2\langle w^0, n \rangle n$. But we cannot have $w(\overline{t}_1) = w^0$ since this would imply by right-continuity that the point mass enters $\mathcal{M}^c$.

It remains to show that the motion is linear for $\overline{t}_1 < t < \delta$. From the above, $(y^1,w(\overline{t}_1)) \in V_{\text{in}}$, and thus by continuity there exists a maximal (possibly infinite) $\eta > 0$ such that $y(t) \in \Int \mathcal{M}$ for $\overline{t}_1 < t < \overline{t}_1 + \eta$. By free motion in the interior of $\mathcal{M}$, it follows that 
\begin{equation} \label{eq6.21}
(y(t),w(t)) = (y^1 + (t - \overline{t}_1)w(\overline{t}_1), w(\overline{t}_1)) \quad \text{ for } \quad \overline{t}_1 < t < \overline{t}_1 + \eta.
\end{equation}
It is easy to see that $\eta = \delta$. For (\ref{eq6.21}) and the definition of $\eta$ imply that $\eta \leq \overline{t}_2 = \delta$. On the other hand, if $\eta < \overline{t}_2$, then $y(\overline{t}_1 + \eta) \in \Int \mathcal{M}$ by continuity of the trajectory, and thus there exists $\eta' > \eta$ such that $y(t) \in \Int \mathcal{M}$ for $\overline{t}_1 < t < \overline{t}_1 + \eta'$, contradicting maximality of $\eta$. We conclude uniqueness.

To prove existence, we take $(y(t),w(t))$ to be defined by (\ref{eq6.16'}) for $0 \leq t < \overline{t}_1 + \delta$. Then $y(t)$ is continuous and $w(t)$ is right-continuous. Further, the only condition in (\ref{eq6.15}) which is not trivial to verify is the third. For this, note that we may take $\lambda = -2R^{-1}\langle w^0, n \rangle$, where $R$ is defined as above. It follows from the expression (\ref{eq3.26}) for the normal vector $n$ that the third condition in (\ref{eq6.15}) holds.
\end{proof}

We close by providing some support for the assumption that the energy of the disk is conserved for all time.  

\begin{proposition} \label{prop_energylimit}
Consider a physical system consisting of two bodies $B_1$ and $B_2$ in the plane, of mass $M_1$ and $M_2$ respectively and with moments of inertia $J_1$ and $J_2$ respectively about their centers of mass. Assume that the kinetic energy, linear momentum, and angular momentum of the system are conserved for all time. Let $t < t'$, let ${v^i}$, ${v^i}'$ be the linear velocity of the center of mass of $B_i$ at times $t, t'$ respectively, let ${\omega^i}, {\omega^i}'$ be the angular velocity of $B_i$ at times $t, t'$ respectively, and assume that ${v^2} = {\omega^2} = 0$. Then, keeping $M_1$, $J_1$, ${v^1}$, and ${\omega^1}$ fixed, the following limit holds:
\begin{equation} \label{eq3.15}
    \lim_{M_2, J_2 \to \infty} M_1||{{v^1}}'||^2 + J_1|{\omega^1}'|^2 = M_1||{v^1}||^2 + J_1|{\omega^1}|^2.
\end{equation}
\end{proposition}
\begin{proof}[Proof of Proposition \ref{prop_energylimit}]
Since we assume ${v^2} = 0$ and ${\omega^2} = 0$,  conservation of linear and angular momentum give us
\begin{equation}
    M_1{{v^1}}' + M_2{v^2}' = M_1{v^1},
\end{equation}
and 
\begin{equation}
    J_1{\omega^1}' + J_2{\omega^2}' = J_1{\omega^1}.
\end{equation}
We obtain from these, 
\begin{equation}
    M_2^2||{v^2}'||^2 \leq 2M_1^2||{v^1}'||^2 + 2M_1^2||{v^1}||^2, \quad\quad J_2^2|{\omega^2}'|^2 \leq 2J_1^2|{\omega^1}'|^2 + 2J_1^2|{\omega^1}|^2,
\end{equation}
and hence 
\begin{equation} \label{eq6.4}
\begin{split}
    M_2||{v^2}'||^2 & + J_2|{\omega^2}'|^2 \leq \frac{2M_1^2}{M_2}(||{v^1}'||^2 + ||{v^1}||^2) + \frac{2J_1^2}{J_2}(|{\omega^1}'|^2 + |{\omega^1}|^2) \\
    & \leq \max\left\{\frac{2M_1}{M_2}, \frac{2J_1}{J_2}\right\}(M_1||{v^1}'||^2 + J_1|{\omega^1}'|^2 + M_1||{v^1}||^2 + J_1|{\omega^1}|^2).
\end{split}
\end{equation}
By conservation of energy, we also have
\begin{equation} \label{eq6.5}
    M_1||{v^1}'||^2 + J_1|{\omega^1}'|^2 + M_2||{v^2}'||^2 + J_2|{\omega^2}'|^2 =  M_1||{v^1}||^2 + J_1|{\omega^1}|^2.
\end{equation}
Therefore, 
\begin{equation} \label{eq6.6}
\begin{split}
    M_1||{v^1}'||^2 + J_1|{\omega^1}'|^2 \leq M_1||{v^1}||^2 + J_1 |{\omega^1}|^2. 
\end{split}
\end{equation}
From (\ref{eq6.6}) and (\ref{eq6.4}) we obtain that 
\begin{equation}
    M_2||{v^2}'||^2 + J_2|{\omega^2}'|^2 \leq \max\left\{\frac{4M_1}{M_2}, \frac{4J_1}{J_2}\right\}(M_1||{v^1}||^2 + J_1|{\omega^1}|^2).
\end{equation}
From this we see that 
\begin{equation}
    M_2||{v^2}'||^2 + J_2|{\omega^2}'|^2 \to 0 \quad \text{ as } M_2, J_2 \to \infty.
\end{equation}
Hence, by conservation of energy (\ref{eq6.5}), the limit (\ref{eq3.15}) holds.
\end{proof}

\section{Elementary Properties of the Billiard System} \label{sec_config}

The goals for this section are first to describe the basic properties of the configuration space $\mathcal{M}$ and its cylindrical approximation $\mathcal{M}_{\cyl}$, and second to provide a rigorous definition for the collision law described in the introduction. \S\ref{ssec_configprops} and \ref{ssec_cylconfig} are concerned exclusively with the geometry of $\mathcal{M}$ and $\mathcal{M}_{\cyl}$, while \S\ref{ssec_collaws} introduces dynamics. 

The collision law turns out to be a special case of the general macro-reflection laws, described in \S\ref{sec_genref}. Results proved in \S\ref{sec_genref} can be applied to show, for example, that the collision law is symmetric with respect to the measure $\Lambda^2$.

In addition to describing the collision law in $\mathcal{M}$, we will also define two auxiliary collision laws, the \textit{cylindrical collision law} and the \textit{modified collision law}, which will play important roles in the proofs of our main results.

\subsection{Properties of the configuration space} \label{ssec_configprops}

Throughout this subsection, the cell $\Sigma$ -- assumed to satisfy conditions B1-B5 of \S\ref{sssec_periodic} -- and the roughness scale $\epsilon > 0$ are fixed. As described in \S\ref{ssec_main_results}, the cell and roughness scale give rise to a fixed wall $W = W(\Sigma,\epsilon)$ and a freely moving disk with satellites $D = D(\epsilon)$, where the satellites are spaced at angles $\rho(\epsilon)$ apart. This physical system has a configuration space $\mathcal{M} = \mathcal{M}(\Sigma,\epsilon) \subset \mathbb{R}^3$, defined by the expression (\ref{eq1.37}). 

We first describe some elementary properties of the configuration space $\mathcal{M}$. In what follows, $\mathcal{H}^2$ denotes 2-dimensional Hausdorff measure on $\mathbb{R}^3$.

\begin{proposition} \label{prop_Melem}
\begin{enumerate} [label = \roman*.]
    \item For all $j,k \in \mathbb{Z}$, $\mathcal{M} + j\epsilon e_1 + k \rho e_3 = \mathcal{M}$.
    
    \item $\{(x_1,x_2,\alpha) : x_2 \geq 0\} \subset \mathcal{M} \subset \{y \in \mathbb{R}^3 : D(y) \cap \Int W = \emptyset\}$.
    
    \item $\Int\mathcal{M} = \{y \in \mathbb{R}^3 : D(y) \cap W = \emptyset\}$; consequently, $\mathcal{M} = \overline{\Int \mathcal{M}}$.
    
    \item $\partial \mathcal{M} \subset \{y \in \mathbb{R}^3 : D(y) \cap \Int W = \emptyset, \text{ and for some $0 \leq k \leq N$, } S_k(y) \in \partial W\}$
    
    \item $\partial \mathcal{M} \subset \{(x_1,x_2,\alpha) : -\epsilon - \frac{1}{8}\rho^2 \leq x_2 \leq 0\}$.
    
    \item Assume $\epsilon$ is sufficiently small. For any $y \in \mathcal{M}$, there are at most two satellites of $D(y)$ which lie on $\partial W$. If two satellites of $D(y)$ lie in $\partial W$, then the satellites are adjacent to each other on the disk.
    
    \item $\mathcal{H}^2(\{(x_1,x_2,\alpha) : x_2 \geq 0\} \cap \overline{\mathcal{M}^c}) = 0$.
\end{enumerate}
\end{proposition}

Next, we will describe the regularity properties of $\mathcal{M}$, and introduce a large subset of $\mathcal{M}$ which has a simple parametrization. This requires some additional notation. We define a disjoint union of open intervals in the line:
\begin{equation} \label{eq4.1}
\mathcal{Z} = \mathcal{Z}(\epsilon) = \bigcup_{k \in \mathbb{Z}} \left(\frac{2k-1}{2}\rho(\epsilon) + \delta_0(\epsilon), \frac{2k + 1}{2}\rho(\epsilon) - \delta_0(\epsilon) \right)
\end{equation}
where 
\begin{equation}
    \delta_0 =\delta_0(\epsilon) = \arcsin\left(\frac{\epsilon}{2\sin(\rho(\epsilon)/2)}\right).
\end{equation}
We also let
\begin{equation}
    \widehat{\mathcal{Z}} = \mathbb{R}^2 \times \mathcal{Z}.
\end{equation}
One may easily check that, for any $R > 0$, the Lebesgue measure of $[-R, R] \smallsetminus \mathcal{Z}$ is of order $O(\delta_0) = O(\epsilon/\rho) = o(\epsilon^{1/2})$. Thus $\widehat{\mathcal{Z}}$ is a large subset of $\mathbb{R}^3$ in the sense that for any bounded subset $B \subset \mathbb{R}^3$, the Lebesgue measure of $B \smallsetminus \widehat{\mathcal{Z}}$ converges to zero as $\epsilon \to 0$. We define 
\begin{equation}
    \mathcal{M}_{\roll} = \mathcal{M} \cap \widehat{\mathcal{Z}},
\end{equation}
\begin{equation}
    \Gamma_{\roll} = \partial \mathcal{M} \cap \widehat{\mathcal{Z}}.
\end{equation}
We have

\begin{proposition} \label{prop_Mreg}
Assume $\epsilon$ is sufficiently small.
\begin{enumerate}[label = \roman*.]
    \item There exists a closed subset $\mathcal{S} \subset \partial \mathcal{M}$ such that 
    \begin{enumerate}
        \item[i.1.] $\mathcal{H}^2(\mathcal{S}) = 0$; and
        
        \item[i.2.] for every $q \in \partial \mathcal{M} \smallsetminus \mathcal{S}$, there exists a neighborhood $U \subset \mathbb{R}^3$ of $q$ and a $C^2$ diffeomorphism $\phi : U \to\mathbb{R}^3$ such that $\phi(q) = 0$ and $\phi(U \cap \mathcal{M}) = \mathbb{H}^3 := \{(x_1,x_2,x_3) : x_3 \geq 0\}$.
    \end{enumerate}
    
    \item For any $y \in \mathcal{M}_{\roll}$, $D(y) \cap W$ has cardinality at most 1.
    
    For any $y \in \Gamma_{\roll}$, the set $D(y) \cap W$ has cardinality 1, and consists of a single satellite of $D$.
    
    \item $\mathcal{M}_{\roll}$ is parametrized by the function
    $F : \overline{W^c} \times \mathcal{Z} \to \mathbb{R}^3$ defined by 
    \begin{equation} \label{eq4.4}
        F(x_1,x_2,\alpha) = (x_1 - \sin\overline{\alpha}, x_2 + \cos\overline{\alpha}, \alpha),
    \end{equation}
    where $\overline{\alpha} := \alpha - \overline{k}\rho$ and $\overline{k} := \argmin\{|\alpha - k\rho| : k \in \mathbb{Z}\}$.
    
    Moreover, the restriction of $F$ to $\partial W \times \mathcal{Z}$ parametrizes $\Gamma_{\roll}$.
\end{enumerate}
\end{proposition}

\begin{remark} \normalfont
The quantity $\overline{k}$ is uniquely defined because $\mathcal{Z}$ does not contain points of form $(\frac{2j+1}{2})\rho$, $j \in \mathbb{Z}$.
\end{remark}

\begin{remark} \normalfont
The parametrization (\ref{eq4.4}) explains the notation $\mathcal{M}_{\roll}$ and $\Gamma_{\roll}$. Suppose that the disk initially has the configuration $(x_1,x_2, k\rho)$ where $k \in \mathbb{Z}$, and imagine rotating the center of the disk counterclockwise about the satellite $S_k$ by a small angle $\Delta \alpha = \alpha - k\rho$, keeping the position of $S_k$ fixed. After ``rolling'' the disk in this way, the final configuration is given by $F(x_1,x_2,\alpha)$.
\end{remark}

We call a point in $\mathcal{M}$ a \textit{regular point} if it belongs to the set 
\begin{equation}
    \mathcal{M}_{\reg} := \mathcal{M} \smallsetminus \mathcal{S}. 
\end{equation}
To avoid ambiguity, we will assume that $\mathcal{S}$ is the minimal subset of $\partial \mathcal{M}$ (with respect to inclusion) such that $i.1$ and $i.2$ in Proposition \ref{prop_Mreg} hold.

In \S\ref{sss_sing}, we define the class of billiard domains $\CES^2_0(\mathbb{R}^3)$. The class is sufficiently large to encompass the kinds of spaces dealt with in this work, while narrow enough that standard results from billiards theory still apply. The acronym $\CES$ stands for ``closed and embedded, with singularities.'' The subscript $0$ means that the codimension in $\mathbb{R}^3$ of a submanifold in this class is zero, and the superscript $2$ indicates that the submanifold is twice differentiable. An immediate consequence of Proposition \ref{prop_Mreg}(i) is the following fact.
 
\begin{corollary} \label{cor_classy}
$\mathcal{M}$ belongs to the class $\CES_0^2(\mathbb{R}^3)$.
\end{corollary}

To prove the propositions, we introduce the following notation. Recall that $N$ denotes the number of satellites of the disk. For $y = (x,\alpha) \in \mathbb{R}^3$ and $0 \leq k \leq N-1$, let $S_k(y)$ denote the position of the satellite $S_k$ after rotating the disk in reference configuration (\ref{eq1.2}) about its center of mass counterclockwise by an angle of $\alpha$ and translating by $x$. Explicitly, 
\begin{equation}
    S_k(x_1,x_2,\alpha) = (x_1 + \sin(\alpha + k\rho), x_2 - \cos(\alpha + k\rho)).
\end{equation}
The map $y \mapsto S_k(y): \mathbb{R}^3 \to \mathbb{R}^2$ is smooth, and for each fixed $\alpha$, the map $x \mapsto S_k(x,\alpha)$ is a translation in the plane $\mathbb{R}^2$. 

As noted in Remark \ref{rem_no_contact}, the inner body $D_0$ cannot come into contact with the wall $W$. Thus, for the purposes of describing the geometry of $\mathcal{M}$, there is no loss of generality if we assume $D_0 = \emptyset$. We use this fact without comment in the proofs below.

\begin{proof}[Proof of Proposition \ref{prop_Melem}]
(i) We must show that $\mathcal{M}$ is invariant under translation in $\mathbb{R}^3$ by $\epsilon e_1$ and by $\rho e_3$. The first of these follows because the wall $W = W(\Sigma,\epsilon)$ is invariant under translation by $\epsilon e_1$. The second follows by noting that a rotation of the disk through an angle of $\rho$ about its center of mass maps the disk onto itself.

(ii) Since the satellites of $D$ lie at unit distance from the center, and $W \subset \{(x_1,x_2) : x_2 \leq -1\}$, it follows that $\{(x_1,x_2,\alpha) : x_2 > 0\} \subset \mathcal{M}$. Since $\mathcal{M}$ is closed, the first inclusion follows. 

By construction, $D \cap \Int W = \emptyset$ if and only if none of the satellites of $D$ lie in $\Int W$. That is,
\begin{equation}
    \{y : D(y) \cap \Int W = \emptyset\} = \{y : \text{for } 0 \leq k \leq N-1, S_k(y) \notin \Int W \}.
\end{equation} 
By continuity of the $S_k$'s, this set is closed. The second inclusion then follows by the definition (\ref{eq1.37}) of $\mathcal{M}$.

(iii) Let $\mathcal{M}_0 = \{y : D(y) \cap W = \emptyset\}$. Equivalently,  
\begin{equation}
    \mathcal{M}_0 = \{y : \text{for } 0 \leq k \leq N-1, S_k(y) \notin W \}.
\end{equation} 
By continuity of the $S_k$'s, this is an open set, and it is also evident that $\mathcal{M}_0 \subset \mathcal{M}$. Thus $\mathcal{M}_0 \subset \Int \mathcal{M}$. For the reverse inclusion, suppose that $y \notin \mathcal{M}_0$. Then for some $k$, $S_k(y) \in W$. Let $V \subset \mathbb{R}^3$ be any neighborhood of $y$. Since $W$ is the closure of its interior  by condition A1 in \S\ref{sssec_uphalf}, it follows that $V \cap \Int W \neq \emptyset$. Therefore, by (ii) $V$ intersects $\mathcal{M}^c$. Hence $y \notin \Int \mathcal{M}$, and this proves $\mathcal{M}_0 \supset \Int\mathcal{M}$.

(iv) It follows from (ii) and (iii) that 
\begin{equation}
    \partial \mathcal{M} \subset \{y \in \mathbb{R}^3 : D(y) \cap \Int W = \emptyset \text{ and } D(y) \cap \partial W \neq \emptyset\}.
\end{equation}
If $D$ intersects $W$, then one of its satellites lies in $W$, so the result follows.

(v) By definition of $W = W(\Sigma,\epsilon)$,
\begin{equation} \label{eq4.9'}
    \{(x_1,x_2) : x_2 \leq -1 - \epsilon\} \subset W \subset \{(x_1,x_2) : x_2 \leq -1\}. 
\end{equation}
Consequently, since the satellites lie at unit distance from the center of mass of $D$, if $y =(x_1,x_2,\alpha) \in \partial \mathcal{M}$, then $x_2 \leq 0$; otherwise $D(y)$ could not intersect $W$. On the other hand, a strict lower bound is obtained by putting $D$ in a configuration such that two adjacent satellites lie on the line $x_2 = -1 - \epsilon$. See Figure \ref{closest_fig}. Trigonometry yields that the $x_2$-coordinate of the center of the disk in this configuration is 
\begin{equation} \label{eq4.9}
    -1 - \epsilon + \cos(\rho/2) \geq -\epsilon - \frac{1}{8}\rho^2.
\end{equation}

\begin{figure}
    \centering
    \includegraphics[width = 0.9\linewidth]{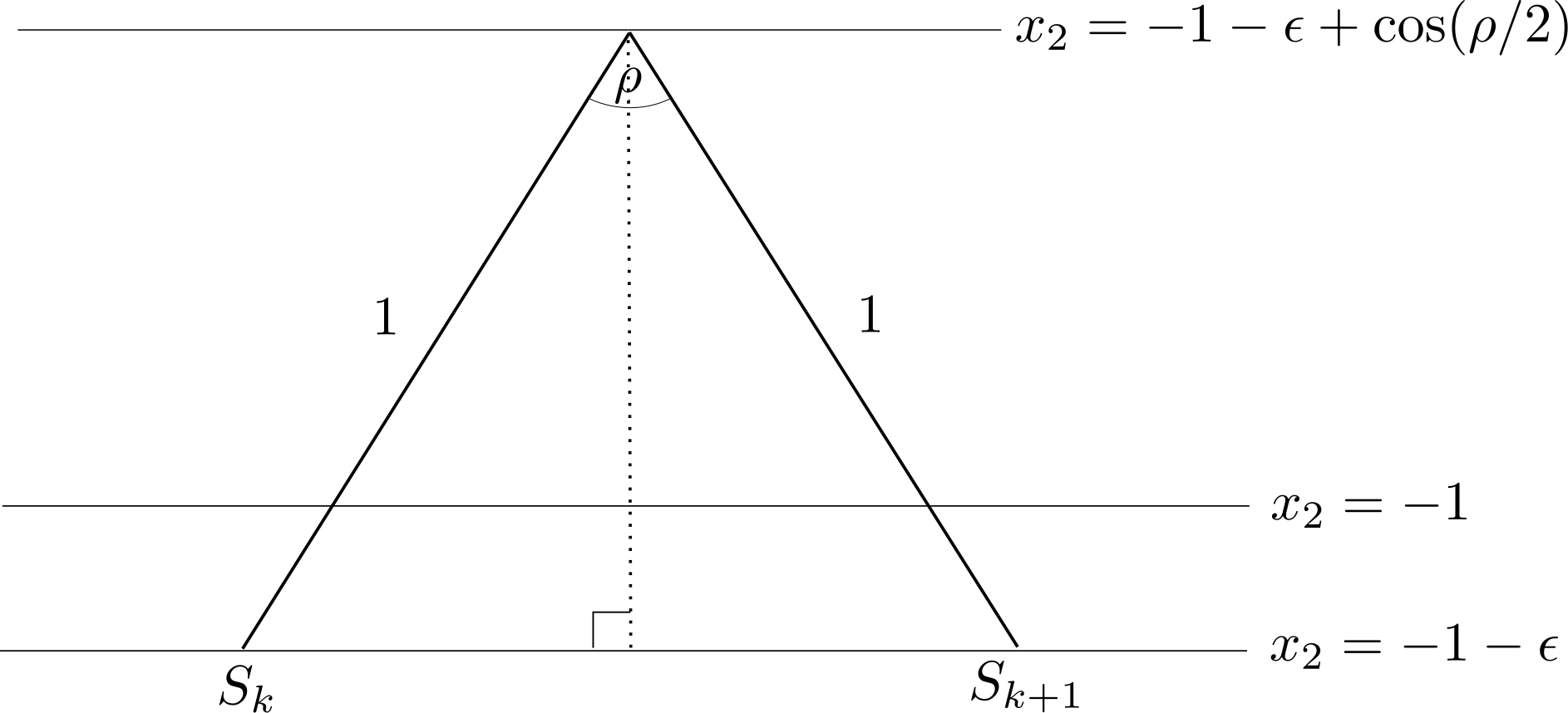}
    \caption{Two satellites lying on the line $x_2 = -1-\epsilon$.}
    \label{closest_fig}
\end{figure}

(vi) Let $y = (x_1,x_2,\alpha) \in \mathcal{M}$, and without loss of generality, suppose that $S_0(y)$, $S_1(y)$, and $S_2(y)$ all lie in $W$. An upper bound for $x_2$ is attained in the case where $S_0(y)$ and $S_2(y)$ both lie on the line $x_2 = -1$. In this case, $x_2 = -1 + \cos\rho \leq -\frac{1}{4}\rho^2$ for $\rho$ sufficiently small. Since $\rho(\epsilon) \to 0$ and $\epsilon/\rho(\epsilon)^2 \to 0$, this contradicts the bound (\ref{eq4.9}) if $\epsilon$ is sufficiently small.

(vii) Suppose $y = (x_1,x_2,\alpha) \in \overline{\mathcal{M}^c} = (\Int \mathcal{M}^c)$, and $x_2 \geq 0$. Then at least one satellite of $D(y)$ lies in $W$. By (\ref{eq4.9'}) and the fact that each satellite lies at unit distance from the center of mass of $D$, this implies that $x_2 = 0$, and $\alpha = k\rho$ for some $k \in \mathbb{Z}$. That is, 
\begin{equation}
    \{(x_1,x_2,\alpha) : x_2 \geq 0\} \cap \overline{\mathcal{M}^c} \subset \{(x_1,x_2,\alpha) : x_2 = 0 \text{ and } \alpha = k\rho \text{ for some } k \in \mathbb{Z}\},
\end{equation}
and the set on the right is $\mathcal{H}^2$-null.
\end{proof}

\begin{proof}[Proof of Proposition \ref{prop_Mreg}]
(i) We begin by introducing the mapping $g : \mathbb{R}^2 \times \mathbb{R} \to \mathbb{R}^3$ defined by 
\begin{equation}
    g(p_1,p_2,\alpha) = (p_1 - \sin\alpha, p_2 + \cos\alpha, \alpha).
\end{equation}
This is the mapping which takes $(p_1,p_2,\alpha)$ to the unique configuration $y$ such that $S_0(y) = (p_1,p_2)$ and the angular orientation of the disk is $\alpha$. 
One may easily check that $g$ is a diffeomorphism. Consequently, $g$ takes $\mathcal{H}^2$-null sets to $\mathcal{H}^2$-null sets.

Consider the following subsets of $\mathcal{M}$:
\begin{equation}
    A = \{y \in \mathcal{M} : \text{$D(y) \cap W$ contains a singular point of the boundary $\partial W$}\},
\end{equation}
\begin{equation}
    B = \{y \in \mathcal{M} : \text{$D(y) \cap W$ contains more than one satellite of $D$}\}. \hspace{0.6in}
\end{equation}
We define
\begin{equation}
    \mathcal{S} = A \cup B.
\end{equation}
By continuity and the fact that $W$ is closed, it is easy to see that $A$ and $B$ are closed; therefore $\mathcal{S}$ is closed.

To prove (i.1) holds, we must show that $A$ and $B$ are $\mathcal{H}^2$-null. To show that $\mathcal{H}^2(A) = 0$, by symmetry, it is enough to show that $\mathcal{H}^2(\{y : S_0(y) \in \partial_s W\}) = 0$. To prove this, it is enough to show that $g^{-1}\{y : S_0(y) \in \partial_s W\}$ is an $\mathcal{H}^2$-null set. But this set is equal to $\partial_s W \times \mathbb{R}$, and the result follows since $\partial_s W$ is a discrete set of points.

It takes more effort to show that $\mathcal{H}^2(B) = 0$. By Proposition \ref{prop_Melem}(ii) and (vi) and symmetry, it is enough to show that 
\begin{equation}
    \{y \in \mathbb{R}^3 : S_0(y) \in \partial W \text{ and } S_1(y) \in \partial W\}
\end{equation}
is $\mathcal{H}^2$-null. By pulling back with respect to $g$, it is equivalent to show that 
\begin{equation}
    K := \{(p,\alpha) \in \partial W \times \mathbb{R} : S_1(g(p,\alpha)) \in \partial W\}
\end{equation}
is $\mathcal{H}^2$-null. Furthermore, by periodicity, it is enough to show that 
\begin{equation}
    K_1 := \{(p,\alpha) \in \partial W_1 \times [0,2\pi) : S_1(g(p,\alpha)) \in \partial W\}
\end{equation}
is $\mathcal{H}^2$-null, where $\partial W_1 := \partial W \cap [0,\epsilon) \times \mathbb{R}$ is a single period of the boundary of the wall $W$. 

Note that the restriction of $\mathcal{H}^2$ to $\partial W_1 \times [0,2\pi)$ is just the product measure induced by Lebesgue measure on each factor. We denote Lebesgue measure on both factors by $m$. For each $p \in \partial W_1$, let 
\begin{equation}
    F_p = \{\alpha \in [0,2\pi) : S_1(g(p,\alpha)) \in \partial W\},
\end{equation}
and let
\begin{equation}
    E = \{p \in \partial W_1 : m(F_p) > 0\}.
\end{equation}
For $r > 0$ and $p \in \mathbb{R}^2$, let $C(r,p)$ denote the circle of radius $r$ centered at $p$. Trigonometry shows that the distance between any two satellites is 
\begin{equation}
    2\sin(\rho/2).
\end{equation}
Therefore for each $p \in \partial W_1$, the mapping $\alpha \mapsto S_1(g(p,\alpha))$ takes $\alpha$ to the point on the circle $C(2\sin(\rho/2), p)$ making an angle of $\alpha$ counterclockwise from the point $S_1(g(p,0))$ in the same circle. Therefore, for any $U \subset [0,2\pi)$, 
\begin{equation}
    S_1(g(p,U)) \subset C(2\sin(\rho/2), p),
\end{equation}
and
\begin{equation} \label{eq4.23}
    m(S_1(g(p,U))) = 2\sin(\rho/2)m(U).
\end{equation}
For each $p \in \partial W_1$, let $G_p = S_1(g(p,F_p))$. If $p \in E$, then 
\begin{equation} \label{eq4.24}
    G_p \subset \partial W_2
\end{equation}
where $\partial W_2 := \partial W \cap [-1, \epsilon + 1]$, a bounded set. Moreover, since each $G_p$ is a subset of a circle centered at $p$, if $p \neq p'$, then $G_p$ and $G_{p'}$ can intersect in at most two points. In particular $m(G_p \cap G_{p'}) = 0$. Applying this and (\ref{eq4.23}) and (\ref{eq4.24}), we obtain
\begin{equation}
    \sum_{p \in E} 2\sin(\rho/2)m(F_p) = \sum_{p \in E} m(G_p) = m\left(\bigcup_{p \in E} G_p\right) \leq m(\partial W_2) < \infty.
\end{equation}
Since $m(F_p) > 0$ for each $p \in E$, it follows that $E$ is countable. Therefore, 
\begin{equation}
    \mathcal{H}^2(K_1) = \int_E m(F_p) m(\dd p) = 0,
\end{equation}
and this concludes the proof that $\mathcal{H}^2(B) = 0$.

It remains to show that (i.2) holds. Let $q = (p_1,p_2,\beta) \in \partial \mathcal{M} \smallsetminus \mathcal{S}$. Then exactly one satellite of $D(q)$ lies in $\partial_{\reg} W$, and all the other satellites lie in $W^c$. Without loss of generality, we may suppose $S_0(q) \in \partial_{\reg} W$. Let $V \subset \mathbb{R}^2$ be a neighborhood of $S_0(q)$ chosen small enough that there is some diffeomorphism $\psi : V \to \mathbb{R}^3$ such that $\psi(S_0(q)) = 0$, and $\psi(V \cap \overline{W^c}) = \{(x_1,x_2) : x_2 \geq 0\}$. Let 
\begin{equation}
    \widehat{U} = S_0^{-1}(V) = g(V \times \mathbb{R}),
\end{equation}
and let $\phi : \widehat{U} \to \mathbb{R}^3$ be defined by
\begin{equation}
    \phi(x_1,x_2,\alpha) = (\psi(S_0(x_1,x_2,\alpha)),\alpha - \beta).
\end{equation}
This is easily seen to be a diffeomorphism, with inverse given by $\phi^{-1}(x_1,x_2,\alpha) = g(\psi^{-1}(x_1,x_2),\alpha + \beta)$ and with $\phi(q) = 0$. We may find a neighborhood $U \subset \widehat{U}$ of $q$ such that $S_k(y) \in W^c$ for all $y \in U$ and $1 \leq k \leq N-1$. Then 
\begin{equation}
\begin{split}
    \phi(U \cap \mathcal{M}) & = \phi(U \cap \overline{\{y \in \mathbb{R}^3 : S_0(y) \in W^c\}}) \\
    & = \phi(U \cap \{y \in \mathbb{R}^3 : S_0(y) \in \overline{W^c}\}) \\
    & = \phi(U) \cap \psi(V \cap \overline{W^c}) \times \mathbb{R} \\
    & = \phi(U) \cap \{(x_1,x_2,x_3) : x_2 \geq 0\}.
\end{split}
\end{equation}
Thus $\phi : U \to \phi(U)$ gives us the desired map.

(ii) Since $\Gamma_{\roll} \subset \partial \mathcal{M}$, it is clear that for any $y \in \Gamma_{\roll}$, at least one satellite of $D(y)$ lies in $W$. Since $\Gamma_{\roll} \subset \mathcal{M}_{\roll} \subset \mathbb{R} \times [-\epsilon,\infty) \times \mathcal{Z}$, (ii) will follow from:

\begin{claim}{4.2.1} \label{claim4.2.1}
Suppose $y = (x_1,x_2,\alpha) \in \mathbb{R} \times [-\epsilon,\infty) \times \mathcal{Z}$. Then at most one satellite of $D(y)$ lies in $W$.
\end{claim}

To prove this, note that if $y \in \mathbb{R} \times (0,\infty) \times \mathcal{Z}$, then none of the satellites can lie in $W$. Thus, to prove the claim, we may suppose that $y \in \mathbb{R} \times [-\epsilon,0] \times \mathcal{Z}$. It is then enough to show that, for such a $y$, at most one satellite of $D(y)$ lies in the strip $\{(x_1,x_2) : -1 - \epsilon < x_2 \leq -1\}$. Suppose for a contradiction that two satellites lie in the strip. By (vi) the two satellites are adjacent. By symmetry, we may without loss of generality take the adjacent pair  to be $S_0$ and $S_1$. Up to symmetry, the situation is as depicted in Figure \ref{two_sat_fig}. We let $\epsilon_1 \geq 0$ be the magnitude of the difference in the $x_2$-coordinates of $S_0$ and $S_1$, and we let $\delta$ denote the unsigned angle between the ray $\overrightarrow {S_0 S_1}$ and the horizontal line through $S_0$. The center of $D$ and the two satellites $S_0$ and $S_1$ form an isosceles triangle. Elementary triangle geometry shows that 
\begin{equation}
    \delta = \frac{\rho}{2} - |\alpha|, \quad\quad\quad 2\sin\left(\frac{\rho}{2}\right)\sin\delta = \epsilon_1.
\end{equation}
Therefore, 
\begin{equation} \label{eq2.14}
\begin{split}
    \frac{\rho}{2} \geq |\alpha| = \frac{\rho}{2} - \delta & = \frac{\rho}{2} - \arcsin\left(\frac{\epsilon_1}{2\sin(\rho/2)}\right) \\
    & \geq \frac{\rho}{2} - \arcsin\left(\frac{\epsilon}{2\sin(\rho/2)}\right) = \frac{\rho}{2} - \delta_0,
\end{split}
\end{equation}
here using the fact that $\epsilon_1 \leq \epsilon$. Hence $\alpha$ is not in $\mathcal{Z}$, giving us the desired contradiction.

\begin{figure}
    \centering
    \includegraphics[width = 0.9\linewidth]{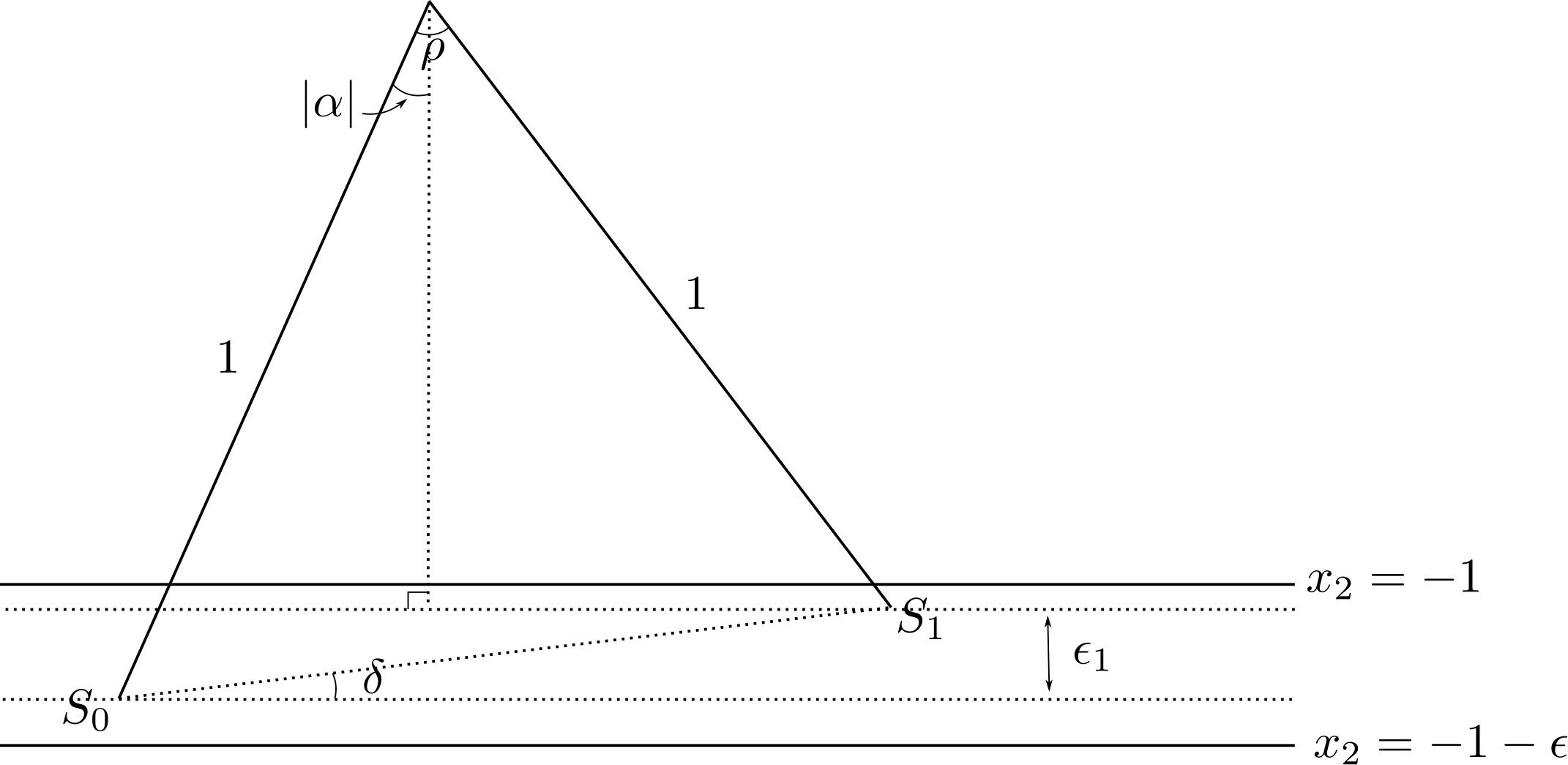}
    \caption{Isosceles triangle formed by the center of the disk and two satellites.}
    \label{two_sat_fig}
\end{figure}

(iii) Let 
\begin{equation} \label{eq4.34}
    \mathcal{Z}^0 = \left(-\frac{1}{2}\rho + \delta_0, \frac{1}{2}\rho - \delta_0 \right), \quad\quad \widehat{\mathcal{Z}}^0 = \mathbb{R}^2 \times \mathcal{Z}^0.
\end{equation}
The set $\mathcal{Z}^0$ is the member of the union (\ref{eq4.1}) corresponding to $k = 0$. Let $F^0$ denote the restriction of $F$ to $\overline{W^c} \times \mathcal{Z}^0$; then 
\begin{equation}
    F^0(x_1,x_2,\alpha) = (x_1 - \sin\alpha, x_2 + \cos\alpha, \alpha), \quad\quad (x_1,x_2,\alpha) \in \overline{W^c} \times \mathcal{Z}^0.
\end{equation}
Let $\mathcal{M}_{\roll}^0 = \mathcal{M} \cap \widehat{\mathcal{Z}}^0$, and $\Gamma_{\roll}^0 = \partial \mathcal{M} \cap\widehat{\mathcal{Z}}^0 \subset \Gamma_{\roll}$. As $\mathcal{M}$ is invariant under translation by $\rho e_3$, so is $\mathcal{M}_{\roll}$ and $\Gamma_{\roll}$. Consequently, it is enough to show that $F^0(\overline{W^c} \times \mathcal{Z}^0) = \mathcal{M}_{\roll}^0$ and $F^0(\partial W \times \mathcal{Z}^0) = \Gamma_{\roll}^0$.

We observe that $F^0$ is just the restriction of $g$ to $\overline{W^c} \times \mathcal{Z}^0$. In particular, $F^0$ maps a point $(p_1,p_2,\alpha) \in \overline{W^c} \times \mathcal{Z}^0$ to a configuration $y = (x_1,x_2,\alpha)$ such that $S_0(y) = (p_1,p_2)$. It follows from Claim \ref{claim4.2.1} in the proof of (ii) that at most one satellite of $D(y)$ can lie in the strip $\mathbb{R} \times [-1 - \epsilon, -1]$ and all other satellites lie in $\mathbb{R} \times (0,\infty)$. If one satellite lies in the strip, then as $S_0(y)$ has the minimal $x_2$-coordinate of all the satellites of $D(y)$, it follows that this satellite must be $S_0(y)$. As $S_0(y) = (p_1,p_2) \in \overline{W^c}$, it follows that $y \in \mathcal{M}$. This proves $F^0(\overline{W^c} \times \mathcal{Z}^0) \subset \mathcal{M}_{\roll}^0$.

Moreover, if we suppose that $(p_1,p_2,\alpha) \in \partial W \times \mathcal{Z}^0$, then the same reasoning goes through, but $S_0(y) \in \partial W$. Hence $y \in \Gamma_{\roll}^0$, and this proves $F^0(\partial W \times \mathcal{Z}^0) \subset\Gamma_{\roll}^0$.

For the reverse inclusions, note that if $y = (x_1,x_2,\alpha) \in \mathcal{M}_{\roll}^0$, then $(p_1,p_2) := S_0(y) \in \overline{W^c}$, so $F^0(p_1,p_2,\alpha) = y$, and this proves $F^0(\overline{W^c} \times \mathcal{Z}^0) \supset \mathcal{M}_{\roll}^0$.

Moreover, if we suppose that $y \in \Gamma_{\roll}^0$, then exactly one satellite must lie in $\partial W$, by (ii). Since $\alpha \in \mathcal{Z}^0$, this satellite is $S_0$. Thus if we take $(p_1,p_2) := S_0(x_1,x_2,\alpha)$, then $F^0(p_1,p_2,\alpha) = (x_1,x_2,\alpha)$, and this proves $F^0(\partial W \times \mathcal{Z}^0) \supset \Gamma_{\roll}^0$.
\end{proof}

\begin{proof}[Proof of Corollary \ref{cor_classy}]
We verify the definition of the class $\CES_0^2(\mathbb{R}^3)$, given in \S\ref{sss_sing}. Taking $\mathcal{S} = \mathcal{S}$ as in Proposition \ref{prop_Melem}(i), it is immediate that condition C1 of the definition holds. To see that condition C2 holds, let $y \in \mathcal{M} \smallsetminus \mathcal{S}$. If $y \in \Int \mathcal{M}$, then trivially there exists a neighborhood of $y$ which is $C^2$-diffeomorphic to $\mathbb{R}^3$. On the other hand, if $y \in \partial \mathcal{M} \smallsetminus \mathcal{S}$, then we apply Proposition \ref{prop_Melem}(i).
\end{proof}

\begin{remark} \label{rem_normal} \normalfont
One can extract from the proof of Proposition \ref{prop_Mreg} a formula for the unit normal vector to the configuration space boundary at a point $q = (p_1,p_2,\alpha) \in \partial \mathcal{M} \smallsetminus \mathcal{S}$. Indeed, for such a $q$, there exists a unique satellite of $D(q)$ which lies in $\partial W$. Suppose that $S_0(q) \in \partial W$. The proof of (i) tells us that in a neighborhood of $q$, for some $\delta > 0$, the boundary $\partial \mathcal{M}$ is parametrized by 
\begin{equation}
    \widetilde{g}(t,\beta) = g(p(t),\beta), \quad\quad (t,\beta) \in (-\delta, \delta) \times (\alpha - \delta, \alpha + \delta),
\end{equation}
where $p : (-\delta, \delta) \to \partial_{\reg} W$ is a parametrization of an open subset of the regular part of $\partial W$, chosen so that $p(0) = p = (p_1,p_2)$. Let $k = (k_1,k_2)$ denote the unit normal to $\partial W$ at the point $p$, with respect to the Euclidean inner product on $\mathbb{R}^2$. Let 
\begin{equation}
    n(q) = \left(m^{-1} + J^{-1}(k_1 \cos\alpha + k_2 \sin\alpha)^2\right)^{-1/2}\begin{pmatrix} m^{-1}k_1 \\ m^{-1}k_2 \\ J^{-1}(k_1 \cos\alpha + k_2 \sin\alpha) \end{pmatrix}.
\end{equation}
One may easily check that $||n|| = 1$, and that $\langle \partial_t \widetilde{g}(0,\alpha), n \rangle = \langle \partial_\alpha \widetilde{g}(0,\alpha), n \rangle = 0$. Consequently, $n$ is the unit normal to $\partial \mathcal{M}$ at the point $q$.

By symmetry, if more generally $q = (p_1,p_2,\alpha)$ is a point in $\partial \mathcal{M} \smallsetminus \mathcal{S}$ such that $S_j(q) \in \partial W$ and $k = (k_1,k_2)$ is the unit normal to $\partial W$ at the point $S_j(q)$, then the unit normal to $\partial \mathcal{M}$ at $q$ is given by the formula 
\begin{equation} \label{eq3.26}
\begin{split}
     n(q) = R^{-1}\begin{pmatrix} m^{-1}k_1 \\ m^{-1}k_2 \\ J^{-1}\left(k_1 \cos(\alpha - j\rho) + k_2 \sin(\alpha - j\rho)\right) \end{pmatrix},
\end{split}
\end{equation}
where $R := \left(m^{-1} + J^{-1}(k_1 \cos(\alpha - j\rho) + k_2 \sin(\alpha - j\rho)^2\right)^{1/2}$.
\end{remark}

\subsection{Cylindrical configuration space} \label{ssec_cylconfig}

For a fixed wall $W = W(\Sigma,\epsilon)$, the \textit{cylindrical configuration space} is defined by
\begin{equation}
    \mathcal{M}_{\cyl} = \{(x_1,x_2,\alpha) : (x_1 + \alpha, x_2 - 1) \in \overline{W^c}\}.
\end{equation}
This is the cylinder with base 
\begin{equation}
    \widehat{B} := \overline{W^c} + e_2 = \mathcal{M} \cap \{(x_1,x_2,\alpha) : \alpha = 0\}
\end{equation}
and axis
\begin{equation}
    \chi := (m + J)^{-1/2}(1,0,-1) \in \mathbb{S}^2.
\end{equation}
The space $\mathcal{M}_{\cyl}$ is much simpler than $\mathcal{M}$. It is parametrized by
\begin{equation}
    f_{\text{lin}}\begin{pmatrix} x_1 \\ x_2 \\ \alpha \end{pmatrix} = \begin{pmatrix} x_1 - \alpha \\ x_2 + 1 \\ \alpha \end{pmatrix}, \quad\quad (x_1,x_2,\alpha) \in \overline{W^c} \times \mathbb{R}.
\end{equation}
The restriction of $F_{\text{lin}}$ to $\overline{W^c} \times \mathbb{R}$ is just the linearization about $\alpha = 0$ of the parametrization $F$ for $\mathcal{M}_{\roll}$, defined by (\ref{eq4.4}).

Recall the definition of $\mathcal{Z}^0$ and $\widehat{\mathcal{Z}}^0$ (\ref{eq4.34}). The subset of the configuration space $\mathcal{M} \cap \widehat{\mathcal{Z}}^0$ can be expressed as a smooth perturbation of $\mathcal{M}_{\cyl} \cap \widehat{\mathcal{Z}}^0$. To make this statement precise, let $H_1 : \widehat{\mathcal{Z}} \to \widehat{\mathcal{Z}}$ be defined by
\begin{equation} \label{eq4.43}
H_1\left(\begin{matrix} x_1 \\ x_2 \\ \alpha \end{matrix}\right) =\begin{pmatrix}
x_1 + \overline{\alpha} - \sin(\overline{\alpha}) \\
x_2 - 1 + \cos(\overline{\alpha}) \\
\alpha
\end{pmatrix},
\end{equation}
where $\overline{\alpha} := \alpha - \overline{k}\rho$ and $\overline{k} := \text{argmin}\{|\alpha - k\rho| : k \in \mathbb{Z}\}$. (Note that $\overline{k}$ is uniquely determined since $\mathcal{Z}$ does not contain points of the form $\rho/2 + k\rho$.) The map $H_1$ is a diffeomorphism with inverse given by 
\begin{equation}
    H_1^{-1}\left(\begin{matrix} x_1 \\ x_2 \\ \alpha \end{matrix}\right) =\begin{pmatrix}
x_1 - \overline{\alpha} + \sin(\overline{\alpha}) \\
x_2 + 1 - \cos(\overline{\alpha}) \\
\alpha
\end{pmatrix}.
\end{equation}
The following identity holds:
\begin{equation}
    H_1 \circ f_{\text{lin}} = f \quad \text{ on } \overline{W^c} \times \mathcal{Z}^0.
\end{equation}
Consequently, by Proposition \ref{prop_Mreg}(iii),
\begin{equation} \label{eq4.46}
\begin{split}
   & H_1(\mathcal{M}_{\cyl} \cap \widehat{\mathcal{Z}}^0) = \mathcal{M} \cap \widehat{\mathcal{Z}}^0, \quad \text{ and } \\
   &H_1(\partial \mathcal{M}_{\cyl} \cap \widehat{\mathcal{Z}}^0) = \partial \mathcal{M} \cap \widehat{\mathcal{Z}}^0.
\end{split}
\end{equation}
Notice that $|\alpha| \leq \rho(\epsilon)/2$ in $\mathcal{Z}^0$. Consequently 
\begin{equation} \label{eq3.6}
\begin{split}
\sup_{\widehat{\mathcal{Z}}^0} ||H_1 - \text{Id}_{\widehat{\mathcal{Z}}^0}|| & \leq |\sin(\rho/2) - \rho/2| + |-1 + \cos(\rho/2)| \\
& = O(\rho^2) = o(1).
\end{split}
\end{equation}

To define the cylindrical collision law in the space $\mathcal{M}_{\cyl}$, we will need the following result.

\begin{lemma} \label{lem_classy0}
$\mathcal{M}_{\cyl}$ belongs to the class $\CES_0^2(\mathbb{R}^3)$.
\end{lemma}
\begin{proof}
Recall the given decomposition of $\partial W$ into countably many compact $C^2$ curve segments: $\partial W = \bigcup_{i \in \mathbb{Z}} \Gamma_i$. For each $i \in \mathbb{Z}$, let $\widehat{\Gamma}_i = \Gamma_i + e_2$. For each $i$, let $p_i$ and $p_i'$ denote the two endpoints of $\widehat{\Gamma}_i$ and let $\Int \widehat{\Gamma}_i = \widehat{\Gamma}_i \smallsetminus \{p_i, p_i'\}$. Let
\begin{equation}
    \mathcal{V}_i =  \{(x_1,x_2,\alpha) : (x_1 + \alpha, x_2 - 1) \in \Int\widehat{\Gamma}_i\},
\end{equation}
\begin{equation}
    \mathcal{S} = \left\{(x_1,x_2,\alpha) : (x_1 + \alpha, x_2 - 1) \in \bigcup_{i \in \mathbb{Z}} \{p_i, p_i'\} \right\}.
\end{equation}
Then the boundary of $\mathcal{M}_{\cyl}$ may be written as the disjoint union
\begin{equation}
    \partial \mathcal{M}_{\cyl} = \mathcal{S} \cup \bigcup_{i \in \mathbb{Z}} \mathcal{V}_i
\end{equation}
The set $\mathcal{S}$ is a countable, locally finite union of parallel lines. Thus $\mathcal{S}$ is closed, and the 2-dimensional Hausdorff measure of $\mathcal{S}$ is zero. 

Suppose $y = (x_1,x_2,\alpha) \in \mathcal{M}_{\cyl} \smallsetminus \mathcal{S}$. If $y \in \Int \mathcal{M}_{\cyl}$, then trivially there exists a neighborhood $U$ of $y$ which is entirely contained in $\mathcal{M}_{\cyl}$ and which is $C^2$-diffeomorphic to $\mathbb{R}^3$. If $y \in \mathcal{V}_i$ for some $i$, then 
\begin{equation}
    p := (x_1 + \alpha, x_2 - 1) \in \Int \widehat{\Gamma}_i.
\end{equation}
It follows from the assumptions on the curve segments $\Gamma_i$ that there exists a neighborhood $\widetilde{U} \subset \mathbb{R}^2$ of $p$ and a diffeomorphism $\widetilde{\phi} : \widetilde{U} \to \mathbb{R}^2$ such that $\widetilde{\phi}(p) = 0$ and $\widetilde{\phi}(\widehat{B} \cap \widetilde{U}) = \mathbb{H}^2 := \mathbb{R} \times [0,\infty)$. Let 
\begin{equation}
    U = \{(x_1,x_2,\alpha) : (x_1 + \alpha, x_2 - 1) \in \widetilde{U}\},
\end{equation}
and define $\phi : U \to \mathbb{R}^3$ by 
\begin{equation}
    \phi(x_1,x_2,\alpha) = (\widetilde{\phi}(x_1 + \alpha, x_2 - 1), \alpha).
\end{equation}
This is a diffeomorphism, and moreover $\phi(\mathcal{M}_{\cyl} \cap U) = \mathbb{R} \times [0,\infty) \times \mathbb{R} \cong \mathbb{H}^3$. This completes the proof.
\end{proof}

\subsection{Collision laws} \label{ssec_collaws}

\subsubsection{Definition and basic properties} \label{sssec_coldef}

As before, let $\Sigma$ be a fixed cell, and let $\epsilon > 0$. Recall the plane $\mathbf{P} = \{(x_1,x_2,\alpha) : x_2 = 0\}$, and the hemispheres $\mathbb{S}^2_{\pm} = \mathbb{S}^2 \cap \{(v_1,v_2,\omega) : \pm v_2 > 0\}$. 

In the introduction, we defined the collision law $K^{\Sigma, \epsilon}$ as follows: Let $(y,w) \in \mathbf{P} \times \mathbb{S}^2_+$, and consider the billiard trajectory starting in initial state $(y,-w)$. Suppose the trajectory hits the boundary $\partial \mathcal{M}$ and reflects specularly a certain number of times before returning to the plane $\mathbf{P}$ in a state $(y',w') \in \mathbf{P} \times \mathbb{S}^2_+$. The \textit{collision law (associated with the cell $\Sigma$ and scale $\epsilon$)} is the mapping 
\begin{equation}
    K^{\Sigma,\epsilon}(y,w) = (y',w').
\end{equation}
The domain of this mapping consists of pairs $(y,w)$ such that the billiard trajectory starting from $(y,-w)$ is well defined for all time and returns to the plane $\mathbf{P}$ after only finitely many reflections at regular points of $\partial \mathcal{M}$. For certain initial conditions, the trajectory may not be well-defined for all time. In particular, the trajectory may not in general be continued beyond a point where it hits a singular point of $\partial \mathcal{M}$ (i.e. the set $\mathcal{S}$ defined above) or where it hits the boundary tangentially.

Useful facts about the collision are obtained by observing that the mapping $K^{\Sigma,\epsilon}$ is a special case of the general macro-reflection laws defined in \S\ref{sec_genref}, and applying the results derived in that section. 

Recall the measure on $\mathbf{P} \times \mathbb{S}^2_+$,
\begin{equation}
    \Lambda^2(\dd y \dd w) = \langle w, e_2 \rangle \dd y \sigma(\dd w),
\end{equation}
where $\dd y$ is Lebesgue measure on $\mathbf{P}$ and $\sigma(\dd w)$ is surface measure on $\mathbb{S}^2$. This is just a special case of the measure $\Lambda$ defined in \S\ref{sssec_detref}.

\begin{proposition} \label{prop_detcol}
There exists a full measure open set $\mathcal{F} \subset \mathbf{P} \times \mathbb{S}^2$ such that:

(i) $K^{\Sigma,\epsilon} : \mathcal{F} \to \mathbf{P} \times \mathbb{S}^2_+$ is a well-defined $C^1$ mapping.

(ii) $K^{\Sigma,\epsilon}$ maps $\mathcal{F}$ into $\mathcal{F}$ and is an involution in the sense that $K^{\Sigma,\epsilon} \circ K^{\Sigma,\epsilon} = \text{Id}_{\mathcal{F}}$. Consequently, $K^{\Sigma,\epsilon} : \mathcal{F} \to \mathcal{F}$ is a $C^1$ diffeomorphism.

(iii) $K^{\Sigma,\epsilon}$ preserves the measure $\Lambda^2$.
\end{proposition}

\begin{proof}
We apply the results of \S\ref{sssec_detref} with $\mathcal{R} = \mathbb{R}^3$, $\mathcal{M}_1 = \mathcal{M}$, $\mathcal{M}_0 = \{(x_1,x_2,\alpha) : x_2 \geq 0\}$, and $\mathcal{N} = \overline{\mathcal{M}_1 \smallsetminus \mathcal{M}_0}$. For these results to hold, we need to check that conditions D1, D2, D3', D4, and D5 from \S\ref{sssec_detref} hold (for the statement of D3', see Remark \ref{rem_extend}). If they do, then in the notation of \S\ref{sssec_detref}
\begin{equation}
    K^{\Sigma,\epsilon} = P^{\mathcal{M},\mathcal{M}_0},
\end{equation}
and Proposition \ref{prop_detcol} is just a special case of Proposition \ref{prop_reflproperties}.

Condition D1 is obvious, and condition D2 follows from Corollary \ref{cor_classy}. 

By Proposition \ref{prop_Melem}(ii), $\mathcal{M}_0 \subset \mathcal{M}$. In the notation of Remark \ref{rem_extend}, 
\begin{equation}
    A_2 := \mathbf{P} \cap (\overline{\mathcal{M} \smallsetminus \mathcal{M}_0}) \smallsetminus \Int \mathcal{M} \subset \{(x_1,x_2,\alpha) : x_2 \geq 0\} \cap \mathcal{M},
\end{equation}
and by Proposition \ref{prop_Melem}(vii), the two-dimensional Hausdorff measure of the set on the right is zero. Therefore, $\mathcal{H}^2(A_2) = 0$, and condition D3' holds.

To verify condition D4, note that the boundary of $\mathcal{N}$ is contained in $\partial \mathcal{M} \cup \mathbf{P}$. It follows from Proposition \ref{prop_Melem}(v) that $\mathcal{N}$ is bounded between the planes $x_2 = 0$ and $x_2 = -\epsilon - \frac{1}{2}\rho^2$. Thus the only way for a billiard trajectory in $\mathcal{N}$ starting from a state $(y,w)$ to not return to the boundary is if the trajectory is parallel to the two planes. Hence $\mathcal{U}^+ \smallsetminus \mathcal{U}_{\fin} \subset \{(y,w) \in \mathcal{U}^+ : \langle w, e_2 \rangle = 0\}$, and this is a measure zero subset of $\mathcal{U}^+$.

To verify condition D5, we use the double periodicity of $\mathcal{M}$. Let $\widetilde{\mathcal{N}}$ be the space obtained by identifying points in $\mathcal{N}$ which are translates of each other by $i \epsilon e_1 + j \rho e_3$, $i,j \in \mathbb{Z}$. Since $\mathcal{N}$ is bounded between the planes $x_2 = 0$ and $x_2 = -\epsilon - \frac{1}{2}\rho^2$, the reduced billiard domain $\widetilde{\mathcal{N}}$ is compact, and the induced invariant measure on $\widetilde{\mathcal{N}}$ is finite. Thus the Poincar\'{e} Recurrence Theorem implies that, except on a null set of initial conditions, the billiard trajectory will return to the plane $\mathbf{P}$ after only finitely many collisions with $\partial \mathcal{M}$. 
\end{proof}

For $(j,k) \in \mathbb{Z}^2$, define translations 
\begin{equation} \label{eq4.58}
    \tau_{jk}(y) = y + j\epsilon e_1 + k \rho e_3, \quad\quad y \in \mathbf{P},
\end{equation}
and 
\begin{equation} \label{eq4.59}
    \overline{\tau}_{jk}(y,w) = (\tau_{jk}(y),w), \quad\quad (y,w) \in \mathbf{P} \times \mathbb{S}^2_+.
\end{equation}
An elementary but important fact is that $K^{\Sigma,\epsilon}$ commutes with these translations.

\begin{proposition} \label{prop_colcomm}
For all $(j,k) \in \mathbb{Z}^2$, $\overline{\tau}_{jk} \circ K^{\Sigma,\epsilon} = K^{\Sigma,\epsilon} \circ \overline{\tau}_{jk}$.
\end{proposition}

\begin{proof}
This is immediate from Proposition \ref{prop_Melem}(i), which tells us that the configuration space $\mathcal{M}$ is invariant under the translations $\overline{\tau}_{jk}$.
\end{proof}

The collision law $K^{\Sigma,\epsilon}$ is a associated with a Markov kernel on $\mathbf{P} \times \mathbb{S}^2_+$:
\begin{equation}
    \mathbb{K}^{\Sigma,\epsilon}(y,w;\dd y'\dd w') := \delta_{K^{\Sigma,\epsilon}(y,w)}(\dd y' \dd w').
\end{equation}
We say that a Markov kernel $\mathbb{K}(y,w;\dd y'\dd w')$ is a \textit{rough collision law} if there exist a sequence of cells $\Sigma_i$ (satisfying conditions B1-B5 of \S\ref{sssec_periodic}) and positive numbers $\epsilon_i \to 0$ such that 
\begin{equation}
    \mathbb{K}^{\Sigma_i,\epsilon_i}(y,w; \dd y' \dd w')\Lambda^2(\dd y \dd w) \to \mathbb{K}(y,w; \dd y' \dd w')\Lambda^2(\dd y \dd w)
\end{equation}
weakly in the space of measures on $\mathbf{P} \times \mathbb{S}^2_+$.

The most important basic property of the rough collision laws is the following:

\begin{proposition} \label{prop_rcol_symm}
If $\mathbb{K}(y,w;\dd y' \dd w')$ is a rough collision law, then it is symmetric with respect to the measure $\Lambda^2$ in the sense that for any $h \in C_c((\mathbf{P} \times \mathbb{S}^2_+)^2)$, 
\begin{equation}
\begin{split}
    &\int_{(\mathbf{P} \times \mathbb{S}^2_+)^2} h(y,w,y',w') \mathbb{K}(y,w; \dd y' \dd w') \Lambda^2(\dd y \dd w) \\
    & \quad\quad \int_{(\mathbf{P} \times \mathbb{S}^2_+)^2} h(y',w',y,w) \mathbb{K}(y,w; \dd y' \dd w') \Lambda^2(\dd y \dd w).
\end{split}
\end{equation}
\end{proposition}

\begin{proof}
It follows from the proof of Proposition \ref{prop_detcol} that the collision law is a special case of the macro-reflection law defined in \S\ref{sssec_detref}. Therefore, the rough collision law is a special case of the rough reflection law defined in \S\ref{sssec_genrfref}. Thus the proposition follows from Proposition \ref{prop_gensymm}.
\end{proof}

\subsubsection{Cylindrical collision law} \label{sssec_cylcol}

Recall the definition of the cylindrical configuration space. By condition A3' of \S\ref{sssec_rigid}, $\mathbb{R} \times (-\infty,-1 - \epsilon] \subset W \subset \mathbb{R} \times (-\infty,-1]$. It thus follows from the definition of $\mathcal{M}_{\cyl}$ that  
\begin{equation}
    \mathbb{R} \times [0,\infty) \times \mathbb{R} \subset \mathcal{M}_{\cyl} \subset \mathbb{R} \times [-\epsilon,\infty) \times \mathbb{R}.
\end{equation}
Let $(y,w) \in \mathbf{P} \times \mathbb{S}^2_+$. Suppose a point particle starting from initial state $(y,-w)$ hits and reflects specularly from $\partial \mathcal{M}_{\cyl}$ a certain number of times before eventually returning to $\mathbf{P}$ in a state $(y',w') \in \mathbf{P} \times \mathbb{S}^2_+$. By definition, the \textit{cylindrical collision law (associated with cell $\Sigma$ and scale $\epsilon$)} is the mapping 
\begin{equation}
    K^{\Sigma,\epsilon}_{\cyl}(y,w) = (y',w').
\end{equation}
In the notation of \S\ref{sssec_detref} with $\mathcal{M}_0 = \{(x_1,x_2,\alpha) : x_2 \geq 0\}$, 
\begin{equation}
    K^{\Sigma,\epsilon}_{\cyl} := P^{\mathcal{M}_{\cyl}, \mathcal{M}_0}.
\end{equation}
As in the previous case, $K^{\Sigma,\epsilon}_{\cyl}$ is well-defined on a full-measure subset of $\mathbf{P} \times \mathbb{S}^2_+$, and the following proposition holds.

\begin{proposition} \label{prop_detcolcyl}
There exists a full measure open set $\mathcal{F}_{\cyl} \subset \mathbf{P} \times \mathbb{S}^2$ such that:

(i) $K^{\Sigma,\epsilon}_{\cyl} : \mathcal{F}_{\cyl} \to \mathbf{P} \times \mathbb{S}^2_+$ is a well-defined $C^1$ mapping.

(ii) $K^{\Sigma,\epsilon}_{\cyl}$ maps $\mathcal{F}_{\cyl}$ into $\mathcal{F}_{\cyl}$ and is an involution in the sense that $K^{\Sigma,\epsilon} \circ K^{\Sigma,\epsilon} = \text{Id}_{\mathcal{F}_{\cyl}}$. Consequently, $K^{\Sigma,\epsilon} : \mathcal{F}_{\cyl} \to \mathcal{F}_{\cyl}$ is a $C^1$ diffeomorphism.

(iii) $K^{\Sigma,\epsilon}_{\cyl}$ preserves the measure $\Lambda^2$.
\end{proposition}

\begin{proof}
The situation is similar to that of Proposition \ref{prop_detcol}. Taking $\mathcal{M}_1 = \mathcal{M}_{\cyl}$, $\mathcal{M}_0 = \{(x_1,x_2,\alpha) : x_2 \geq 0\}$, and $\mathcal{N} = \overline{\mathcal{M}_1 \smallsetminus \mathcal{M}_0}$, it is sufficient to check that conditions D1, D2, D3', D4, and D5 from \S\ref{sssec_detref} hold. The proposition then follows as a special case of Proposition \ref{prop_reflproperties}.

The verification of conditions D1, D2, D4, and D5 is almost identical to the verification of the same conditions in the proof of Proposition \ref{prop_detcol}, so we omit it.

To verify condition D3', we must show that the 2-dimensional Hausdorff measure of 
\begin{equation}
    A_2 := \mathbf{P} \cap (\overline{\mathcal{M}_{\cyl} \smallsetminus \mathcal{M}_0}) \smallsetminus \Int \mathcal{M}_{\cyl}
\end{equation}
is zero. Let $G : \mathbb{R}^3 \to \mathbb{R}^2$ be defined by
\begin{equation}
    G(x_1,x_2,\alpha) = (x_1 + \alpha, x_2).
\end{equation}
Since $\mathcal{M}_{\cyl}$ is the cylinder with base $\widehat{B}$ and axis $\chi$, we have $\mathcal{M}_{\cyl} = G^{-1}(\widehat{B})$. Let $\mathbf{L}$ be the line $\{(x_1,x_2) : x_2 = 0\} \subset \mathbb{R}^2$. As the set $A_2$ is cylindrical with axis $\chi$, the projection of $A_2$ under $G$ onto $\mathbb{R}^2$ is 
\begin{equation}
    \widetilde{A}_2 := \mathbf{L} \cap \overline{\{(x_1,x_2) \in \widehat{B} : x_2 < 0\}} \smallsetminus \Int \widehat{B}.
\end{equation}
Since $\mathbf{L}$ is a subset of $\widehat{B}$, we have $\widetilde{A}_2 \subset \partial \widehat{B}$.

Fix any $q \in \widetilde{A}_2$, and let $\{\widehat{\Gamma}_i, i \in \mathbb{Z}\}$, be the collection of curve segments constituting the boundary of $\widehat{B}$, as in the proof of Lemma \ref{lem_classy0}.   Evidently, $q \in \widehat{\Gamma}_i$ for some $i$. But by our assumptions on the $\Gamma_i$, if the interior of $\widehat{\Gamma}_i$ intersects $\mathbf{L}$, then $\widehat{\Gamma}_i \subset \mathbf{L}$. Thus either $q$ is an endpoint of $\widehat{\Gamma}_i$, or $q \in \Int \widehat{\Gamma}_i \subset \mathbf{L}$. In the second case we may choose a neighborhood $U$ of $q$ small enough that $U \cap \widehat{B} = U \cap \{(x_1,x_2) : x_2 \geq 0\}$. But then $q$ is not in the closure of $\{(x_1,x_2) \in \widehat{B} : x_2 < 0\}$, contrary to our assumption. We conclude that $q$ is an endpoint of $\widehat{\Gamma}_i$.

Let $E$ be the set of endpoints of the $\Gamma_i$'s. By the argument above, $\widetilde{A}_2 \subset E$, and thus $A_2 \subset G^{-1}(E)$ is a discrete collection of lines. Therefore $\mathcal{H}^2(A_2) = 0$.
\end{proof}

For $(j,k) \in \mathbb{Z}^2$ and $s \in \mathbb{R}$, define translations 
\begin{equation}
    \tau_{jk}^{(s)}(y) = y + j\epsilon e_1 + k s \chi, \quad\quad y \in \mathbf{P},
\end{equation}
and 
\begin{equation}
    \overline{\tau}_{jk}^{(s)}(y,w) = (\tau_{jk}^{(s)}(y),w), \quad\quad (y,w) \in \mathbf{P} \times \mathbb{S}^2_+.
\end{equation}
The cylindrical collision law $K_{\cyl}^{\Sigma,\epsilon}$ commutes with these translations.

\begin{proposition} \label{prop_cylcomm}
For all $(j,k) \in \mathbb{Z}^2$, $\overline{\tau}_{jk}^{(s)} \circ K_{\cyl}^{\Sigma,\epsilon} = K_{\cyl}^{\Sigma,\epsilon} \circ \overline{\tau}_{jk}^{(s)}$.
\end{proposition}

\begin{proof}
Since $\mathcal{M}_{\cyl}$ is a cylinder with base $\widehat{B}$ and axis $\chi$, and the base is $\epsilon$-periodic in the $e_1$ direction, it follows that $\mathcal{M}_{\cyl}$ is invariant under the translations $\tau_{jk}^{(s)}$. This implies that $K_{\cyl}^{\Sigma,\epsilon}$ commutes with the translations $\overline{\tau}_{jk}^{(s)}$.
\end{proof}

\subsubsection{Modified collision law} \label{sssec_modcol}

When we compare the mappings $K^{\Sigma,\epsilon}$ and $K^{\Sigma,\epsilon}_{\cyl}$, the modified collision law will serve as a kind of intermediate mapping to which we may compare both. As the definition of the modified collision law is somewhat technical, the reader may wish to refer to \S\ref{ssec_pure_scaling} before reading this section for better motivation.

The plane $\mathbf{P}$ forms the boundary of the two half-spaces $\mathbb{R}^3_{+} := \{(x_1,x_2,\alpha) : x_2 > 0\}$ and $\mathbb{R}^3_{-} := \{(x_1,x_2,\alpha) : x_2 < 0\}$. Recall the mapping $H_1 : \widehat{\mathcal{Z}} \to \widehat{\mathcal{Z}}$ defined by (\ref{eq4.43}). We define subsets of $\widehat{\mathcal{Z}}$:
\begin{equation}
    \widetilde{\mathbf{P}} = H_1(\mathbf{P} \cap \widehat{\mathcal{Z}}),
\end{equation}
\begin{equation}
    \mathcal{O}_+ = H_1(\mathbb{R}^3_+ \cap \widehat{\mathcal{Z}}),
\end{equation}
\begin{equation}
    \mathcal{O}_- = H_1(\mathbb{R}^3_- \cap \widehat{\mathcal{Z}}).
\end{equation}
Then $\mathcal{O}_+$ and $\mathcal{O_-}$ are disconnected open subsets of $\widehat{\mathcal{Z}}$ satisfying 
\begin{equation}
    \overline{\mathcal{O}_+} \cup \overline{\mathcal{O}_-} = \widehat{\mathcal{Z}} \quad \text{ and } \quad \overline{\mathcal{O}_+} \cap \overline{\mathcal{O}_-} = \widetilde{\mathbf{P}},
\end{equation}
where closure is taken in $\widehat{\mathcal{Z}}$.

As $H_1$ fixes the $\alpha$-coordinate, $H_1$ maps $\mathbb{R}^3_- \cap \widehat{\mathcal{Z}}$ into itself. Since $\partial \mathcal{M}_{\cyl} \subset \overline{\mathbb{R}^3_-}$, and $\mathbf{P} \subset \overline{\mathbb{R}^3_+}$, it follows from (\ref{eq4.46}) that 
\begin{equation}
    \partial\mathcal{M} \cap \widehat{\mathcal{Z}} \subset \overline{\mathcal{O}_-} \text{ and } \mathbf{P} \cap \widehat{\mathcal{Z}} \subset \overline{\mathcal{O}_+}.
\end{equation}

The set $\widetilde{\mathbf{P}}$ is a smoothly embedded surface in $\widehat{\mathcal{Z}}$. For $(x_1,\alpha) \in \mathbf{P} \cap \widehat{\mathcal{Z}}$, we define 
\begin{equation} 
    u(x_1,\alpha) = \cos\overline{\alpha} - 1,
\end{equation}
where $\overline{\alpha} = \alpha - \overline{k}\rho$ and $\overline{k} = \text{argmin}\{|\alpha - k\rho| : k \in \mathbb{Z}\}$. One may verify directly from the definition of $H_1$ that $\widetilde{\mathbf{P}}$ is the graph of $u$, i.e. 
\begin{equation}
    \widetilde{\mathbf{P}} = \{(x_1,x_2,\alpha) \in \widehat{\mathcal{Z}} : x_2 = u(x_1,\alpha)\}.
\end{equation}
Two useful observations about $u$ are the following: First, for any $k \in \mathbb{Z}$, the restriction of $u$ to $\mathbf{P} \cap \mathbb{R}^2 \times (k\rho -\rho/2 + \delta_0, k\rho + \rho/2 - \delta_0)$ is concave. Second, we have the estimate on the gradient:
\begin{equation} \label{eq4.65}
    ||\grad u(x_1,\alpha)|| \leq |\sin\overline{\alpha}| \leq \sin(\rho/2) = o(1), \quad \text{ as } \epsilon \to 0.
\end{equation}
Thus $\widetilde{\mathbf{P}}$ is almost ``flat'' and almost parallel to $\mathbf{P}$. 

Let $\mathbb{S}^2_{+\rho/2} = \{w \in \mathbb{S}^2 : \langle w, e_2 \rangle > \sin(\rho/2)\}$, and define $\Psi : \widetilde{\mathbf{P}} \times \mathbb{S}^2_{+\rho/2} \to \mathbf{P} \times \mathbb{S}^2_{+\rho/2}$ by 
\begin{equation} \label{eq4.66}
    \Psi(y,w) = \left(y - \frac{\langle y, e_2\rangle}{\langle w,e_2 \rangle}w, w\right).
\end{equation}
That is, $\Psi$ maps $(y,w)$ to $(y',w)$, where $y'$ is the point of intersection of the ray $\{y + tw : t \geq 0\}$ with $\mathbf{P}$. The fact that, for $(y,w) \in \widetilde{\mathbf{P}} \times \mathbb{S}^2_{+\rho/2}$, the ray makes an angle of at most $\pi/2 - \rho/2$ with $e_2$ guarantees that the ray cannot intersect $\widetilde{\mathbf{P}}$ at any point other than its initial point and thus $\Psi$ is injective. In fact, the following is true:

\begin{lemma} \label{lem_modification}
$\Psi$ is a diffeomorphism onto an open subset of $\mathbf{P} \times \mathbb{S}^2_{+\rho/2}$.
\end{lemma}

\begin{proof}
By injectivity, it is enough to show that $\Psi$ is a local diffeomorphism. Let $\overline{\Psi}$ denote a function given by the formula (\ref{eq4.66}) and defined on an open subset of $\mathbb{R}^3 \times \mathbb{R}^3$ containing $\widetilde{\mathbf{P}} \times \mathbb{S}^2_{+\rho/2}$. The differential of $\overline{\Psi}$ takes the form
\begin{equation}
    \dd\overline{\Psi}(x_1,x_2,\alpha, v_1, v_2,\omega) = \begin{bmatrix}
    1 & -\frac{v_1}{v_2} & 0 & & & \\
    0 & 0 & 0 & & * & \\
    0 & -\frac{\omega}{v_2} & 1 & & & \\
    & & & & & \\
    & 0_{3 \times 3} & & & I_{3 \times 3} & \\
    & & & & & 
    \end{bmatrix},
\end{equation}
where $0_{3 \times 3}$ and $I_{3 \times 3}$ are the $3 \times 3$ zero matrix and identity matrix respectively. The null space of this matrix is 
\begin{equation}
    \text{Span}(\overline{w}), \quad \text{where } \overline{w} = (v_1,v_2,\omega,0,0,0) \in \mathbb{R}^6.
\end{equation}
A consequence of the estimate (\ref{eq4.65}) is that if $w = (v_1,v_2,\omega) \in \mathbb{S}^2_{+\rho/2}$, then $w$ cannot be tangent to $\widetilde{\mathbf{P}}$ at $(x_1,x_2,\alpha)$; otherwise $v_2 \leq \grad g(x_1,\alpha) \leq \sin(\rho/2)$. Thus $\overline{w}$ is not tangent to $\widetilde{\mathbf{P}} \times \mathbb{S}^2_+ \subset \mathbb{R}^3 \times \mathbb{R}^3$. Since $\Psi$ is the restriction of $\overline{\Psi}$ to $\widetilde{\mathbf{P}} \times \mathbb{S}^2_{+\rho/2}$, we conclude that the differential of $\Psi$ has full rank, and the result follows by the inverse function theorem.
\end{proof}

Set 
\begin{equation} \label{eq4.69}
    \overline{H}_1(y,w) = (H_1(y,w),w) \quad\quad \text{ for } (y,w) \in \widehat{\mathcal{Z}} \times \mathbb{S}^2.
\end{equation}
This is a diffeomorphism of $\widehat{\mathcal{Z}} \times \mathbb{S}^2$. Define $\eta : \mathbf{P} \cap \widehat{\mathcal{Z}} \times \mathbb{S}^2_{+\rho/2} \to \mathbf{P} \cap \widehat{\mathcal{Z}} \times \mathbb{S}^2_{+\rho/2}$ by 
\begin{equation} \label{eq4.70}
    \eta(y,w) = \Psi \circ \overline{H}_1(y,w).
\end{equation}
This is a diffeomorphism onto its image because $\overline{H}_1$ and $\Psi$ are.

Recall the full-measure subset open subset $\mathcal{F} \subset \mathbf{P} \times \mathbb{S}^2_+$ on which the collision law $K^{\Sigma,\epsilon} : \mathcal{F} \to \mathcal{F}$ is a well-defined involutive $C^1$ diffeomorphism, and let
\begin{equation} 
    \widetilde{\mathcal{F}} = \eta^{-1}(\mathcal{F} \cap K^{\Sigma, \epsilon}(\Img \eta))) \subset \mathbf{P} \times \mathbb{S}^2_{+\rho/2}.
\end{equation}
We define the \textit{modified collision law} $\widetilde{K}^{\Sigma,\epsilon} : \widetilde{\mathcal{F}} \to \widetilde{\mathcal{F}}$ by 
\begin{equation} \label{eq4.72}
    \widetilde{K}^{\Sigma,\epsilon} = \eta^{-1} \circ K^{\Sigma,\epsilon} \circ \eta.
\end{equation}
This is an involutive $C^1$ diffeomorphism of $\widetilde{\mathcal{F}}$ because $K^{\Sigma,\epsilon}$ is an involutive $C^1$ diffeomorphism of $\mathcal{F}$.

Note that $\widetilde{\mathcal{F}}$ is not a full-measure subset of $\mathbf{P} \times \mathbb{S}^2$. Nonetheless, we will see below that $\Lambda^2(B \smallsetminus \widetilde{\mathcal{F}}) \to 0$ as $\epsilon \to 0$ for any set $B$ of finite $\Lambda^2$-measure (see Lemma \ref{lem_mod} and Remark \ref{rem_3.2}).

Recall the translation maps $\tau_{jk}$ and $\overline{\tau}_{jk}$ defined by (\ref{eq4.58}) and (\ref{eq4.59}). It is easy to show directly that $\eta$ commutes with $\overline{\tau}_{jk}$. This observation and Proposition \ref{prop_colcomm} give us

\begin{proposition} \label{prop_modcomm}
For all $(j,k) \in \mathbb{Z}^2$, $\overline{\tau}_{jk} \circ \widetilde{K}^{\Sigma,\epsilon} = \widetilde{K}^{\Sigma,\epsilon} \circ \overline{\tau}_{jk}$.
\end{proposition}

We think of $\widetilde{K}^{\Sigma,\epsilon}$ as a collision law obtained by ``modifying'' $K^{\Sigma,\epsilon}$ in the spatial coordinates by $\eta$. Let us consider the perturbation $\eta$ in more detail. We denote its domain more succinctly by 
\begin{equation}
    \mathcal{G} := (\mathbf{P} \cap \widehat{\mathcal{Z}}) \times \mathbb{S}^2_{+\rho/2}.
\end{equation}
Let $(y,w) \in \mathcal{G}$ and let $(y',w') = \eta(y,w) \in \mathbf{P} \times \mathbb{S}^2_+$. Then $w' = w$, and 
\begin{equation}
\begin{split}
    ||y' - y|| & \leq ||y - H_1(y)|| + \left|\frac{\langle H_1(y), e_2 \rangle}{\langle w, e_2 \rangle}\right| \\
    & \leq |\alpha - \sin\overline{\alpha}| + |-1 + \cos\overline{\alpha}| + \frac{|-1 + \cos\overline{\alpha}|}{\sin(\rho/2)} \leq C\rho
\end{split}
\end{equation}
for some constant $C$, here using $|\overline{\alpha}| < \rho/2$. This shows that 
\begin{equation} \label{eq4.75}
    ||\eta - \text{Id}_{\mathcal{G}}||_{L^{\infty}(\mathcal{G})} \leq C\rho = o(1) \quad \text{ as } \epsilon \to 0.
\end{equation}
Since $\eta : \mathcal{G} \to \eta(\mathcal{G})$ is a diffeomorphism, we also have the following estimate:
\begin{equation} \label{eq4.76}
    ||\text{Id}_{\eta(\mathcal{G})} - \eta^{-1}||_{L^{\infty}(\eta(\mathcal{G}))} \leq C\rho = o(1) \quad \text{ as } \epsilon \to 0.
\end{equation}

\section{Proofs of Main Results} \label{sec_main_results}

In this section we prove Theorems \ref{thm_pure_scaling}, \ref{thm_dichotomy}, and \ref{thm_classification}. For a short summary of our arguments, see \S\ref{sssec_minformal}.

\subsection{Cylindrical configuration space} \label{ssec_cyl}

Our first task is to show that a version of Theorem \ref{thm_classification} holds if we replace $\mathcal{M}$ with $\mathcal{M}_{\cyl}$. 

\begin{theorem} \label{thm_cylindrical}
Given a sequence of cells $\Sigma_i$, there exists decreasing sequence of positive numbers $\{b_i\}$ such that exactly one of the following is true:
\begin{enumerate}[label = (\Alph*)]
    \item There exists a Markov kernel $\mathbb{K}$ such that, for any sequence $\epsilon_i \leq b_i$ with $\epsilon_i \to 0$, the limit $\lim_{i \to \infty} \mathbb{K}^{\Sigma_i,\epsilon_i}_{\cyl}$ exists and is equal to $\mathbb{K}$.
    
    \item For any sequence of positive numbers $\epsilon_i \leq b_i$ with $\epsilon_i \to 0$, the limit $\lim_{i \to \infty} \mathbb{K}^{\Sigma_i,\epsilon_i}_{\cyl}$ does not exist.
\end{enumerate}
If (A) holds, then $\mathbb{K}$ takes the following form
\begin{equation} \label{eq3.7}
\mathbb{K}(y_1,y_2,\theta,\psi ; \dd y_1' \dd y_2' \dd\theta' \dd\psi') = \delta_{(y_1,y_2)}(y_1',y_2')\dd y_1'\dd y_2' \times \widetilde{\mathbb{P}}(\theta,\dd\theta') \times \delta_{\pi - \psi}(\psi') \dd\psi',
\end{equation}
where $\widetilde{\mathbb{P}}$ is a Markov kernel on $\mathbb{S}^1_+$ satisfying the following properties:
\begin{enumerate}[label = \roman*.]
    \item $\widetilde{\mathbb{P}}$ is symmetric with respect to the measure $\sin\theta \dd\theta$ on $\mathbb{S}^1_+$.
    \item Let
    \begin{equation}
    \begin{split}
        \widetilde{\Sigma}_i & = \{(y_1,y_2) : (y_1,(1+mJ^{-1})^{-1/2}y_2) \in \Sigma_i\}, \\
        \widetilde{\epsilon}_i & = (1+mJ^{-1})^{-1/2}\epsilon_i, \\
        \mathbb{P} & = \delta_{y_1}(y_1') \dd y_1' \times \widetilde{\mathbb{P}}(\theta,\dd\theta').
    \end{split}
    \end{equation}
    Then 
    \begin{equation}
        \mathbb{P} = \lim_{i \to \infty} \mathbb{P}^{\widetilde{\Sigma}_i, \widetilde{\epsilon}_i}.
    \end{equation}
\end{enumerate}
Consequently, $\mathbb{K}$ is symmetric with respect to the measure $\Lambda^2$.

Moreover, if the sequence of cells $\Sigma_i = \Sigma$ is constant, then we may take $b_i = \infty$ and (A) always holds.

Conversely, if $\mathbb{K}$ is a Markov kernel on $\mathbf{P} \times \mathbb{S}^2_+$ of form (\ref{eq3.7}) such that the measure $\sin\theta \dd\theta$ on $\mathbb{S}^1_+$ is invariant with respect to $\widetilde{\mathbb{P}}$, then there exist a sequence of cells $\{\Sigma_i\}$ and a decreasing sequence of positive numbers $\{b_i\}$ such that $\mathbb{K} = \lim_{i \to \infty} \mathbb{K}^{\Sigma_i,\epsilon_i}_{\cyl}$ whenever $\epsilon_i \to 0$ and $0 < \epsilon_i \leq b_i$.
\end{theorem}

\begin{proof} First, we will define the sequence $b_i$. Define the parallelogram
\begin{equation}
    R_{1 \times 1} = \{(x_1,\alpha) \in \mathbf{P} : 0 \leq x_1 + \alpha \leq 1, 0 \leq \alpha \leq 1\}.
\end{equation}
By Proposition \ref{prop_detcolcyl}, $K^{\Sigma_i,1}_{\cyl}$ is well-defined and finite $\Lambda^2$-almost surely. Since $\Lambda^2(R_{1 \times 1} \times \mathbb{S}^2_+) < \infty$, for each $i$ there exists a constant $C_i$ such that 
\begin{equation} \label{eq4.3}
    \Lambda^2(\{(y,w) \in R_{1 \times 1} \times \mathbb{S}^2_+ : ||K^{\Sigma_i,1}_{\cyl}(y,w) - (y,w)|| \geq C_i\}) \leq i^{-1}.
\end{equation}
By replacing each $C_i$ with $\max\{C_1, \dots, C_i\}$, we may suppose that the constants $C_i$ are increasing. We define 
\begin{equation} \label{eq5.4}
    b_i = C_i^{-1}i^{-1}.
\end{equation}
We assume possibility (B) does not hold, i.e. we assume that $\mathbb{K} = \lim_{i \to \infty} \mathbb{K}^{\Sigma_i,\epsilon_i^{(0)}}_{\cyl}$ does exist for some sequence of positive numbers $\epsilon_i^{(0)} \leq b_i$ such that $\epsilon_i^{(0)} \to 0$. We will argue that (A) holds and that $\mathbb{K}$ takes the form described above. 

Fix any sequence of positive numbers $\epsilon_i \leq b_i$ with $\epsilon_i \to 0$, and consider the billiard trajectory in $\mathcal{M}_{\cyl} = \mathcal{M}_{\cyl}(\epsilon_i)$.

\textit{Step 1a.} The key fact is that the billiard evolution decouples into two independent evolutions. To describe this decoupling, first recall that a point particle in $\mathcal{M}_{\cyl}$ moves linearly in $\Int \mathcal{M}_{\cyl}$ and reflects specularly from the boundary $\partial \mathcal{M}_{\cyl}$. Note that, for every $p \in \partial \mathcal{M}_{\cyl}$, $\chi$ is tangent to the boundary $\partial \mathcal{M}_{\cyl}$ at $p$. Therefore, specular reflection preserves the angle between the velocity of the point particle and $\chi$. Hence, the angle between the velocity and $\chi$ is conserved for all time.

In the coordinates $(y_1,y_2, y_3)$, the $y_3$-axis is parallel to the cylindrical axis $\chi$.  Consequently, $\mathcal{M}_{\cyl}$ may be identified with the product space $\mathcal{M}_{\cyl}^{1,2} \times \mathbb{R}$, where $\mathcal{M}_{\cyl}^{1,2} = \chi^\perp \cap \mathcal{M}_{\cyl} = \{(y_1,y_2) : (y_1,y_2,0) \in \mathcal{M}_{\cyl}\}$. The trajectory $y(t)$ of the point particle in $\mathcal{M}_{\cyl}$ decouples into a pair of independent trajectories $(y_{1,2}(t), y_3(t)) \in \mathcal{M}_{\cyl}^{1,2} \times \mathbb{R}$. The trajectory $y_{1,2}(t)$ moves linearly in the interior of $\mathcal{M}_{\cyl}^{1,2}$ with velocity $\dot{y}_{1,2}(t) = \dot{y}(t) - \langle \dot{y}(t) ,\chi \rangle \chi$, and reflects specularly from $\partial \mathcal{M}_{\cyl}^{1,2}$. The trajectory $y_3(t)$ moves freely in $\mathbb{R}$ with constant velocity $\dot{y}_3 = \langle \dot{y}(t), \chi \rangle = \langle \dot{y}(0), \chi \rangle$ for all time.

\textit{Step 1b.} Let us consider the two-dimensional billiard in $\mathcal{M}_{\cyl}^{1,2}$ in more detail. Define planes $\mathbf{Q}_0 = \{(x_1,x_2,\alpha) : \alpha = 0\}$ and $\mathbf{Q}_1 = \chi^\perp =\{(y_1,y_2,y_3) : y_3 = 0\}$. The angle between these two planes (with respect to the kinetic energy inner product) is $\gamma = \arccos(\langle e_3,\chi\rangle) = \arccos((1+mJ^{-1})^{-1/2})$. Let $\pi_0 : \mathbf{Q}_0 \to \mathbf{Q}_1$ denote orthogonal projection from $\mathbf{Q}_0$ onto $\mathbf{Q}_1$. The base of the cylinder $\mathcal{M}_{\cyl}$ is $\widehat{B} \subset \mathbf{Q}_0$, and therefore $\mathcal{M}_{\cyl}^{1,2} = \pi_0(\widehat{B})$. Identifying $\mathbf{Q}_0$ with $\mathbb{R}^2$ with coordinates $(x_1,x_2)$ and $\mathbf{Q}_1$ with $\mathbb{R}^2$ with coordinates $(y_1,y_2)$, $\pi_0$ is just the ``foreshortening map''
\begin{equation}
    \pi_0 : (x_1,x_2) \mapsto (y_1,y_2) = ((1+mJ^{-1})^{-1/2}x_1, x_2).
\end{equation}
Thus, the particle with position $y_{1,2}(t)$ moves freely in the complement of the ``foreshortened'' wall 
\begin{equation} \label{eq3.3}
    \widetilde{W} := \{(y_1,y_2) \in \mathbf{Q}_1 : ((1+mJ^{-1})^{1/2}y_1,y_2) \in W + e_2\},
\end{equation}
and reflects specularly from $\partial \widetilde{W}$. The boundary $\partial \widetilde{W}$ is piecewise $C^2$ and bounded between the lines $y_2 = -\epsilon$ and $y_2 = 0$, because $\partial W$ is piecewise $C^2$ and bounded between the lines $y_2 = -1 - \epsilon$ and $y_2 = -1$. Letting $\mathbf{L} = \{(y_1,y_2) : y_2 = 0\}$, we may therefore define a macro-reflection law $P^{\widetilde{\Sigma}_i, \widetilde{\epsilon}_i} : \mathbf{L} \times (0,\pi) \to \mathbf{L} \times (0,\pi)$, in the way described in \S\ref{sssec_uphalf} (with $\mathbf{L}$ playing the role of $\mathbb{R}$). 

\textit{Step 1c.} We now describe the marginals of the collision law. The relationship between the macroscopic reflection law $P^{\widetilde{\Sigma}_i, \widetilde{\epsilon}_i}$ and the collision law $K^{\Sigma_i,\epsilon_i}_{\cyl}$ is as follows. Note that $\mathbf{L} = \mathbf{P} \cap \mathbf{Q}_1$. The macro-reflection law $P^{\widetilde{\Sigma}_i, \widetilde{\epsilon}_i}$ describes the orthogonal projection onto $\mathbf{Q}_1$ of the state of the point particle after returning to the plane $\mathbf{P}$. In more detail, letting $\pi_{1,3} : \mathbf{P} \times (0,\pi)^2 \to \mathbb{R} \times (0,\pi)$ be the mapping $(y_1,y_3,\theta,\psi) \mapsto (y_1,\psi)$, we have 
\begin{equation} \label{eq3.10}
    \pi_{1,3} \circ K^{\Sigma_i,\epsilon_i}_{\cyl}(y_1,y_3,\theta,\psi) = P^{\widetilde{\Sigma}_i, \widetilde{\epsilon}_i}(y_1,\theta).
\end{equation}
Note that the right-hand side only depends on $y_1$ and $\theta$.

\indent We also consider the projections $\pi_2 : \mathbf{P} \times (0,\pi)^2 \to \mathbb{R}$ mapping $(y_1,y_3,\theta,\psi) \mapsto y_3$, and $\pi_4 : \mathbf{P} \times (0,\pi)^2 \to (0,\pi)$ mapping $(y_1,y_2,\theta,\psi) \mapsto \psi$. We define
\begin{equation} 
    Q_2(y_1,y_3,\theta,\psi) = \pi_2 \circ K^{\Sigma_i,\epsilon_i}_{\cyl}(y_1,y_3,\theta,\psi), 
\end{equation}
\begin{equation} 
    Q_4(y_1,y_3,\theta,\psi) = \pi_4 \circ K^{\Sigma_i,\epsilon_i}_{\cyl}(y_1,y_3,\theta,\psi).
\end{equation}
Suppose $(y,w) \in \mathbf{P} \times \mathbb{S}^2_+$ with coordinates $(y_1,y_2,\theta,\psi)$, and let $(y',w') = K^{\Sigma_i,\epsilon_i}(y,w)$ with coordinates $(y_1',y_2',\theta',\psi')$. Then, since the angle between the velocity and $\chi$ is conserved for all time, $\cos\psi' = \langle w',\chi \rangle = \langle -w,\chi\rangle = \cos(\pi -\psi)$. Hence, $\psi' = \pi - \psi$. Therefore, $Q_4$ only depends on $\psi$, and 
\begin{equation} \label{eq3.13}
    Q_4(\psi) = \pi - \psi.
\end{equation}

\indent To describe $Q_2$, let us write 
\begin{equation} \label{eq3.14}
    Q_2(y_1,y_3,\theta,\psi) = y_3 + E(y_1,y_3,\theta,\psi).
\end{equation}
The definition of $E$ makes sense for any choice of cell $\Sigma$ and roughness scale $\epsilon$. In the argument below, we will indicate explicitly the dependence of $E$ on $\Sigma$ and $\epsilon$ by writing $E = E_{\Sigma}^{\epsilon}$. We will prove 

\begin{claim}{5.1.1} \label{claim5.1.1}
For any bounded set $B \subset \mathbf{P} \times \mathbb{S}^2_+$, there exists a constant $C_B$ depending only on $B$ such that 
\begin{equation}
    \Lambda^2(B \cap \{(y,w) : |E_{\Sigma_i}^{\epsilon_i}(y,w)| \geq i^{-1}\}) \leq C_B i^{-1}.
\end{equation}
\end{claim}

\begin{subproof}[Proof of Claim \ref{claim5.1.1}]
The key observation is that the change of spatial coordinates $y \mapsto \epsilon_i^{-1} y$ maps $\mathcal{M}_{\cyl}(\epsilon_i)$ to $\mathcal{M}_{\cyl}(1)$. Consequently $K^{\Sigma_i,1}_{\cyl}$ and $E_{\Sigma_i}^{1}$ may be viewed respectively as $K^{\Sigma_i,\epsilon_i}_{\cyl}$ and $E_{\Sigma_i}^{\epsilon_i}$ re-expressed in ``zoomed'' coordinates. Let 
\begin{equation}
    R_{\epsilon_i \times \epsilon_i} = \epsilon_i R_{1 \times 1} = \{(x_1,\alpha) \in \mathbf{P} : 0 \leq x_1 \leq \epsilon_i, 0 \leq x_1 + \alpha \leq \epsilon_i\}.
\end{equation}
Making the change of coordinates $y \mapsto \epsilon_iy$, recalling that the spatial factor of $\Lambda^2$ is just Lebesgue measure on $\mathbf{P}$, and using the observation above, we have 
\begin{equation} \label{eq4.14}
\begin{split}
    \Lambda^2(R_{\epsilon_i \times \epsilon_i} \times \mathbb{S}^2_+ \cap \{|E_{\Sigma_i}^{\epsilon_i}| \geq i^{-1}\}) & = \epsilon_i^2\Lambda^2(R_{1 \times 1} \times \mathbb{S}^2_+ \cap \{|E_{\Sigma_i}^1| \geq \epsilon_i^{-1} i^{-1}\}) \\
    & \leq \epsilon_i^2\Lambda^2(R_{1 \times 1} \times \mathbb{S}^2_+ \cap \{|E_{\Sigma_i}^1| \geq C_i\}) \\
    & \leq \epsilon_i^2\Lambda^2(R_{1 \times 1} \times \mathbb{S}^2_+ \cap \{||K^{\Sigma_i,1}_{\cyl} - \text{Id}_{\mathbf{P} \times \mathbb{S}^2_+}|| \geq C_i\}) \\
    & \leq \epsilon_i^2 i^{-1},
\end{split}
\end{equation}
where the second line follows from (\ref{eq5.4}) and $\epsilon_i \leq b_i$, the third line follows because $||K^{\Sigma_i,1}_{\cyl} - \text{Id}_{\mathbf{P} \times \mathbb{S}^2_+}|| \geq |E_{\Sigma_i}^{\epsilon_i}|$, and the fourth line follows from (\ref{eq4.3}). For $(j,k) \in \mathbb{Z}^2$, define translations, 
\begin{equation}
    \tau_{jk}(y) = y + j \epsilon_i e_1 + (1 + mJ^{-1})^{1/2}k \epsilon_i \chi, \quad\quad y \in \mathbf{P}.
\end{equation}
The plane $\mathbf{P}$ is tessellated by the translates $\tau_{jk} R_{\epsilon_i \times \epsilon_i}$ in the sense that 
\begin{equation}
    \mathbf{P} = \bigcup_{(i,k) \in \mathbb{Z}^2} \tau_{jk} R_{\epsilon_i \times \epsilon_i},
\end{equation} 
and for $(j,k) \neq (j',k')$, the set $\tau_{jk} R_{\epsilon_i \times \epsilon_i} \cap \tau_{j'k'} R_{\epsilon_i \times \epsilon_i}$ has measure zero. We also note that the cylindrical set $\mathcal{M}_{\cyl}(\epsilon_i)$ is invariant under the translates $\tau_{jk}$. Consequently, $E_{\Sigma_i}^{\epsilon_i}$ is invariant under the $\tau_{jk}$ in the sense that $E_{\Sigma_i}^{\epsilon_i} \circ \tau_{jk} = E_{\Sigma_i}^{\epsilon_i}$. Therefore, for any $(j,k) \in \mathbb{Z}^2$, 
\begin{equation} \label{eq4.15}
    \Lambda^2(\tau_{jk} R_{\epsilon_i \times \epsilon_i} \times \mathbb{S}^2_+ \cap \{|E_{\Sigma_i}^{\epsilon_i}| \geq i^{-1}\}) = \Lambda^2(R_{\epsilon_i \times \epsilon_i} \times \mathbb{S}^2_+ \cap \{|E_{\Sigma_i}^{\epsilon_i}| \geq i^{-1}\}) \leq \epsilon_i^2 i^{-1}.
\end{equation}
Fix a bounded set $B \subset \mathbf{P} \times \mathbb{S}^2_+$. Note that there exists a constant $C_B$ depending only on $B$ such that $B$ may be covered by $C_B \epsilon_i^{-2}$ sets of form $\tau_{jk} P_{\epsilon_i \times \epsilon_i} \times \mathbb{S}^2_+$. Consequently, by (\ref{eq4.15}),
\begin{equation}
    \Lambda^2(B \cap \{|E_{\Sigma_i}^{\epsilon_i}| \geq i^{-1}\}) \leq C_B i^{-1}.
\end{equation}
This implies the claim.
\end{subproof}

\textit{Step 1d.} We now conclude the proof that $\lim_{i \to \infty}\mathbb{K}^{\Sigma_i,\epsilon_i}_{\cyl}$ exists and is of form (\ref{eq3.7}). We write 
\begin{equation} \label{eq4.17}
\begin{split}
   &\int_{(\mathbf{P} \times \mathbb{S}^2_+)^2} g(y',w')\mathbb{K}^{\Sigma_i,\epsilon_i}_{\cyl}(y,w; \dd y' \dd w')f(y,w) \Lambda^2(\dd y \dd w) \\
    & = \int_{\mathbf{P} \times \mathbb{S}^2_+} g(K^{\Sigma_i,\epsilon}_{\cyl}(y,w))f(y,w)\Lambda^2(\dd y \dd w) \\
    & = \int_{\{E < i^{-1}\}} g(K^{\Sigma_i,\epsilon_i}_{\cyl}(y,w))f(y,w)\Lambda^2(\dd y \dd w) \\
    & \quad + \int_{\{E \geq i^{-1}\}} g(K^{\Sigma_i,\epsilon_i}_{\cyl}(y,w))f(y,w)\Lambda^2(\dd y \dd w).
\end{split}
\end{equation}
Noting that $g$ is bounded and $f$ has compact support, we see that by Claim \ref{claim5.1.1}, with $B = \supp f$, the second term above converges to zero as $i \to \infty$.

To handle the first term in (\ref{eq4.17}), we introduce the following notation: For $h \in C_c(\mathbf{P} \times \mathbb{S}^2_+)$, let $\widehat{h}(y_1,\theta,y_3,\psi) = h(y_1,y_3,\theta,\psi)$. Using (\ref{eq3.10}), (\ref{eq3.13}), and (\ref{eq3.14}), the first term in (\ref{eq4.17}) is equal to
\begin{equation} \label{eq4.18}
\begin{split}
    &\int_{\{E < i^{-1}\}} g(K^{\Sigma_i,\epsilon_i}_{\cyl}(y,w))f(y,w)\Lambda^2(\dd y \dd w) \\
    & = \int_{\{E < i^{-1}\}} \widehat{g}(P^{\widetilde{\Sigma}_i, \widetilde{\epsilon}_i}(y_1,\theta), y_3 + E, \pi - \psi) f(y_1,y_3,\theta,\psi)\sin\theta\sin^2\psi \dd y_1 \dd y_2 \dd\theta \dd\psi \\
    & = \int_{\{E < i^{-1}\}} \widehat{g}(P^{\widetilde{\Sigma}_i, \widetilde{\epsilon}_i}(y_1,\theta),y_3,\pi-\psi) f(y_1,y_3,\theta,\psi)\sin\theta\sin^2\psi \dd y_1 \dd y_2 \dd\theta \dd\psi \\
    & \quad + \int_{\{E < i^{-1}\}} (\tau_{(0,0,E,0)}\widehat{g} - \widehat{g})(P^{\widetilde{\Sigma}_i, \widetilde{\epsilon}_i}(y_1,\theta),y_3,\pi-\psi) f(y_1,y_3,\theta,\psi) \\
    & \hspace{3in}\times\sin\theta\sin^2\psi \dd y_1 \dd y_2 \dd\theta \dd\psi \\
    & = \int_{\{E < i^{-1}\}} \widehat{g}(P^{\widetilde{\Sigma}_i, \widetilde{\epsilon}_i}(y_1,\theta),y_3,\pi-\psi) f(y_1,y_3,\theta,\psi)\sin\theta\sin^2\psi \dd y_1 \dd y_2 \dd\theta \dd\psi \\
    & \quad + \int_{P^{\widetilde{\Sigma}_i, \widetilde{\epsilon}_i}(\{E < i^{-1}\})} (\tau_{(0,0,E,0)}\widehat{g} - \widehat{g})(y_1,\theta,y_3,\pi-\psi) \widehat{f}(P^{\widetilde{\Sigma}_i, \widetilde{\epsilon}_i}(y_1,\theta), y_3,\psi)\\
    & \hspace{3in}\times\sin\theta\sin^2\psi \dd y_1 \dd y_2 \dd\theta \dd\psi,
\end{split}
\end{equation}
using in the last line the fact that $P^{\widetilde{\Sigma}_i, \widetilde{\epsilon}_i}(y_1,\theta)$ is an involution which preserves the measure $\sin\theta \dd y_1 \dd\theta$ by Proposition \ref{prop_det_ref_properties} in \S\ref{ssec_rough_bill}. The second term above is bounded in absolute value by 
\begin{equation}
\begin{split}
    &\sup_{||h|| \leq i^{-1}} ||\tau_h\widehat{g} - \widehat{g}||_{L^\infty} \int_{\mathbf{P} \times \mathbb{S}^2_+} |\widehat{f}(P^{\widetilde{\Sigma}_i, \widetilde{\epsilon}_i}(y_1,\theta), y_3,\psi)|\sin\theta\sin^2\psi \dd y_1 \dd y_2 \dd\theta \dd\psi \\
    & = \sup_{||h|| \leq i^{-1}} ||\tau_h\widehat{g} - \widehat{g}||_{L^\infty} ||f||_{L^1} \to 0 \quad \text{ as } i \to \infty,
\end{split}
\end{equation}
again using invariance of $P^{\widetilde{\Sigma}_i, \widetilde{\epsilon}_i}$ with respect to $\sin\theta \dd\theta \dd y_1$ to obtain the equality. On the other hand, by Claim \ref{claim5.1.1} and the fact that $f$ and $g$ are bounded and $f$ has compact support, the first term in (\ref{eq4.18}) is equal to 
\begin{equation} 
\begin{split}
    &\int_{\mathbf{P} \times \mathbb{S}^2_+} \widehat{g}(P^{\widetilde{\Sigma}_i, \widetilde{\epsilon}_i}(y_1,\theta),y_3,\pi-\psi) f(y_1,y_3,\theta,\psi)\sin\theta\sin^2\psi \dd y_1 \dd y_2 \dd\theta \dd\psi + o_{i \to \infty}(1) \\
    & = \int_{\mathbf{P} \times \mathbb{S}^2_+} \left(\int_{\mathbf{P} \times \mathbb{S}^2_+} g(y_1',y_3', \theta',\psi') \mathbb{P}^{\widetilde{\Sigma}_i, \widetilde{\epsilon}_i}(y_1,\theta; \dd y_1' \dd\theta') \delta_{y_3}(\dd y_3') \delta_{\pi - \psi}(\dd\psi') \right)\\
    & \quad\quad\quad\quad\quad\quad\quad\quad\quad \times f(y_1,y_3,\theta,\psi)\sin\theta\sin^2\psi \dd y_1 \dd y_2 \dd\theta \dd\psi + o_{i \to \infty}(1).
\end{split}
\end{equation}
To summarize, we have shown that 
\begin{equation} \label{eq4.21}
\begin{split}
    &\int_{\mathbf{P} \times \mathbb{S}^2_+} \left(\int_{\mathbf{P} \times \mathbb{S}^2_+} g(y',w')\mathbb{K}^{\Sigma_i,\epsilon_i}_{\cyl}(y,w; \dd y' \dd w') \right) f(y,w) \Lambda^2(\dd y \dd w) \\
    & = \int_{\mathbf{P} \times \mathbb{S}^2_+} \left(\int_{\mathbf{P} \times \mathbb{S}^2_+} g(y_1',y_3', \theta',\psi') \mathbb{P}^{\widetilde{\Sigma}_i, \widetilde{\epsilon}_i}(y_1,\theta; \dd y_1' \dd\theta') \delta_{y_3}(\dd y_3') \delta_{\pi - \psi}(\dd\psi') \right)\\
    & \quad\quad\quad\quad\quad\quad\quad\quad\quad \times f(y_1,y_3,\theta,\psi)\sin\theta\sin^2\psi \dd y_1 \dd y_2 \dd\theta \dd\psi + o_{i \to \infty}(1).
\end{split}
\end{equation}

Since by assumption $\mathbb{K} = \lim_{i \to \infty} \mathbb{K}^{\Sigma_i,\epsilon_i^{(0)}}_{\cyl}$ exists, the equality above implies that the following limit exists:
\begin{equation}
    \mathbb{P} = \lim_{i \to \infty} \mathbb{P}^{W(\widetilde{\Sigma}_i,\widetilde{\epsilon}_i^{(0)})},
\end{equation}
and $\mathbb{K} = \mathbb{P} \times \delta_{y_3} \times \delta_{\pi - \psi}$. By Corollary \ref{cor_rough_ref_char} and Remark \ref{rem_indep}, $\mathbb{P} = \lim \mathbb{P}^{\gamma(\widetilde{\Sigma}_i,\widetilde{\epsilon}_i)}$ does not depend on the sequence $\epsilon_i \to 0$, and moreover $\mathbb{P}(y_1,\theta;\dd y_1' \dd\theta') = \delta_{y_1}(\dd y_1') \widetilde{\mathbb{P}}(\theta, \dd\theta')$, for some Markov kernel $\widetilde{\mathbb{P}}$ on $\mathbb{S}^1_+$ which is symmetric with respect to the measure $\sin\theta \dd\theta$. Thus by (\ref{eq4.21}) we see that $\lim_{i \to \infty} \mathbb{K}^{\Sigma_i,\epsilon_i}$ exists for any $\epsilon_i \leq b_i$, and this limit is equal to $\mathbb{K} = \delta_{y_1} \times \widetilde{\mathbb{P}} \times \delta_{y_3} \times \delta_{\pi - \psi}$. This proves the first part of the theorem.

\textit{Step 2.} Now let us assume that $\Sigma_i = \Sigma$ is constant, let $\epsilon_i \to 0$ be arbitrary (not necessarily bounded by $b_i$) and consider how the proof above goes through in this case. In steps 1a-1d, to obtain the equality (\ref{eq4.21}) we did not use the assumption that $\mathbb{K}^{\Sigma_i,\epsilon_i}_{\cyl}$ converges. Thus (\ref{eq4.21}) always holds. On the other hand, $\widetilde{\Sigma}_i = \widetilde{\Sigma}$ is constant, it follows from Lemma \ref{lem_construct} that the limit $\mathbb{P} = \lim_{i\to \infty} \mathbb{P}^{\widetilde{\Sigma},\widetilde{\epsilon}_i}$ exists. This implies that the limit $\lim_{i \to \infty} \mathbb{K}^{\Sigma, \epsilon_i}_{\cyl}$ exists. This proves the second statement in the theorem.

\textit{Step 3.} Finally, the converse may be proved as follows. Suppose $\mathbb{K}$ takes the form (\ref{eq3.7}) where $\widetilde{\mathbb{P}}$ is a Markov kernel on $\mathbb{S}^1_+$ such that $\sin\theta \dd\theta$ is invariant with respect to $\widetilde{\mathbb{P}}$. By Corollary \ref{cor_rough_ref_char} in \S\ref{ssec_rough_bill} there is a sequence $\widetilde{\Sigma}_i$ of cells such that for any sequence $\widetilde{\epsilon}_i \to 0$, $\delta_{y_1} \times \widetilde{\mathbb{P}} = \lim_{i \to \infty} \mathbb{P}^{\widetilde{\Sigma}_i, \widetilde{\epsilon}_i}$. Let 
\begin{equation}
    \Sigma_i = \{(x_1,x_2) : (x_1,(1+mJ^{-1})^{1/2}x_2) \in \widetilde{\Sigma}_i\}, 
\end{equation}
and let $b_i$ be chosen with respect to the $\Sigma_i$ above as in the beginning of the proof. Take a sequence of positive numbers $\epsilon_i \to 0$ with $\epsilon_i \leq b_i$, and let $\widetilde{\epsilon}_i = (1+mJ^{-1})^{1/2} \epsilon_i$. Let $W_i = W(\Sigma_i,\epsilon_i)$ and $\widetilde{W}_i = W(\widetilde{\Sigma}_i, \widetilde{\epsilon}_i)$, and observe that $\widetilde{W}_i$ is just foreshortened version of $W_i$ in the sense of (\ref{eq3.3}). Consider the cylindrical configuration space $\mathcal{M}_{\cyl}(\Sigma_i,\epsilon_i)$ and the cylindrical collision law $K^{\Sigma_i,\epsilon_i}_{\cyl}$ which the wall $W_i$ gives rise to. By Steps 1a-1d, the equality (\ref{eq4.21}) holds. Since $\mathbb{P}^{\widetilde{\Sigma}_i, \widetilde{\epsilon}_i} \to \mathbb{P} = \delta_{y_1} \times \widetilde{\mathbb{P}}$, it follows that $\mathbb{K}^{\Sigma_i,\epsilon_i}_{\cyl} \to \delta_{y_1} \times \widetilde{\mathbb{P}} \times \delta_{y_3} \times \delta_{\pi - \psi} = \mathbb{K}$. 
\end{proof}

\subsection{Cylindrical approximation in the pure scaling case} \label{ssec_pure_scaling}

In this section we prove Theorem \ref{thm_pure_scaling}, delegating the proof of a key lemma to a subsequent subsection. 

Since the sequence of cells is constant, we take $\Sigma$ to be fixed throughout this section, and we denote the collision law in $\mathcal{M} = \mathcal{M}(\Sigma, \epsilon)$ by $K^\epsilon$, dropping the explicit dependence on $\Sigma$ from our notation. Similarly, we denote the collision law in the cylindrical configuration space $\mathcal{M}_{\cyl} = \mathcal{M}_{\cyl}(\Sigma,\epsilon)$ by $K^\epsilon_{\cyl}$, and we denote the modified collision law by $\widetilde{K}^{\epsilon}$.

The following lemma gives us the sense in which the collision laws $K^\epsilon$ and $K^\epsilon_{\cyl}$ approximate each other.

\begin{lemma} \label{lem_comparison}
For any $f, g \in C_c^\infty(\mathbf{P} \times \mathbb{S}^2_+)$,  
\begin{equation}\label{eq5.25}
    \int_{\mathbf{P} \times \mathbb{S}^2_+} g [f \circ K^\epsilon - f \circ K^\epsilon_{\cyl}] \dd\Lambda^2 \to 0 \quad \text{ as } \epsilon \to 0.
\end{equation}
\end{lemma}

The proof of the lemma involves first throwing away a small set of ``bad'' inputs, and then making two comparisons: (i) a comparison between the true collision law $K^\epsilon$ and the modified collision law $\widetilde{K}^{\epsilon}$, and (ii) a comparison between the modified collision law $\widetilde{K}^{\epsilon}$ and the cylindrical collision law $K^{\epsilon}_{\cyl}$. 

Recall the definition (\ref{eq4.72}) of $\widetilde{K}^{\epsilon}$. The main idea for making the comparison (i) is to use estimates for $\eta$ and its differential to argue that $\widetilde{K}^{\epsilon} = \eta^{-1} \circ K^\epsilon \circ \eta$ approximates $K^\epsilon$ as $\epsilon \to 0$ in the sense of (\ref{eq5.25}). 

Most of the work in this section is concerned with making the comparison (ii). 

Consider the parallelogram
\begin{equation} \label{eq5.26}
    R_{\epsilon} = \{(x_1,\alpha) \in \mathbf{P} : 0 \leq x_1 + \alpha \leq \epsilon, -\rho(\epsilon)/2 \leq \alpha \leq \rho(\epsilon)/2\}.
\end{equation}
For $(j,k) \in \mathbb{Z}^2$, define translations 
\begin{equation}
    \tau_{jk}(y) = y + j\epsilon e_1 + k\rho e_3, \quad\quad y \in \mathbf{P}.
\end{equation}
We also let 
\begin{equation}
    \overline{\tau}_{jk}(y,w) = (\tau_{jk}(y),w), \quad\quad (y,w) \in \mathbf{P} \times \mathbb{S}^2_+.
\end{equation}
The plane $\mathbf{P}$ may be tessellated by the parallelograms $\tau_{jk} R_\epsilon$ in the sense that $\mathbf{P} = \bigcup_{(j,k) \in \mathbb{Z}^2} \tau_{jk}R_\epsilon$ and $\tau_{jk}R_\epsilon \cap \tau_{j'k'}R_\epsilon$ has Lebesgue measure zero whenever $(j,k) \neq (j',k')$. 

Recall that $\widetilde{K}^\epsilon$ is defined on an open subset $\widetilde{\mathcal{F}} \subset \mathbf{P} \times \mathbb{S}^2_+$.

\begin{lemma} \label{lem_mod}
For $\epsilon > 0$ sufficiently small, there exist sets $\Omega(\epsilon) \subset \widetilde{\mathcal{F}}$ such that
\begin{enumerate}[label = \roman*.]
    \item $\Omega(\epsilon)$ is invariant under the translations $\overline{\tau}_{jk}$;
    \item $\frac{1}{\epsilon\rho(\epsilon)}\Lambda^2(R_{\epsilon} \times \mathbb{S}^2_+ \smallsetminus \Omega(\epsilon)) \to 0$ as $\epsilon \to 0$; and 
    \item the following limit holds:
    \begin{equation}
        \lim_{\epsilon \to 0} \sup_{(y,w) \in R_\epsilon \cap \Omega(\epsilon)} ||\widetilde{K}^{\epsilon}(y,w) - K^\epsilon_{\cyl}(y,w)|| = 0.
    \end{equation}
\end{enumerate}
\end{lemma}

\begin{remark} \label{rem_3.2} \normalfont
Lemma \ref{lem_mod}(i) and (ii), together with Claim \ref{claim5.2.1} from Step 1 of its proof, below, imply that for any bounded set $B \subset \mathbf{P} \times \mathbb{S}^2_+$, $\Lambda^2(B \smallsetminus \Omega(\epsilon)) \to 0$ as $\epsilon \to 0$. Consequently, $\Lambda^2(B \smallsetminus \widetilde{\mathcal{F}}) \to 0$ as $\epsilon \to 0$. 
\end{remark}

The proof of this lemma is given in \S\ref{ssec_zoom}. We will now prove Lemma \ref{lem_comparison} using the lemma above.

\begin{proof}[Proof of Lemma \ref{lem_comparison}]
Let $\Omega(\epsilon) \subset \widetilde{\mathcal{F}}$ be as in Lemma \ref{lem_mod}. We split up the integral $\int g[f \circ K^\epsilon - f \circ K^\epsilon_{\cyl}] \dd\Lambda^2$ as follows:
\begin{equation}
    \begin{split}
        & \int_{\mathbf{P} \times \mathbb{S}^2_+} g[f \circ K^\epsilon - f \circ K^\epsilon_{\cyl}] \dd\Lambda^2 \\
        & = \underbrace{\int_{\mathbf{P} \times \mathbb{S}^2_+ \smallsetminus \Omega} g[f \circ K^\epsilon - f \circ K^\epsilon_{\cyl}] \dd\Lambda^2}_{\text{\normalsize $=: I_1$}} + \underbrace{\int_{\Omega} g[f \circ K^\epsilon - f \circ \widetilde{K}^{\epsilon}] \dd\Lambda^2}_{\text{\normalsize $=: I_2$}} \\
        & \hspace{1.5in} + \underbrace{\int_{\Omega} g[f \circ \widetilde{K}^{\epsilon} - f \circ K^\epsilon_{\cyl}] \dd\Lambda^2}_{\text{\normalsize $=: I_3$}}.
    \end{split}
\end{equation}
Note that the modified collision law $\widetilde{K}^{\epsilon}$ is defined on $\Omega$ because $\Omega \subset \widetilde{\mathcal{F}}$. We will show separately that each of $I_1$, $I_2$, and $I_3$ converge to zero as $\epsilon \to 0$.

\textit{Step 1.} To show $I_1 \to 0$, first we write  
\begin{equation} \label{eq3.60}
\begin{split}
|I_1| & \leq \int_{\mathbf{P} \times \mathbb{S}^2_+ \smallsetminus \Omega} |g|\cdot |f\circ K^\epsilon|\dd\Lambda + \int_{\mathbf{P} \times \mathbb{S}^2_+ \smallsetminus \Omega}|g|\cdot |f \circ K^\epsilon_{\cyl}|\dd\Lambda \\
& \leq 2||f||_{L^\infty} \Lambda^2(\supp g \smallsetminus \Omega).
\end{split}
\end{equation}

\begin{claim}{5.2.1} \label{claim5.2.1}
Let $B$ be a compact set. There exists a constant $C = C_{B} < \infty$ depending only on $B$ such that, for $\epsilon$ sufficiently small, the set $B$ can be covered by $\frac{C}{\epsilon\rho(\epsilon)}$ sets of form $\tau_{jk} R_\epsilon \times \mathbb{S}^2_+$.
\end{claim}

\begin{subproof}[Proof of Claim \ref{claim5.2.1}]
Since $B$ is compact, there exist minimal $a,b < \infty$ such that 
\begin{equation}
B \subset \{(x_1,\alpha) \in \mathbf{P} : -a \leq x_1 + \alpha \leq a, -b \leq \alpha \leq b\} \times \mathbb{S}^2_+.
\end{equation}
Note that parallelogram $R_{\epsilon}$ has length $\epsilon$ and height $\rho$, so the parallelogram appearing on the right-hand side above can be covered by $\frac{2a + 1}{\epsilon} \cdot \frac{2b + 1}{\rho(\epsilon)}$ translates $\tau_{jk} R_\epsilon$. The claim follows by taking $C = (2a + 1)(2b + 1)$.
\end{subproof}

Since $\Omega$ is invariant with respect to the translations $\tau_{jk}$ and $\supp g$ is compact, it follows that 
\begin{equation} \label{eq3.62}
    \Lambda^2(\supp g \smallsetminus \Omega) \leq \frac{C}{\epsilon\rho(\epsilon)} \Lambda^2(R_\epsilon \smallsetminus \Omega),
\end{equation}
where $C = C_{\supp g}$. The right-hand side converges to zero by Lemma \ref{lem_mod}.

\textit{Step 2.} Next we show $I_2 \to 0$. Recall the definition (\ref{eq4.70}) of the diffeomorphism $\eta : \mathcal{G} \to \eta(\mathcal{G}) \subset \mathbf{P} \times \mathbb{S}^2_{+\rho/2}$. We may write $\eta(y,w) = (\varphi_{w}(y), w)$, where for each $w = (v_1,v_2,\omega) \in \mathbb{S}^2_{+\rho/2}$, $\varphi_{w} : \mathbf{P} \cap \widehat{\mathcal{Z}} \to \mathbf{P}$ is given by the formula
\begin{equation}
   \varphi_{w}(x_1,\alpha) = \begin{pmatrix}
   x_1 -\overline{\alpha} + \sin\overline{\alpha} - (1 - \cos\overline{\alpha})\frac{v_1}{v_2} \\
   \alpha - (1 - \cos\overline{\alpha})\frac{\omega}{v_2}
   \end{pmatrix},
\end{equation}
where $\overline{\alpha} = \alpha - \overline{k}\rho$ and $\overline{k} = \text{ argmin}\{|\alpha - k\rho| : k \in \mathbb{Z}\}$. The differential of $\varphi_{w}$ is thus given by the formula:
\begin{equation}
    \dd(\varphi_{w})(x_1,\alpha) = \begin{bmatrix}
    1 & -1 + \cos\overline{\alpha} - \sin(\overline{\alpha})\frac{v_1}{v_2} \\
    0 & 1 - \sin(\overline{\alpha})\frac{\omega}{v_2}
    \end{bmatrix}.
\end{equation}
For $(y, w) \in \mathcal{G}$ , we define
\begin{equation}
    U(y,w) = |\det \dd(\varphi_{w})(x_1,\alpha)| = \left|1 - \sin(\overline{\alpha})\frac{\omega}{v_2} \right|.
\end{equation}
Then for $(y,w) \in \eta(\mathcal{G})$, 
\begin{equation}
    |\det \dd(\varphi_{w}^{-1})(x_1,\alpha)| = \frac{1}{U \circ \eta^{-1}(y,w)}.
\end{equation}

\begin{claim}{5.2.2} \label{claim5.2.2}
If $B$ is a compact subset of $\mathbf{P} \times \mathbb{S}^2_+$, then 
\begin{equation}
    \sup_{(y,w) \in B \cap \mathcal{G}} |U(y,w) - 1| \to 0 \quad \text{ as } \epsilon \to 0.
\end{equation}
\end{claim}

\begin{subproof}[Proof of Claim \ref{claim5.2.2}]
By compactness, for all $(y,w) \in B$, $v_2 = \langle w, e_2 \rangle > c > 0$ for some fixed constant $c$. Also note that $|\overline{\alpha}| \leq \rho(\epsilon)/2 \to 0$ as $\epsilon \to 0$. Therefore, 
\begin{equation}
    |U(y,w) - 1| \leq \sin(\rho/2) \cdot c^{-1} \to 0 \quad \text{ as } \epsilon \to 0,
\end{equation}
and the claim is proved.
\end{subproof}

Recall also the definition (\ref{eq4.72}) of the modified collision law $\widetilde{K}^\epsilon$. Making the change of variables $(y,w) \mapsto \eta^{-1}(y,w) = (\varphi_{w}^{-1}(y),w)$ and noting that the first factor of $\Lambda^2$ is just Lebesgue measure on $\mathbf{P}$, we have 
\begin{equation}
    \int_{\Omega} g[f \circ \widetilde{K}^\epsilon]\dd\Lambda^2 = \int_{\eta(\Omega)} [g \circ \eta^{-1}][f \circ \eta^{-1} \circ K^\epsilon]\frac{1}{U \circ \eta^{-1}}\dd\Lambda^2.
\end{equation}
Using this, we have
\begin{equation} \label{eq3.40}
\begin{split}
    &\left|\int_{\Omega} g[f \circ K^\epsilon - f \circ \widetilde{K}^\epsilon]\dd\Lambda^2 \right| \\
    & = \left|\int_{\Omega} g[f \circ K^\epsilon]\dd\Lambda^2 - \int_{\eta(\Omega)} [g \circ \eta^{-1}][f \circ \eta^{-1} \circ K^\epsilon]\frac{1}{U \circ \eta^{-1}} \dd\Lambda^2\right| \\
    & \leq \left| \int_{\Omega \cap \eta(\Omega)} g[f \circ K^\epsilon - f \circ \eta^{-1} \circ K^\epsilon]\dd\Lambda^2 \right| \\
    & \quad + \left| \int_{\Omega \cap \eta(\Omega)} [g - g \circ \eta^{-1}][f \circ \eta^{-1} \circ K^\epsilon]\dd\Lambda^2 \right| \\
    & \quad + \left|\int_{\Omega \cap \eta(\Omega)} [g \circ \eta^{-1}][f \circ \eta^{-1} \circ K^\epsilon]\left[1 - \frac{1}{U \circ \eta^{-1}}\right]\dd\Lambda^2 \right| \\
    & \quad + \left| \int_{\Omega \smallsetminus \eta(\Omega)} g[f \circ K^\epsilon]\dd\Lambda^2 \right| + \left| \int_{\eta(\Omega) \smallsetminus \Omega} [g \circ \eta^{-1}][f \circ K^\epsilon]\frac{1}{U \circ \eta^{-1}} \dd\Lambda^2 \right|.
\end{split}
\end{equation}
We will estimate separately each term appearing above. First, recall (\ref{eq4.75}) and (\ref{eq4.76}) which say that $\eta$ and $\eta^{-1}$ converges uniformly to the identity on their respective domains $\mathcal{G}$ and $\eta(\mathcal{G})$ as $\epsilon \to 0$. This has two important consequences: First, since $f$ and $g$ are continuous and compactly supported, 
\begin{equation} \label{eq3.41}
    ||f - f \circ \eta||_{L^\infty(\mathcal{G})} \to 0 \quad \text{ and } \quad ||g - g \circ \eta||_{L^\infty(\mathcal{G})} \to 0 \quad \text{ as } \quad \epsilon \to 0,
\end{equation}
and
\begin{equation} \label{eq3.43'}
    ||f - f \circ \eta^{-1}||_{L^\infty(\eta(\mathcal{G}))} \to 0 \quad \text{ and } \quad ||g - g \circ \eta^{-1}||_{L^\infty(\eta(\mathcal{G}))} \to 0 \quad \text{ as } \quad \epsilon \to 0.
\end{equation}
Second, there exists a fixed compact set $B \subset \mathbf{P} \times \mathbb{S}^2_+$ such that, for $\epsilon$ sufficiently small,
\begin{equation} \label{eq3.42}
\begin{split}
    &\mathcal{G} \cap \big\{\supp f \cup \supp g \cup \supp(f \circ \eta) \cup \supp(g \circ \eta)\big\} \subset B, \quad \text{ and } \\
    & \eta(\mathcal{G}) \cap \big\{\supp f \cup \supp g \cup \supp(f \circ \eta^{-1}) \cup \supp(g \circ \eta^{-1}) \big\} \subset B.
\end{split}
\end{equation}
We recall that by Lemma \ref{lem_mod}, $\Omega \subset \widetilde{\mathcal{F}} \subset \mathcal{G}$, and hence
\begin{equation}
    \widetilde{K}^\epsilon(\Omega) \subset \widetilde{K}^\epsilon(\widetilde{\mathcal{F}}) = \widetilde{\mathcal{F}} \subset \mathcal{G} \quad \Rightarrow \quad  K^\epsilon(\eta(\Omega)) \subset \eta(\mathcal{G}).
\end{equation}
By invariance of $\Lambda^2$ with respect to $K^\epsilon$ and (\ref{eq3.43'}),
\begin{equation} \label{eq3.45}
\begin{split}
    &\left| \int_{\Omega \cap \eta(\Omega)} g[f \circ K^\epsilon - f \circ \eta^{-1} \circ K^\epsilon]\dd\Lambda^2 \right| \\
    & \quad\quad = \int_{K^\epsilon(\Omega \cap \eta(\Omega))} |g \circ K^\epsilon|\cdot |f - f \circ \eta^{-1}|\dd\Lambda^2 \\
    & \quad\quad \leq ||g \circ K^\epsilon||_{L^1_{\Lambda^2}(\mathbf{P} \times \mathbb{S}^2_+)} ||f - f \circ \eta^{-1}||_{L^\infty(\mathcal{G})} \\
    & \quad\quad = ||g||_{L^1_{\Lambda^2}(\mathbf{P} \times \mathbb{S}^2_+)} ||f - f \circ \eta^{-1}||_{L^\infty(\mathcal{G})} \to 0 \quad \text{ as } \epsilon \to 0.
\end{split}
\end{equation}
Also, by (\ref{eq3.43'}) and (\ref{eq3.42}), we have 
\begin{equation} \label{eq3.46}
\begin{split}
    &\left| \int_{\Omega \cap \eta(\Omega)} [g - g \circ \eta^{-1}][f \circ \eta^{-1} \circ K^\epsilon]\dd\Lambda^2 \right| \\
    & \leq ||f||_{L^\infty}\left| \int_{B \cap \Omega \cap \eta(\Omega)} [g - g \circ \eta^{-1}]\dd\Lambda^2 \right| \to 0 \quad \text{ as } \epsilon \to 0.
\end{split}
\end{equation}
In addition, note that $\supp g \subset \eta(B)$, and by Claim \ref{claim5.2.2}, $\frac{1}{U} \to 1$ uniformly on $B \cap \mathcal{G}$; thus
\begin{equation} \label{eq3.47}
\begin{split}
    &\left|\int_{\Omega \cap \eta(\Omega)} [g \circ \eta^{-1}][f \circ \eta^{-1} \circ K^\epsilon]\left[1 - \frac{1}{U \circ \eta^{-1}}\right]\dd\Lambda^2 \right| \\
    & \leq ||f||_{L^\infty} ||g||_{L^\infty} \int_{\Omega \cap \eta(B \cap \Omega)} \left|1 - \frac{1}{U \circ \eta^{-1}}\right|\dd\Lambda^2 \to 0 \quad \text{ as } \epsilon \to 0.
\end{split}
\end{equation}
By Claim \ref{claim5.2.2}, there exists a constant $d_1 < \infty$ such that $U \leq d_1$ on $B \cap \mathcal{G}$, for $\epsilon$ sufficiently small. Thus
\begin{equation} \label{eq3.48}
\begin{split}
    \left|\int_{\Omega \smallsetminus \eta(\Omega)} g[f \circ K^\epsilon]\dd\Lambda^2\right| & = \left|\int_{\eta^{-1}(\Omega) \smallsetminus \Omega} (g \circ \eta)[f \circ K^\epsilon \circ \eta] U \dd\Lambda^2\right| \\
    & = d_1||f||_{L^\infty} \int_{\eta^{-1}(\Omega) \smallsetminus \Omega} |g \circ \eta| \dd\Lambda^2 \\
    & \leq d_1||f||_{L^\infty} ||g||_{L^\infty} \Lambda^2(B \smallsetminus \Omega) \\
    & \leq d_1||f||_{L^\infty} ||g||_{L^\infty} \frac{C}{\epsilon\rho(\epsilon)}\Lambda^2(R_\epsilon \smallsetminus \Omega) \to 0 \text{ as } \epsilon \to 0,
\end{split}
\end{equation}
where $C = C_{B}$ as in Claim \ref{claim5.2.1} from Step 1, and the convergence to zero follows by Lemma \ref{lem_mod}. 

Also by Claim \ref{claim5.2.2} there exists $d_2 < \infty$ such that $\frac{1}{U} \leq d_2$ on $B \cap \mathcal{G}$. Noting that $\supp (g \circ \eta^{-1}) \subset B \cap \eta(B)$ by (\ref{eq3.42}), we similarly have
\begin{equation} \label{eq3.49}
\begin{split}
    &\left| \int_{\eta(\Omega) \smallsetminus \Omega} [g \circ \eta^{-1}][f \circ K^\epsilon]\frac{1}{U \circ \eta^{-1}} \dd\Lambda^2 \right| \\
    & = \left| \int_{B \cap \eta(\Omega \cap B) \smallsetminus \Omega} [g \circ \eta^{-1}][f \circ K^\epsilon]\frac{1}{U \circ \eta^{-1}} \dd\Lambda^2 \right| \\
    & \leq d_2||f||_{L^\infty} ||g||_{L^\infty} \Lambda^2(B \smallsetminus \Omega) \\
    & \leq d_2||f||_{L^\infty} ||g||_{L^\infty} \frac{C}{\epsilon\rho(\epsilon)}\Lambda^2(R_\epsilon \smallsetminus \Omega) \to 0 \text{ as } \epsilon \to 0.
\end{split}
\end{equation}
From (\ref{eq3.45}), (\ref{eq3.46}), (\ref{eq3.47}), (\ref{eq3.48}), and (\ref{eq3.49}), we see that each term in the bound (\ref{eq3.40}) converges to zero. Thus $I_2 \to 0$.

\textit{Step 3.} Finally, we show that $I_3 \to 0$. We must deal with the fact that $\widetilde{K}^\epsilon$ and $K^\epsilon_{\cyl}$ have different periodicity (see Propositions \ref{prop_cylcomm} and \ref{prop_modcomm}).
Let $\tau_{jk}$ be defined as above, and for $(j,k) \in \mathbb{Z}^2$, let 
\begin{equation}
    \tau_{jk}'(y) = y + \left(j + \left\lceil\frac{j(1+mJ^{-1})^{1/2}\rho}{\epsilon}\right\rceil\right)\epsilon e_1 + k(1+mJ^{-1})^{1/2}\rho \chi, \quad\quad y \in \mathbf{P}.
\end{equation}
In the lattice in $\mathbf{P}$ spanned by integer combinations of $e_1$ and $(1+mJ^{-1})^{1/2}\rho\chi$, the point $\tau_{jk}'(0)$ is the closest point in the lattice to the right of the point $\tau_{jk}(0) = j\epsilon e_1 + k\rho e_3$. We also let 
\begin{equation}
    \overline{\tau}_{jk}'(y,w) = (\tau_{jk}'(y), w), \quad\quad (y,w) \in \mathbf{P} \times \mathbb{S}^2_+.
\end{equation}
Then
\begin{equation}
    K^\epsilon \circ \overline{\tau}_{jk} = \overline{\tau}_{jk} \circ K^\epsilon, \quad\text{ and } \quad K^\epsilon_{\cyl} \circ \overline{\tau}_{jk}' = \overline{\tau}_{jk}' \circ K^\epsilon_{\cyl}.
\end{equation}
Also, for $k \in \mathbb{Z}$, define 
\begin{equation}
    \mathbf{P}_k = \{(x_1,\alpha) \in \mathbf{P} : k\rho - \rho/2 \leq \alpha < k\rho + \rho/2\},
\end{equation}
and define $\theta : \mathbf{P} \times \mathbb{S}^2_+ \to \mathbf{P} \times \mathbb{S}^2_+$ by
\begin{equation}
    \theta(y,w) = \left(y - \epsilon\left\{\frac{k (1+mJ^{-1})^{1/2}\rho}{\epsilon}\right\}e_1,w\right), \text{ if } (y,w) \in \mathbf{P}_k \times \mathbb{S}^2_+.
\end{equation}
Here $\left\{\frac{k(1+mJ^{-1})^{1/2}\rho}{\epsilon}\right\}$ is the fractional part of $\frac{k(1+mJ^{-1})^{1/2}\rho}{\epsilon}$. We see that $\theta$ is an order $\epsilon$ shift which ``corrects'' for the difference of periods, in the sense that
\begin{equation}
    \theta \circ \overline{\tau}_{jk}' = \overline{\tau}_{jk} \quad \text{ on } \mathbf{P}_k \times \mathbb{S}^2_+.
\end{equation}
We may write
\begin{equation} \label{eq3.55}
\begin{split}
    \int_{\Omega} g[f \circ \widetilde{K}^\epsilon - f \circ K^\epsilon_{\cyl}] & = \int_{\Omega} g[f \circ \widetilde{K}^\epsilon - f \circ \theta \circ K^\epsilon_{\cyl} \circ \theta^{-1}]\dd\Lambda^2 \\
    & + \int_{\Omega} g[f \circ \theta \circ K^\epsilon_{\cyl} \circ \theta^{-1} - f \circ K^\epsilon_{\cyl} \circ \theta^{-1}]\dd\Lambda^2 \\
    & + \int_{\Omega} g[f \circ K^\epsilon_{\cyl} \circ \theta^{-1} - f \circ K^\epsilon_{\cyl}]\dd\Lambda^2.
\end{split}
\end{equation}
We show separately that each of the above terms converges to zero. If $u$ is any Lipschitz function on $\mathbf{P} \times \mathbb{S}^2_+$, we let
\begin{equation} \label{eq5.61}
    \Lip(u) = \sup_{(y,w) \neq (y',w')} \frac{|u(y,w) - u(y',w')|}{||(y,w) - (y',w')||}. 
\end{equation}
By Claim \ref{claim5.2.1} from Step 1, there is an index set $I \subset \mathbb{Z}^2$ with $|I| \leq C/\epsilon\rho$ such that $\supp g$ is covered by the parallelograms $\tau_{jk}R_\epsilon$, $(j,k) \in I$. Using the fact that $g$ is bounded and $f$ is Lipschitz (since it is in $C_c^\infty$), we have 
\begin{equation}
\begin{split}
    &\left| \int_{\Omega} g[f \circ \widetilde{K}^\epsilon - f \circ \theta \circ K^\epsilon_{\cyl} \circ \theta^{-1}]\dd\Lambda^2\right| \\
    & \leq ||g||_{L^\infty} \Lip(f) \int_{\Omega \cap \supp g} ||\widetilde{K}^\epsilon - \theta \circ K^\epsilon_{\cyl} \circ \theta^{-1}||\dd\Lambda^2 \\
    & \leq ||g||_{L^\infty} \Lip(f) \sum_{(j,k) \in I} \int_{\overline{\tau}_{jk}(R_\epsilon \times \mathbb{S}^2_+) \cap \Omega} ||\widetilde{K}^\epsilon - \theta \circ K^\epsilon_{\cyl} \circ \theta^{-1}||\dd\Lambda^2 \\
    & = ||g||_{L^\infty} \Lip(f) \sum_{(j,k) \in I} \int_{(R_\epsilon \times \mathbb{S}^2_+) \cap \Omega} ||\widetilde{K}^\epsilon \circ \overline{\tau}_{jk} - \theta \circ K^\epsilon_{\cyl} \circ \overline{\tau}_{jk}'||\dd\Lambda^2 \\
    & = ||g||_{L^\infty} \Lip(f) \sum_{(j,k) \in I} \int_{(R_\epsilon \times \mathbb{S}^2_+) \cap \Omega} ||\overline{\tau}_{jk} \circ \widetilde{K}^\epsilon - \overline{\tau}_{jk} \circ K^\epsilon_{\cyl}|| \dd\Lambda^2 \\
    & = ||g||_{L^\infty} \Lip(f) |I| \int_{(R_\epsilon \times \mathbb{S}^2_+) \cap \Omega} ||\widetilde{K}^\epsilon - K^\epsilon_{\cyl}|| \dd\Lambda^2 \\
    & \leq ||g||_{L^\infty}\Lip(f)|I|\Lambda^2(R_\epsilon \times \mathbb{S}^2_+) \sup_{(R_\epsilon \times \mathbb{S}^2_+) \cap \Omega} ||\widetilde{K}^{\epsilon} - K^\epsilon|| \\
    & = 2\pi C||g||_{L^\infty} \Lip(f) \sup_{(R_\epsilon \times \mathbb{S}^2_+) \cap \Omega} ||\widetilde{K}^{\epsilon} - K^\epsilon||,
\end{split}
\end{equation}
using in the last line the fact that $|I| = C/\epsilon\rho$ and $\Lambda^2(R_\epsilon \times \mathbb{S}^2_+) = 2\pi\rho\epsilon$.
The last quantity converges to zero by Lemma \ref{lem_mod}. To show that the other two terms in (\ref{eq3.55}) converge to zero, observe that 
\begin{equation}
\begin{split}
    &\left|\int_{\Omega} g[f \circ \theta \circ K^\epsilon_{\cyl} \circ \theta^{-1} - f \circ K^\epsilon_{\cyl} \circ \theta^{-1}]\dd\Lambda^2\right| \\
    & \leq \int_{\mathbf{P} \times \mathbb{S}^2_+} |g||f \circ \theta \circ K^\epsilon_{\cyl} \circ \theta^{-1} - f \circ K^\epsilon_{\cyl} \circ \theta^{-1}|\dd\Lambda^2 \\
    & = \int_{\mathbf{P} \times \mathbb{S}^2_+} |g||f \circ \theta - f|\dd\Lambda^2,
\end{split}
\end{equation}
where the last equality follows by invariance of $\Lambda^2$ under translation and $K^\epsilon_{\cyl}$. The quantity above converges to zero because $\theta \to 0$ uniformly as $\epsilon \to 0$. To handle the last term in (\ref{eq3.55}), we make the change of variables $(y,w) \mapsto \theta(y,w)$ to obtain
\begin{equation}
\begin{split}
    &\left| \int_{\Omega} g[f \circ K^\epsilon_{\cyl} \circ \theta^{-1} - f \circ K^\epsilon_{\cyl}]\dd\Lambda^2 \right| \\
    & = \left| \int_{\theta^{-1}(\Omega)} [g \circ \theta][f \circ K^\epsilon_{\cyl}]\dd\Lambda^2 - \int_{\Omega} g[f \circ K^\epsilon_{\cyl}]\dd\Lambda^2 \right| \\
    & \leq \int_{\mathbf{P} \times \mathbb{S}^2_+} |g \circ \theta - g||f \circ K^\epsilon_{\cyl}|\dd\Lambda^2 \\
    & \quad\quad + ||g||_{L^\infty}||f||_{L^\infty}\left\{\Lambda^2(\theta^{-1}(\supp g) \smallsetminus \theta^{-1}(\Omega)) + \Lambda^2(\supp g \smallsetminus \Omega)\right\} \\
    & \leq ||g \circ \theta - g||_{L^\infty}||f||_{L^1_{\Lambda^2}} + 2||g||_{L^\infty}||f||_{L^\infty}\Lambda^2(\supp g \smallsetminus \Omega),
\end{split}
\end{equation}
using invariance of $\Lambda^2$ with respect to $K^\epsilon_{\cyl}$ and $\theta$ to obtain the last line. The first term in the line above converges to zero because $\theta \to 0$ uniformly, and the second term converges to zero by the same argument as in Step 1. Thus all three terms on the right-hand side of (\ref{eq3.55}) converge to zero, so $I_3 \to 0$.
\end{proof}

We now apply Lemma \ref{lem_comparison} and Theorem \ref{thm_cylindrical} to prove our main results.

\begin{proof}[Proof of Theorem \ref{thm_pure_scaling}]
Fix a constant sequence of cells $\Sigma_i = \Sigma$ and a sequence of positive numbers $\epsilon_i \to 0$. By Theorem \ref{thm_cylindrical}, there exists a Markov kernel $\mathbb{K}$ not depending on the sequence $\epsilon_i$ such that the limit $\lim_{i \to \infty} \mathbb{K}^{\Sigma_i,\epsilon_i}_{\cyl}$ exists and is equal to $\mathbb{K}$. Moreover, $\mathbb{K}$ takes the form \ref{rough_col0}. We may write 
\begin{equation}
\begin{split}
&\int_{\mathbf{P} \times \mathbb{S}^2_+} g(y,w) \left(\int_{\mathbf{P} \times \mathbb{S}^2_+} f(y,w)\mathbb{K}^{\Sigma_i,\epsilon_i}(y,w;\dd y'\dd w') \right) \dd\Lambda^2(\dd y \dd w) \\
& = \int_{\mathbf{P} \times \mathbb{S}^2_+} g(y,w) \left(\int_{\mathbf{P} \times \mathbb{S}^2_+} f(y,w)\mathbb{K}^{\Sigma_i,\epsilon_i}_{\cyl}(y,w;\dd y'\dd w') \right) \dd\Lambda^2(\dd y \dd w) \\
& \quad\quad + \int_{\mathbf{P} \times \mathbb{S}^2_+} g(y,w) [f(K^{\Sigma_i,\epsilon_i}(y,w)) - f(K^{\Sigma_i,\epsilon_i}_{\cyl}(y,w))] \dd\Lambda^2(\dd y \dd w).
\end{split}
\end{equation}
The second term above converges to zero as $\epsilon_i \to 0$ by Lemma \ref{lem_comparison}. It follows that $\mathbb{K}^{\Sigma_i,\epsilon_i} \to \mathbb{K}$, and the theorem is proved. 
\end{proof}

\begin{proof}[Proof of Theorem \ref{thm_dichotomy}]
Fix a sequence of cells $\{\Sigma_i\}_{i \geq 1}$. Lemma \ref{lem_comparison} implies that, for each fixed $i \geq 1$,
\begin{equation}
    d_{\mathcal{G}}^{\Lambda^2}(\mathbb{K}^{\Sigma_i,\epsilon}_{\cyl}, \mathbb{K}^{\Sigma_i,\epsilon}) \to 0 \quad \text{ as } \epsilon \to 0,
\end{equation}
where $d^{\Lambda^2}_{\mathcal{G}}$ is the pseudometric introduced in \ref{sssec_pseudo}. Thus we may choose a decreasing sequence $\{b_i^{(1)}\}_{i \geq 1}$ such that, for each $i \geq 1$, if $\epsilon \leq b_i^{(1)}$, then 
\begin{equation}
    d_{\mathcal{G}}^{\Lambda^2}(\mathbb{K}^{\Sigma_i,\epsilon}_{\cyl}, \mathbb{K}^{\Sigma_i,\epsilon}) \leq i^{-1}.
\end{equation}
Let another decreasing sequence $\{b_i^{(2)}\}_{i \geq 1}$ be chosen as in Theorem \ref{thm_cylindrical}, and let $b_i = \min\{b_i^{(1)}, b_i^{(2)}\}$. 

To prove Theorem \ref{thm_dichotomy}, we will show that $\neg (B) \Rightarrow (A)$. Assume that there exists a sequence $\epsilon_i^{(0)} \to 0$, $\epsilon_i^{(0)} \leq b_i$ such that $\mathbb{K} := \lim_{i \to \infty} \mathbb{K}^{\Sigma_i,\epsilon_i^{(0)}}$ exists. Then by the triangle inequality, 
\begin{equation} \label{eq4.62}
\begin{split}
    d_{\mathcal{G}}^{\Lambda^2}(\mathbb{K}^{\Sigma_i,\epsilon_i^{(0)}}_{\cyl}, \mathbb{K}) & \leq d_{\mathcal{G}}^{\Lambda^2}(\mathbb{K}^{\Sigma_i,\epsilon_i^{(0)}}_{\cyl}, \mathbb{K}^{\Sigma_i,\epsilon_i^{(0)}}) + d_{\mathcal{G}}^{\Lambda^2}(\mathbb{K}^{\Sigma_i,\epsilon_i^{(0)}}, \mathbb{K}) \\
    & \quad\quad \to 0 \quad \text{ as } i \to \infty.
\end{split}
\end{equation}
By Theorem \ref{thm_cylindrical}, it follows that, for any sequence $\epsilon_i \to 0$, $\epsilon_i \leq b_i$, $\mathbb{K}^{\Sigma_i,\epsilon_i}_{\cyl} \to \mathbb{K}$ as $i \to \infty$; so by the triangle inequality again 
\begin{equation} \label{eq4.63}
\begin{split}
    d_{\mathcal{G}}^{\Lambda^2}(\mathbb{K}^{\Sigma_i,\epsilon_i}, \mathbb{K}) & \leq d_{\mathcal{G}}^{\Lambda^2}(\mathbb{K}^{\Sigma_i,\epsilon_i}_{\cyl}, \mathbb{K}^{\Sigma_i,\epsilon_i}) + d_{\mathcal{G}}^{\Lambda^2}(\mathbb{K}^{\Sigma_i,\epsilon_i}_{\cyl}, \mathbb{K}) \\
    & \quad\quad \to 0 \quad \text{ as } i \to \infty.
\end{split}
\end{equation}
This proves the theorem.
\end{proof}

\begin{remark} \normalfont
The pseudometric $d^{\Lambda^2}_{\mathcal{G}}$ is an important ingredient in the proof. A direct application of (\ref{eq5.25}) to handle convergence would not allow for a choice of $\{b_i\}_{i \geq 1}$ which is independent of $f$ and $g$.
\end{remark}

\begin{proof}[Proof of Theorem \ref{thm_classification}]
Suppose that $\mathbb{K} \in \mathcal{A}_0$. Then there exists a sequence of cells $\{\Sigma_i\}_{i \geq 1}$ and a decreasing sequence of positive numbers $\{b_i\}_{i \geq 1}$ such that (A) holds, i.e. for any sequence $\epsilon_i \to 0$ with $0 < \epsilon_i \leq b_i$, $\lim \mathbb{K}^{\Sigma_i,\epsilon_i} = \mathbb{K}$. By the triangle inequality argument giving us (\ref{eq4.62}) in the proof of Theorem \ref{thm_dichotomy}, if $\epsilon_i \to 0$ sufficiently fast then we also have $\lim \mathbb{K}^{\Sigma_i,\epsilon_i}_{\cyl}$ exists and is equal to $\mathbb{K}$, and it follows from Theorem \ref{thm_cylindrical} that $\mathbb{K}$ must take the form (\ref{rough_col}). The proves the forward direction.

To prove the converse statement, suppose $\mathbb{K}$ is a Markov kernel on $\mathbf{P} \times \mathbb{S}_+^2$ which takes the form (\ref{rough_col}) where $\widetilde{\mathbb{P}}$ preserves the measure $\sin\theta \dd\theta$. By Theorem \ref{thm_cylindrical}, there is a sequence of cells $\Sigma_i$ and a decreasing sequence of positive numbers $b_i^{(1)}$ such that $\mathbb{K} = \lim_{i \to \infty} \mathbb{K}^{\Sigma_i,\epsilon_i}_{\cyl}$ whenever $0 < \epsilon_i \leq b_i^{(1)}$ and $\epsilon_i \to 0$. By the triangle inequality argument giving us (\ref{eq4.63}) in the proof of Theorem \ref{thm_dichotomy}, we have that $\mathbb{K}^{\Sigma_i,\epsilon_i} \to \mathbb{K}$ provided that $\epsilon_i \to 0$ sufficiently fast. This implies that $\mathbb{K} \in \mathcal{A}_0$.
\end{proof}

\subsection{Zooming argument} \label{ssec_zoom}

In this subsection, we will define the set $\Omega$ and prove Lemma \ref{lem_mod}. The important idea from this section is to compare the billiard evolutions in the respective configuration spaces $\mathcal{M}$ and $\mathcal{M}_{\cyl}$ after ``zooming in'' by a factor of $\epsilon^{-1}$.

Throughout this subsection, the cell $\Sigma$ is fixed. The scale $\epsilon > 0$ is also fixed, except at the very end.

\subsubsection{Notation and elementary observations} \label{sssec_zoomnote}

For $r > 0$, we define  
\begin{equation}
    \sigma_r(y,w) = (ry, w), \quad\quad (y,w) \in \mathbb{R}^3 \times \mathbb{R}^3.
\end{equation}

Throughout this subsection, if $A$ is a subset of $\mathbb{R}^3$ and $B$ is a subset of $\mathbb{R}^3 \times \mathbb{R}^3$ we will denote 
\begin{equation}
\begin{split}
    & A^* = \epsilon^{-1}A = \{y \in \mathbb{R} : \epsilon y \in A\}, \\
    & B^* = \sigma_{\epsilon^{-1}}(B) = \{(y,w) \in \mathbb{R}^3 \times \mathbb{R}^3 : (\epsilon y, w) \in B\}.
\end{split}
\end{equation}
Thus, for example, we will consider $\mathcal{M}^*$, $\mathcal{M}^*_{\cyl}$, $\mathcal{Z}^*$, $\widehat{\mathcal{Z}}^*$, $\Gamma_{\roll}^* \subset \partial \mathcal{M}^*$, and so forth.

As before, we let $\mathbf{Q}_0 = \{(x_1,x_2,\alpha) : \alpha = 0\}$. We will often identify this plane with $\mathbb{R}^2$.

One motivation for ``zooming'' is that the zoomed cylindrical configuration space $\mathcal{M}_{\cyl}^*$ does not depend on $\epsilon$. In particular, $\mathcal{M}_{\cyl}^*$ is the cylinder with base 
\begin{equation} \label{eq5.70}
    \widehat{B}^* := \epsilon^{-1}\widehat{B}(\epsilon) = \overline{W(\Sigma,1)^c} + e_2 \subset \mathbf{Q}_0
\end{equation}
and axis $\chi$. Neither the base nor the axis depend on $\epsilon$. Another important observation is that 
\begin{equation}
     \mathcal{M}_{\cyl}^* \cap \mathbf{Q}_0 = \widehat{B}^* = \mathcal{M}^* \cap \mathbf{Q}_0.
\end{equation}
(The second equality follows from the parametrization of the configuration space given by Proposition \ref{prop_Mreg}(iii).) Thus, while $\mathcal{M}^*$ does depend on $\epsilon$, its intersection with $\mathbf{Q}_0$ does not.

By Proposition \ref{prop_Mreg}(iii), the subset of the zoomed configuration space $\mathcal{M}^* \cap \widehat{\mathcal{Z}}^*$ is parametrized by $f^* : \overline{W(\Sigma,1)^c} \times \mathcal{Z}^* \to \mathbb{R}^3$, defined by 
\begin{equation} \label{eq3.56}
    f^*(x_1,x_2,\alpha) = \epsilon^{-1}f(\epsilon x_1,\epsilon x_2,\epsilon\alpha) =  \begin{pmatrix}
    x_1 - \epsilon^{-1}\sin(\epsilon \overline{\alpha}^*) \\
    x_2 + 1 + \epsilon^{-1}(-1+\cos(\epsilon\overline{\alpha}^*)) \\
    \alpha
    \end{pmatrix},
\end{equation}
where $\overline{\alpha}^* = \alpha - \overline{k}^*\rho$ and $\overline{k}^* = \text{argmin}\{|\alpha - k\epsilon^{-1}\rho| : k \in \mathbb{Z}\}$. In particular, note that the restriction of $f^*$ to $\partial W(\Sigma,1) \times \mathcal{Z}^* \to \mathbb{R}^3$ parametrizes $\Gamma_{\roll}^* = \partial \mathcal{M}^* \cap \widehat{\mathcal{Z}}^*$.

Recall the diffeomorphism $H_1 : \widehat{\mathcal{Z}} \to \widehat{\mathcal{Z}}$ defined by (\ref{eq4.43}). We define $H_\epsilon : \widehat{\mathcal{Z}}^* \to \widehat{\mathcal{Z}}^*$ by 
\begin{equation} \label{eq3.54}
    H_\epsilon(x_1,x_2,\alpha) = \epsilon^{-1}H_1(\epsilon x_1, \epsilon x_2, \epsilon\alpha) = \begin{pmatrix}
    x_1 + \overline{\alpha}^* - \epsilon^{-1}\sin(\epsilon\overline{\alpha}^*) \\
    x_2 + \epsilon^{-1}(-1 + \cos(\epsilon\overline{\alpha}^*)) \\
    \alpha
    \end{pmatrix}.
\end{equation}
(Arguably, we should denote this map instead by $H^*$. However, the parameter $\epsilon$ will play an important role in our arguments.)

Corresponding to (\ref{eq4.46}), we have 
\begin{equation} \label{eq5.71}
\begin{split}
    &H_\epsilon(\mathcal{M}_{\cyl}^* \cap \widehat{\mathcal{Z}}^*) = \mathcal{M}^* \cap \widehat{\mathcal{Z}}^*, \text{ and } \\
    &H_\epsilon(\partial \mathcal{M}_{\cyl}^* \cap \widehat{\mathcal{Z}}^*) = \partial \mathcal{M}^* \cap \widehat{\mathcal{Z}}^*.
\end{split}
\end{equation}

The motivation behind ``zooming'' becomes clearer by observing that 
\begin{equation}
    H_\epsilon(y) = y + \text{Error}(y,\epsilon),
\end{equation}
where the error term and its first derivatives converge uniformly to zero as $\epsilon \to 0$. Thus $H_\epsilon$ maps straight chords in $\mathcal{M}_{\cyl}^*$ to ``almost straight'' chords in $\mathcal{M}^*$, and this suggests that a billiard trajectory in $\mathcal{M}^*$ will in some sense be closely tracked by a billiard trajectory in $\mathcal{M}^*_{\cyl}$ as $\epsilon \to 0$.

\subsubsection{Projection onto the cylindrical base} \label{sssec_cylproject}

We will analyze separately the projection of the billiard trajectory onto the cylindrical base and the cylindrical axis. Define a projection map $G_0 : \mathbb{R}^3 \to \mathbf{Q}_0$ by
\begin{equation}
    G_0(x_1,x_2,\alpha) = (x_1 + \alpha, x_2).
\end{equation}
Note that $G_0(\chi) = 0$, and consequently, $G_0$ projects the cylinder $\mathcal{M}_{\cyl}^*$ onto its base $\widehat{B}^*$ and the cylindrical boundary $\partial \mathcal{M}_{\cyl}^*$ onto  $\partial \widehat{B}^*$. 

Therefore, by (\ref{eq5.71}), the ``perturbed'' projection map
\begin{equation}
    G_\epsilon := G_0 \circ H_\epsilon^{-1} : \widehat{\mathcal{Z}}^* \to \mathbf{Q}_0
\end{equation}
maps $\mathcal{M}^* \cap \widehat{\mathcal{Z}}^*$ onto $\widehat{B}^*$ and $\partial \mathcal{M}^* \cap \widehat{\mathcal{Z}}^*$ onto $\partial \widehat{B}^*$. Explicitly, $G_\epsilon$ is given by the formula 
\begin{equation} \label{eq3.60'}
G_\epsilon(x_1,x_2,\alpha) = \begin{pmatrix}
x_1 + \epsilon^{-1}\sin(\epsilon\overline{\alpha}^*) \\
x_2 - \epsilon^{-1}(-1 + \cos(\epsilon\overline{\alpha}^*))
\end{pmatrix}.
\end{equation}

The  goal is to compare the projection under $G_0$ of the billiard trajectory in $\mathcal{M}_{\cyl}^*$ to the projection under $G_\epsilon$ of the billiard trajectory in $\mathcal{M}^*$. The advantage of considering the projection is that the image of $G_\epsilon$ is $\widehat{B}^*$ which does not depend on $\epsilon$. 

Towards our goal, we first provide estimates on the differential of $G_\epsilon$.

\begin{lemma} \label{dGlem}
Let $\epsilon \geq 0$, $y = (x_1,x_2,\alpha) \in \mathbb{R}^3$ and $w = (v_1,v_2,\omega) \in \mathbb{S}^2$. Set $\varphi = \angle(w, \chi)$. Then  
\begin{equation} \label{eq5.79}
\begin{split}
[1 - ((mJ)^{-1/2} + 1)|\epsilon\overline{\alpha}^*|]\sin\varphi & \leq ||\dd(G_\epsilon)_{y}(w)|| \\
& \leq (1 + mJ^{-1})^{1/2}[1 + ((mJ)^{-1/2} + 1)|\epsilon\overline{\alpha}^*|]\sin\varphi
\end{split}
\end{equation}
\end{lemma}

\begin{proof}
One may easily prove the result if $\epsilon = 0$, because then $G_0$ is just the linear map which projects vectors in the direction $\chi$ onto the plane $\mathbf{Q}_0$, and
\begin{equation}
\dd(G_0)_{y}(w) = \begin{pmatrix} v_1 + \omega \\
v_2
\end{pmatrix}.
\end{equation}
Indeed, if $\theta$ is the angle between $\chi$ and the plane $\mathbf{Q}_0$, then  $\dd(G_0)_{y}(w)$ must lie on the boundary of the ellipse with major and minor axes $\sin\varphi$ and $\frac{\sin\varphi}{\sin\theta}$ respectively; hence
\begin{equation}
    \sin\varphi \leq \dd(G_0)_{y}(w) \leq \frac{\sin\varphi}{\sin\theta}.
\end{equation}
Noting that $\widehat{e}_3 := (0,0,J^{-1/2})$ is the unit normal to the plane $\mathbf{Q}_0$, we have
\begin{equation}
    \sin\theta = \langle \chi, n_0 \rangle = \left(\frac{J}{m + J} \right)^{1/2}
\end{equation}
Substituting this into the previous equality yields
\begin{equation} \label{eq3.62'}
    \sin\varphi \leq \dd(G_0)_{y}(w) \leq (1 + mJ^{-1})^{1/2}\sin\varphi.
\end{equation}

Now assume $\epsilon > 0$. The differential of $G_\epsilon$ is 
\begin{equation} \label{eq3.66}
\dd(G_\epsilon)_{y} = \begin{bmatrix} 
1 & 0 & \cos(\epsilon\overline{\alpha}^*) \\
0 & 1 & \sin(\epsilon\overline{\alpha}^*)
\end{bmatrix}.
\end{equation}
Thus
\begin{equation} \label{eq3.67}
||\dd(G_\epsilon)_{y}(w)||^2 = (v_1 + \omega)^2 + (v_2)^2 + E(\epsilon, y,w) = ||\dd(G_0)_{y}(w)||^2 + E(\epsilon,y,w),
\end{equation}
where 
\begin{equation}
E(\epsilon,y,w) = 2\omega[-v_1(1 - \cos(\epsilon\overline{\alpha}^*)) + v_2\sin(\epsilon\overline{\alpha}^*)] + \omega^2\sin^2(\epsilon\overline{\alpha}^*).
\end{equation}
We estimate the middle term
\begin{equation}
\begin{split}
& 2\omega[-v_1(1 - \cos(\epsilon\overline{\alpha}^*)) + v_2\sin(\epsilon\overline{\alpha}^*)] \\
& \leq 2\omega \sqrt{(v_1)^2 + (v_2)^2}\sqrt{(1 - \cos(\epsilon\overline{\alpha}^*))^2 + \sin^2(\epsilon\overline{\alpha}^*)} \quad\quad (\text{Cauchy-Schwarz}) \\
& = 2\omega\sqrt{(v_1)^2 + (v_2)^2}\sqrt{2 - 2\cos(\epsilon\overline{\alpha}^*)} \\
& \leq 2\omega\sqrt{(v_1)^2 + (v_2)^2}|\epsilon\overline{\alpha}^*|  \quad\quad\quad\quad\quad\quad\quad (\text{using } 1 - \cos(x) \leq x^2/2) \\
& = 2(mJ)^{-1/2}(J^{1/2}\omega m^{1/2}\sqrt{(v_1)^2 + (v_2)^2}|\epsilon\overline{\alpha}^*| \\
& \leq (mJ)^{-1/2}[m(v_1)^2 + m(v_2)^2 + J \omega^2]|\epsilon\overline{\alpha}^*| \\
& = (mJ)^{-1/2}||w||^2 |\epsilon\overline{\alpha}^*| = (mJ)^{-1/2}|\epsilon\overline{\alpha}^*|,
\end{split}
\end{equation}
where the second-to-last line above follows from the Arithmetic-Geometric Mean Inequality. Since $\omega^2 \leq ||w||^2 = 1$, we have 
\begin{equation}
|E(\epsilon,y,w)| \leq ((mJ)^{-1/2} + 1)|\epsilon\overline{\alpha}^*|.
\end{equation}
By (\ref{eq3.67}), we obtain 
\begin{equation}
\begin{split}
||\dd(G_0)_{y}(w)||^2 - ((mJ)^{-1/2} + 1)|\epsilon\overline{\alpha}^*| & \leq ||\dd(G_\epsilon)_{y}(w)||^2 \\
& \leq ||\dd(G_\epsilon)_{y}(w)||^2 + ((mJ)^{-1/2} + 1)|\epsilon\overline{\alpha}^*|.
\end{split}
\end{equation}
Taking the square root of each side, and using the fact that $\sqrt{x + h} \leq \sqrt{x}(1+h)$ for $h > 0$ and $\sqrt{x + h} \geq \sqrt{x}(1+h)$ for $h < 0$, we conclude that 
\begin{equation}
\begin{split}
||\dd(G_0)_{y}(w)||(1 - ((mJ)^{-1/2} & + 1)|\epsilon\overline{\alpha}^*|) \leq ||\dd(G_\epsilon)_{y}(w)|| \\
& \leq ||\dd(G_0)_{y}(w)||(1 + ((mJ)^{-1/2} + 1)|\epsilon\overline{\alpha}^*|).
\end{split}
\end{equation}
By (\ref{eq3.62'}) and the above, we obtain (\ref{eq5.79}).
\end{proof}

\subsubsection{Radius of transversality}

To state our next result, we must introduce some additional terminology. Consider a closed set $\mathcal{D} \subset \mathbb{R}^2$. Assume that $\mathcal{D}$ is the closure of its interior and has a piecewise $C^2$ boundary (in the sense of \S\ref{sssec_uphalf}). Assume that some decomposition $\partial \mathcal{D} = \bigcup_{i = 1}^\infty I_i$ into closed $C^2$ curve segments has been fixed.

We will call a line segment between two points $p$ and $p'$ in $\partial \mathcal{D}$ such that the interior of the line segment lies in $\Int \mathcal{D}$ a \textit{chord} of $\mathcal{D}$.

Consider a chord $C$ of the domain $\mathcal{D}$ with endpoints $p,p' \in \partial \mathcal{D}$, and let $v = \frac{p' - p}{||p' - p||}$. (For example, $C$ might be the segment of the billiard trajectory starting from the initial state $(p,v)$.) 

Let $I, I' \subset \partial\mathcal{D}$ denote $C^2$ curve segments from the fixed decomposition $\{I_i\}$ containing $p, p'$ respectively, and let $L(p,p')$ denote the line through points $p$ and $p'$. 

Let $Z \subset I$ be the set of all points $q \in \Int I$ such that the ray starting from $q$ with direction $v$ first returns to the boundary at a point $q' \in \Int I'$ and such that the line $L(q,q')$ is not tangent to the curve segments $I$, $I'$ at the points $q$, $q'$ respectively. 

We define the \textit{radius of transversality associated with} $C$ to be the number 
\begin{equation}
r_{\mathcal{D}}(C) = r_{\mathcal{D}}(p,v) = \sup\{r \geq 0 : (\forall q \in I) \  \text{dist}(q,L(p,p')) \leq r \Rightarrow q \in Z\}.
\end{equation}
This quantity is well-defined up to the choice of decomposition of $\partial \mathcal{D}$ into $C^2$ curve segments $\{I_i\}$. Indeed, if either $p$ or $p'$ lies in at the endpoint of a curve segment in the fixed decomposition of $\partial \mathcal{D}$, then $r_{\mathcal{D}}(p,v) = 0$ for any choice of $I$ and $I'$. On the other hand, if $p$ and $p'$ both lie in the interiors of curve segments, then $I$ and $I'$ are uniquely determined.
 
The concept is illustrated in Figure \ref{transversality_fig}. The term ``transversality'' refers to the fact that the lines $L(q,q')$ must cross the segments $I, I'$ transversally. We will also sometimes refer to this as the radius of transversality associated with state $(p,v)$ if $C$ arises as a segment of a billiard trajectory.

\begin{figure}
    \centering
    \includegraphics[width = 0.7\linewidth]{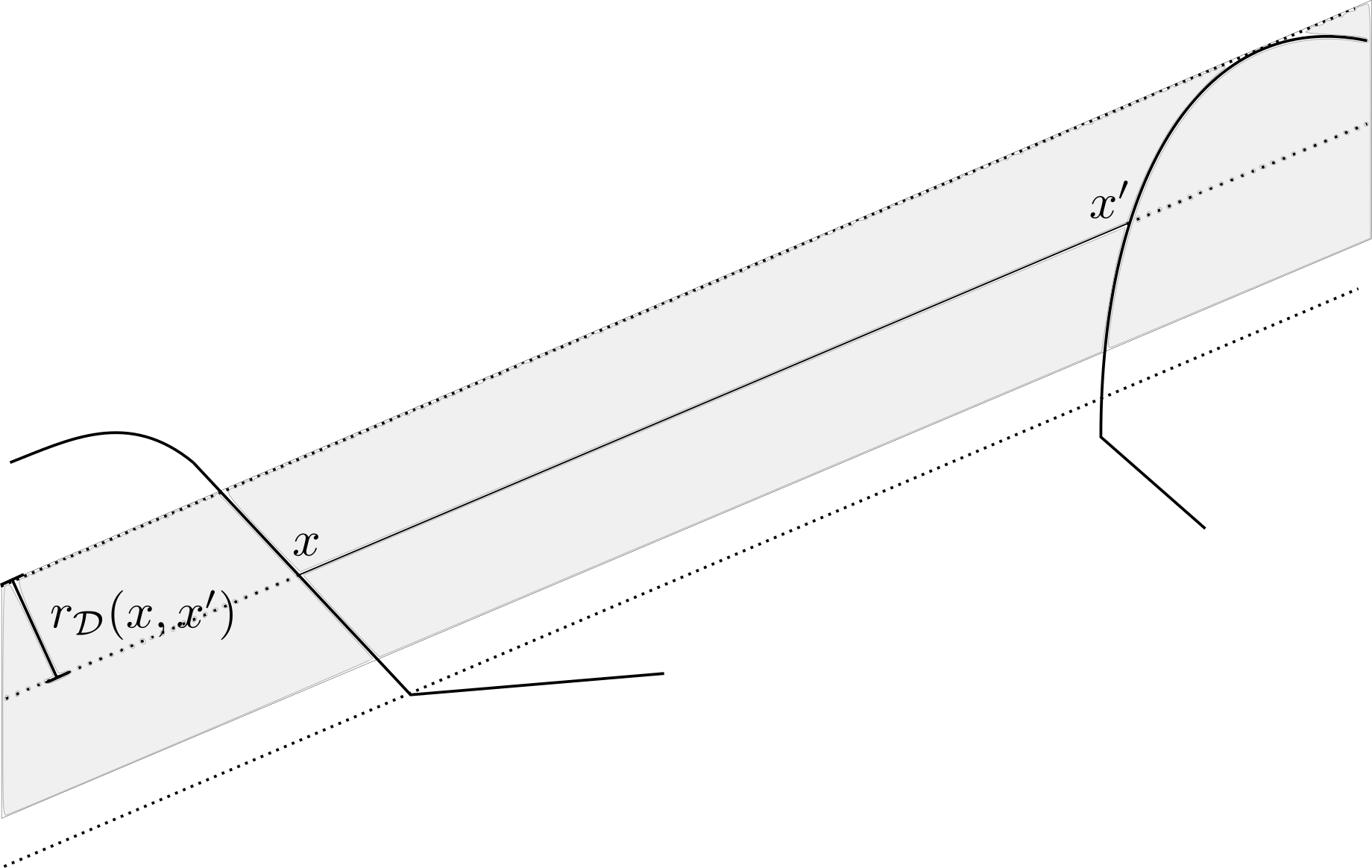}
    \caption{Radius of transversality}
    \label{transversality_fig}
\end{figure}

\subsubsection{Modulus of continuity lemmas} \label{sssec_modcon}

To state our next results, we need to refer to some aspects of the construction of the cylindrical collision law $K^\epsilon_{\cyl}$. According to the definition of $K^\epsilon_{\cyl}$ given in \S\ref{sssec_cylcol}, the cylindrical collision law is a special case of the general macro-reflection laws defined in \S\ref{sssec_detref} (taking $\mathcal{M}_1 = \mathcal{M}_{\cyl}$ and $\mathcal{M}_0 = \{(x_1,x_2,\alpha) : x_2 \geq 0\}$). Adapting notation from Remark \ref{rem_extend} to this setting, we have
\begin{equation}
    \mathcal{N} = \overline{\{(x_1,x_2,\alpha) \in \mathcal{M}_{\cyl} : x_2 < 0\}},
\end{equation}
and the plane $\mathbf{P}$ may be expressed as the disjoint union
\begin{equation}
    \mathbf{P} = A_1 \cup A_2 \cup A_3,
\end{equation}
where 
\begin{equation}
    A_1 = \mathbf{P} \cap \mathcal{N} \cap \Int \mathcal{M}_{\cyl}, \quad\quad A_2 = \mathbf{P} \cap \mathcal{N} \smallsetminus \Int \mathcal{M}_{\cyl}, \quad\quad A_3 = \mathbf{P} \smallsetminus \mathcal{N}.
\end{equation}
The set $A_2$ has Lebesgue measure zero, and consequently $\mathcal{N}$ belongs to the class $\CES^2_0(\mathbb{R}^3)$. (See the proof of Proposition \ref{prop_detcolcyl}.)

The billiard domain $\mathcal{N}$ plays a key role in the definition of $K_{\cyl}^\epsilon$. If $(y,w) \in A_3 \times \mathbb{S}^2_+$, then the collision law maps $(y,w)$ to $R(y,-w)$, where $R$ denotes specular reflection. And if $(y,w) \in A_1 \times \mathbb{S}^2_+$, then (except on a null set) the billiard trajectory starting from $(y,-w)$ will enter $\mathcal{N}$ and reflect from $\partial \mathcal{M}_{\cyl}$ a finite number of times before eventually returning to $\mathbf{P}$ in a state $(y',w')$; in this case $K_{\cyl}^\epsilon(y,w) := (y',w')$. The details of the more general case are covered in \S\ref{sssec_detref} and Remark \ref{rem_extend}.

We would like to study the billiard trajectory in $\mathcal{N}$ after zooming. That is, we introduce ``zoomed'' versions of the sets defined above:
\begin{equation} \label{eq5.92}
    \mathcal{N}^* = \epsilon^{-1}\mathcal{N} = \overline{\{(x_1,x_2,\alpha) \in \mathcal{M}_{\cyl}^* : x_2 < 0\}}, \quad\quad A_i^* = \epsilon^{-1}A_i, \quad i = 1,2,3.
\end{equation}
The set $\mathcal{N}^*$ is cylindrical, with axis $\chi$ and base
\begin{equation} \label{eq5.93}
    \mathcal{D}^* := G_0(\mathcal{N}^*)  = \overline{\{(x_1,x_2) \in \widehat{B}^* : x_2 < 0\}} \subset \mathbf{Q}_0,
\end{equation}
where we recall that $\widehat{B}^*$ is the base of the cylinder $\mathcal{M}_{\cyl}^*$ in the plane $\mathbf{Q}_0$; see (\ref{eq5.70}). By our assumptions on the wall, the boundary $\partial \widehat{B}^*$ is piecewise $C^2$; it decomposes as a union of compact $C^2$ curves $\partial \widehat{B}^* = \bigcup_{i \in \mathbb{Z}} \widehat{\Gamma}_i^*$; for each $i$, if the interior of $\widehat{\Gamma}_i^*$ intersects the line $\mathbf{L} := \{(x_1,x_2) : x_2 = 0\}$, then $\widehat{\Gamma}_i^* \subset \mathbf{L}$. Consequently, $\mathcal{D}^*$ also has a piecewise $C^2$ boundary, consisting of a union of some subcollection of the curve segments $\{\widehat{\Gamma}_i^*, i \in \mathbb{Z}\}$ and a collection of line segments in $\mathbf{L}$. Fixing such a decomposition of $\partial \mathcal{D}^*$, we can unambiguously refer to the radius of transversality of a chord of $\mathcal{D}^*$.

We will now very carefully study a single segment of the billiard trajectory in $\mathcal{N}^*$ and its projected image in $\mathcal{D}^*$.

Let $y^0 \in \partial \mathcal{N}^*$ and $w^0 \in \mathbb{S}^2$. We assume that the billiard trajectory starting from initial state $(y^0,w^0)$ immediately enters the interior of $\mathcal{N}^*$ and first returns to $\partial \mathcal{N}^*$ at a point $y^1$. We also assume that if $y^i \in \partial \mathcal{M}_{\cyl}^*$ then the unit normal vector at $y^i$ exists, for $i = 0,1$. These assumptions will be satisfied for example if $(y^0,w^0)$ is in the domain of the billiard map for the billiard domain $\mathcal{N}^*$.

If $y^1 \in \partial \mathcal{M}_{\cyl}^*$, let $w^1$ denote the velocity of the point particle after reflecting specularly at $y^1$. Otherwise (if $y^1 \in \mathbf{P}$) let $w^1 = w^0$.

Recall the surface $\widetilde{\mathbf{P}} = H_1(\mathbf{P} \cap \widehat{\mathcal{Z}})$, which is the boundary of the open sets $\mathcal{O}_{\pm} = H_1(\mathbb{R}^3_\pm\cap \widehat{\mathcal{Z}}^*)$ (see \S\ref{sssec_modcol}). We have
\begin{equation}
    \widetilde{\mathbf{P}}^* := \epsilon^{-1}\widetilde{\mathbf{P}} = H_{\epsilon}(\mathbf{P} \cap \widehat{\mathcal{Z}}^*), \quad\quad \mathcal{O}_{\pm}^* := \epsilon^{-1}\mathcal{O}_{\pm} = H_\epsilon(\mathbb{R}^3_{\pm} \cap \widehat{\mathcal{Z}}^*).
\end{equation}

We want to compare the segment of the billiard trajectory running from $(y^0,w^0)$ to $(y^1,w^1)$ to a segment of a billiard trajectory in $\mathcal{M}^*$. However, rather than taking as given that the segment of the trajectory in $\mathcal{M}^*$ has nice properties, as we have for the segment in $\mathcal{M}_{\cyl}^*$, we will state conditions in the lemmas below which guarantee that this is the case. Our goal in Lemmas \ref{lem_control} and \ref{lem_control2} is to be able to control the billiard trajectory in $\mathcal{M}^*$ in terms of properties of the billiard trajectory in $\mathcal{M}_{\cyl}^*$. See Figure \ref{projections_fig}

\begin{figure}
    \centering
    \includegraphics[width = \linewidth]{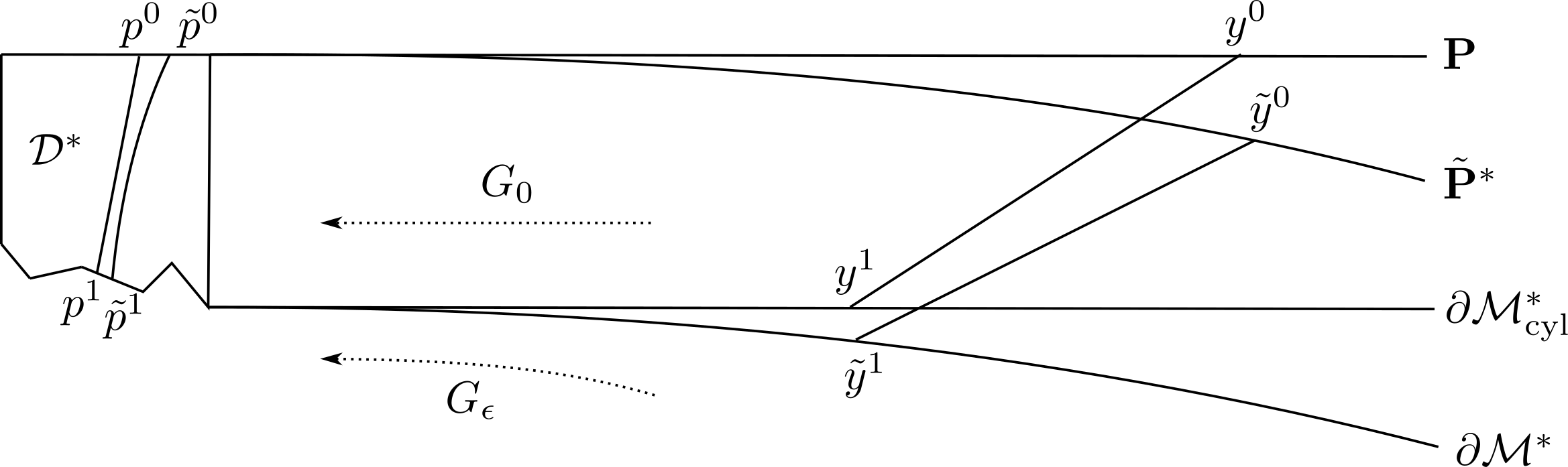}
    \caption{The line segment from $y^0$ to $y^1$ projects under $G_0$ to the chord of the cylindrical base $\mathcal{D}^*$ from $p^0$ to $p^1$. The line segment from $\widetilde{y}^0$ to $\widetilde{y}^1$ projects under $G_\epsilon$ to a (slightly) curved segment from $\widetilde{p}^0$ to $\widetilde{p}^1$ whose interior lies in $\Int \mathcal{D}^*$.}
    \label{projections_fig}
\end{figure}

Let $\widetilde{y}^0 \in (\partial \mathcal{M}^* \cap \widehat{\mathcal{Z}}^*) \cup \widetilde{\mathbf{P}}^*$, and let $\widetilde{w}^0 \in \mathbb{S}^2$. We say that the pair $(\widetilde{y}^0,\widetilde{w}^0)$ is \textit{licit} if (i) the billiard trajectory starting from initial state $(\widetilde{y}^0,\widetilde{w}^0)$ immediately enters $\Int \mathcal{M}^* \cap \mathcal{O}_-$ and first returns to $ (\partial \mathcal{M}^* \cap \widehat{\mathcal{Z}}^*) \cup \widetilde{\mathbf{P}}^*$ at a point $\widetilde{y}^1$; and (ii) if $\widetilde{y}^i \in \partial \mathcal{M}^*$ then there is a well-defined inward-pointing unit normal vector at $\widetilde{y}^i$, $i = 0,1$. If $\widetilde{y}^1 \in \partial \mathcal{M}^*$, let $\widetilde{w}^1$ denote the velocity of the point particle after reflecting at $y^1$. Otherwise (if $\widetilde{y}^1 \in \widetilde{\mathbf{P}}^*$) let $\widetilde{w}^1 = \widetilde{w}^0$. 

For $i = 0,1$, we use the following notation:
\begin{itemize}
    \item $\alpha^i$ denotes the angular coordinate of $y^i$, and $\widetilde{\alpha}^i$ denotes the angular coordinate of $\widetilde{y}^i$.
    \item $\omega^i$ denotes the angular coordinate of $w^i$, and $\widetilde{\omega}^i$ denotes the angular coordinate of $\widetilde{w}^i$.
    \item $(p^i,u^i) = \dd G_0(y^i,w^i)$ and $(\widetilde{p}^i,\widetilde{u}^i) = \dd G_\epsilon(\widetilde{y}^i,\widetilde{w}^i)$.
\end{itemize}
The quantities above whose definition depends on $(\widetilde{y}^0,\widetilde{w}^0)$ being licit are: $\widetilde{y}^1$, $\widetilde{w}^1$, $\widetilde{\alpha}^1$, $\widetilde{\omega}^1$, $\widetilde{p}^1$, and $\widetilde{u}^1$.

The projection $G_0$ maps $\mathcal{N}^* = \overline{\mathbb{R}_- \cap \mathcal{M}_{\cyl}}$ onto $\mathcal{D}^*$, mapping the boundary onto the boundary and the interior onto the interior. Consequently, $G_\epsilon$ maps $H_\epsilon(\overline{\mathbb{R}_- \cap \mathcal{M}_{\cyl}}) = \overline{\mathcal{M}^* \cap\mathcal{O}_-^*}$ onto $\mathcal{D}^*$, mapping the boundary onto the boundary and the interior onto the interior. The line segment from $p^0$ to $p^1$ is therefore a chord of $\mathcal{D}^*$. Similarly, the image under $G_\epsilon$ of the line segment from $y^0$ to $y^1$ is a closed (not necessarily straight) curve segment whose interior lies in $\Int\mathcal{D}^*$. 

We also use the following notation:
\begin{itemize}
    \item $L = ||p^1 - p^0||$.
    \item $\theta_0$ is the angle between $u^0$ and the inward pointing normal vector at $p^0 \in \partial \mathcal{D}^*$, and $\theta_1$ is the angle between $-u^0$ and the inward pointing normal at $p^1 \in \partial \mathcal{D}^*$.
    \item $\nu$ is the angle between $w$ and the line spanned by $\chi$.
    \item $\widehat{u}^0 = u^0/||u^0||$.
    \item $\xi$ is the unique unit vector such that $(\xi,\widehat{u}^0)$ is a positively oriented orthonormal basis.
    \item $r_0 = r_{\mathcal{D}^*}(p,\widehat{u}^0)$ is the radius of transversality associated with the chord from $p^0$ to $p^1$ in $\mathcal{D}^*$.
\end{itemize}
All of these quantities depend only on the segment of the billiard trajectory in the cylindrical configuration space. See Figure \ref{projection2_fig}.

\begin{figure}
    \centering
    \includegraphics[width = 0.9\linewidth]{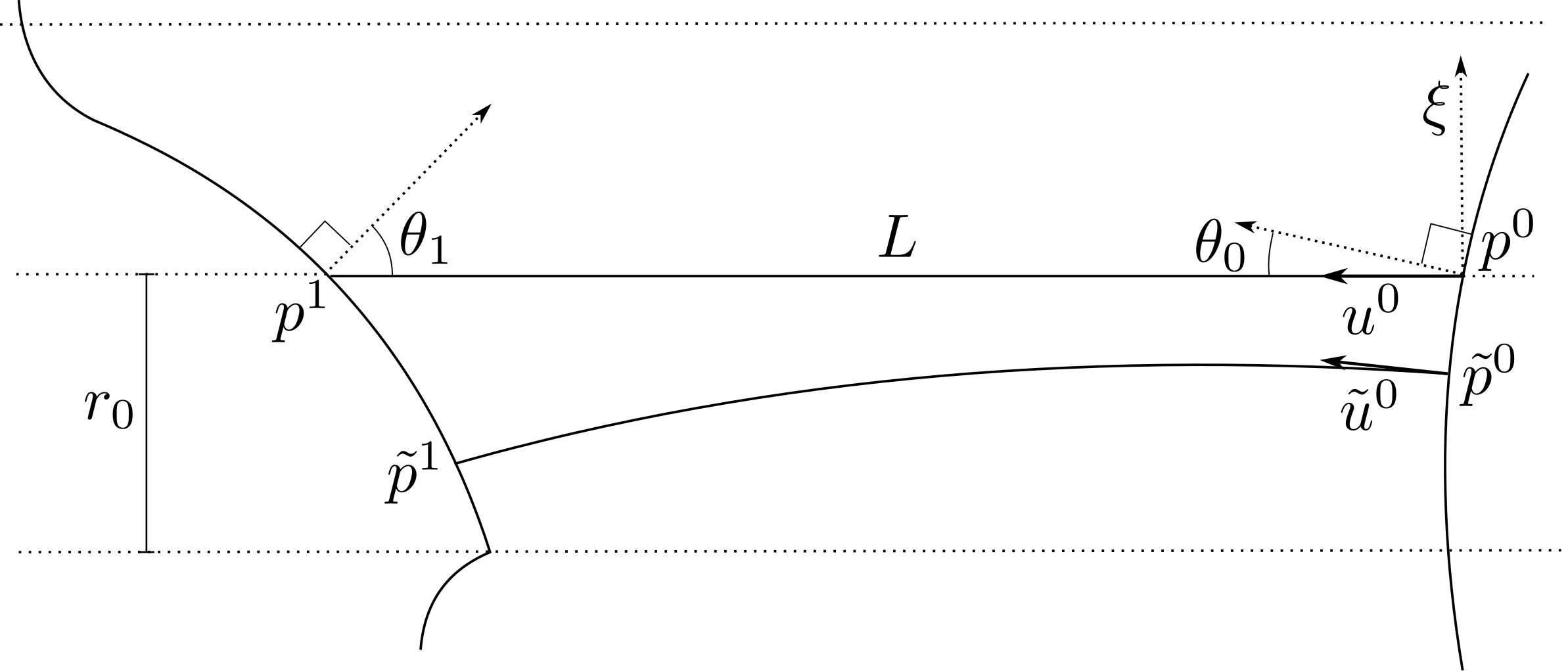}
    \caption{Projected segments of the billiard trajectories.}
    \label{projection2_fig}
\end{figure}

By the definition of the radius of transversality and the implicit function theorem, there exist $C^2$ functions $f_0, f_1 : (-r_0, r_0) \to \mathbb{R}$ such that for all $h \in (-r_0,r_0)$, $p^0 + h\xi + f_0(h)\widehat{u}^0$ lies in the $C^2$ curve segment containing $p^0$ and  $p^0 + h\xi + f_1(h)\widehat{u}^0$ lies in the $C^2$ curve segment containing $p^1$. Moreover, the strip 
\begin{equation}
    \{h \xi + h' \widehat{u}^0 : -r_0 < h < r_0, f_0(h) < h' < f_1(h)\} \subset \mathbb{R}^2
\end{equation}
does not intersect the wall $W(\Sigma, 1)$.

Recall that the maximum curvature $\kappa_{\max}$ of the $C^2$ curve segments constituting $\partial W(\Sigma, 1)$ is finite (see Remark \ref{rem_finkappa}). Let $L_{\max}$ denote the supremum of the lengths of all chords of $\partial W(\Sigma, 1)$. Recall that by construction, $W(\Sigma,1)$ satisfies conditions A5a and A5b from \S\ref{sssec_uphalf}. These conditions imply that $L_{\max}$ is finite. Let \begin{equation} \label{eq5.101}
    \overline{\kappa} = \max\{\kappa_{\max}, 1\}, \quad\quad \overline{L} = \max\{L_{\max}, 1\}.
\end{equation}

In our formulas below, we would like to consider the inverse of the function $\rho$, but this is not possible since $\rho$ is not strictly increasing (by assumption, it is increasing but its values are always some fraction of $2\pi$). Therefore, we instead consider
\begin{equation} \label{eq5.102}
    \widecheck{\rho}^{-1}(s) := \sup\{t \in (0,\infty) : \rho(t) < s\}, \quad\quad s \in (0,\infty).
\end{equation}
It follows from the assumptions on the growth of $\rho(\epsilon)$, stated in \ref{sssec_rigid} that $\frac{\widecheck{\rho}^{-1}(s)}{s^2} \to 0$ as $s \to 0$.

\begin{lemma} \label{lem_control}
Let $0 < s_2, s_3 \leq 1$ and \begin{equation} \label{eq3.74}
    0 < s_1 \leq \min\left\{r_0, \frac{5\cos\theta_1}{2\overline{\kappa}}, \frac{5s_2\sin\nu}{12}, \frac{s_3}{24\sqrt{10}\overline{\kappa}}\right\},
\end{equation}
and assume that 
\begin{equation} \label{eq3.75}
\epsilon \leq \min\left\{ \widecheck{\rho}^{-1}\left(\frac{s_1\cos\theta_1\sin\nu}{10(L+1)}\right), \frac{s_1\cos\theta_1\sin^2\nu}{10(L+1)^2}, \frac{5s_2\sin\nu}{12}, \widecheck{\rho}^{-1}\left(\frac{s_3}{12\sqrt{2}}\right), \right\}, 
\end{equation}
and 
\begin{equation} \label{eq3.76'}
|\alpha^0| < \epsilon^{-1}\left(\frac{\rho}{2} - \delta_0\right) - \left(\frac{L}{\sin\nu} + 1 \right).
\end{equation}
Assume in addition that 
\begin{equation} \label{eq3.76}
\begin{split}
& ||\widetilde{p}^0 - p^0|| \leq \frac{s_1\cos\theta_1}{20}, \\
& |\widetilde{\alpha}^0 - \alpha^0| \leq \frac{s_2}{3}, \quad\quad \text{ and } \\
& ||\widetilde{w}^0 - w^0|| \leq \min\left\{ \frac{s_1\cos\theta_1\sin\nu}{20(1+mJ^{-1})^{1/2}(2+(mJ)^{-1/2})(L+1)}, \frac{s_2\sin\nu}{3L},\frac{s_3}{12} \right\}.
\end{split}
\end{equation}
Then the pair $(\widetilde{y}^0,\widetilde{w}^0)$ is licit, in the sense defined above. Moreover, the point $\widetilde{p}^1$ lies on the same curve segment of $\partial W(\Sigma,1)$ as $p^1$ (with respect to the given decomposition), and 
\begin{equation}
||\widetilde{p}^1 - p^1|| \leq s_1, \quad |\widetilde{\alpha}^1 - \alpha^1| \leq s_2, \quad\quad ||\widetilde{w}^1 - w^1|| \leq s_3.
\end{equation}
\end{lemma}

A weaker but more convenient form of the lemma is as follows:

\begin{lemma} \label{lem_control2}
Assume that 
\begin{equation} \label{eq3.79'}
    s \leq 79\min\left\{ r_0, (\overline{\kappa})^{-1}\cos\theta_1\right\},
\end{equation}
and
\begin{equation} \label{eq3.80}
    \epsilon \leq \min\left\{ \widecheck{\rho}^{-1}\left( \frac{s\cos\theta_1\sin^2\nu}{800\overline{\kappa}(\overline{L} + 1)} \right), \frac{s\cos\theta_1\sin^3\nu}{800\overline{\kappa}(\overline{L} + 1)^2} \right\},
\end{equation}
and 
\begin{equation} \label{eq3.81}
    |\alpha^0| \leq \epsilon^{-1}\left(\frac{\rho}{2} - \delta_0 \right) - \left( \frac{\overline{L}}{\sin\nu} + 1 \right).
\end{equation}
Assume in addition that 
\begin{equation} \label{eq3.82'}
    ||\widetilde{p}^0 - p^0|| + |\widetilde{\alpha}^0 - \alpha^0| + ||\widetilde{w}^0 - w^0|| \leq \frac{s\cos\theta_1\sin^2\nu}{1600(1+mJ^{-1})^{1/2}(2+(mJ)^{-1/2})\overline{\kappa}(\overline{L}+1)}.
\end{equation} 
Then the pair $(\widetilde{y}^0,\widetilde{w}^0)$ is licit. Moreover, $\widetilde{p}^1$ lies on the same curve segment of $\partial W(\Sigma,1)$ as $p^1$, and
\begin{equation} \label{eq3.83}
    ||\widetilde{p}^1 - p^1|| + |\widetilde{\alpha}^1 - \alpha^1| + ||\widetilde{w}^1 - w^1|| \leq s.
\end{equation}
\end{lemma}

We now prove the two lemmas.

\begin{proof}[Proof of Lemma \ref{lem_control}]
\textit{Step 1.} Recall that if $I'$ denotes the curve segment containing $p^1$, and $I_1$ denotes the connected component of the intersection of $I'$ with the strip $\{x \in \mathbb{R}^2 : |\langle x - p^0, \xi \rangle| < r_0\}$, then the transversality radius $I_1$ may be represented as the graph of the function $f_1 : (-r_0, r_0) \to \mathbb{R}$, in the sense that 
\begin{equation}
    I_1 = \{h\xi + f_1(h)\widehat{u}^0 : h \in (-r_0, r_0)\}.
\end{equation}

The proof of the lemma depends on the following two claims.

\begin{claim}{5.7.1} \label{claim5.7.1}
Let $p : [a,b] \to \mathbb{R}^2$ be a unit speed parametrization of the curve segment $I_1$ containing $p^1$, with $a < 0 < b$ and $p(0) = p^1$. Let $0 < c < \min\{|a|, |b|\}$ be such that 
\begin{equation}
|\langle p(c) - p(0), \xi \rangle| < \min\left\{ r_0, \frac{(\overline{\kappa})^{-1}}{2} \right\}.
\end{equation}
Then the functional $\sigma \mapsto ||p(\sigma) - p(0)||$ is monotone increasing on the interval $[0,c]$ and monotone decreasing on the interval $[-c,0]$.
\end{claim}

\begin{claim}{5.7.2} \label{claim5.7.2}
Consider the projected billiard trajectories $\gamma(t) := G_0(y^0 + tw^0)$, and $\widetilde{\gamma}(t) := G_{\epsilon}(\widetilde{y}^0 + t\widetilde{w}^0)$, where $t \geq 0$. Let
\begin{equation}
    t_1 = \inf\{t > 0 : \gamma(t) \in \partial\mathcal{D}^*\}, \quad \widetilde{t}_1 = \inf\{t > 0 : \widetilde{\gamma}(t) \in \partial\mathcal{D}^*\}
\end{equation}
(defined equal to $\infty$ if the set over which we take the infimum is empty). Then $t_1 < \infty$. Moreover, if we assume that
\begin{equation}
    ||\widetilde{\gamma}(t) - \gamma(t)|| \leq s_1\cos\theta_1/5 \quad \text{ for all } 0 \leq t \leq t_1 + \frac{4s_1}{5||u^0||},
\end{equation}
then 
\begin{enumerate}[label = (\roman*)]
    \item $t_1 - \frac{4s_1}{5||u^0||} \leq \widetilde{t}_1 \leq t_1 + \frac{4s_1}{5||u^0||}$, and
    \item $||\widetilde{p}^1 - p^1|| \leq s_1$.
\end{enumerate}
\end{claim}

\begin{subproof}[Proof of Claim \ref{claim5.7.1}]
By symmetry it is enough to prove that the functional is increasing on $[0,c]$. Let $g_1(\sigma) = \langle p(\sigma) - p(0), \xi \rangle$ and $g_2(\sigma) = \langle p(\sigma) - p(0), \widehat{u}^0 \rangle$ for $\sigma \in [0,c]$. Since $|\langle p(c) - p(0), \xi \rangle| < r_0$, it follows that the part of the curve segment between $p(0)$ and $p(c)$ lies on the graph of the function $f_1$. Consequently, $g_1$ is nondecreasing on $[0,c]$. We may write 
\[
||p(\sigma) - p(0)||^2 = g_1(\sigma)^2 + g_2(\sigma)^2.
\]
If $g_2$ does not attain an extreme value in $(0,c)$, then the above functional is nondecreasing, and we are done. Suppose $g_2$ does attain an extreme value at some point $\sigma^* \in (0, c)$. Then $g_2'(\sigma^*) = \langle p'(\sigma^*), \widehat{u}^0 \rangle = 0$ and $\langle p'(\sigma^*), \xi \rangle = \pm 1$. Taylor expanding about $\sigma^*$ gives us
\begin{equation}
    p(\sigma) - p(\sigma^*) = p'(\sigma^*)(\sigma - \sigma^*) + E(\sigma), \quad ||E(\sigma)|| \leq \frac{\overline{\kappa}}{2}|\sigma - \sigma^*|^2.
\end{equation}
Since the curve segment between $\sigma = 0$ and $\sigma = c$ lies on the graph of the function $f_1$, for all $\sigma \in [0,c]$, $|\langle p(c) - p(0), \xi \rangle| \geq |\langle p(\sigma) - p(\sigma^*), \xi \rangle|$, and hence using the hypothesis
\begin{equation}
\begin{split}
    \frac{(\overline{\kappa})^{-1}}{2} & > |\langle p(c) - p(0), \xi \rangle| \\
    & \geq |\langle p(\sigma) - p(\sigma^*), \xi \rangle| \\
    & = |\langle p'(\sigma^*), \xi \rangle (\sigma - \sigma^*) - \langle E(\sigma), \xi \rangle| \\
    & \geq |\sigma - \sigma^*| - \frac{\overline{\kappa}}{2}|\sigma - \sigma^*|^2.
\end{split}
\end{equation}
The function $f(s) = s - \frac{\overline{\kappa}}{2}s^2$ achieves a maximum value of $\frac{(\overline{\kappa})^{-1}}{2}$ at $s = (\overline{\kappa})^{-1}$. It follows that $|\sigma - \sigma^*| < (\overline{\kappa})^{-1}$ for all $\sigma \in [0,c]$. Therefore, $c < 2(\overline{\kappa})^{-1}$.

Set $h(\sigma) = ||p(\sigma) - p(0)||^2$, $\sigma \in [0,c]$. Then 
\begin{equation}
\begin{split}
    h'(\sigma) & = 2\langle p(\sigma) - p(0), p'(\sigma) \rangle \\
    & \geq \sigma - \frac{\overline{\kappa}}{2}\sigma^2,
\end{split}
\end{equation}
where we Taylor expand about $\sigma$ to obtain the second line. Since $0 \leq \sigma \leq c < 2(\overline{\kappa})^{-1}$, we see that the last quantity above is nonnegative, and $h'(\sigma) \geq 0$ for $\sigma \in [0,c]$. This proves the claim.
\end{subproof}

\begin{subproof}[Proof of Claim \ref{claim5.7.2}]
Since $G_0$ is linear, we have \begin{equation} \label{eq3.79}
    \gamma(t) = p^0 + tu^0,
\end{equation}
and $t_1 = L/||u^0|| < \infty$. We also have for $0 \leq t \leq t_1 + \frac{4s_1}{5||u^0||}$,
\begin{equation}
\begin{split}
    |\langle \widetilde{\gamma}(t) - p^0, \xi \rangle| & = |\langle \widetilde{\gamma}(t) - \gamma(t), \xi \rangle + \langle \gamma(t) - p^0, \xi \rangle| \\
    & = |\langle \widetilde{\gamma}(t) - \gamma(t), \xi \rangle| \\
    & \leq ||\langle \widetilde{\gamma}(t) - \gamma(t)|| \leq \frac{s_1\cos\theta_1}{5} < r_0, 
\end{split}
\end{equation}
using (\ref{eq3.74}) in the last line. Thus $\widetilde{\gamma}(t)$ lies in the strip $\{x : |\langle x - p^0 \rangle| < r_0\}$, for all times $0 \leq t \leq t_1 + \frac{4s_1}{5||u^0||}$. 

As before, let $p : [a,b] \to \mathbb{R}^2$ be a unit speed parametrization of $I_1$ with $p(0) = p^1$. Since the segment $I_1$ coincides with the graph of the function $f_1 : (-r_0,r_0) \to \mathbb{R}$, as described above, and $s_1\cos\theta_1/5 < r_0$ by (\ref{eq3.74}), there is a maximal interval $[c_1, c_2]$ with $c_1 < 0 < c_2$ such that $p(\sigma)$ lies in the strip $\{x : |\langle x - p^0 \rangle| \leq s_1\cos\theta_1/5\}$ for all $\sigma \in [c_1,c_2]$. Reversing the direction of the parametrization if necessary, we may assume that $\sigma \mapsto \langle p(\sigma) - p(0), \xi \rangle$ is increasing on $[c_1,c_2]$. By Taylor expansion, we have 
\begin{equation}
    p(\sigma) = p(0) + p'(0)\sigma + E(\sigma), \quad ||E(\sigma)|| \leq \frac{\overline{\kappa}}{2}\sigma^2. 
\end{equation}
Thus for all $\sigma \in [c_1,c_2]$,
\begin{equation} \label{eq3.82}
\begin{split}
    \frac{s_1\cos\theta_1}{5} & \geq |\langle p(\sigma) - p(0), \xi \rangle| \\
    & = |(\cos\theta_1) \sigma + \langle E(\sigma), \xi \rangle |\\
    & \geq \cos\theta_1 |\sigma| - \frac{\overline{\kappa}}{2}|\sigma|^2.
\end{split}
\end{equation}
The function $f(s) = (\cos\theta_1)s - \frac{\overline{\kappa}}{2}s^2$ is increasing on the interval $[0,(\overline{\kappa})^{-1}\cos\theta_1]$ and achieves a maximum of $\frac{\cos^2\theta_1}{2\overline{\kappa}}$ at $s = (\overline{\kappa})^{-1}\cos\theta_1$. By (\ref{eq3.74}) $\frac{s_1\cos\theta_1}{5} < \frac{\cos^2\theta_1}{2\overline{\kappa}}$, and therefore $|\sigma| < (\overline{\kappa})^{-1}\cos\theta_1$. Thus by (\ref{eq3.82}), for all $\sigma \in [c_1,c_2]$
\begin{equation}
\begin{split}
    \frac{s_1\cos\theta_1}{5} & \geq \cos\theta_1|\sigma| - \frac{\overline{\kappa}}{2}|\sigma|[(\overline{\kappa})^{-1}\cos\theta_1] \geq \frac{\cos\theta_1|\sigma|}{2} \\
    & \Rightarrow \quad \frac{2s_1}{5} \geq |\sigma|.
\end{split}
\end{equation}
We conclude that 
\begin{equation}
    \max\{|c_1|, |c_2|\} \leq \frac{2s_1}{5}.
\end{equation} 
Let $t^- = t_1 - \frac{4s_1}{5||u^0||}$, and let $t^+ = t_1 + \frac{4s_1}{5||u^0||}$. Fix $t_0 < t^-$ and $\sigma \in [c_1,c_2]$. Let $\widetilde{x}$ be a point in the disk of radius $s_1\cos\theta_1/5$ centered at $\gamma(t_0)$. We have
\begin{equation} \label{eq3.85}
\begin{split}
    ||p(\sigma) - \widetilde{x}|| \geq ||\gamma(t_0) - p(0)|| - ||\gamma(t_0) - \widetilde{x}|| - ||p(\sigma) - p(0)||.
\end{split}
\end{equation}
Note that 
\begin{equation} \label{eq3.86}
    ||\gamma(t_0) - p(0)|| = ||\gamma(t_0) - \gamma(t_1)|| = |t_0 - t_1| \cdot ||u^0|| \geq |t^- - t_1|\cdot||u^0|| = \frac{4s_1}{5}.
\end{equation}
Also, by Taylor expansion, 
\begin{equation} \label{eq3.87}
    ||p(\sigma) - p(0)|| \leq |\sigma| + \frac{\overline{\kappa}}{2}|\sigma|^2 \leq \frac{3}{2}|\sigma| \leq \frac{3}{5}s_1,
\end{equation}
here using the fact proved above that $|\sigma| \leq (\overline{\kappa})^{-1}$ and $|\sigma| \leq 2s_1/5$. Substituting (\ref{eq3.86}) and (\ref{eq3.87}) as well as the hypothesized bound on $||\gamma(t) - \widetilde{\gamma}(t)||$ into (\ref{eq3.85}) gives us
\begin{equation}
    ||p(\sigma) - \widetilde{x}|| \geq \frac{4s_1}{5} - \frac{s_1\cos\theta_1}{5} - \frac{3s_1}{5} > 0.
\end{equation}
Since $\gamma(t_0)$ lies below the graph of $f_1$ in the center of the strip $\{x : |\langle x - p^0, \xi \rangle| \leq s_1\cos\theta_1/5\}$, by connectedness this proves that the entire disk of radius $s_1\cos\theta_1/5$ centered at $\gamma(t_0)$ lies below the graph of $f_1$ in the strip $\{x : |\langle x - p^0, \xi \rangle| \leq s_1\cos\theta_1/5\}$. In particular, by the hypothesized bound on $||\gamma(t) - \widetilde{\gamma}(t)||$, we see that $\gamma(t_0)$ lies below the graph of $f_1$ in the strip $\{x : |\langle x - p^0, \xi \rangle| \leq s_1\cos\theta_1/5\}$. This proves that $\widetilde{t}_1 > t^-$.

A symmetric argument shows that the disk of radius $s_1\cos\theta_1/5$ and centered at $\gamma(t^+)$ lies above the graph of $f_1$ in the strip $\{x : |\langle x - p^0, \xi \rangle| \leq s_1\cos\theta_1/5\}$. Consequently, $\widetilde{\gamma}(t^+)$ lies above the graph of $f_1$ in the strip $\{x : |\langle x - p^0, \xi \rangle| \leq s_1\cos\theta_1/5\}$. By the intermediate value theorem, we conclude that $t^- < \widetilde{t}_1 < t^+$, and this proves part (i) of the claim.

To prove part (ii), we make the following estimate using part (i):
\begin{equation}
\begin{split}
    ||\widetilde{p}^1 - p^1|| & = ||\widetilde{\gamma}(\widetilde{t}_1) - \gamma(t_1)|| \\
    & \leq ||\widetilde{\gamma}(\widetilde{t}_1) - \gamma(\widetilde{t}_1)|| + ||\gamma(\widetilde{t}_1) - \gamma(t_1)|| \\
    & \leq \frac{s_1\cos\theta_1}{5} + |\widetilde{t}_1 - t_1|\cdot ||u^0|| \\
    & \leq \frac{s_1\cos\theta_1}{5} + \frac{4s_1}{5} \leq s_1.
\end{split}
\end{equation}
The claim is proved.
\end{subproof}

\textit{Step 2.} We will now show that, under the hypotheses of the lemma, $||\widetilde{p}^1 - p^1|| \leq s_1$. Let $\gamma$ and $\widetilde{\gamma}$ be as in Claim \ref{claim5.7.2} above. By Taylor expansion,
\begin{equation} \label{eq3.90}
\widetilde{\gamma}(t) = \widetilde{p}^0 + t \widetilde{u}^0 + E(t), \quad\quad \text{ where } ||E(t)|| \leq \frac{t^2}{2}\sup_{u \in [0,t]}||\gamma''(u)||.
\end{equation}
To compute $\widetilde{\gamma}''(t)$, write $\widetilde{y}^0 = (\widetilde{x}_1^0, \widetilde{x}_2^0, \widetilde{\alpha}^0)$ and $\widetilde{w}^0 = (\widetilde{v}_1^0, \widetilde{v}_2^0, \widetilde{\omega}^0)$, and note that by (\ref{eq3.66})
\begin{equation}
\widetilde{\gamma}'(t) = \dd(G_\epsilon)_{\widetilde{y}^0 + t\widetilde{w}^0}(\widetilde{w}^0) = \begin{pmatrix}
\widetilde{v}_1^0 + \widetilde{\omega}^0\cos(\epsilon(\widetilde{\alpha}^0 + t\widetilde{\omega}^0)) \\
\widetilde{v}_2^0 + \widetilde{\omega}^0\sin(\epsilon(\widetilde{\alpha}^0 + t\widetilde{\omega}^0))
\end{pmatrix},
\end{equation}
and therefore
\begin{equation}
\widetilde{\gamma}''(t) = \begin{pmatrix}
- \epsilon(\widetilde{\omega}^0)^2\sin(\epsilon(\widetilde{\alpha}^0 + t\widetilde{\omega}^0)) \\
\epsilon(\widetilde{\omega}^0)^2\cos(\epsilon(\widetilde{\alpha}^0 + t\widetilde{\omega}^0))
\end{pmatrix}.
\end{equation}
Thus $||\widetilde{\gamma}''(t)|| = \epsilon(\widetilde{\omega}^0)^2$ and $||E(t)|| \leq \epsilon t^2 (\widetilde{\omega}^0)^2 / 2$. This, together with (\ref{eq3.79}) and (\ref{eq3.90}), gives us
\begin{equation} \label{eq3.93}
||\widetilde{\gamma}(t) - \gamma(t)|| \leq ||\widetilde{p}^0 - p^0|| + ||\widetilde{u}^0 - u^0|| t + \frac{\epsilon t^2(\widetilde{\omega}^0)^2}{2}.
\end{equation}
Using $(\ref{eq3.66})$ and Lemma \ref{dGlem}, we compute
\begin{equation} \label{eq3.99}
\begin{split}
||\widetilde{u}^0 - u^0|| & = ||(\dd G_0)_{y^0}(\widetilde{w}^0) - (\dd G_\epsilon)_{\widetilde{y}^0}(w^0)|| \\
& \leq ||(\dd G_0)_{y^0}(w^0) - (\dd G_\epsilon)_{\widetilde{y}^0}(w^0)|| + ||(\dd G_\epsilon)_{\widetilde{y}^0}(\widetilde{w}^0 - w^0)|| \\
& \leq |\omega^0|\sqrt{(1 - \cos(\epsilon\widetilde{\alpha}^0))^2 + \sin^2(\epsilon\widetilde{\alpha}^0)} \\
& + (1+mJ^{-1})^{1/2}(1 + ((mJ)^{-1/2} + 1)|\epsilon\widetilde{\alpha}^0|)||\widetilde{w}^0 - w^0|| \\
& \leq |\omega^0||\epsilon\widetilde{\alpha}^0| + (1+mJ^{-1})^{1/2}(1 + ((mJ)^{-1/2} + 1)|\epsilon\widetilde{\alpha}^0|)||\widetilde{w}^0 - w^0||. 
\end{split}
\end{equation}
Using the hypotheses, we have 
\begin{equation}
    |\epsilon\widetilde{\alpha}^0| \leq \epsilon(|\alpha^0| + |\widetilde{\alpha}^0 - \alpha^0|) \leq \frac{\rho(\epsilon)}{2} - \delta_0 - \epsilon(\frac{L}{\sin\nu}+1) + \frac{\epsilon s_2}{3} \leq \frac{\rho(\epsilon)}{2} \leq 1.
\end{equation}
and $|\omega^0| \leq ||w^0|| = 1$. Hence from (\ref{eq3.99}) we obtain
\begin{equation} \label{eq3.95}
||\widetilde{u}^0 - u^0|| \leq \frac{\rho(\epsilon)}{2} + (1+mJ^{-1})^{1/2}(2 + (mJ)^{-1/2})||\widetilde{w}^0 - w^0||.
\end{equation}
Using the same notation as in the proof of Claim \ref{claim5.7.2}, we let
\begin{equation}
t^+ = t_1 + \frac{4s_1}{5||u^0||} = \frac{L + 4s_1/5}{||u^0||}.
\end{equation}
By Lemma \ref{dGlem},
\begin{equation} \label{eq3.97}
||u^0|| = ||\dd G_0(w^0)|| \geq \sin\nu,
\end{equation}
and hence 
\begin{equation} \label{eq3.98}
t^+ \leq \frac{L + 4s_1/5}{\sin\nu} \leq \frac{L+1}{\sin\nu}
\end{equation}
(note that $s_1 \leq 5/4$ by the given bounds (\ref{eq3.74})). Substituting (\ref{eq3.95}) into (\ref{eq3.93}) gives us, for $0 \leq t \leq t^+$, 
\begin{equation}
\begin{split}
||\widetilde{\gamma}(t) - \gamma(t)|| & \leq ||\widetilde{p}^0 - p^0|| + \frac{\rho(\epsilon)}{2}t^+ \\
& \quad + (1+mJ^{-1})^{1/2}(2 + (mJ)^{-1/2})||\widetilde{w}^0 - w^0||t^+ + \frac{\epsilon(t^+)^2}{2}.
\end{split}
\end{equation}
By the bounds in the hypothesis (\ref{eq3.76}) and the bound (\ref{eq3.98}) on $t^+$, we obtain
\begin{equation}
\begin{split}
||\widetilde{\gamma}(t) - \gamma(t)|| & \leq \frac{s_1\cos\theta_1}{20} + \frac{1}{2}\left(\frac{s_1\cos\theta_1\sin\nu}{10(L+1)} \right)\left(\frac{L+1}{\sin\nu}\right) \\
& + (1+mJ^{-1})^{1/2}(2 + (mJ)^{-1/2}) \\
&\quad\times \left(\frac{s_1\cos\theta_1\sin\nu}{20(1+mJ^{-1})^{1/2}(2 + (mJ)^{-1/2})(L+1)}\right)\left(\frac{L+1}{\sin\nu}\right) \\
& + \frac{1}{2}\left(\frac{s_1\cos\theta_1\sin^2\nu}{10(L+1)^2}\right)\left(\frac{L+1}{\sin\nu}\right)^2 \\
& = \frac{s_1\cos\theta_1}{5}.
\end{split}
\end{equation}
By Claim \ref{claim5.7.2}, this implies $||\widetilde{p}^1 - p^1|| \leq s_1$.

\textit{Step 3.} To bound the difference in the angle coordinates $|\widetilde{\alpha}^1 - \alpha^1|$, we proceed as follows. Let $t_1, \widetilde{t}_1 > 0$ be the times defined above at which the trajectory hits the boundary in the cylindrical configuration space and the true configuration space respectively. As we argued above $||\widetilde{\gamma}(t) - \gamma(t)|| \leq s_1 \cos\theta_1 / 5$ for all $0 \leq t \leq (L + 4s_1 /5)/ ||u^0||$. By Claim \ref{claim5.7.2}, this implies that $t_1 - 4s_1/5||u^0|| \leq \widetilde{t}_1 \leq t_1 + 4s_1/5||u^0||$. We want to bound $|\widetilde{\alpha}(\widetilde{t}_1) - \alpha(t_1)|$. Since $\widetilde{\alpha}(t) = \widetilde{\alpha}^0 + t \widetilde{\omega}^0$ is linear, the function $t \mapsto \widetilde{\alpha}(t) - \alpha(t_1)$ is either increasing or decreasing. Consequently, by the bounds on $\widetilde{t}_1$ here described, 
\begin{equation}
|\widetilde{\alpha}(\widetilde{t}_1) - \alpha(t_1)| \leq \max\left\{\left|\widetilde{\alpha}\left(t_1 + \frac{4s_1}{5||u^0||}\right) - \alpha(t_1) \right|, \left|\widetilde{\alpha}\left(t_1 - \frac{4s_1}{5||u^0||}\right) - \alpha(t_1)\right|\right\}.
\end{equation}
Recall that $t_1 = L/||u^0||$. Substituting $\alpha(t) = \alpha^0 + t\omega^0$ and $\widetilde{\alpha}(t) = \widetilde{\alpha}^0 + t \widetilde{\omega^0}$ into the above and applying the triangle inequality, we obtain
\begin{equation}
\begin{split}
|\widetilde{\alpha}(\widetilde{t}_1) - \alpha(t_1)| & \leq  |\widetilde{\alpha}^0 - \alpha^0| + |\widetilde{\omega}^0 - \omega^0|\left(\frac{L}{||u^0||}\right) + |\widetilde{\omega}^0|\left(\frac{4s_1}{5||u^0||}\right) \\
& \leq |\widetilde{\alpha}^0 - \alpha^0| + ||\widetilde{w}^0 - w^0||\left(\frac{L}{\sin\nu}\right) + \left(\frac{4s_1}{5\sin\nu}\right),
\end{split}
\end{equation}
using (\ref{eq3.97}) in the last line. Substituting the hypothesized bounds (\ref{eq3.76}) for the quantities above, we obtain 
\begin{equation}
\begin{split}
|\widetilde{\alpha}^1 - \alpha^1| \leq \frac{s_2}{3} + \frac{s_2\sin\nu}{3L}\left(\frac{L}{\sin\nu}\right) +  \frac{4}{5\sin\nu}\left(\frac{5s_2\sin\nu}{12}\right) = s_2.
\end{split}
\end{equation}

\textit{Step 4.} Finally, we bound the difference in velocities $||\widetilde{w}^1 - w^1||$. Let $\widetilde{n}^1$ be the inward-pointing unit normal at $\widetilde{y}^1 \in \partial \mathcal{M}^*$ and let $n^1$ be the unit normal at $y^1 \in \partial \mathcal{M}^*_{\cyl}$. Then by specular reflection
\begin{equation} \label{eq3.104}
\widetilde{w}^1 = \widetilde{w}^0 - 2\langle \widetilde{w}^0, \widetilde{n}^1 \rangle \widetilde{n}^1 \quad \text{ and } \quad w^1 = w^0 - 2\langle w^0, n^1 \rangle n^1.
\end{equation}
Therefore, writing
\begin{equation}
\langle \widetilde{w}^0, \widetilde{n}^1\rangle\widetilde{n}^1 - \langle w^0, n^1 \rangle n^1 = \langle \widetilde{w}^0 - w^0, \widetilde{n}^1 \rangle n^1 + \langle w^0, \widetilde{n^1} - n^1 \rangle\widetilde{n}^1 + \langle w^0, n^1 \rangle (\widetilde{n}^1 - n^1), 
\end{equation}
we obtain by (\ref{eq3.104}) and Cauchy-Schwarz
\begin{equation} \label{eq3.106}
\begin{split}
||\widetilde{w}^1 - w^1|| \leq  3||\widetilde{w}^0 - w^0|| + 4||\widetilde{n}^1 - n^1||.
\end{split}
\end{equation}

To bound the difference in the unit normal vectors, first we obtain an expression for $\widetilde{n}^1$ in terms of $n^1$. Suppose once again that $p(\sigma)$, $\sigma \in (a,b)$ is a unit speed parametrization of the curve segment $I_1$ with $p(0) = p^1$. Note that by Step 2, $\widetilde{p}^1 \in I_1$, and we may assume without loss of generality that $p(c) = \widetilde{p}^1$ for some $c \in (0,b)$. 

We obtain a parametrization of a neighborhood of $\widetilde{y}^1$ in $\partial \mathcal{M}^*$ as follows. First, from the hypothesis and the conclusion of Step 3, we have
\begin{equation} \label{eq3.113}
\begin{split}
    |\widetilde{\alpha}^1| & \leq |\alpha^0| + |\alpha^1 - \alpha^0| + |\widetilde{\alpha}^1 - \alpha^1| \\
    & < \frac{\epsilon^{-1}\rho}{2} - \epsilon^{-1}\delta_0 - (\frac{L}{\sin\nu} + 1) + |t_1||\omega^0| + s_2 \\
    & \leq \frac{\epsilon^{-1}\rho}{2} - \epsilon^{-1}\delta_0 - (\frac{L}{\sin\nu} + 1) + \frac{L}{||u^0||} + s_2 \\
    & \leq \frac{\epsilon^{-1}\rho}{2} - \epsilon^{-1}\delta_0,
\end{split}
\end{equation}
using $t_1 = L/||u^0||$ and $|\omega^0| \leq 1$ to obtain the second to last line, and (\ref{eq3.97}) for the last line. Consequently, $\widetilde{y}^1 \in \Gamma_{\roll}^*$. Recall the map $f^*$ defined by (\ref{eq3.56}) which parametrizes the surface $\partial \mathcal{M}^*_{\roll}$, and observe that $f^*(\widetilde{p}^1, \widetilde{\alpha}^1) = \widetilde{y}^1$. Thus, the map 
\begin{equation}
    Q(\sigma,\alpha) = f(p(\sigma),\alpha), \quad (\sigma,\alpha) \in (a,b) \times (-\frac{\epsilon^{-1}\rho}{2} + \epsilon^{-1}\delta_0, \frac{\epsilon^{-1}\rho}{2} - \epsilon^{-1}\delta_0),
\end{equation}
parametrizes a neighborhood of $\widetilde{y}^1$ in $\partial \mathcal{M}^*$. Explicitly,
\begin{equation}
Q(\sigma,\alpha) = p(\sigma) + \begin{pmatrix} \epsilon^{-1}\sin(\epsilon\alpha) \\ 1 + \epsilon^{-1}(-1 + \cos(\epsilon\alpha)) \\ \alpha \end{pmatrix},
\end{equation}
Observe that $\widetilde{y}^1 = Q(c,\widetilde{\alpha}^1)$. The tangent space of $\partial \mathcal{M}^*$ at $\widetilde{y}^1$ is therefore spanned by the following unit vectors
\begin{equation}
\partial_\sigma Q(c,\widetilde{\alpha}^1) = p'(c), \quad (m+J)^{-1/2}\partial_\alpha Q(c,\widetilde{\alpha}^1) = (m+J)^{-1/2}\begin{pmatrix}
\cos(\epsilon\widetilde{\alpha}^1) \\
\sin(\epsilon\widetilde{\alpha}^1) \\
1
\end{pmatrix}  =: \widetilde{\chi}.
\end{equation} 
(These are linearly independent since the $\alpha$-component of $\widetilde{\chi}$ is nonzero.) Let 
\begin{equation} \label{eq3.109}
\widehat{\chi} = \frac{\widetilde{\chi} - \langle p'(c), \widetilde{\chi} \rangle p'(c)}{||\widetilde{\chi} - \langle p'(c), \widetilde{\chi} \rangle p'(c)||}.
\end{equation}
Then $(p'(c), \widehat{\chi})$ is an orthonormal basis for the tangent space to $\Gamma_{\roll}$ at $\widetilde{y}^1$. Consequently, the unit normal at $\widetilde{y}^1$ is given by the formula
\begin{equation}
\begin{split}
\widetilde{n}^1 & = \frac{n^1 - \langle n^1, p'(c)\rangle p'(c) - \langle n^1, \widehat{\chi} \rangle \widehat{\chi}}{||n^1 - \langle n^1, p'(c)\rangle p'(c) - \langle n^1, \widehat{\chi} \rangle \widehat{\chi}||} \\
& = \frac{n^1 - \langle n^1, p'(c)\rangle p'(c) - \langle n^1, \widehat{\chi} \rangle \widehat{\chi}}{\sqrt{1 - \langle n^1, p'(c)\rangle^2 - \langle n^1, \widehat{\chi} \rangle^2}}.
\end{split}
\end{equation}
Therefore, 
\begin{equation} \label{eq3.111}
\langle n^1, \widetilde{n}^1\rangle = \frac{1  - \langle n^1, p'(c)\rangle^2 - \langle n^1, \widehat{\chi} \rangle^2}{\sqrt{1 - \langle n^1, p'(c)\rangle^2 - \langle n^1, \widehat{\chi} \rangle^2}} = \sqrt{1 - \langle n^1, p'(c)\rangle^2 - \langle n^1, \widehat{\chi} \rangle^2}.
\end{equation}
It is even easier to specify a frame at $y^1 \in \partial \mathcal{M}^*_{\cyl}$. Because of the cylindrical structure of $\partial \mathcal{M}^*_{\cyl}$, both $p'(0)$ and $\chi = (m+J)^{-1/2}(-1,0,1)$ are tangent to $\partial \mathcal{M}^*_{\cyl}$ at $y^1$. Hence $\langle p'(0), n^1 \rangle = \langle \chi , n^1\rangle = 0$, and 
\begin{equation} \label{eq3.112}
|\langle n^1, p'(c) \rangle|  = |\langle n^1, p'(c) - p'(0)\rangle| \leq ||p'(c) - p'(0)||.
\end{equation}
To bound the quantity on the right, first note that 
\begin{equation}
    |\langle p(c) - p(0), \xi \rangle| \leq ||p(c) - p(0)|| = ||\widetilde{p}^1 - p^1|| \leq s_1 < \min\{r_0, \frac{1}{2}(\overline{\kappa})^{-1}\},
\end{equation}
by (\ref{eq3.74}). Thus, by Claim \ref{claim5.7.1}, $\sigma \mapsto ||p(\sigma) - p(0)||$ is monotone increasing on $[0,c]$. Consequently, for all $\sigma \in [0,c]$,
\begin{equation} \label{eq3.114}
    s_1 \geq ||p(\sigma) - p(0)|| \geq |\sigma| - \frac{\overline{\kappa}}{2}|\sigma|^2,
\end{equation}
using Taylor's theorem for the last estimate. Since the function $f(s) = s^2 - \frac{\overline{\kappa}}{2}s^2$ is monotone increasing on $[0, (\overline{\kappa})^{-1}]$ and achieves a maximum of $\frac{(\overline{\kappa})^{-1}}{2} > s_1$ at $s = (\overline{\kappa})^{-1}$, it follows that $c < (\overline{\kappa})^{-1}$. Thus by (\ref{eq3.114})
\begin{equation}
    s_1 \geq |c| - \frac{\overline{\kappa}}{2}|c|(\overline{\kappa})^{-1} = \frac{1}{2}|c|.
\end{equation}
It follows that 
\begin{equation}
    ||p'(c) - p(0)|| \leq \int_0^c ||p''(u)||\dd u \leq \overline{\kappa}|c| \leq 2\overline{\kappa}s_1.
\end{equation}
Substitution into (\ref{eq3.112}) yields
\begin{equation} \label{eq3.117}
|\langle n^1, p'(c) \rangle| \leq 2\overline{\kappa}s_1.
\end{equation}
In addition, we have 
\begin{equation} 
\begin{split}
|\langle n^1, \widetilde{\chi} \rangle| & = |\langle n^1, \widetilde{\chi} - \chi \rangle| \leq ||\widetilde{\chi} - \chi|| \\
& = \sqrt{(1 - \cos(\epsilon\widetilde{\alpha}^1))^2 + \sin^2(\epsilon\widetilde{\alpha}^1)} \\
& = \sqrt{2 - 2\cos(\epsilon\widetilde{\alpha}^1)} \leq |\epsilon\widetilde{\alpha}^1|.
\end{split}
\end{equation}
Therefore, by (\ref{eq3.109})
\begin{equation} \label{eq3.119}
\begin{split}
|\langle n^1, \widehat{\chi} \rangle| & = \frac{|\langle n^1, \widetilde{\chi}\rangle - \langle p'(c), \widetilde{\chi} \rangle \langle n^1, p'(c) \rangle|}{\sqrt{1 - \langle p'(c), \widetilde{\chi}\rangle^2}} \\
& \leq \frac{|\epsilon\widetilde{\alpha}^1| + 2\overline{\kappa}s_1}{\sqrt{1 - 4(\overline{\kappa})^2s_1^2}} \leq \sqrt{2}(|\epsilon\widetilde{\alpha}^1| + 2\overline{\kappa}s_1),
\end{split}
\end{equation}
where the last line follows by the assumption that $s_1 \leq 1/(2\sqrt{2}\overline{\kappa})$ by (\ref{eq3.74}). By (\ref{eq3.111}), (\ref{eq3.117}), and (\ref{eq3.119}), we conclude that 
\begin{equation}
\begin{split}
\langle \widetilde{n}^1, n^1 \rangle^2 & \geq 1 - 4(\overline{\kappa})^2s_1^2 - 2(|\epsilon\widetilde{\alpha}^1| + 2\overline{\kappa}s_1)^2 \\
& \geq 1 - 20(\overline{\kappa})^2s_1^2 - 4\epsilon^2|\widetilde{\alpha}^1|^2 \\
& \geq 1 - 20(\overline{\kappa})^2s_1^2 - \rho(\epsilon)^2,
\end{split}
\end{equation}
noting that $|\widetilde{\alpha}^1| \leq \epsilon^{-1}\rho/2$ by (\ref{eq3.113}). From this we obtain the estimate 
\begin{equation}
\begin{split}
||\widetilde{n}^1 - n^1|| & = \sqrt{2 - 2\langle \widetilde{n}^1, n^1 \rangle} \\
& \leq \sqrt{40(\overline{\kappa})^2s_1^2 + 2\rho(\epsilon)^2} \\
& \leq 2\sqrt{10}\overline{\kappa}s_1 + \sqrt{2}\rho(\epsilon).
\end{split}
\end{equation}
From (\ref{eq3.106}) we obtain the following bound on the difference of velocities
\begin{equation}
||\widetilde{w}^1 - w^1|| \leq 3||\widetilde{w}^0 - w^0|| + 8\sqrt{10}\overline{\kappa}s_1 + 4\sqrt{2}\rho(\epsilon).
\end{equation}
Substituting the bounds (\ref{eq3.74}), (\ref{eq3.75}), and (\ref{eq3.76}) for the quantities above, we obtain
\begin{equation}
||\widetilde{w}^1 - w^1|| \leq 3(\frac{s_3}{12}) + 8\sqrt{10}\overline{\kappa}(\frac{s_3}{24\sqrt{10}\overline{\kappa}}) + 4\sqrt{2}(\frac{s_3}{12\sqrt{2}}) = s_3.
\end{equation}
This completes the proof.
\end{proof}

\begin{proof}[Proof of Lemma \ref{lem_control2}]
This lemma is corollary of the previous one. Take 
\begin{equation} \label{eq3.138}
\begin{split}
    &s_1 = \left( 1 + \frac{12}{5\sin\nu} + 24\sqrt{10}\overline{\kappa} \right)^{-1}s, \\
    &s_2 = \frac{12s_1}{5\sin\nu}, \\
    &s_3 = 24\sqrt{10}\overline{\kappa}s_1.
\end{split}
\end{equation}
Since by the above and (\ref{eq3.79}) and the assumption $s \leq 1$, we have $s_2 \leq 1$ and $s_3 \leq 1$. Since $0 < \sin\nu \leq 1$ and $\overline{\kappa} \geq 1$, we have 
\begin{equation} \label{eq3.139}
    79 \leq 1 + \frac{12}{5\sin\nu} + 24\sqrt{10}\overline{\kappa} \leq \frac{80\overline{\kappa}}{\sin\nu}.
\end{equation}
Thus by (\ref{eq3.79'}),
\begin{equation}
    s_1 \leq 79^{-1}s \leq \min\{r_0,(\overline{\kappa})^{-1}\cos\theta_1\} \leq \min\left\{r_0,\frac{5\cos\theta_1}{2\overline{\kappa}}\right\}.
\end{equation}
Also, by definition of $s_2$ and $s_3$ above, we have
\begin{equation}
    s_1 = \frac{5s_2\sin\nu}{12} = \frac{s_3}{24\sqrt{10}\overline{\kappa}}
\end{equation}
Therefore the bound (\ref{eq3.74}) is satisfied.

To see that $(\ref{eq3.75})$ is satisfied, first observe that 
\begin{equation} \label{eq3.142}
\begin{split}
    \frac{s\cos\theta_1\sin^2\nu}{800\overline{\kappa}(\overline{L} + 1)} & = \left(1 + \frac{12}{5\sin\nu} + 24\sqrt{10}\overline{\kappa}\right) \frac{s_1\cos\theta_1\sin^2\nu}{800\overline{\kappa}(\overline{L} + 1)} \\
    & \leq \frac{s_1\cos\theta_1\sin\nu}{10(\overline{L}+1)} \leq \frac{s_1\cos\theta_1\sin\nu}{10(L+1)},
\end{split}
\end{equation}
where the first inequality above follows from (\ref{eq3.139}). Consequently, since $\widecheck{\rho}^{-1}$ is increasing, by (\ref{eq3.80}) we have 
\begin{equation} \label{eq3.143}
    \epsilon \leq \widecheck{\rho}^{-1}\left( \frac{s\cos\theta_1\sin^2\nu}{800\overline{\kappa}(\overline{L} + 1)} \right) \leq \widecheck{\rho}^{-1}\left( \frac{s_1\cos\theta_1\sin\nu}{10(L+1)} \right).
\end{equation}
By (\ref{eq3.80}) and the estimate (\ref{eq3.142}), we also have 
\begin{equation} \label{eq3.144}
    \epsilon \leq \frac{s\cos\theta_1\sin^3\nu}{800\overline{\kappa}(\overline{L} + 1)^2} \leq \frac{s_1\cos\theta_1\sin^2\nu}{10(\overline{L}+1)^2} \leq \frac{s_1\cos\theta_1\sin^2\nu}{10(L+1)^2}.
\end{equation}
It also follows from (\ref{eq3.144}) and the definition of $s_2$ that 
\begin{equation} \label{eq3.145}
    \epsilon \leq s_1 = \frac{5s_2\sin\nu}{12},
\end{equation}
and from (\ref{eq3.143}) and the definition of $s_3$ that 
\begin{equation} \label{eq3.146}
    \epsilon \leq \widecheck{\rho}^{-1}(s_1) \leq \widecheck{\rho}^{-1}\left(2\sqrt{5}s_1\right) = \widecheck{\rho}^{-1}\left(\frac{s_3}{12\sqrt{2}}\right).
\end{equation}
Combining (\ref{eq3.143}), (\ref{eq3.144}), (\ref{eq3.145}), and (\ref{eq3.146}), we obtain the desired bound (\ref{eq3.75}) on $\epsilon$.

To see that (\ref{eq3.76'}) is satisfied, observe that by (\ref{eq3.81}),
\begin{equation}
    |\alpha^0| \leq \epsilon^{-1}\left(\frac{\rho}{2} - \delta_0\right) - \left(\frac{\overline{L}}{\sin\nu} + 1\right) \leq \epsilon^{-1}\left(\frac{\rho}{2} - \delta_0\right) - \left(\frac{L}{\sin\nu} + 1\right).
\end{equation}

From (\ref{eq3.82'}) we get the following bounds on $||\widetilde{p}^0 - p^0||$, $|\widetilde{\alpha}^0 - \alpha^0|$, and $||\widetilde{w}^0 - w^0||$. First, by definition of $s_1$ and (\ref{eq3.82'}), we have 
\begin{equation} \label{eq3.149}
\begin{split}
   ||\widetilde{p}^0 - p^0|| & + |\widetilde{\alpha}^0 - \alpha^0| + ||\widetilde{w}^0 - w^0|| \\
   & \leq \left(1 + \frac{12}{5\sin\nu} + 24\sqrt{10}\overline{\kappa}\right) \\
   & \quad \times\frac{s_1\cos\theta_1\sin^2\nu}{1600(1 + mJ^{-1})^{1/2}(2 + (mJ)^{-1/2})\overline{\kappa}(\overline{L} + 1)} \\
   & \leq \left(\frac{80\overline{\kappa}}{\sin\nu}\right)\frac{s_1\cos\theta_1\sin^2\nu}{1600(1 + mJ^{-1})^{1/2}(2 + (mJ)^{-1/2})\overline{\kappa}(\overline{L} + 1)} \\
   & = \frac{s_1\cos\theta_1\sin\nu}{20(1 + mJ^{-1})^{1/2}(2 + (mJ)^{-1/2})(\overline{L} + 1)}.
\end{split}
\end{equation}
It follows from the above that 
\begin{equation}
    ||\widetilde{p}^0 - p^0|| \leq \frac{s_1\cos\theta_1}{20}.
\end{equation}
Since $s_1 = \frac{5s_2\sin\nu}{12}$, we also obtain from (\ref{eq3.149})
\begin{equation}
\begin{split}
    |\widetilde{\alpha}^0 - \alpha^0| & \leq \frac{s_2\cos\theta_1\sin^2\nu}{48(1 + mJ^{-1})^{1/2}(2 + (mJ)^{-1/2})(\overline{L} + 1)} \leq \frac{s_2}{3}.
\end{split}
\end{equation}
In addition, (\ref{eq3.149}) implies 
\begin{equation}
    ||\widetilde{w}^0 - w^0|| \leq \frac{s_1\cos\theta_1\sin\nu}{20(1 + mJ^{-1})^{1/2}(2 + (mJ)^{-1/2})(L + 1)}.
\end{equation}
Substituting $s_1 = \frac{5s_2\sin\nu}{12}$ into (\ref{eq3.149}) gives us 
\begin{equation}
    ||\widetilde{w}^0 - w^0|| \leq \frac{s_2\cos\theta_1\sin^2\nu}{48(1 + mJ^{-1})^{1/2}(2 + (mJ)^{-1/2})(\overline{L} + 1)} \leq \frac{s_2\sin\nu}{3L},
\end{equation}
and substituting $s_1 = \frac{s_3}{24\sqrt{10}\overline{\kappa}}$ gives us 
\begin{equation}
    ||\widetilde{w}^0 - w^0|| \leq \frac{s_3\cos\theta_1\sin\nu}{48\sqrt{10}\overline{\kappa}(1 + mJ^{-1})^{1/2}(2 + (mJ)^{-1/2})(\overline{L} + 1)} \leq \frac{s_3}{12}.
\end{equation}
We conclude from the last five displays that the bounds (\ref{eq3.76}) are satisfied. Therefore, the hypotheses of Lemma \ref{lem_control} hold, and we conclude that 
\begin{equation}
\begin{split}
    ||\widetilde{p}^1 - p^1|| + |\widetilde{\alpha}^1 - \alpha^1| + ||\widetilde{w}^1 - w^1|| & \leq s_1 + s_2 + s_3 \\
    & = \left(1 + \frac{12}{5\sin\nu} + 24\sqrt{10}\overline{\kappa}\right)s_1 = s,
\end{split}
\end{equation}
here using the definitions (\ref{eq3.138}).
\end{proof}

\subsubsection{Definition of $\Omega$ and proof of Lemma \ref{lem_mod}} \label{sssec_omegadef}

Recall the sets $\mathcal{N}^*$ and $A_1^*, A_2^*, A_3^*$ defined by (\ref{eq5.92}). The set $A_2^*$ has Lebesgue measure zero. (See the proof of Proposition \ref{prop_detcolcyl}.)

Recall that $K^{\Sigma,\epsilon}_{\cyl}$ is a well-defined $C^1$ involutive diffeomorphism on a $\Lambda^2$-full measure open subset $\mathcal{F}_{\cyl} \subset \mathbf{P} \times \mathbb{S}^2_+$. For any $(y,w) \in \mathcal{F}_{\cyl}$, the billiard trajectory in $\mathcal{M}_{\cyl}$ starting from $(y,-w)$ is well-defined for all time and hits $\partial \mathcal{M}_{\cyl}$ only finitely many times before returning to $\mathbf{P}$. (See Propositions \ref{prop_detcolcyl} and \ref{prop_reflproperties}.)

Let $\mathcal{F}_{\cyl}^* = \sigma_{\epsilon^{-1}}(\mathcal{F}_{\cyl}) = \{(y,w) \in \mathbf{P} \times \mathbb{S}^2_+ : (\epsilon y, w) \in \mathcal{F}_{\cyl}\}$. For any $(y,w) \in \mathcal{F}_{\cyl}^*$, the billiard trajectory in $\mathcal{M}_{\cyl}^*$ starting from initial state $(y,-w)$ is defined for all time and returns to $\mathbf{P}$ after only finitely many collisions with $\partial \mathcal{M}_{\cyl}^*$.

Fix $(y,w) \in (A_1^* \times \mathbb{S}^2) \cap \mathcal{F}_{\cyl}^*$, and let $(y^0,w^0) = (y,-w)$. Let $N$ denote the number of times which the point particle hits $\partial \mathcal{M}_{\cyl}^*$ before returning to the plane $\mathbf{P}$. For $1 \leq j \leq N$, let $y^j$ denote the point in $\partial \mathcal{M}_{\cyl}^*$ where the billiard trajectory starting from $(y^{j-1},w^{j-1})$ first returns to the boundary, and let $w^j$ denote the velocity of the point particle immediately after reflecting from the boundary at $y^j$. Let $(y^{N+1},w^{N+1})$ denote the state of the point particle upon returning to the plane $\mathbf{P}$ (thus $w^{N+1} = w^N$). 

We will also use the following notation: For $0 \leq j \leq N+1$
\begin{itemize}
    \item $(p^j,u^j) = \dd G_0(y^j,w^j)$.
    \item $k^j$ is the unit normal vector to $\partial\widehat{B}^*$ at $p^j$ if $1 \leq j \leq N$. We also let $k^0 = k^{N+1} = -e_2$. (Thus each $k^j$ is the inward-pointing unit normal at $p^j \in \partial \mathcal{D}^*$, where $\mathcal{D}^*$ is the region defined by (\ref{eq5.93})).
    \item $\theta_j$ is the angle between $k^j$ and $-u^{j-1}$ (i.e. $\cos \theta_j = \langle k^j, -u^{j-1} \rangle$) if $1 \leq j \leq N+1$. We let $\theta_0$ be the angle between $w^0$ and $k^0 = -e_2$.
    \item $r_j = r_{\mathcal{D}^*}(p_j, \widehat{u}_j)$ is the radius of transversality associated with the chord from $p_j$ to $p_{j+1}$ if $0 \leq j \leq N$.
    \item $\nu$ is the angle between $w^0$ and the line spanned by $\chi$. 
\end{itemize}

Recall that in the cylindrical configuration space the projection of the point particle velocity onto $\chi$ is preserved. It follows that $\nu$ is the angle between $w^j$ and the line spanned by $\chi$ for $0 \leq j \leq N+1$.

In the notation above, we let $F = F(\epsilon)$ denote the set of all states states $(y,w) \in (A_1 \times \mathbb{S}^2) \cap \mathcal{F}_{\cyl}^*$ such that 
\begin{enumerate}[label = F\arabic*.]
    \item $N \leq \frac{\log(1/\rho)}{4\log\log(1/\rho)}$,
    \item $r_j \geq \frac{1}{\log(1/\rho)}$ for $0 \leq j \leq N$,
    \item $\cos\theta_j \geq \frac{\overline{\kappa}}{\log(1/\rho)}$ for $0 \leq j \leq N+1$, and
    \item $\sin\nu \geq \frac{1}{\log(1/\rho)}$.
\end{enumerate}
We also define $\Xi = \Xi(\epsilon)$ to be the set of all states $(y,w) \in (A_1 \times \mathbb{S}^2) \cap \mathcal{F}_{\cyl}^*$ such that if $\alpha$ is the angular coordinate of $y$, then the following holds:
\begin{equation}
\begin{split}
    \text{for all } j \in \mathbb{Z}, &\text{ if } \alpha \in \Big(\epsilon^{-1}\Big(-\frac{\rho}{2} + j\rho\Big), \epsilon^{-1}\Big(\frac{\rho}{2} + j\rho\Big)\Big], \\
    &\text{ then } |\alpha - j\epsilon^{-1}\rho| \leq \epsilon^{-1}\left(\frac{\rho}{2} - \delta_0 - \frac{12(\overline{L} + 1)}{5\log(1/\rho)^2} \right) - \frac{\log(1/\rho)^2 \overline{L}}{4\log\log(1/\rho)}.
\end{split}
\end{equation}

Recall the parallelogram $R_\epsilon$, defined by (\ref{eq5.26}). Let 
\begin{equation}
    R_{\epsilon}^* = \epsilon^{-1}R_\epsilon = \{(x_1,\alpha) \in \mathbf{P} : 0 \leq x_1 + \alpha \leq 1 \text{ and } -\epsilon^{-1}\rho(\epsilon)/2 \leq \alpha \leq \epsilon^{-1}\rho(\epsilon)/2\}.
\end{equation}
Also define translation maps 
\begin{equation}
    \tau_{jk}^*(y) = \epsilon^{-1}\tau_{jk}(\epsilon y) = y + j e_1 + k \epsilon^{-1}\rho e_3, \quad\quad y \in \mathbf{P},
\end{equation}
\begin{equation}
    \overline{\tau}_{jk}^*(y,w) = (\tau_{jk}^*(y),w), \quad\quad (y,w) \in \mathbf{P} \times \mathbb{S}^2_+.
\end{equation}
The translates $\tau_{jk}^* R_\epsilon^*$ tessellate the plane $\mathbf{P}$, in the sense that $\mathbf{P} = \bigcup_{(j,k) \in \mathbb{Z}^2} \tau_{jk}^*R_{\epsilon}^*$, and $\tau_{jk}^*R_{\epsilon}^* \cap \tau_{j'k'}^*R_{\epsilon}^*$ has Lebesgue measure zero whenever $(j,k) \neq (j',k')$.

We let $\widehat{R}_{\epsilon}^*$ denote the parallelogram $R_\epsilon^*$ minus its upper and right boundary segments, i.e.
\begin{equation}
    \widehat{R}_{\epsilon}^* = \{(x_1,\alpha) \in \mathbf{P} : 0 \leq x_1 + \alpha < 1 \text{ and } -\epsilon^{-1}\rho(\epsilon) \leq \alpha < \epsilon^{-1}\rho(\epsilon)/2\}.
\end{equation}
Then the translates $\tau_{jk}^* \widehat{R}_\epsilon^*$ form a collection of disjoint sets which tessellate the plane $\mathbf{P}$.

We define $\Omega^*$ to be the subset of $\mathbf{P} \times \mathbb{S}^2_+$ which is invariant under the translations $\overline{\tau}_{jk}^*$, $(j,k) \in \mathbb{Z}^2$ and such that 
\begin{equation}
    \Omega^* \cap (\widehat{R}_{\epsilon}^* \times \mathbb{S}^2_+) = \left(F \cup \Xi \cup (A_3^* \times \mathbb{S}^2_+)  \right) \cap (\widehat{R}_{\epsilon}^* \times \mathbb{S}^2_+).
\end{equation}
We then define $\Omega \subset \mathbf{P} \times \mathbb{S}^2$ by 
\begin{equation}
    \Omega = \sigma_{\epsilon}(\Omega^*) = \{(y,w) \in \mathbf{P} \times \mathbb{S}^2_+ : (\epsilon^{-1}y, w) \in \Omega^*\}.
\end{equation}

\begin{proof}[Proof of Lemma \ref{lem_mod}]
(i) follows immediately from the definition of $\Omega$.

(ii) Changing to zoomed coordinates, it is enough to show that $\epsilon \rho(\epsilon)^{-1}\Lambda^2(\mathbf{P} \times \mathbb{S}^2 \smallsetminus \Omega^*) \to 0$ as $\epsilon \to 0$. Note that $A_1^* \cup A_3^*$ is a full measure subset of $\mathbf{P}$. Consequently, it is enough to show that $\epsilon \rho(\epsilon)^{-1}\Lambda^2((R_\epsilon^* \cap A_1^*)  \times \mathbb{S}^2 \smallsetminus F) \to 0$ and $\epsilon\rho(\epsilon)^{-1}\Lambda^2((R_\epsilon^* \cap A_1^*)  \times \mathbb{S}^2 \smallsetminus \Xi) \to 0$ as $\epsilon \to 0$.

First we deal with $F$. We let $k = \lceil \epsilon^{-1}\rho(\epsilon)/2 \rceil$, and we chop $R_{\epsilon}$ into $2k$ parts 
\begin{equation}
    R_{\epsilon}^{*j} = R_{\epsilon}^* \cap \left\{(x_1,\alpha) : \frac{(j-1)\epsilon^{-1}\rho}{2k} \leq \alpha \leq \frac{j\epsilon^{-1}\rho}{2k} \right\}, \quad\quad -k + 1 \leq j \leq k. 
\end{equation}
By definition of the parallelogram $R_\epsilon^*$, we see that 
\begin{equation}
    R_\epsilon^{*j} = R_{\epsilon}^{*0} + \frac{(1 + mJ^{-1})^{-1/2}\epsilon^{-1}\rho j}{k}\chi.
\end{equation}
By observing that the sets $F$ and $A_1^*$ are invariant under any translation in the direction $\chi$, it follows that 
\begin{equation}
    \Lambda^2((R_\epsilon^{*j} \cap A_1^*) \times \mathbb{S}^2_+ \smallsetminus F) = \Lambda^2((R_\epsilon^{*0} \cap A_1^*)  \times \mathbb{S}^2_+ \smallsetminus F), \quad \text{ for } -k+1 \leq j \leq k.
\end{equation}
Therefore, 
\begin{equation}
    \Lambda^2((R_\epsilon^{*0} \cap A_1^*)  \times \mathbb{S}^2_+ \smallsetminus F) = (2k)^{-1}\Lambda^2((R_\epsilon^* \cap A_1^*) \times \mathbb{S}^2_+ \smallsetminus F).
\end{equation}
Noting that $(2k)^{-1} \sim \frac{\epsilon}{\rho}$, we see it is enough to show that $\Lambda^2((R_\epsilon^{*0} \cap A_1^*)  \times \mathbb{S}^2_+ \smallsetminus F) \to 0$ as $\epsilon \to 0$. To this end, note that $R_\epsilon^{*0}  \times \mathbb{S}^2_+$ lies in the fixed bounded set $\{(x_1,\alpha) : 0 \leq x_1 \leq 1, -1 \leq x_1 + \alpha \leq 0\}  \times \mathbb{S}^2_+$ which has finite $\Lambda^2$-measure. Also notice that $F(\epsilon) \supset F(\epsilon')$ whenever $\epsilon < \epsilon'$. Consequently, by continuity from above, it is enough to show that $\bigcup_{\epsilon > 0} F(\epsilon)$ has full $\Lambda^2$ measure in $A_1^* \times \mathbb{S}^2_+$. Now $\bigcup_{\epsilon > 0} F(\epsilon)$ is the set of all $(y,w) \in A_1^* \times \mathbb{S}^2_+$ such that 
\begin{enumerate}
    \item $N < \infty$, 
    \item $r_j > 0$ for $1 \leq j \leq N-1$,
    \item $\theta_j < \frac{\pi}{2}$ for $0 \leq j \leq N+1$, and 
    \item $\nu > 0$.
\end{enumerate}
By construction of the cylindrical collision law, the set of initial states $(y,w)$ such that $N = \infty$ is a measure zero subset of $A_1^* \times \mathbb{S}^2_+$ (see \S\ref{sssec_cylcol} and \S\ref{sssec_detref}). If $r_j = 0$, this means that the billiard trajectory hits $\partial\mathcal{M}_{\cyl}$ at $y^j$ either tangentially or at a singularity of $\partial\mathcal{M}_{\cyl}$, and the set of initial states such that this happens has measure zero (see Lemma \ref{lem_iter}). If $\theta_j = \pi/2$, this means the billiard trajectory hits $\partial\mathcal{M}_{\cyl}$ tangentially at $y^j$, and again the set of initial states such that this happens has measure zero. Finally, since $\nu$ is a conserved quantity, we see that $\nu = 0$ if and only if $w = \pm\chi$, and $\Lambda^2(A_1^* \times \{\pm\chi\}) = 0$. This proves $\bigcup_{\epsilon > 0} F(\epsilon)$ is a full-measure subset of $A_1^* \times \mathbb{S}^2_+$, as desired.

To show that $\frac{\epsilon}{\rho(\epsilon)}\Lambda^2(R_\epsilon^* \times \mathbb{S}^2_+ \smallsetminus \Xi(\epsilon)) \to 0$, observe that $R_\epsilon^* \times \mathbb{S}^2_+ \smallsetminus \Xi(\epsilon)) = (R_\epsilon^* \smallsetminus R_\epsilon') \times \mathbb{S}^2_+$, where $R_\epsilon'$ is the parallelogram defined by 
\begin{equation}
\begin{split}
    R_\epsilon' = \Bigg\{(x_1,\alpha) \in \mathbf{P} : & \text{ } 0 \leq \alpha + x_1 \leq 1, \\
    & \text{ and } |\alpha| \leq \epsilon^{-1}\left(\frac{\rho}{2} - \delta_0 \right) - \left(\frac{\overline{L}}{\sin\nu} + 1\right) - \frac{\log(1/\rho)^2 \overline{L}}{4\log\log(1/\rho)} \Bigg\}.
\end{split}
\end{equation}
It is enough to show that $\frac{\epsilon}{\rho(\epsilon)}m(R_\epsilon \smallsetminus R_\epsilon') \to 0$, where $m$ denotes Lebesgue measure on $\mathbf{P}$. Note that $R_\epsilon' \subset R_\epsilon$. Computing the area of each rectangle, we have
\begin{equation}
\begin{split}
    \frac{\epsilon}{\rho(\epsilon)}m(R_\epsilon \smallsetminus R_\epsilon') & = \frac{\epsilon}{\rho(\epsilon)}\left[\epsilon^{-1}\rho - 2\epsilon^{-1}\left(\frac{\rho}{2} - \delta_0 \right) + 2\left(\frac{\overline{L}}{\sin\nu} + 1\right) + \frac{2\log(1/\rho)^2 \overline{L}}{4\log\log(1/\rho)}\right] \\
    & = \frac{2\delta_0}{\rho} + \frac{2\epsilon}{\rho}\left(\frac{\overline{L}}{\sin\nu} + 1\right) + \frac{\epsilon\log(1/\rho)^2 \overline{L}}{2\rho\log\log(1/\rho)}.
\end{split}
\end{equation}
Recalling that $\frac{\epsilon^{1/2}}{\rho} = o(1)$, and $\delta_0 = O(\frac{\epsilon}{\rho}) = o(\epsilon^{1/2})$, we see that the above quantity converges to zero as $\epsilon \to 0$.

Part (iii) of Lemma \ref{lem_mod} is proved by iterating the result of Lemma \ref{lem_control2}. Fix $(y,w) \in \Omega$ and let $N$, $(y^j, w^j)$, $\theta_j$, and $r_j$ be defined as in the definition of $\Omega$ above. Let $\widetilde{y}^0 = H_{\epsilon}(y^0)$ and let $\widetilde{w}^0 = w^0$. Let $\Phi^*$ denote the billiard map in the zoomed configuration space $\mathcal{M}^*$. We will use Lemma \ref{lem_control2} to prove that 
\begin{equation}
    (\widetilde{y}^j,\widetilde{w}^j) := \Phi^*(\widetilde{y}^{j-1},\widetilde{w}_{j-1}), \quad 1 \leq j \leq N.
\end{equation} 
are well-defined, and that the trajectory starting from $(\widetilde{y}^0,\widetilde{w}^0)$ will return to the surface $\widetilde{\mathbf{P}}^*$ in some state $(\widetilde{y}^{N+1},\widetilde{w}^{N+1})$ after $N$ collisions with the boundary. The same lemma will also give us bounds on $||(\widetilde{y}^{N+1},\widetilde{w}^{N+1}) - (y^{N+1},w^{N+1})||$. Indeed, we define by backwards induction 
\begin{equation}
\begin{split}
    s_{N+1} & = 79\log(1/\rho)^{-1}, \\ 
    s_j & = \frac{s_{j+1}\cos\theta_{j+1}\sin^2\nu}{1600(1 + mJ^{-1})^{1/2}(2 + (mJ)^{-1/2})\overline{\kappa}(\overline{L}+1)}, \quad 0 \leq j \leq N.
\end{split}
\end{equation}
Then by definition of $\Omega^*$, for $0 \leq j \leq N$, 
\begin{equation} \label{eq3.91}
    s_j \leq s_{N+1} = 79\log(1/\rho)^{-1} \leq 79\min\{r_j, (\overline{\kappa})^{-1}\cos\theta_{j+1}\}.
\end{equation}
Also observe that for $1 \leq j \leq N+1$, 
\begin{equation}
    s_j \geq s_1 = s_{N+1} \prod_{j = 1}^N \frac{\cos\theta_{j+1}\sin^2\nu}{1600(1 + mJ^{-1})^{1/2}(2 + (mJ)^{-1/2})\overline{\kappa}(\overline{L} + 1)}.
\end{equation}
Hence, using the bounds in the definition of $\Omega^*$,
\begin{equation}
\begin{split}
    &\frac{s_j \cos\theta_{j}\sin^2\nu}{800\overline{\kappa}(\overline{L} + 1)} \\
    &\geq \left( s_{N+1} \prod_{j = 1}^N \frac{\cos\theta_{j+1}\sin^2\nu}{1600(1 + mJ^{-1})^{1/2}(2 + (mJ)^{-1/2})\overline{\kappa}(\overline{L} + 1)} \right)\frac{ \cos\theta_{j}\sin^2\nu}{800\overline{\kappa}(\overline{L} + 1)} \\
    & \geq \frac{79}{\log(1/\rho)}\left(\frac{\log(1/\rho)^{-3}}{1600(1+mJ^{-1})^{1/2}(2 + (mJ)^{-1/2})(\overline{L} + 1)}\right)^N\frac{\log(1/\rho)^{-3}}{800(\overline{L} + 1)} \\
    & = \log(1/\rho)^{-4N}\cdot \frac{79 \log(1/\rho)^{N - 4}}{[1600(1 + J^{-1})^{1/2}(2 + J^{-1/2})(\overline{L} + 1)]^N 800(\overline{L} + 1)}.
\end{split}
\end{equation}
Substituting in the given upper bound for $N$ in the definition of $\Omega^*$, and noting the that the right factor above converges to infinity as $\rho = \rho(\epsilon) \to 0$, we obtain that for $\epsilon$ sufficiently small, the last line above is bounded below by
\begin{equation}
    \log(1/\rho)^{-\frac{\log(1/\rho)}{\log\log(1/\rho)}} = \rho = \rho(\epsilon).
\end{equation}
Thus 
\begin{equation} \label{eq3.95'}
    \widecheck{\rho}^{-1}\left( \frac{s_j \cos\theta_{j}\sin^2\nu}{800\overline{\kappa}(\overline{L} + 1)} \right) \geq \epsilon.
\end{equation}
An almost identical calculation gives us, for $\epsilon$ sufficiently small, 
\begin{equation} \label{eq3.96}
    \frac{s_j\cos\theta_j\sin^3\nu}{800\overline{\kappa}(\overline{L} + 1)^2} \geq \log(1/\rho)^{-5N} \geq \rho(\epsilon)^{\frac{5}{4}} \geq \epsilon,
\end{equation}
here recalling that $\epsilon^{1/2}/\rho = o(1)$. Assume $\alpha^0 \in [-\epsilon^{-1}\rho/2 + k\epsilon^{-1}\rho, \epsilon^{-1}\rho/2+k\epsilon^{-1}\rho)$. For $0 \leq j \leq N-1$, 
\begin{equation}
    |\alpha^{j+1} - \alpha^j| \leq ||w^j|| t_j = t_j,
\end{equation}
where $t_j$ is the time for the billiard trajectory to go from $y^j$ to $y^{j+1}$. We have $t_j = \frac{||p^{j+1} - p^j||}{||u^j||} \leq \frac{\overline{L}}{\sin\nu}$ by Lemma \ref{dGlem}, recalling that by definition $u^j = \text{d}(G_0)_{y^j}(w^j)$. Hence, $|\alpha^{j+1} - \alpha^j| \leq \frac{\overline{L}}{\sin\nu}$.  Thus, for $0 \leq j \leq N$, applying the bound on $|\alpha^0 - k\epsilon^{-1}\rho|$ coming from the definition of $\Omega$, we obtain
\begin{equation} \label{eq3.99'}
\begin{split}
    |\alpha^j - k\epsilon^{-1}\rho| & \leq |\alpha^0 - k\epsilon^{-1}\rho| + \sum_{j = 0}^{N-1} |\alpha^{j+1} - \alpha^j| \\
    & \leq \epsilon^{-1}\left(\frac{\rho}{2} - \delta_0 \right) - \left( \frac{\overline{L}}{\sin\nu} + 1 \right) - \frac{\log(1/\rho)^2 \overline{L}}{4\log\log(1/\rho)} + \frac{N \overline{L}}{\sin\nu} \\
    & \leq \epsilon^{-1}\left(\frac{\rho}{2} - \delta_0 \right) - \left( \frac{\overline{L}}{\sin\nu} + 1 \right).
\end{split}
\end{equation}
By definition of $y^0$ and $\widetilde{y}^0$, and the fact that $H_\epsilon$ fixes the $\alpha$-coordinate, we have 
\begin{equation}
    p^0 = G_0(y^0) = G_0 \circ H_\epsilon^{-1}(H_\epsilon(y^0)) = G_\epsilon(\widetilde{y}^0) = \widetilde{p}^0.
\end{equation}
Also $\widetilde{w}^0 = w^0$ by definition. Thus trivially 
\begin{equation}
    ||\widetilde{p}^0 - p^0|| + |\widetilde{\alpha}^0 - \alpha^0| + ||\widetilde{w}^0 - w^0|| \leq s_0.
\end{equation}
From (\ref{eq3.91}), (\ref{eq3.95'}), (\ref{eq3.96}), and (\ref{eq3.99'}) we see that the hypotheses of Lemma \ref{lem_control2} are satisfied. Thus, by induction, for $1 \leq j \leq N+1$, $\widetilde{p}^j$, the pair $(\widetilde{y}^{j-1},\widetilde{w}^{j-1})$ is licit, in the sense of \S\ref{sssec_modcon}. Hence, $\widetilde{\alpha}^j$, and $\widetilde{w}^j$ are well-defined, $p^j$ and $\widetilde{p}^j$ lie on the same smooth curve segment in $\partial \mathcal{D}^*$, and 
\begin{equation}
    ||\widetilde{p}^j - p^j|| + |\widetilde{\alpha}^j - \alpha^j| + ||\widetilde{w}^j - w^j|| \leq s_j.
\end{equation}
In particular, $p^j$ and $\widetilde{p}^j$ lie on $\partial\widehat{B}^*$ for $1 \leq j \leq N$; $p^0$ and $\widetilde{p}^{N+1}$ both lie on the line $\{(x_1,x_2) : x_2 = 0\}$, and 
\begin{equation} \label{eq3.111'}
    ||\widetilde{p}^{N+1} - p^{N+1}|| + |\widetilde{\alpha}^{N+1} - \alpha^{N+1}| + ||\widetilde{w}^{N+1} - w^{N+1}|| \leq s_{N+1} = \frac{79}{\log(1/\rho)}.
\end{equation}
We will use this estimate shortly. 

Recall the diffeomorphism $H_\epsilon : \widehat{\mathcal{Z}}^* \to \widehat{\mathcal{Z}}^*$, and define a diffeomorphism $\overline{H}_\epsilon : \widehat{\mathcal{Z}}^* \times \mathbb{S}^2 \to \widehat{\mathcal{Z}}^* \times \mathbb{S}^2$ by 
\begin{equation} \label{eq5.212}
    \overline{H}_{\epsilon}(y,w) = (H_\epsilon(y), w).
\end{equation}
In other words, $\overline{H}_{\epsilon} = \sigma_{\epsilon^{-1}} \circ \overline{H}_1 \circ \sigma_{\epsilon}$, where $\overline{H}_1$ is defined as in (\ref{eq4.69}). We also define ``zoomed versions'' of $\Psi$, $K^{\Sigma,\epsilon}$, $K^{\Sigma,\epsilon}_{\cyl}$, and $\widetilde{K}^{\Sigma,\epsilon}$ as follows:
\begin{equation}
\begin{split}
    \Psi^* & = \sigma_{\epsilon^{-1}} \circ \Psi \circ \sigma_{\epsilon} : \widetilde{\mathbf{P}}^* \times \mathbb{S}^2_{+\rho/2} \to \mathbf{P}^* \times \mathbb{S}^2_{+\rho/2}, \\
    K^* & = \sigma_{\epsilon^{-1}} \circ K^\epsilon \circ \sigma_{\epsilon} : \sigma_{\epsilon^{-1}}(\mathcal{F}) \to \sigma_{\epsilon^{-1}}(\mathcal{F}), \\
    K^*_{\cyl} & = \sigma_{\epsilon^{-1}} \circ K^\epsilon_{\cyl} \circ \sigma_{\epsilon} : \sigma_{\epsilon^{-1}}(\mathcal{F}_{\cyl}) \to \sigma_{\epsilon^{-1}}(\mathcal{F}_{\cyl}), \\
    \widetilde{K}^* & = \sigma_{\epsilon^{-1}} \circ \widetilde{K}^{\epsilon} \circ \sigma_{\epsilon} : \sigma_{\epsilon^{-1}}(\widetilde{\mathcal{F}}) \to \sigma_{\epsilon^{-1}}(\widetilde{\mathcal{F}}).
\end{split}
\end{equation}
Then from (\ref{eq4.72}) we see that
\begin{equation}
    \widetilde{K}^* = (\overline{H}_{\epsilon})^{-1} \circ (\Psi^*)^{-1} \circ K^* \circ \Psi^* \circ \overline{H}_{\epsilon}.
\end{equation}
For the rest of our argument to be valid, we need $(\widetilde{y}^0,-\widetilde{w}^0)$ and $(\widetilde{y}^{N+1},\widetilde{w}^{N+1})$ to lie in the domain of $\Psi^*$. This follows from:

\begin{claim}{5.3.1} \label{claim5.3.1}
For $\epsilon$ sufficiently small, $(\widetilde{y}^0,-\widetilde{w}^0) \in \widetilde{\mathbf{P}}^* \times \mathbb{S}^2_{+\rho/2}$ and $(\widetilde{y}^{N+1},\widetilde{w}^{N+1}) \in \widetilde{\mathbf{P}}^* \times \mathbb{S}^2_{+\rho/2}$.
\end{claim}

\begin{subproof}[Proof of Claim \ref{claim5.3.1}]
As explained above $\widetilde{y}^0$ and $\widetilde{y}^{N+1}$ lie in $\widetilde{\mathbf{P}}^*$. Let $\psi_0 = \angle(-\widetilde{w}^0, e_2)$, and let $\psi_{N+1} = \angle(\widetilde{w}^{N+1}, e_2)$. We must show that $\psi_0 \leq \frac{\pi}{2} - \rho$ and $\psi_{N+1} \leq \frac{\pi}{2} - \rho$. We have 
\begin{equation} \label{eq3.108}
\begin{split}
    \cos\psi_0 = \langle -\widetilde{w}^0, e_2 \rangle & = \langle \dd(G_\epsilon)_{\widetilde{y}^0}(-\widetilde{w}^0), e_2 \rangle + \langle \dd(G_0)_{\widetilde{y}^0}(-\widetilde{w}^0) - \dd(G_\epsilon)_{\widetilde{y}^0}(-\widetilde{w}^0), e_2 \rangle \\
    & \quad\quad\quad + \langle -\widetilde{w}^0 - \dd(G_0)_{\widetilde{y}^0}(-\widetilde{w}^0), e_2 \rangle \\
    & \geq ||\dd(G_\epsilon)(-\widetilde{w}^0)||\cos\theta_0 - ||\dd(G_0)_{\widetilde{y}^0}(-\widetilde{w}^0) - \dd(G_\epsilon)_{\widetilde{y}^0}(-\widetilde{w}^0)||,
\end{split}
\end{equation}
here using that $\langle -\widetilde{w}^0 - \dd(G_0)_{\widetilde{y}^0}(-\widetilde{w}^0), e_2 \rangle = 0$. Let $\widetilde{\alpha}^0$ denote the angular coordinate of $\widetilde{y}^0$. By Lemma \ref{dGlem},
\begin{equation}
\begin{split}
    ||\dd(G_\epsilon)(-\widetilde{w}^0)|| & \geq [1 - ((mJ)^{-1/2}+1)|\epsilon\widetilde{\alpha}_0|]\sin\varphi \\
    & \geq [1 - ((mJ)^{-1/2}+1)\epsilon\rho/2]\sin\varphi.
\end{split}
\end{equation}
Also, by (\ref{eq3.66}) we deduce that 
\begin{equation}
    ||\dd(G_0)_{\widetilde{y}^0}(-\widetilde{w}^0) - \dd(G_\epsilon)_{\widetilde{y}^0}(-\widetilde{w}^0)|| \leq |\epsilon\widetilde{\alpha}^0| \leq \frac{\epsilon\rho}{2}.
\end{equation}
Substituting these bounds into (\ref{eq3.108}), we obtain
\begin{equation}
\begin{split}
    \cos\psi_0 & \geq [1 - ((mJ)^{-1/2}+1)\epsilon\rho/2]\sin\varphi\cos\theta_0 - \frac{\epsilon\rho}{2} \\
    & \geq [1 - ((mJ)^{-1/2}+1)\epsilon\rho/2](\sin\varphi)\log(1/\rho)^{-1} - \frac{\epsilon\rho}{2} \geq \sin(\rho),
\end{split}
\end{equation}
for $\epsilon$ sufficiently small. Therefore, $\psi_0 \leq \frac{\pi}{2} - \rho(\epsilon)$ for $\epsilon$ sufficiently small. We omit the argument showing that $\psi_{N+1} \leq \frac{\pi}{2} - \rho(\epsilon)$ since it is almost identical.
\end{subproof}

\begin{sloppypar} 
By Claim \ref{claim5.3.1}, we may define $(\widehat{y}^0,\widehat{w}^0) = \Psi^*(\widetilde{y}^0,-\widetilde{w}^0)$, and $(\widehat{y}^{N+1},\widehat{w}^{N+1}) = \Psi^*(\widetilde{y}^{N+1},\widetilde{w}^{N+1})$. By definition of $\Psi$ and $\Psi^*$, $\widehat{y}^0$ and $\widehat{y}^{N+1}$ both lie in the plane $\mathbf{P}$; the interior of the line segment from $\widehat{y}^0$ to $\widetilde{y}^1$ lies in $\Int \mathcal{M}^*$ and contains the point $\widetilde{y}^0$; and the interior of the line segment from $\widetilde{y}^N$ to $\widehat{y}^{N+1}$ lies in the interior of $\mathcal{M}^*$ and contains the point $\widetilde{y}^{N+1}$ in its interior. Thus the billiard trajectory in $\mathcal{M}^*$ starting in state $(\widehat{y}^0,-\widehat{w}^0)$ hits the boundary $N$ times before returning to the plane $\mathbf{P}$ in state $(y^{N+1},w^{N+1})$. Therefore, \end{sloppypar}
\begin{equation}
    K^*(\widehat{y}^0,\widehat{w}^0) = (\widehat{y}^{N+1}, \widehat{w}^{N+1}), \quad \text{ and } \quad K^*_{\cyl}(y^0,-w^0) = (y^{N+1}, w^{N+1}).
\end{equation}
Consequently, 
\begin{equation}
\begin{split}
    \widetilde{K}^*(y^0,-w^0) & = (\overline{H}_{\epsilon})^{-1} \circ (\Psi^*)^{-1} \circ K^* \circ \Psi^* \circ \overline{H}_\epsilon(y^0, -w^0) \\
    & = (\overline{H}_{\epsilon})^{-1} \circ (\Psi^*)^{-1} \circ K^* \circ \Psi^*(\widetilde{y}^0, -\widetilde{w}^0) \\
    & = (\overline{H}_{\epsilon})^{-1} \circ (\Psi^*)^{-1} \circ K^*(\widehat{y}^0, \widehat{w}^0) \\
    & = (\overline{H}_{\epsilon})^{-1} \circ (\Psi^*)^{-1}(\widehat{y}^{N+1}, \widehat{w}^{N+1}) \\
    & = (\overline{H}_{\epsilon})^{-1}(\widetilde{y}^{N+1}, \widetilde{w}^{N+1}) = (H_\epsilon^{-1}(\widetilde{y}^{N+1}), w^{N+1}).
\end{split}
\end{equation}
Therefore, 
\begin{equation} \label{eq3.107}
\begin{split}
    ||\widetilde{K}^*(y^0, -w^0) & - K_{\cyl}^*(y^0, -w^0)|| \leq ||H_{\epsilon}^{-1}(\widetilde{y}^{N+1}) - y^{N+1}|| + ||\widetilde{w}^{N+1} - w^{N+1}|| \\
    & \leq ||G_0(H_{\epsilon}^{-1}(\widetilde{y}^{N+1}) - y^{N+1})|| + |\widetilde{\alpha}^{N+1} - \alpha^{N+1}| + ||\widetilde{w}^{N+1} - w^{N+1}||,
\end{split}
\end{equation}
here using the definition of $G_0$ and the fact that $H_\epsilon$ fixes the angular coordinate. But by definition of $G_\epsilon$, 
\begin{equation}
    G_0(H_\epsilon^{-1}(\widetilde{y}^{N+1})) = G_\epsilon(\widetilde{y}^{N+1}) = \widetilde{p}^{N+1}.
\end{equation}
Thus 
\begin{equation}
    ||G_0(H_\epsilon^{-1}(\widetilde{y}^{N+1}) - y^{N+1})|| = ||\widetilde{p}^{N+1} - p^{N+1}||.
\end{equation}
Substituting the above into (\ref{eq3.107}), and applying the bound (\ref{eq3.111}), we obtain
\begin{equation}
\begin{split}
||\widetilde{K}^*(y^0,-w^0) & - K_{\cyl}^*(y^0,-w^0)|| \\
& \leq ||\widetilde{p}^{N+1} - p^{N+1}|| + |\widetilde{\alpha}^{N+1} - \alpha^{N+1}| + ||\widetilde{w}^{N+1} - w^{N+1}|| \\
& \quad\quad \leq \frac{80}{\log(1/\rho)^3}.
\end{split}
\end{equation}
Note that this bound does not depend on the choice of $(y,w) = (y^0,-w^0) \in \Omega^*$. Thus
\begin{equation}
\begin{split}
    \sup_{\Omega} ||\widetilde{K}^{\Sigma,\epsilon} - K^{\Sigma, \epsilon}|| & \leq \sup_{\Omega} ||\sigma_{\epsilon^{-1}} \circ \widetilde{K}^{\Sigma,\epsilon} - \sigma_{\epsilon^{-1}} \circ K^{\Sigma, \epsilon}|| \\
    & = \sup_{\Omega^*} ||\widetilde{K}^* - K^*_{\cyl}|| \leq \frac{80}{\log(1/\rho)^3},
\end{split}
\end{equation}
and the right-hand side converges to zero as $\epsilon \to 0$, as desired.
\end{proof}

\section{Rough Reflections in General Billiard Domains} \label{sec_genref}

\subsection{The billiard map and the invariant measure}

The results stated in this section are well-known, but proofs under the weak regularity conditions adopted in this work may not be found in standard references. We therefore provide careful proofs here. More traditional treatments of billiards may be found in [\ref{Chernov&Markarian}] and [\ref{Tabachnikov}]. For billiards in Riemannian manifolds, see [\ref{CFS_ergodic}, Chapter 5].

\subsubsection{A billiard domain with singularities} \label{sss_sing}

Let $\mathcal{R}$ be a $d$-dimensional $C^2$ Riemannian manifold with metric $\langle \cdot, \cdot \rangle$. Assume that $\mathcal{R}$ is geodesically complete. 

Our billiard domain will be taken from a certain class of closed $C^2$ submanifolds of $\mathcal{R}$ with singularities. Namely, let $\mathcal{N}$ be a subset of $\mathcal{R}$, and let $\partial \mathcal{N}$ and $\Int \mathcal{N}$ denote the topological boundary and interior of $\mathcal{N}$. We say that $\mathcal{N}$ belongs to the class $\CES_0^2(\mathcal{R})$ if $\mathcal{N}$ is a closed subset of $\mathcal{R}$, and there exists a closed subset $\mathcal{S} \subset \partial \mathcal{N}$ such that the following two conditions hold:
\begin{enumerate}[label = C\arabic*.]
     \item The $(d-1)$-dimensional Hausdorff measure of $\mathcal{S}$ is zero.
    \item For every point $q \in \mathcal{N} \smallsetminus \mathcal{S}$, there exists a neighborhood $U \subset \mathcal{R}$ of $q$ and a $C^2$ diffeomorphism $\phi : U \to \mathbb{R}^d$ such that $\phi(q) = 0$ and the image of $\mathcal{N} \cap U$ is either $\mathbb{R}^d$ or the upper half-space $\mathbb{H}^d := \{(x_1,\dots, x_d)\ : x_d \geq 0\}$.
\end{enumerate}

The second condition means that $\mathcal{N} \smallsetminus \mathcal{S}$ is a codimension 0, embedded, $C^2$ submanifold of $\mathcal{R}$ with boundary. We let $\mathcal{N}_{\reg} = \mathcal{N} \smallsetminus \mathcal{S}$.

We will refer to points in $\mathcal{S}$ as \textit{singular points} of $\mathcal{N}$, and to points in $\mathcal{N} \smallsetminus \mathcal{S}$ as \textit{regular points} of $\mathcal{N}$. To avoid ambiguities, we will always assume that $\mathcal{S}$ is chosen to be minimal. That is, $\mathcal{S}$ is a smallest subset of $\partial \mathcal{N}$, with respect to inclusion, such that condition C2 holds.

In condition C2, we will refer to points $q$ with a neighborhood mapping to $\mathbb{R}^d$ as \textit{interior points} of $\mathcal{N}$, and we will refer to points $q$ with a neighborhood mapping to $\mathbb{H}^d$ as \textit{regular boundary points} of $\mathcal{N}$. These notions do not depend on $\phi$. The set of regular boundary points coincides with $\partial \mathcal{N} \smallsetminus \mathcal{S}$, and the set of interior points coincides with $\Int \mathcal{N}$.

\subsubsection{Billiard map} \label{sssec_bilmap}

Let $T\mathcal{R}$ denote the tangent space over $\mathcal{R}$, and let $T^1\mathcal{R} = \{(q,p) \in T\mathcal{R} : \langle p,p \rangle_q = 1\}$ denote the unit tangent bundle. We denote the \textit{geodesic flow} on $T^1\mathcal{R}$ by $G^t$. That is, for each $(q,p)$, if $\gamma(t)$ is the unique (unit speed) geodesic starting from state $(q,p)$ (defined for all time by completeness), then by definition $G^t(q,p) = (\gamma(t),\dot{\gamma}(t)) \in T^1\mathcal{R}$. In local coordinates $(q^i,p^i)$ on the tangent space, the geodesics satisfy:
\begin{equation}
    \dot{q}^i = p^i, \quad\quad \dot{p}^i = - \sum_{i,k = 1}^d \Gamma^i_{jk} p^j p^k, 
\end{equation}
where $\Gamma^i_{jk}$ are the Christoffel symbols associated with the Riemannian metric, and ``dot'' denotes the time derivative. The flow $G^t$ is generated by the vector field on $T^1\mathcal{R}$, expressed in local coordinates as
\begin{equation} \label{eq6.2'}
    X = \sum_{i = 1}^d p^i \pdv{}{q^i} - \sum_{j,k = 1}^d \Gamma^i_{jk} p^j p^k\pdv{}{p^i}.
\end{equation}
It follows from this expression that
\begin{equation} \label{eq6.3}
    \dd\pi(X(q,p)) = \sum_{i = 1}^d p^i \pdv{}{q_i} = (q,p),
\end{equation}
where 
\begin{equation}
    \pi : T\mathcal{R} \to \mathcal{R}   
\end{equation} 
is natural projection.

For each $q \in \partial \mathcal{N}_{\reg}$, let $n(q) \in T_p^1\mathcal{R}$ denote the unit normal vector pointing into $\mathcal{N}$. We denote the restriction of the unit tangent bundle on $\mathcal{R}$ to the regular boundary of $\mathcal{N}$ by 
\begin{equation} \label{eq6.5'}
    \mathcal{U} = \{(q,p) \in T^1 \mathcal{R} : q \in \partial \mathcal{N}_{\reg}\}.
\end{equation}
We also denote the inward ($+$) and outward ($-$)-pointing unit tangent bundles on the regular boundary by
\begin{equation} \label{eq6.6'}
    \mathcal{U}^\pm = \{(q,p) \in T^1\mathcal{R} : q \in \partial \mathcal{N}_{\reg} \text{ and } \pm\langle p, n(q) \rangle_q > 0\}.
\end{equation}

An important observation is that the vector field $X$ is nowhere tangent to $\mathcal{U}^+$. Indeed, if $f(q)$ is a local defining function for $\partial \mathcal{N}$, then $\grad f(q) = c n(q)$ for some $c \neq 0$, and $F(q,p) = f \circ \pi(q,p)$ is a local defining function for $\mathcal{U}$. Thus, for $(q,p) \in \mathcal{U}^+$,  
\begin{equation}
    \dd F \circ X(q,p) = \dd f \circ \dd\pi \circ X(q,p) = \dd f(q,p) = \dd f_q(p) = c\langle p, n(q) \rangle_q \neq 0,
\end{equation}
where the second equality follows by (\ref{eq6.3}). This implies $X(q,p)$ does not lie in $\text{ker } \dd F_{(q,p)} = T_{(q,p)}\mathcal{U}$. A consequence of this observation is that the flowout
\begin{equation} \label{eq6.8}
    G : \mathcal{U}^+ \times \mathbb{R} \to T^1\mathcal{R}, \quad\quad G(q,p,t) := G^t(q,p),
\end{equation}
is a local diffeomorphism.

A point particle starting from $(q,p) \in \mathcal{U}^+$ first returns to $\partial M$ at time 
\begin{equation} \label{eq6.2}
    \overline{t}(q,p) := \inf\{ t > 0 : \pi \circ G^t(q,p) \notin \Int\mathcal{N} \}.
\end{equation}
(If $\pi \circ G^t(q,p) \in \Int\mathcal{N}$ for all $t > 0$, then by convention the above quantity is infinite.) Since the geodesic starting from a point $(q,p) \in \mathcal{U}^+$ initially points into $\mathcal{N}$, the time $\overline{t}$ is strictly positive. In fact, if $\overline{t}(q,p)$ is finite, then since the complement of $\Int \mathcal{N}$ is closed, the infimum in (\ref{eq6.2}) must be a minimum. 

We introduce the following subbundles of $\mathcal{U}^+$:
\begin{equation} \label{eq6.10}
\begin{split}
& \mathcal{U}^+_{\fin} = \{(q,p) \in \mathcal{U}^+ : \overline{t}(q,p) < \infty\}, \\
& \mathcal{U}_1^+ = \{(q,p) \in \mathcal{U}^+_{\fin} : G^{\overline{t}(q,p)}(q,p) \in \mathcal{U}^-\}.
\end{split}
\end{equation}
In applications, $\mathcal{U}^+_{\fin}$ will be a full-measure subset of $\mathcal{U}^+$. The set $\mathcal{U}_1^+$ is an open subset of $\mathcal{U}$ (not just of $\mathcal{U}^+_{\fin}$). This may be proved by showing that $\overline{t}$ is finite and continuous on a neighborhood in $\mathcal{U}$ of a point $(q,p) \in \mathcal{U}_1^+$. (See the proof of Lemma \ref{lem_bilmap}).

The specular reflection map $R : \mathcal{U} \to \mathcal{U}$ is defined by
\begin{equation}
    R(q,p) = (q, p - 2\langle p,n(q) \rangle_q n(q)).
\end{equation}
Note that $R$ maps $\mathcal{U}^\pm$ onto $\mathcal{U}^{\mp}$ and is involutive (that is, $R \circ R = \text{Id}$).

The \textit{billiard map} $\Phi : \mathcal{U}_1^+ \to \mathcal{U}^+$ is defined by
\begin{equation}
    \Phi(q,p) = R \circ G^{\overline{t}(q,p)}(q,p).
\end{equation}

\begin{lemma} \label{lem_bilmap}
(i) $\Phi$ is injective, and its left-inverse is given by $-R \circ \Phi \circ -R$, where $-R(q,p) := R(q,-p)$.

(ii) Suppose $(q_1,p_1) := \Phi(q_0,p_0)$ lies in $\mathcal{U}_1^+$. Then $\Phi$ restricts to a $C^1$ diffeomorphism from a neighborhood of $(q_0,p_0)$ to a neighborhood of $(q_1,p_1)$.
\end{lemma}

\begin{proof}
(i) Consider the unit speed geodesic running from $(q,p) \in \mathcal{U}_1^+$ to $(q',p') := G^{\overline{t}(q,p)}(q,p) \in \mathcal{U}^+$. Noting that $R$ is involutive, we have $-R \circ \Phi(q,p) = (q', -p')$. The geodesic starting from $(q',-p')$ retraces the original geodesic and returns to the boundary in the state $(q,-p)$. Thus by symmetry $-R \circ \Phi(q',-p') = (q,p)$. Hence $-R \circ \Phi \circ -R \circ \Phi(q,p) = (q,p)$.

(ii) Write $\Phi = R \circ \Phi_-$, where $\Phi_-(q,p) = G^{\overline{t}(q,p)}(q,p)$. Let $(q_1,p^-_1) = \Phi_-(q_0,p_0)$. Since $\partial \mathcal{N}_{\reg}$ is an $C^2$ submanifold of $\mathcal{R}$ of codimension 1, it has a $C^2$ local defining function $f$ defined in a neighborhood of $q_1$. Then $F(q,p) := f(q) = f\circ \pi(q,p)$ is a local defining function for $\mathcal{U}$. The function 
\begin{equation}
    K(q,p,t) := F(G^t(q,p))
\end{equation}
is defined in a neighborhood of $(q_0,p_0,\overline{t}(q_0,p_0))$, and
\begin{equation}
\begin{split}
    {\pdv{K}{t}}(q_0,p_0,\overline{t}(q_0,p_0)) & = \dd F(q_1,p^-_1) X(q_1,p^-_1) = \dd f \circ \dd\pi \circ X(q_1,p^-_1) \\
    & = \dd f_{q_1}(p^-_1) = c\langle n(q_1), p^-_1 \rangle_{q_1} \neq 0,
\end{split}
\end{equation}
for some $c \neq 0$. The third equality above follows from (\ref{eq6.3}), and the last equality follows because the gradient of $f$ is parallel to $n(q_1)$ at $q_1$. By the Implicit Function Theorem, there is a $C^2$ function $\widehat{t}$ defined on a neighborhood $U \subset \mathcal{U}_1^+$ of $(q_0,p_0)$, and a finite open interval $I = (a,b)$ containing $\overline{t}(q_0,p_0)$ such that solutions to 
\begin{equation} \label{eq6.15'}
    K(q,p,t) = 0, \quad (q,p,t) \in U \times I
\end{equation}
take the form $(q,p,\widehat{t}(q,p))$.

We claim $\overline{t} = \widehat{t}$ in a neighborhood of $(q_0,p_0)$. To see this, let $V \subset\subset U$ be a precompact open set containing $(q_0,p_0)$, and let 
\begin{equation}
    A = \{(q,p,t) \in \overline{V} \times [0,a] : \pi \circ G^t(q,p) \notin \Int M\}.
\end{equation}
For $(q,p) \in \overline{V}$, it is evident that if $\overline{t}(q,p) \neq \widehat{t}(q,p)$, then $\overline{t}(q,p) \leq a$, since otherwise $(q,p,\overline{t}(q,p))$ would solve (\ref{eq6.15'}). Therefore $\{(q,p) \in V : \overline{t}(q,p) \neq \overline{t}(q,p)\} \subset \pi_{\mathcal{U}}(A)$, where $\pi_{\mathcal{U}} : \mathcal{U} \times \mathbb{R} \to \mathcal{U}$ is projection. Note that $\pi_{\mathcal{U}}(A)$ is compact. Since $\overline{t}(q_0,p_0) > a$, we have $(q_0,p_0) \notin \pi_{\mathcal{U}}(A)$, and $\overline{t} = \widehat{t}$ on $V \smallsetminus \pi_{\mathcal{U}}(A)$.

Consequently, $\Phi_-$ is $C^2$ in a neighborhood of $(q_0,p_0)$. The unit normal field $n$ and the specular reflection map $R$ are $C^1$. Therefore $\Phi = R \circ \Phi_-$ is $C^1$ in a neighborhood of $(q_0,p_0)$.

By the same reasoning, $\Phi$ is $C^1$ in a neighborhood of $(q_1,-p^-_1)$. By (i), $\Phi$ has a left inverse $-R \circ \Phi \circ -R$ which is $C^1$ in a neighborhood of $(q_1,p_1)$. Thus $\Phi$ is a local $C^1$ diffeomorphism.
\end{proof}

\subsubsection{Invariant measure}

At each point $q \in \mathcal{R}$, the tangent space $T_q\mathcal{R}$ inherits a Euclidean metric from the metric on the manifold. Let $\sigma_q$ denote surface measure on the unit sphere $T_q^1\mathcal{R} \subset T_q\mathcal{R}$ with respect to this metric, and let $\tau$ denote the measure on $\mathcal{R}$ generated by the Riemannian metric on the manifold. Define a measure $\mu$ on the unit tangent bundle by
\begin{equation}
    \mu(\dd p\dd q) = \sigma_q(\dd p) \tau(\dd q).
\end{equation}
In other words,
\begin{equation}
    \int_{T^1\mathcal{R}} f \dd\mu = \int_{\mathcal{R}} \int_{T^1_q \mathcal{R}} f(q,p) \sigma_q(\dd p) \tau(\dd q), \quad \text{ for all } f \in C_c(T^1 \mathcal{R}).
\end{equation}
It is well-known that the geodesic flow preserves the measure $\mu$, in the sense that $\int_{T^1\mathcal{R}} f \circ G^t \dd\mu = \int_{T^1\mathcal{R}} f \dd\mu$ for any $f$ and $t \in \mathbb{R}$. This is Liouville's Theorem, stated for a Hamiltonian system with energy function $H(q,p) = \frac{1}{2}\langle p,p \rangle_q$. See [\ref{Arnold}, \S 16], or [\ref{CFS_ergodic}, Ch. 5]. 

Let $\tau_1$ denote the surface measure on $\partial \mathcal{N}_{\reg}$ induced by the metric on the ambient space $\mathcal{R}$. We define a measure $\mu_1$ on $\mathcal{U}_+$ as follows:
\begin{equation}
    \mu_1(\dd p \dd q) = \langle q, n(p) \rangle_p \sigma_p(\dd q) \tau_1(\dd p). 
\end{equation}

\begin{lemma} \label{lem_bilinvar}
The billiard map $\Phi$ preserves the measure $\mu_1$, in the sense that, for any $A \subset \mathcal{U}_1^+$, $\mu_1(\Phi(A)) = \mu_1(A)$.
\end{lemma}

\begin{proof} We equip $T \mathcal{R}$ with the \textit{Sasaki metric} $\langle \cdot,\cdot \rangle_S$. This is the metric, first introduced in [\ref{Sasaki}], which restricts on $\mathcal{R}$ to the given Riemannian metric and on the fibers $T_q\mathcal{R}$ to the inherited Euclidean metric, and which makes $\mathcal{R}$ and $T_q\mathcal{R}$ orthogonal. Given $\xi, \zeta \in T_{(q,p)}T\mathcal{R}$ and curves $\gamma = (\alpha, U)$ and $\eta = (\beta, V)$ with $\dot{\gamma}(0) = \xi$ and $\dot{\eta}(0) = \zeta$, the metric is defined by 
\begin{equation}
    \langle \xi,\zeta \rangle_S = \langle \dot{\alpha}(0), \dot{\beta}(0) \rangle + \langle D_{\dot{\alpha}}U(0), D_{\dot{\beta}}V(0) \rangle,
\end{equation}
where $D$ denotes covariant differentiation.

The metric $\langle \cdot, \cdot \rangle_S$ induces a metric on the unit tangent bundle $T^1\mathcal{R}$. The volume measure on $T^1\mathcal{R}$ coming from this metric is just the measure $\mu$ defined above.

The following simple argument comes from [\ref{CFS_ergodic}, Chapter 6].  With respect to the Sasaki metric, the unit normal vector field on the codimension 1 submanifold $\mathcal{U}_1^+ \subset T^1\mathcal{R}$ is given by $\iota^*n$, where $\iota : \mathcal{R} \to T\mathcal{R}$ is inclusion. Consequently, the flux of the vector field $X$ through any $A \subset \mathcal{U}_1^+$ is 
\begin{equation}
    \int_A \langle X(q,p), \iota^*n(q,p) \rangle_S \mu(\dd q \dd p) = \int_A \langle p, n(q) \rangle \mu(\dd q \dd p) = \mu_1(A).
\end{equation}
As above, write $\Phi = R \circ \Phi_-$. Let $A' = \Phi_-(A)$. Since the flow $G^t$ preserves $\mu$, this implies that the flux is preserved:
\begin{equation}
    \mu_1(A) = \mu_1(A').
\end{equation}
Thus $\Phi_-$ preserves the measure $\mu_1$. Since specular reflection $R$ is an isometry and preserves the quantity $\langle p, n(q) \rangle_q$, it too preserves the measure $\mu_1$.
\end{proof}

\begin{lemma} \label{lem_iter}
(i) $\mu_1(\mathcal{U}_{\fin}^+ \smallsetminus \mathcal{U}_1^+) = 0$.

(ii) For each $m \geq 1$, let $\mathcal{U}_m^+ = \Phi^{-1}(\mathcal{U}_{m-1}^+)$. Then $\mu_1(\mathcal{U}_{\fin}^+ \smallsetminus \mathcal{U}_m^+) = 0$, and the iterate $\Phi^m = \Phi \circ \Phi \circ \cdots \circ \Phi \text{ ($m$ times)} : \mathcal{U}_m^+ \to \mathcal{U}^+$ is a $C^1$ diffeomorphism onto its image.
\end{lemma}

\begin{proof}
(i) For $(q,p) \in \mathcal{U}_{\fin}^+ \smallsetminus \mathcal{U}_1^+$, set $(q',p') = \Phi_-(q,p)$. Let $A$ be the set of all such points $(q,p)$ such that $q'$ lies in the singular part of the boundary $\mathcal{S}$. Let $B$ be the set of all such points $(q,p)$ such that $q' \in \partial \mathcal{N}_{\reg}$ and $p'$ is tangent to $\partial \mathcal{N}_{\reg}$ at $q'$. We see that $\mathcal{U}_{\fin}^+ \smallsetminus \mathcal{U}_1^+$ is the disjoint union of $A$ and $B$.  Thus it enough to show that $A$ and $B$ have measure zero in $\mathcal{U}^+$.

Recall that the flowout $G$ defined by (\ref{eq6.8}) is a local diffeomorphism. Let $\widetilde{\mathcal{S}} \subset \mathcal{U}^+ \times \mathbb{R}$ denote the preimage of $\mathcal{S}$ under the composition
\begin{equation} \label{eq6.9'}
    \mathcal{U}^+ \times \mathbb{R} \xrightarrow{G} T^1\mathcal{R} \xrightarrow{\pi} \mathcal{R}, 
\end{equation}
where $\pi$ is natural projection. Then $A \subset \pi_{\mathcal{U}^+}(\widetilde{\mathcal{S}})$, where $\pi_{\mathcal{U}^+}$ is projection onto the first factor of $\mathcal{U}^+ \times \mathbb{R}$.

Let $s$ denote the Hausdorff dimension of $\mathcal{S}$. By assumption, $s < d - 1$. Since the first map in (\ref{eq6.9'}) is a local $C^2$ diffeomorphism, and the second is a $C^2$ submersion from a $(2d-1)$-dimensional space to a $d$-dimensional space, we see that the Hausdorff dimension of $\widetilde{\mathcal{S}}$ is $s + d - 1$. Thus the Hausdorff dimension of $\pi_{\mathcal{U}^+}(\widetilde{\mathcal{S}})$ is at most $s + d - 1 < 2d - 2$. Since $\mathcal{U}^+$ is a $(2d-2)$-dimensional submanifold of $\mathcal{R}$, we conclude that $A \subset \pi_{\mathcal{U}^+}(\widetilde{\mathcal{S}})$ has measure zero in $\mathcal{U}^+$.

To see that $B$ also has measure zero in $\mathcal{U}^+$, let $\mathcal{T} \subset \mathcal{U}^+ \times \mathbb{R}$ be the preimage under $G$ of $T^1\partial \mathcal{N}_{\reg} \subset T^1\mathcal{R}$. Observe that $B \subset \pi_{\mathcal{U}^+}(\mathcal{T})$. The dimension of $T^1\partial \mathcal{N}_{\reg}$ is a $2d - 3$, and therefore the dimension of $\mathcal{T}$ is $2d - 3$, and the dimension of $\pi_{\mathcal{U}^+}(\mathcal{T})$ is at most $2d - 3$. It follows that $B \subset \pi_{\mathcal{U}^+}(\mathcal{T})$ has measure zero in $\mathcal{U}^+$. 

(ii) Recall that $\mathcal{U}_1^+$ is an open subset of $\mathcal{U}^+$ and $\Phi : \mathcal{U}_1^+ \to \Phi(\mathcal{U}_1^+)$ is a $C^1$ diffeomorphism. It clearly follows that each $\mathcal{U}_m^+$ is an open subset of $\mathcal{U}^+$ and each iterate $\Phi^m$ is a $C^1$ diffeomorphism from $\mathcal{U}_m^+$ onto its image. To see that $\mu_1(\mathcal{U}_{\fin}^+ \smallsetminus \mathcal{U}_1^+) = 0$, let $F \subset \mathcal{U}_1^+$ be any measurable set with $\mu_1(F) < \infty$. Using (i) and the fact that $\Phi$ preserves $\mu_1$, we have 
\begin{equation}
\begin{split}
    \mu_1(F \smallsetminus \mathcal{U}_2^+) & = \mu_1(F \smallsetminus \Phi^{-1}(\Phi(F) \cap \mathcal{U}_1^+)) \\
    & = \mu_1(F) - \mu_1(\Phi^{-1}(\Phi(F) \cap \mathcal{U}_1^+)) \\
    & = \mu_1(F) - \mu_1(\Phi(F) \cap \mathcal{U}_1^+) \\
    & = \mu_1(F) - \mu_1(\Phi(F)) = 0.
\end{split}
\end{equation}
It follows that $\mu_1(\mathcal{U}_1^+ \smallsetminus \mathcal{U}_2^+) = 0$. Inductively, we similarly obtain that $\mu_1(\mathcal{U}_{m-1} \smallsetminus \mathcal{U}_m) = 0$ for all $m \geq 1$. As $\mathcal{U}_1^+$ is a full measure subset of $\mathcal{U}_{\fin}^+$, the result follows.
\end{proof}

\subsection{Rough Reflection Laws} \label{ssec_reflaws}

The construction of rough reflections described here generalizes a construction appearing in [\ref{ABS_refl}]. It is also similar to constructions appearing in [\ref{feres_rw}] and [\ref{Plakhov5}].

\subsubsection{Macro-reflection laws} \label{sssec_detref}

Consider two nested billiard domains $\mathcal{M}_0 \subset \mathcal{M}_1 \subset \mathcal{R}$. The boundary of $\mathcal{M}_0$ provides an interface from which we can observe reflections from the boundary of $\mathcal{M}_1$. The goal for this subsection is to describe precisely what it means to ``observe'' the reflections from $\partial \mathcal{M}_1$ from the interface $\partial \mathcal{M}_0$. In the next subsection, we will describe what happens when $\mathcal{M}_1$ approaches $\mathcal{M}_0$ uniformly.

We suppose that $\mathcal{M}_1, \mathcal{M}_0$ are subsets of $\mathcal{R}$ which satisfy the following conditions:
\begin{enumerate}[label = D\arabic*.]
    \item $\mathcal{M}_0$ is closed, embedded $C^2$ submanifold of $\mathcal{R}$ of codimension 0, with boundary (but no corners or singularities),
    \item $\mathcal{M}_1$ belongs to the class $\CES_0^2(\mathcal{R})$, and 
    \item $\Int \mathcal{M}_1 \supset \mathcal{M}_0$
\end{enumerate}
We would also like to consider the case where $\mathcal{M}_0$ is allowed to intersect the boundary of $\mathcal{M}_1$, but this introduces some technicalities. We describe how to modify our arguments to accommodate this situation in Remark \ref{rem_extend}.

Let 
\begin{equation} \label{eq6.27}
    \mathcal{N} = \overline{\mathcal{M}_1 \smallsetminus \mathcal{M}_0}.
\end{equation}
Here the closure is taken in the ambient space $\mathcal{R}$. The fundamental fact in our analysis is that:

\begin{lemma} \label{lem_classy}
$\mathcal{N}$ belongs to the class $\CES_0^2(\mathcal{R})$.
\end{lemma}

\begin{proof}
Let $\mathcal{S}$ be the set of singular points of $\mathcal{M}_1$, and let $\mathcal{S}_1 = \mathcal{S} \cap \mathcal{N}$. It is enough to show that $\mathcal{N} \smallsetminus \mathcal{S}_1$ is a codimension 0, embedded $C^2$ submanifold of $\mathcal{R}$ with boundary. Since $\mathcal{M}_0 \subset \Int \mathcal{M}_1$, we may write
\begin{equation} \label{eq6.28}
    \mathcal{N} \smallsetminus \mathcal{S}_1 = (\Int \mathcal{M}_1 \smallsetminus \mathcal{M}_0) \cup \partial_{\reg} \mathcal{M} \cup \partial \mathcal{M}_0,
\end{equation}
and the above union is disjoint.

If $q \in \Int \mathcal{M}_1 \smallsetminus \mathcal{M}_0$, then there exists a neighborhood $U_1 \subset \mathcal{R}$ of $q$ which lies entirely in $\Int \mathcal{M}_1 \smallsetminus \mathcal{M}_0$ and which is $C^2$-diffeomorphic to $\mathbb{R}^d$. 

If $q \in \partial_{\reg} \mathcal{M}$, then there exists a neighborhood $U_2 \subset \mathcal{R}$ of $q$ and a $C^2$-diffeomorphism $\phi_2 : U_2 \to \mathbb{R}^d$ such that $\phi_2(0) = 0$ and $\phi_2(\mathcal{M}_1 \cap U_2) = \mathbb{H}^d$. Shrinking $U$ if necessary, we may suppose that $U_2 \cap \mathcal{M}_0 = \emptyset$, in which case $\phi_2(\mathcal{N} \cap U_2) = \mathbb{H}^d$. 

Finally, if $q \in \partial \mathcal{M}_0$, then there exists a neighborhood $U_3 \subset \mathcal{R}$ of $q$ and a $C^2$-diffeomorphism $\phi_3 : U_3 \to \mathbb{R}^d$ such that $\phi_3(0) = 0$ and $\phi_3(\mathcal{M}_0 \cap U_2) = \mathbb{H}^d$. Shrinking $U_3$ if necessary, we may suppose that $U_3 \subset \Int \mathcal{M}_1$, in which case 
\begin{equation}
    \phi_3(\mathcal{N} \cap U_3) = \Phi_3(U \smallsetminus \Int \mathcal{M}_0) = \mathbb{H}_-^d := \{(x_1,...,x_d) : x_d \leq 0\} \cong \mathbb{H}^d,
\end{equation}
and we are done.
\end{proof}

Notation introduced in the previous subsection specializes to our choice $\mathcal{N}$ in (\ref{eq6.27}). The boundary of $\mathcal{N}$ decomposes as the disjoint union $\partial \mathcal{M}_1 \cup \partial \mathcal{M}_0$. We denote restrictions of $\mathcal{U}^\pm$ to $\partial \mathcal{M}_0$ by 
\begin{equation}
    \mathcal{V}^\pm = \{(q,p) \in \mathcal{U}^{\mp} : q \in \partial \mathcal{M}_0\}.
\end{equation}
Note that our sign conventions are reversed with $\mathcal{V}^\pm \subset \mathcal{U}^{\mp}$. This is because we want to think of vectors in $\mathcal{V}^+$ as pointing \textit{into} the billiard domain $\mathcal{M}_0$, even though they point \textit{out} of $\mathcal{N}$.

For our definition of a macro-reflection law to make sense, two additional conditions are needed:
\begin{enumerate}
    \item[D4.] $\mathcal{U}_{\fin}^+$ is a full-measure subset of $\mathcal{U}^+$, with respect to the measure $\mu_1$.
    \item[D5.] Except on a set of $\mu_1$-measure zero, a point particle in $\mathcal{N}$ which starts with initial state $(q,p) \in \mathcal{V}^- \subset \mathcal{U}^+$ eventually returns to $\mathcal{V}^-$ after only finitely many collisions with $\partial \mathcal{M}_1$.
\end{enumerate}
For various special cases, there are straightforward conditions which guarantee that conditions D4 and D5 hold. Typically these conditions will reduce our setting to the finite measure case, and D5 will follow from the Poincar\'{e} Recurrence Theorem. See \S\ref{sssec_special}, below.

As a consequence of Lemma \ref{lem_iter}, condition D4 implies that $\Phi$ and its iterates are well-defined on $\mathcal{U}^+$, except on a null set. 

For $(q,p) \in \mathcal{U}^+$, set $M(q,p) = \min\{m \geq 1 : (q,p) \in \mathcal{U}^+_m \text{ and } \Phi^m(q,p) \in \mathcal{U}_{\partial C}^+\}$ (and set $M = \infty$ if the minimum does not exist). The \textit{macro-reflection law associated with $\mathcal{M}_0$ and $\mathcal{M}_1$} is the mapping $P^{\mathcal{M}_1,\mathcal{M}_0} = P^{\mathcal{M}_1} : \mathcal{V}^+ \to \mathcal{V}^+ \cup \{\Delta\}$ defined by 
\begin{equation} \label{eq6.24}
    P^{\mathcal{M}_1}(q,p) = \begin{cases}
    R \circ \Phi^{M(q,-p)}(q,-p) & \text{ if } M(q,-p) < \infty, \\
    \Delta & \text{otherwise},
    \end{cases}
\end{equation}
where $\Delta$ is a ``cemetery state.'' To understand this definition, consider a point particle moving freely in $\mathcal{M}_1$ and reflecting specularly from $\partial \mathcal{M}_1$. If the point particle starts in initial state $(q,-p) \in \mathcal{V}^-$ and after hitting $\partial \mathcal{M}_1$ some number of times returns to $\partial \mathcal{M}_0$ in state $(q',p') \in \mathcal{V}^+$, then by definition $P^{\mathcal{M}_1}(q,p) = (q',p')$. In (\ref{eq6.24}), the map $\Phi$ is defined for the billiard in $\mathcal{N}$ and not $\mathcal{M}_1$; thus the reason for post-composing with $R$ is that this ``undoes'' the reflection from $\partial \mathcal{M}_0$ of the point particle in $\mathcal{N}$.

For $q \in \partial \mathcal{M}_0$, let $k(q) = -n(q)$ denote the outward pointing normal vector field on $\partial \mathcal{M}_0$, let $\tau_{\partial \mathcal{M}_0}$ be surface measure on $\partial \mathcal{M}_0$, and define a measure on $\mathcal{V}^+$ by 
\begin{equation}
    \Lambda(\dd q \dd p) = \langle p, k(q) \rangle_q \sigma_q(\dd p) \tau_{\partial \mathcal{M}_0}(\dd q).
\end{equation}
Clearly the restriction of $\mu_1$ to $\mathcal{V}^-$ pushes forward to $\Lambda$ under negation $(q,p) \mapsto (q,-p)$. It is also easy to see that reflection $R$ pushes forward $\mu_1$ to $\Lambda$.

\begin{proposition} \label{prop_reflproperties}
(i) There exists a $\Lambda$-full measure, open subset $\mathcal{V}_1^+ \subset \mathcal{V}^+$ such that $P^{\mathcal{M}_1} : \mathcal{V}_1^+ \to \mathcal{V}_1^+$ is a $C^1$ diffeomorphism, and $P^{\mathcal{M}_1} \circ P^{\mathcal{M}_1} = \text{Id}_{\mathcal{V}_1^+}$.

(ii) $P^{\mathcal{M}_1}$ preserves the measure $\Lambda$.
\end{proposition}

\begin{proof}
For $m \geq 1$, let $V_m = \{(q,p) \in \mathcal{V}^+ : M(q,-p) = m\}$. For each $(q,p) \in V_m$, $M(q,p) = m < \infty$, and therefore by Lemma \ref{lem_iter} there is a neighborhood $N$ of $(q,p)$ such that $\Phi^m : N \to \Phi(N)$ is a diffeomorphism. Since $\mathcal{V}^-$ is an open subset of $\mathcal{U}^+$, $\Phi^{-m}(N \cap \mathcal{V}^-)$ is a neighborhood of $(q,p)$ contained in $-V_m$. Thus $V_m$ is open.

Let $\mathcal{V}_1^+ = \bigcup_{m \geq 1} V_m = \{(q,p) \in \mathcal{V}^+ : M(q,-p) < \infty\}$. By conditions D4 and D5, there exists a full measure set $\mathcal{F} \subset \mathcal{V}^-$ such that 
\begin{equation}
\mathcal{F} \cap \bigcap_{m = 1}^\infty \mathcal{U}^+_m \subset -\mathcal{V}_1^+.
\end{equation}
Since each of the sets in the intersection on the left has full $\mu_1$-measure in $\mathcal{V^-} \subset \mathcal{U}^+$ (see Lemma \ref{lem_iter}), it follows that $\mathcal{V}_1^+$ has full $\Lambda$-measure in $\mathcal{V}^+$.

To see that $P^{\mathcal{M}_1}$ is involutive, let $1 \leq n \leq m$. By Lemma \ref{lem_bilmap} and the fact that $R \circ R = \text{Id}$, we have
\begin{equation}
    -R \circ \Phi^n \circ -R \circ \Phi^m = (-R \circ \Phi \circ -R)^n \circ \Phi^m = \Phi^{m-n}.
\end{equation}
Thus, on $V_m$,
\begin{equation} \label{eq6.29'}
    R \circ \Phi^n \circ -P^{\mathcal{M}_1}(q,p) = -\Phi^{m-n}(q,-p),
\end{equation}
and the first $n$ such that $-\Phi^{m-n}(q,-p) \in \mathcal{V}^+$ is $n = m$. Thus $P^{\mathcal{M}_1}(q,p) \in V_m$, and taking $n = m$ in (\ref{eq6.29'}), we obtain $P^W \circ P^W(q,p) = (q,p)$.

Finally, since $P^{\mathcal{M}_1} = R \circ \Phi^m \circ -$ on $V_m$, and negation and $R^{-1}$ pushforward $\Lambda$ to $\mu_1$, and $\Phi$ preserves $\mu_1$ (see Lemma \ref{lem_bilinvar}), we see that $P^{\mathcal{M}_1}$ preserves $\Lambda$.
\end{proof}

\begin{remark} \label{rem_extend} \normalfont
The definition (\ref{eq6.24}) of $P^{\mathcal{M}_1}$ can be extended to cases where $\partial \mathcal{M}_1$ is allowed to intersect $\partial \mathcal{M}_0$. The only difficulty is that, if we are not careful, the space $\mathcal{N}$ defined by (\ref{eq6.27}) may no longer belong to the class $\CES_0^2(\mathcal{R})$. Problems can occur if the intersection of $\partial \mathcal{M}_1$ and $\partial \mathcal{M}_0$ is a ``fat Cantor set'' for example. Nonetheless, $\mathcal{N}$ will belong to the class $\CES_0^2(\mathcal{R})$ if appropriate assumptions are introduced.

Let $\mathcal{M} \supset \mathcal{M}_0$ and assume conditions D1 and D2 are satisfied. Then $\partial \mathcal{M}_0$ decomposes as the disjoint union $A_1 \cup A_2 \cup A_3$, where 
\begin{equation}
    A_1 = \partial \mathcal{M}_0 \cap \mathcal{N} \cap \Int \mathcal{M}, \quad\quad A_2 = \partial \mathcal{M}_0 \cap \mathcal{N} \smallsetminus \Int \mathcal{M}, \quad\quad A_3 = \partial \mathcal{M}_0 \smallsetminus \mathcal{N}.
\end{equation}
We replace condition D3 with 
\begin{enumerate}
    \item[D3'.] $\mathcal{M} \supset \mathcal{M}_0$, and the $(d-1)$-dimensional Hausdorff measure of $A_2$ is zero. 
\end{enumerate}
This of course reduces to condition D3 when $\Int \mathcal{M} \supset \mathcal{M}_0$ (because $A_2$ is the empty set). 

Under conditions D1, D2, and D3', $\mathcal{N}$ will belong to the class $\CES_0^2(\mathcal{R})$. To see this, replace the set $\mathcal{S}_1$ in the proof of Lemma \ref{lem_classy} with $\mathcal{S}_2 := \mathcal{S}_1 \cup A_2$. Then in analogy to (\ref{eq6.28}), we may write 
\begin{equation}
    \mathcal{N} \smallsetminus \mathcal{S}_2 = (\Int \mathcal{M} \smallsetminus \mathcal{M}_0) \cup (\partial_{\reg} \mathcal{M} \smallsetminus \mathcal{M}_0) \cup A_1.
\end{equation}
The rest of the argument proceeds as in the proof of the lemma, except that we replace $\partial_{\reg} \mathcal{M}$ with $\partial_{\reg} \mathcal{M} \smallsetminus \mathcal{M}_0$ and $\partial \mathcal{M}_0$ with $A_1$.

From the way that we have defined $\mathcal{N}$ in (\ref{eq6.27}), we see that $A_1 \subset \partial \mathcal{N}_{\reg} := \partial\mathcal{N} \smallsetminus \mathcal{S}_2$ and $A_2 \subset \partial \mathcal{N}$, but $A_3$ is not a subset of $\partial \mathcal{N}$. In fact, $A_3 \subset \partial_{\reg} \mathcal{M} \cap \partial \mathcal{M}_0$. We extend the definition of $P^{\mathcal{M}_1}$ so that the point particle reflects specularly from $\partial \mathcal{M}_0$ whenever it hits $A_3$. More precisely, we define $P^{\mathcal{M}_1} : \mathcal{V}^+ \to \mathcal{V}^+ \cup \{\Delta\}$ as follows:
\begin{equation}
    P^{\mathcal{M}_1}(q,p) = \begin{cases}
    R \circ \Phi^{M(q,-p)}(q,-p) & \text{ if } q \in A_1 \text{ and } M(q,-p) < \infty, \\
    -p + 2\langle p, k(q) \rangle_q k(q) & \text{ if } q \in A_3, \\
    \Delta & \text{ otherwise}.
    \end{cases}
\end{equation}
This reduces to (\ref{eq6.24}) when $A_1 = \partial \mathcal{M}_0$ (in which case $A_3 = \emptyset$). 

Let 
\begin{equation}
    \mathcal{V}_{A_i}^\pm = \{(q,p) \in \mathcal{V}^\pm : q \in A_i\}, \quad i = 1,2,3.
\end{equation}
Under conditions D1, D2, D3', D4, and D5, Proposition \ref{prop_reflproperties} still holds word for word. The proof gets modified as follows: We let $\mathcal{V}_1^+ = \mathcal{V}^+_{A_3} \cup \bigcup_{m = 1}^\infty V_m$, where now $V_m := \{(q,p) \in \mathcal{V}_{A_1}^+: M(q,-p) = m\}$. For essentially the same reason as before $\mathcal{V}_1^+$ is open, and $\bigcup_{m = 1}^\infty V_m$ is a $\Lambda$-full measure subset of $\mathcal{V}_{A_1}^+$, and by condition D3', $\mathcal{V}^+_{A_1} \cup \mathcal{V}^+_{A_3}$ is a full measure subset of $\mathcal{V}_1^+$. Therefore $\mathcal{V}^+$ is a full measure subset of $\mathcal{V}_1^+$. The rest of the proposition follows simply by noting that specular reflection from $\partial \mathcal{M}_0$ is involutive and preserves $\Lambda$.
\end{remark}

\subsubsection{Rough reflections} \label{sssec_genrfref}

From now on we assume that $\mathcal{M}_1$ satisfies conditions D1, D2, D3', D4, and D5, stated above.

The macro-reflection law $P^{\mathcal{M}_1}$ is associated with the deterministic Markov kernel on $\mathcal{V}_1^+$,  
\begin{equation}
    \mathbb{P}^\mathcal{M}(q,p; \dd q' \dd p') := \delta_{P^{\mathcal{M}_1}(q,p)}(\dd q' \dd p').
\end{equation}
We call a Markov kernel $\mathbb{P}(q,p; \dd q' \dd p')$ on $\mathcal{V}^+$ a \textit{rough collision law} if there exists a sequence $\mathcal{M}_i \supset \mathcal{M}_0$ (satisfying hypotheses 1' and 2-5) and a sequence of positive numbers $\epsilon_i \to 0$ such that
\begin{equation} \label{eq6.36}
    \partial \mathcal{M}_i \subset \overline{N}_{\epsilon_i}(\partial \mathcal{M}_0) := \{q \in \mathcal{R} : \text{dist}(q,\partial \mathcal{M}_0) \leq \epsilon_i\}
\end{equation} 
and such that the following limit holds:
\begin{equation} \label{eq6.29}
    \mathbb{P}^{\mathcal{M}_i}(q,p;\dd q' \dd p') \Lambda(\dd q \dd p) \to \mathbb{P}(q,p; \dd q' \dd p')\Lambda(\dd q \dd p),
\end{equation}
weakly in the space of measures on $\mathcal{V}^+ \times \mathcal{V}^+$. 

The limit (\ref{eq6.29}) means that for any function $h \in C_c(\mathcal{V}^+ \times \mathcal{V}^+)$, 
\begin{equation} \label{eq6.30}
\begin{split}
    &\lim_{i \to \infty} \int_{\mathcal{V}^+ \times \mathcal{V}^+} h(q,p;\dd q' \dd p') \mathbb{P}^{\mathcal{M}_i}(q,p;\dd q' \dd p') \Lambda(\dd q \dd p) \\
    & \quad\quad = \int_{\mathcal{V}^+ \times \mathcal{V}^+} h(q,p;\dd q' \dd p') \mathbb{P}(q,p;\dd q' \dd p') \Lambda(\dd q \dd p).
\end{split}
\end{equation}
\begin{remark} \normalfont
Since with respect to the uniform norm $C_c^\infty(\mathcal{V}^+ \times \mathcal{V}^+)$ is dense in $C_c(\mathcal{V}^+ \times \mathcal{V}^+)$, and the tensor product $C_c^\infty(\mathcal{V}^+) \otimes C_c^\infty(\mathcal{V}^+)$ is dense in $C_c^\infty(\mathcal{V}^+ \times \mathcal{V}^+)$, it is sufficient to verify (\ref{eq6.30}) for functions $h$ of form $f(q,p) g(q',p')$, where $f, g \in C_c^\infty(\mathcal{V}^+)$.
\end{remark}

\begin{proposition} \label{prop_gensymm}
A rough reflection law $\mathbb{P}(q,p;\dd q' \dd p')$ is symmetric with respect to the measure $\Lambda$, in the sense that, for any $h \in C_c(\mathcal{V}^+ \times \mathcal{V}_+)$,
\begin{equation} \label{eq6.33}
\begin{split}
&\int_{\mathcal{V}^+ \times \mathcal{V}^+} h(q,p, q',p') \mathbb{P}(q,p; \dd q' \dd p')\Lambda(\dd q \dd p) \\
& \quad\quad = \int_{\mathcal{V}^+ \times \mathcal{V}^+} h(q',p',q,p) \mathbb{P}(q,p; \dd q' \dd p')\Lambda(\dd q \dd p).
\end{split}
\end{equation}
\end{proposition}

\begin{proof}
Suppose $\mathbb{P}^{\mathcal{M}_i} \to \mathbb{P}$, as in (\ref{eq6.29}). We claim that an equality of form (\ref{eq6.33}) holds if we replace $\mathbb{P}$ with $\mathbb{P}^{\mathcal{M}_i}$. This is equivalent to showing that
\begin{equation}
    \int_{\mathcal{V}^+} h(P^{\mathcal{M}_i}(q,p),q,p)\Lambda(\dd q \dd p) = \int_{\mathcal{V}^+} h(q,p, P^{\mathcal{M}_i}(q,p))\Lambda(\dd q \dd p).
\end{equation}
But making the change of variables $(q',p') = P^{W_i}(q,p)$, the above equality is a consequence of the fact that $P^{\mathcal{M}_i}$ is an involution which preserves $\Lambda$. The result follows by taking the weak limit as $i \to \infty$.
\end{proof}

\begin{corollary}
A rough collision law $\mathbb{P}(q,p;\dd q' \dd p')$ preserves the measure $\Lambda$ in the sense that, for all $f \in C_c(\mathcal{V}^+)$,
\begin{equation}
    \int_{\mathcal{V}^+} f(q',p') \mathbb{P}(q,p;\dd q' \dd p') \Lambda(\dd q \dd p) = \int_{\mathcal{V}^+} f(q,p) \Lambda(\dd q \dd p).
\end{equation}
\end{corollary}
\begin{proof}
Let $h(q,p,q',p') \uparrow f(q,p)$ in (\ref{eq6.33}).
\end{proof}

\subsubsection{Pseudometric topology} \label{sssec_pseudo}

In (\ref{eq6.30}), there is a sense in which $\mathbb{P}^{\mathcal{M}_i}$ converges to $\mathbb{P}$ as a limit with respect to a pseudometric topology. To describe this topology, let $X = T^1\mathcal{R}$, let $\mathcal{X}$ denote the Borel $\sigma$-algebra on $X$, and let $\mathcal{G}$ denote the set of all Markov kernels on $(X,\mathcal{X})$, i.e. the set of all functions  $\mathbb{G} : X \times \mathcal{X} \to \mathbb{R}$ satisfying: (i) for each $x \in X$, $A \mapsto \mathbb{G}(x,A)$ is a Borel probability measure on $X$, and (ii) for each $A \in \mathcal{X}$, $x \mapsto \mathbb{G}(x,A)$ is a measurable function.
 
Equip $C_c(X \times X)$ with the uniform topology. Since $X \times X$ is locally compact and $\sigma$-compact, the space $C_c(X \times X)$ is separable. Let $\{f_i\}_{i \geq 1} \subset C_c(X \times X)$ be a countable dense subset, and let $\mathcal{N}$ denote the space of all signed Borel measures on $X \times X$ which are finite on compact sets. (Since $X \times X$ is a locally compact Hausdorff space, this also implies that the measures in $\mathcal{N}$ are Radon -- see [\ref{Folland}, Theorem 7.8].) Equip $\mathcal{N}$ with the metric 
\begin{equation}
     d(\mu_1, \mu_2) := \sum_{i = 1}^\infty 2^{-i} \frac{\left|\int f_i \dd(\mu_1 -\mu_2) \right|}{1 + \left|\int f_i \dd(\mu_1 -\mu_2) \right|}, \quad\quad \mu_1, \mu_2 \in \mathcal{N}.
\end{equation}
With respect to this metric, $\mu_k \to \mu$ in $\mathcal{N}$ if and only if for all $i \geq 1$, $\lim_{k \to \infty}\int f_i \dd\mu_k = \int f_i \dd\mu$ if and only if $\lim_{k \to \infty}\int f \dd\mu_k = \int f \dd\mu$ for all $f \in C_c(X \times X)$. In particular, $d$ restricted to the dual space $C_c(X \times X)^*$ induces the weak topology.
 
Given a Borel measure $\nu$ on $X$ which is finite on compact sets and an element $\mathbb{G} \in \mathcal{G}$, let $\mathbb{G}\nu$ denote the Borel measure on $X \times X$ defined by $\int f \dd(\mathbb{G}\nu) = \int (\int f(x,y) \mathbb{G}(y,\dd x))\nu(\dd y)$, $f \in C_c(X \times X)$. We observe that $\mathbb{G}\nu \in \mathcal{N}$; in particular if $K$ is any compact subset of $X$, then since $\mathbb{G}(x,\cdot)$ is a probability measure
\begin{equation}
\mathbb{G}\nu(K \times X) = \nu(K) < \infty.
\end{equation}
The measure $\nu$ together with the metric $d$ induce a pseudometric on $\mathcal{G}$ defined by 
\begin{equation}
    d_{\mathcal{G}}^{\nu}(\mathbb{G}_1,\mathbb{G}_2) = d(\mathbb{G}_1\nu,\mathbb{G}_2\nu).
\end{equation}
With respect to the induced pseudometric topology, $\mathbb{G}_j \to \mathbb{G}$ if and only if $\mathbb{G}_j \nu \to \mathbb{G} \nu$ in duality to $C_c(X \times X)$. Hence taking $\nu = \Lambda$ yields the desired sense of convergence. 

\begin{sloppypar}
Note that $d_{\mathcal{G}}^{\nu}$ fails to satisfy the non-degeneracy condition of a metric. If $d_{\mathcal{G}}^{\nu}(\mathbb{G}_1, \mathbb{G}_2) = 0$, the Markov kernels $\mathbb{G}_1(\cdot, \dd x)$ and $\mathbb{G}_2(\cdot, \dd x)$ are still allowed to disagree on a $\nu$-null set.
\end{sloppypar}

\subsubsection{Some special cases} \label{sssec_special}

Here we discuss how conditions D1, D2, D3 (or D3'), D4, and D5 are verified for particular cases. In all of our examples, the ambient space is $\mathcal{R} = \mathbb{R}^d$ or $\mathbb{T}^{d-1} \times \mathbb{R}$ with the Euclidean metric.  

\textit{Example 1: Bounded, convex bodies.} Let $B_0$ be a closed, bounded, strictly convex subset of $\mathbb{R}^d$, and let $B$ be some closed, connected subset of $B_0$ such that $B \subset \Int B_0$. Assume that both $B_0$ and $B$ have smooth or piecewise smooth boundary; the essential conditions are that $\mathcal{M}_0 := \mathbb{R}^d \smallsetminus \Int B_0$ should be an embedded $C^2$ submanifold of $\mathbb{R}^d$ with boundary, and $\mathcal{M}_1 := \mathbb{R}^d \smallsetminus \Int B$ should belong to the class $\CES_0^2(\mathbb{R}^d)$. Then conditions D1-D3 of \S\ref{sssec_detref} are clearly satisfied.

Condition D4 holds because a linear trajectory of a point particle must eventually leave the bounded region $B_0$. Condition D5 holds because, by boundedness, the invariant measure on $\mathcal{U}^-$ (the set of pairs $(y,w) \in \partial B_0 \times \mathbb{S}^2$ such that $w$ points out of $B_0$) is finite. Thus we may apply the Poincar\'{e} Recurrence Theorem to conclude that the point particle returns almost surely.

Hence one can define rough reflection laws on $B_0$ by taking weak limits, as in \S\ref{sssec_genrfref}. Plakhov's work on scattering laws, reviewed in \S\ref{sssec_Plakhov}, falls within this setting.

\textit{Example 2: Half-space billiards.} Consider the following subset of $\mathbb{R}^d$:
\begin{equation}
    W_0 = \{(x_1,...,x_d) : x_d \leq 0\}.
\end{equation}
Let $W$ be a subset of $W_0$ such that 
\begin{equation}
    \{(x_1,...,x_d) : x_d \leq -1\} \subset W \subset \{(x_1,...,x_d) : x_d < 0\}.
\end{equation}
Assume that $W$ has a piecewise smooth boundary; at minimum $\mathcal{M}_1 := \mathbb{R}^d \smallsetminus W$ should belong to the class $\CES^2_0(\mathbb{R}^d)$. Let $\mathcal{M}_0 = \{(x_1,...,x_d) : x_d \geq 0\}$. Then conditions D1-D3 are satisfied.

To verify condition D4, note that the only way for a billiard trajectory in $\mathcal{N} = \overline{\mathcal{M}_1 \smallsetminus \mathcal{M}_0}$ to escape to infinity without returning to $\mathcal{M}_0$ or hitting $\partial \mathcal{M}_1$ is if the trajectory is parallel to plane $x_d = 0$. The set of states $(y,w)$ such that $w$ is parallel to this plane is a $\Lambda$-null.

For condition D5 to hold, additional assumptions must be imposed. One condition which implies D5 is the following:
\begin{enumerate}
    \item[D5a.] There exist linearly independent vectors $v_1, \dots, v_{d-1}$ spanning the plane $x_d = 0$ such that $W$ is invariant under translation by the vectors $v_i$, i.e. $W + v_i = W$ for $1 \leq i \leq d-1$.
\end{enumerate}
Indeed, in this case we may reduce to the quotient space obtained by identifying points which are translates of each other by integer combinations of the $v_i$. We can thereby view the billiard domain as a subset of the ambient space $\mathbb{T}^{d-1} \times \mathbb{R}$. The space $\mathcal{N}$ is a compact subset of $\mathbb{T}^{d-1} \times \mathbb{R}$, since it is bounded in the direction of the $x_d$-axis. Consequently, the invariant measure is finite, and as in the previous case we can use Poincar\'{e}'s Theorem to verify condition D5.

The reflections laws considered in \S\ref{sssec_uphalf} are rough reflections in an upper half-space billiard of dimension $d = 2$. The rough collision laws defined in this book are examples of rough reflection laws in an upper half-space of dimension $d = 3$.

\textit{Extensions.} We can generalize both of the above examples to allow $\partial \mathcal{M}_1$ to intersect $\partial \mathcal{M}_0$ (thus $B$ can intersect $\partial B_0$ in Example 1, and $W$ can intersect the plane $x_d = 0$ in Example 2). However, some assumption must be imposed to ensure that condition D3' holds. One simple condition is the following: \textit{The $(d-1)$-dimensional Hausdorff measure of $\mathcal{M}_1 \cap \partial \mathcal{M}_0$ is zero.} Since $A_2 \subset \mathcal{M}_1 \cap \partial \mathcal{M}_0$, this implies condition D3'.

In this book, both the configuration space $\mathcal{M}$ and its cylindrical approximation $\mathcal{M}_{\cyl}$ satisfy condition D3' (with $\mathcal{M}_0 = \{(x_1,x_2,\alpha) : x_2 \geq 0\}$). In the former case, it is because $\mathcal{M}$ satisfies the condition stated in the previous paragraph. In the latter case, the 2-dimensional Hausdorff measure of $A_2$ is zero because it is a countable union of lines. (See Propositions \ref{prop_detcol} and \ref{prop_detcolcyl} respectively.)

In the upper half-space case, if we allow $W$ to touch the plane $x_d = 0$, then there is another condition we can impose which implies condition D5.
\begin{enumerate}
    \item[D5b.] There exists a countable collection $\{R_j\}_{j \geq 1}$ of disjoint bounded open subsets of $\mathbb{R}^d$ such that $\Int W_0 \smallsetminus W = \bigcup_{j = 1}^\infty R_j$. 
\end{enumerate}
If this condition holds, then the point particle will get trapped in one of the regions $R_j$ when it enters $\mathcal{N}$. By boundedness, the restriction of the invariant measure to any one of the regions $R_j$ is finite, and we can once again apply the Poincar\'{e} Recurrence Theorem.

\section{Index of Notation}

{
\setlength\parskip{1em plus 0.1em minus 0.2em}
\setlength\parindent{0pt}

$\langle \cdot, \cdot \rangle$ -- Kinetic energy inner product, \S\ref{sssec_phasespace}

$||\cdot||$ -- Norm induced by the kinetic energy inner product, \S\ref{sssec_phasespace}

$||\cdot||_{L^p} = ||\cdot||_{L^p_\mu(X)}$ -- $L^p$-norm on $L^p_{\mu}(X)$, \S\ref{sssec_noteconv}

$Y \subset\subset X$ -- The closure of $Y$ lies in $X$, \S\ref{sssec_bilmap}

$\mathcal{A}_0$ -- A special subclass of rough collision laws, \S\ref{sssec_main}

$\widehat{B} = \overline{W^c} + e_2$ -- Base of the cylinder $\mathcal{M}_{\cyl}$, \S\ref{ssec_cylconfig}

$C_c(X)$ -- Continuous, compactly supported functions, \S\ref{sssec_noteconv}

$C^k(X)$ -- $k$-times continuously differentiable functions, \S\ref{sssec_noteconv}

$C^\infty(X)$ -- Infinitely differentiable functions, \S\ref{sssec_noteconv}

$C_c^k(X) = C_c(X) \cap C^k(X)$, \S\ref{sssec_noteconv}

$C_c^\infty(X) = C_c(X) \cap C^\infty(X)$, \S\ref{sssec_noteconv}

$\CES_0^2(\mathcal{R})$ -- A special class of billiard domains with singularities, \S\ref{sss_sing}

$\widetilde{\chi} = (-1,0,1)$ -- Rolling velocity, \S\ref{sssec_cylapprox}

$\chi = (m + J)^{-1/2}(-1,0,1)$ -- Normalized rolling velocity, \S\ref{sssec_tilted}

$d_\mathcal{G}^\nu$ -- A pseudometric on a space of Markov kernels $\mathcal{G}$, \S\ref{sssec_pseudo}

$D$ -- Disk with satellites, \S\ref{sssec_rigid}

$D_0$ -- Inner body of the disk $D$, \S\ref{sssec_rigid}

$D(y)$ -- Subset of $\mathbb{R}^2$ occupied by $D$ in configuration $y$, \S\ref{sssec_rigid}

$\partial$ -- Topological boundary operator, \S\ref{sssec_noteconv}

$\partial_{\reg} \mathcal{M}$ -- Set of regular points of $\partial \mathcal{M}$, \S\ref{sssec_config}, \S\ref{ssec_configprops}

$\partial_{\reg} \Sigma$ -- Set of regular points of $\partial \Sigma$, \S\ref{sssec_rigid}

$\partial_s \Sigma$ -- Set of singular points $\partial \Sigma$, \S\ref{sssec_rigid}

$\partial_{\reg} W$ -- Set of regular points of $\partial W$, \S\ref{sssec_rigid}

$\partial_s W$ -- Set of singular points of $\partial W$, \S\ref{sssec_rigid}

$\delta_0 = \delta_0(\epsilon)$ -- Radius of a ``gap region'' \S\ref{ssec_configprops}

$e_j$ -- $j$'th member of the standard basis for $\mathbb{R}^3$, \S\ref{sssec_config}

$\widehat{e}_j = e_j/||e_j||$, \S\ref{sssec_tilted}

$\eta$ -- ``Correction'' diffeomorphism, \S\ref{sssec_modcol}

$\mathcal{F}$ -- Full-measure open set on which $K^{\Sigma,\epsilon}$ is defined, \S\ref{sssec_collaws}, \S\ref{sssec_coldef}

$\mathcal{F}_{\cyl}$ -- Full-measure open set on which $K^{\Sigma,\epsilon}_{\cyl}$ is defined, \S\ref{sssec_cylcol}

$\widetilde{\mathcal{F}}$ -- Open set on which $\widetilde{K}^{\Sigma,\epsilon}$ is defined, \S\ref{sssec_modcol}

$G_\epsilon$ -- ``Perturbed'' projection onto $\mathbf{Q}_0$, \S\ref{sssec_cylproject}

$\Gamma_{\roll} = \partial \mathcal{M} \cap \widehat{\mathcal{Z}}$, \S\ref{ssec_configprops}

$H_1$ -- Smooth perturbation of configuration space, \S\ref{ssec_cylconfig}

$\overline{H}_1$ -- Smooth perturbation of phase space, \S\ref{sssec_modcol}

$H_\epsilon$ -- Zoomed version of $H_1$, \S\ref{sssec_zoomnote}

$\overline{H}_\epsilon$ -- Zoomed version of $\overline{H}_1$, (\ref{eq5.212})

$\mathcal{H}^d$ -- $d$-dimensional Hausdorff measure, \S\ref{ssec_configprops}, \S\ref{sss_sing}

$\Int$ -- Topological interior operator, \S\ref{sssec_noteconv}

$J$ -- Moment of inertia of the disk $D$, \S\ref{sssec_rigid}

$K^{\Sigma,\epsilon}$ -- Collision law associated with the wall $W(\Sigma,\epsilon)$, \S\ref{sssec_collaws}, \S\ref{sssec_coldef}

$\mathbb{K}^{\Sigma,\epsilon}$ -- Markov kernel representation of $K^{\Sigma,\epsilon}$, \S\ref{sssec_collaws}, \S\ref{sssec_coldef}

$\mathbb{K}$ -- Rough reflection law, \S\ref{sssec_collaws},\S\ref{sssec_coldef}

$K^{\Sigma,\epsilon}_{\cyl}$ -- Cylindrical collision law associated with $W(\Sigma,\epsilon)$, \S\ref{sssec_cylcol}

$\mathbb{K}^{\Sigma,\epsilon}_{\cyl}$ -- Markov kernel representation of $K^{\Sigma,\epsilon}_{\cyl}$, \S\ref{sssec_cylcol}

$\widetilde{K}^{\Sigma,\epsilon}$ -- Modified collision law associated with $W(\Sigma,\epsilon)$, \S\ref{sssec_modcol}

$\kappa(p)$ -- Unsigned curvature of $\partial W$ at point $p$, \S\ref{sssec_rigid}

$\kappa_{\max} = \sup\{\kappa(p) : p \in \partial_{\reg} W\}$, \S\ref{sssec_rigid}

$\overline{\kappa}$ -- Max curvature with lower cutoff, (\ref{eq5.101})

$\overline{L}$ -- Max chord length with lower cutoff, (\ref{eq5.101})

$\Lip(f)$ -- Lipschitz constant of the function $f$, (\ref{eq5.61})

$L^p = L_{\mu}^p(X)$ -- Set of $p$-integrable real functions on measure space $(X,\mu)$, \S\ref{sssec_noteconv}

$\Lambda^1$ -- Billiard invariant measure on $\mathbb{R} \times \mathbb{S}^1_+$, \S\ref{sssec_uphalf}

$\Lambda^2$ -- Billiard invariant measure on $\mathbf{P} \times \mathbb{S}^2_+$, \S\ref{sssec_collaws}

$m$ -- Mass of disk $D$, \S\ref{sssec_rigid}

$\mathcal{M}$ -- Configuration space of the disk and wall system, \S\ref{sssec_config}, \S\ref{ssec_configprops}

$\mathcal{M}_{\cyl}$ -- Cylindrical configuration space, \S\ref{sssec_cylapprox} \S\ref{ssec_cylconfig}

$\mathcal{M}_{\reg} = \mathcal{M} \smallsetminus \mathcal{S}$, \S\ref{sssec_config}, \S\ref{ssec_configprops}

$\mathcal{M}_{\roll} = \mathcal{M} \cap \widehat{\mathcal{Z}}$, \S\ref{ssec_configprops}

$N$ -- Number of satellites of disk $D$, \S\ref{sssec_rigid}

$\mathcal{O}_{\pm} = H_1(\mathbb{R}^3_{\pm} \cap \widehat{\mathcal{Z}})$, \S\ref{sssec_modcol}

$\Omega = \Omega(\epsilon)$ -- \S\ref{ssec_pure_scaling}, \S\ref{sssec_omegadef}
    
$P^W$ -- Macro-reflection law associated with wall $W$, \S\ref{sssec_uphalf}

$P^{\Sigma,\epsilon} = P^{W(\Sigma,\epsilon)}$, \S\ref{sssec_periodic}

$\mathbb{P}^{\Sigma,\epsilon} = \mathbb{P}^{W(\Sigma,\epsilon)}$, \S\ref{sssec_periodic}

$\mathbb{P}^W$ -- Markov kernel representation of $P^W$, \S\ref{sssec_uphalf}

$\mathbb{P}$ -- Rough reflection law, \S\ref{sssec_uphalf}, \S\ref{sssec_genrfref}

$\widetilde{\mathbb{P}}$ -- Velocity component of rough reflection law, \S\ref{sssec_periodic}

$\widetilde{\mathbb{P}}^{\Sigma,\epsilon}$ -- Macro-reflection law averaged over one period, \S\ref{ssec_lemconstruct}

$\mathbf{P} = \{(x_1,x_2,\alpha) : x_2 = 0\}$, \S\ref{sssec_config}

$\widetilde{\mathbf{P}} = H_1(\mathbf{P} \cap \widehat{\mathcal{Z}})$, \S\ref{sssec_modcol}

$\mathbf{Q}_0 = \{(x_1,x_2,\alpha) : \alpha - 0\}$, \S\ref{sssec_cylapprox}

$R, R_q$ -- Specular reflection maps,  \S\ref{sssec_phasespace}, \S\ref{sssec_bilmap}

$R_\epsilon$ -- Narrow parallelogram, \S\ref{ssec_pure_scaling}

$\mathbb{R}^2_{\pm} = \{(x_1,x_2) \in \mathbb{R}^2 : \pm x_2 > 0\}$, \S\ref{sssec_uphalf}

$\mathbb{R}^3_{\pm} = \{(x_1,x_2,\alpha) \in \mathbb{R}^3 : \pm x_2 > 0\}$, \S\ref{sssec_collaws}

$\rho = \rho(\epsilon)$ -- Angle between consecutive satellites in $D$, \S\ref{sssec_rigid}

$\widecheck{\rho}^{-1}$ -- Generalized inverse of $\rho$, (\ref{eq5.102})

$S_k$ -- $k$'th satellite in $D$, \S\ref{sssec_rigid}

$S_k(y)$ -- Point in $\mathbb{R}^2$ occupied by $S_k$ in $D(y)$, \S\ref{ssec_configprops}

$\mathbb{S}^1$ -- Unit circle in $\mathbb{R}^2$, \S\ref{sssec_uphalf}

$\mathbb{S}^1_{\pm} = \mathbb{S}^1 \cap \mathbb{R}^2_{\pm}$, \S\ref{sssec_uphalf}

$\mathbb{S}^2$ -- Unit sphere in $\mathbb{R}^3$ with respect to the norm $||\cdot||$, \S\ref{sssec_collaws}

$\mathbb{S}^2_{\pm} = \mathbb{S}^2 \cap \mathbb{R}^3_{\pm}$, \S\ref{sssec_collaws}

$\mathcal{S}$ -- Set of singular points of $\mathcal{M}$, \S\ref{ssec_configprops}

$\sigma_r$ -- Scaling map, \S\ref{sssec_zoomnote}

$\Sigma$ -- Cell for building a periodic wall, \S\ref{sssec_periodic}

$\widetilde{\Sigma}$ -- Foreshortened cell, \S\ref{sssec_main}

$\tau_{jk}, \overline{\tau}_{jk}$ -- Translation maps, \S\ref{sssec_coldef}

$\tau_{jk}^{(s)}, \overline{\tau}_{jk}^{(s)}$ -- Translation maps, \S\ref{sssec_cylcol}

$W$ -- Wall of billiard table in $\mathbb{R}^2$, \S\ref{sssec_uphalf}

$W(\Sigma,\epsilon)$ -- Periodic wall determined by $\Sigma$ and $\epsilon$,
    \S\ref{sssec_periodic}, \S\ref{sssec_rigid}

$\widetilde{W} = W(\widetilde{\Sigma},\widetilde{\epsilon})$ -- Foreshortened wall, \S\ref{sssec_main}

$\mathcal{Z}$ -- Set of angular coordinates where cylindrical approximation is feasible, \S\ref{ssec_configprops}

$\widehat{\mathcal{Z}} = \mathbb{R}^2 \times \mathcal{Z}$, \S\ref{ssec_configprops}

$\mathcal{Z}^0$ -- Component of $\mathcal{Z}$ containing $\alpha = 0$, (\ref{eq4.34})

$\widehat{\mathcal{Z}}^0 = \mathbb{R}^2 \times \mathcal{Z}^0$, (\ref{eq4.34})

}

\vfill

\section{References}

\begin{enumerate}
    \item O. Angel, K. Burdzy, S. Sheffield, \textit{Deterministic approximations of random reflectors}, Trans. Amer. Math. Soc. \textbf{365}(2013), no. 12, 6367–83. MR3105755\label{ABS_refl}
    
    \item V. I. Arnold, \textit{Mathematical Methods of Classical Mechanics}, 2nd ed., Springer Science+Business, Inc., New York, 1989. MR1345386 (96c:70001) \label{Arnold}
    
    \item P. Ballard, \textit{The dynamics of discrete mechanical systems with perfect unilateral constraints}, Arch. Ration. Mech. Anal. \textbf{154}(2000), no. 3, 199-274. MR1785473 (2002a:70011) \label{Ballard}
    
    \item D. S. Broomhead, E. Gutkin, \textit{The dynamics of billiards with no-slip collisions}, Phys. D \textbf{67}(1993), no. 1-3, 188-197. MR1234441 (94g:58162) \label{GutkinBroomhead}
    
    \item A. Champneys and P. V\'{a}rkonyi, \textit{The Painlev\'{e} paradox in contact mechanics}, IMA J. Appl. Math. \textbf{81}(2016), no. 3, 538-588. MR3564666\label{PainleveParadox_review}
    
    \item N. Chernov and R. Markarian, \textit{Chaotic Billiards}, Mathematical Surveys and Monographs, 127. American Mathematical Society, Providence, RI, 2006. MR2229799 (2007f:37050) \label{Chernov&Markarian} 
    
    \item S. Cook and R. Feres, \textit{Random billiards with wall temperature and associated Markov chains}, Nonlinearity \textbf{25}(2012), no. 9, 2503-2541. MR2967115 \label{CookFeres_temp}
    
    \item I. P. Cornfield, S. V. Fomin, and Ya. G. Sinai, \textit{Ergodic Theory}, tr. A. B. Sossinski, Springer-Verlag, Inc., New York, 1982. MR0832433 (87f:28019) \label{CFS_ergodic}
    
    \item \begin{sloppypar}
    C. Cox, R. Feres, \textit{No-slip billiards in dimension two}, Contemp. Math. \textbf{698}(2017), 91-110. MR3716087
    \end{sloppypar}\label{NoSlip}
    
    \item C. Cox, R. Feres and W. Ward, \textit{Differential geometry of rigid body collisions and non-standard billiards}, Discrete Contin. Dyn. Syst. \textbf{36}(2016), no. 11, 6065-6099. MR3543581 \label{diffgeo_rbcollisions}
    
    \item E. D. Demaine and J. O'Rourke, \textit{Geometric Folding Algorithms: Linkages, Origami, Polyhedra}, Cambridge University Press, New York, 2007. MR2354878 (2008g:52001) \label{DemaineORourke}
    
    \item R. Feres and G. Yablonsky, \textit{Knudsen's cosine law and random billiards}, Chemical Engineering Science, \textbf{59}(2004), 1541-1556. \label{FeresYab}

    \item R. Feres, Random walks derived from billiards. \textit{Dynamics, ergodic theory, and geometry}, 179-222, \textit{Math. Sci. Res. Inst. Publ.} \textbf{54}, Cambridge University Press, Cambridge, 2007. MR2369447 (2009c:37033) \label{feres_rw}
    
    \item R. Feres and H-K. Zhang, \textit{The spectrum of the billiard Laplacian of a family of random billiards}, J. Stat. Phys. \textbf{141}(2010), no. 6, 1030-1054. MR2740402 (2011h:37052) \label{FeresZhang_spectrum1}
    
    \item R. Feres and H-K. Zhang, \textit{Spectral gap for a class of random billiards}, Commun. Math. Phys. \textbf{313}(2012), no. 2, 479-515. MR2942958 \label{FeresZhang_spectrum2}
    
    \item R. C. Fetecau, J. E. Marsden, M. Ortiz, and  M. West, \textit{Nonsmooth Lagrangian mechanics and variational collision integrators}, SIAM J. Appl. Dyn. Syst. \textbf{2}(2003), no. 3, 381-416. MR2031279 (2005e:37197) \label{FMOW}

    \item G. Folland, \textit{Real Analysis: Modern Techniques and Their Applications}, 2nd Ed., John Wiley \& Sons, 1999. MR1681462 (2000c:00001) \label{Folland}
    
    \item C. Hahlweg, B. Mei{\ss}ner, W. Zhao, and H. Rothe, \textit{The idea of the Lambertian surface: history, idealization, and system theoretical aspects and part 1 of a lost chapter on multiple reflection}, Proc. SPIE 7792, Reflection, Scattering, and Diffraction from Surfaces II, 779202 (2010).\label{Lambert_history}
    
    \item M. Knudsen, \textit{Kinetic Theory of Gases: Some Modern Aspects}, Methuen’s Monographs on Physical Subjects, London, 1952. \label{Knudsen}
    
    \item M. Lapidus and R. Niemeyer, \textit{Towards the Koch snowflake fractal billiard: computer experiments and mathematical conjectures}, Gems in experimental mathematics, 231–263, Contemp. Math., \textbf{517}, Amer. Math. Soc., Providence, RI, 2010. MR2731085 (2012b:37101) \label{LapNiem}
    
    \item M. Mabrouk, \textit{A unified variational model for the dynamics of perfect unilateral constraints}, Eur. J. Mech. A Solids, \textbf{17}(1998), no. 5, 819–842. MR1650957 (99j:70019) \label{Mab1}
    
    \item  D. P. Monteiro Marques, \textit{Chocs in\'{e}lastiques standards: un r\`{e}sultat d'existence}, S\`{e}m. Anal. Convexe, \textbf{15}(1985), no. 4, 1-32. MR0857776 (88c:70010) \label{Marq1}
    
    \item D. P. Monteiro Marques, \textit{Differential inclusions in nonsmooth mechanical problems: shocks and dry friction}, Progr. Nonlinear Differential Equations Appl. \textbf{9}(1993), Birkh\"{a}user Verlag, Basel. MR1231975 (94g:34003)\label{Marq2}
    
    \item J.-J. Moreau, \textit{Standard inelastic shocks and dynamics of unilateral constraints}, in Unilateral Problems in Structural Analysis, G. Del Piero and F. Maceri, eds., 1985. \label{Mor1}
    
    \item J.-J. Moreau, \textit{Unilateral contact and dry friction in finite freedom dynamics}, in Nonsmooth Mechanics and Applications, J.-J. Moreau and P. G. Panagiotopoulos, eds., CISM Courses and Lectures 302, Springer-Verlag, Vienna, 1988. \label{Mor3}
    
    \item P. Palffy-Muhoray, E. G. Virga, M. Wilkinson and X. Zheng, \textit{On a paradox in the impact dynamics of smooth rigid bodies}, Math. Mech. Solids, \textbf{24}(2019), no. 3, 573-597. MR3935008 \label{impactparadox}
    
    \item L. Paoli and M. Schatzman, \textit{Sch\'{e}ma num\'{e}rique pour un mod\`{e}le de vibrations avec constraints unilat\'{e}rale et perte d’ \'{e}nergie aux impacts, en dimension finie}, C. R. Acad. Sci. Paris S\'{e}r. I Math. \textbf{317}(1993), no. 2, 211–215. MR1231424 (94f:70022) \label{PaoSch1}
    
    \item L. Paoli and M. Schatzman, \textit{Mouvement \`{a} un nombre fini de degr\'{e}s de libert\'{e} avec contraintes unilat\'{e}rales: Cas avec perte d'\'{e}nergie}, RAIRO Mod\'{e}l. Math. Anal. Num\'{e}r. \textbf{27}(1993), no. 6, 673–717. MR1246995 (94m:34038) \label{PaoSch2}
    
     \item A. Yu. Plakhov, \textit{Newton's problem of the body of minimal mean resistance}, Sb. Math. \textbf{195}(2004), no. 7-8, 1017-1037. MR2101335 (2005g:49073) \label{Plakhov1}
    
    \item A. Yu. Plakhov, \textit{Billiards and two-dimensional problems of optimal resistance}, Arch. Ration. Mech. \textbf{194}(2009), no. 2, 349-382. MR2563633 (2010m:49067) \label{Plakhov2}
    
    \item A. Yu. Plakhov, \textit{Billiard scattering on rough sets: two-dimensional case}, SIAM J. Math. Anal. \textbf{40}(2009), no. 6, 2155-2178. MR2481290 (2010i:37081) \label{Plakhov3}
    
    \item A. Yu. Plakhov, \textit{Scattering in billiards and problems of Newtonian aerodynamics}, Russ. Math. Surv. \textbf{64}(2009), no. 5, 873-938. MR2588685 (2011b:37067)\label{Plakhov4}
    
    \item A. Yu. Plakhov, \textit{Exterior Billiards: Systems with Impacts Outside Bounded Domains}, Springer, 2013. MR2931647 \label{Plakhov5}
    
    \item L. Saint-Raymond and M. Wilkinson, \textit{On Collision Invariants for Linear Scattering}, Comm. Pure Appl. Math., \textbf{71}(2018), no. 8, 1494-1534. MR3847748\label{SrW_colinv}
    
    \item S. Sasaki, \textit{On the differential geometry of tangent bundles of Riemannian manifolds II}, T\^{o}hoku Math. J., \textbf{14}(1962), no. 2, 146-155. MR0145456 (26 \#2987)\label{Sasaki}
    
    \item J.-M. Strelcyn, \textit{Plane billiards as smooth dynamical systems with singularities}, in Invariant Manifolds, Entropy and Billiards: Smooth Maps with Singularities, Lecture Notes in Mathematics, 1222, A. Dold and B. Eckmann, eds. Springer-Verlag, Berlin, 1986. MR0872698 (88k:58075) \label{Strelcyn}
    
    \item S. Tabachnikov, \textit{Geometry and Billiards}, Student Mathematical Library, 30. American Mathematical Society, Providence, RI; Mathematics Advanced Study Semesters, University Park, PA, 2005. MR2168892 (2006h:51001) \label{Tabachnikov}
    
    \item M. Wilkinson, \textit{On the non-uniqueness of physical scattering for hard non-spherical particles}, Arch. Rational Mech. Anal. \textbf{235}(2020), no. 3, 2055-2083. MR4065657\label{Wilk_nonunique}
    
    \item M. Wilkinson, \textit{On the initial boundary value problem in the kinetic theory of hard particles I: Non-existence}, 2018. arXiv:1805.04611
    
    \item Wikipedia, \textit{Retroreflector}, \url{https://en.wikipedia.org/wiki/Retroreflector} Online; accessed 5 Nov 2021. \label{retroref_wiki}
\end{enumerate}

\end{document}